\documentclass[11pt]{amsart}
\usepackage{amssymb,amsthm,amsmath}
\usepackage{hyperref}
\usepackage{latexsym}
\usepackage{tikz-cd}
\usepackage{array}
\usepackage{longtable}
\usetikzlibrary{arrows,fit,positioning,calc}
\usepackage[utf8]{inputenc}
\usepackage[margin=1in]{geometry}
\usepackage[shortlabels]{enumitem}
\usepackage{xcolor}
\usepackage{cleveref}

\let\oldtocsection=\tocsection
\let\oldtocsubsection=\tocsubsection

\renewcommand{\tocsection}[2]{\hspace{0em}\oldtocsection{#1}{#2}}
\renewcommand{\tocsubsection}[2]{\hspace{2em}\oldtocsubsection{#1}{#2}}

\input{cluster_picture.sty}
\input{math_commands.sty}

\newtheorem{thm}{Theorem}[section]
\newtheorem{prop}[thm]{Proposition}

\newtheorem{lemma}[thm]{Lemma}
\newtheorem{cor}[thm]{Corollary}
\newtheorem{dfn}[thm]{Definition}

\theoremstyle{definition} \newtheorem{ex}[thm]{Example}

\newtheorem{hyp}[thm]{Hypothesis}
\theoremstyle{definition} \newtheorem{rmk}[thm]{Remark}
\newtheorem{algo}[thm]{Algorithm}

\title{Clusters and semistable models of hyperelliptic curves in the wild case}
\author{Leonardo Fiore}
\address{\parbox{\linewidth}{Department of Mathematics ``Federigo Enriques", The University of Milan \newline Via Cesare Saldini, 50, 20133 Milano MI, Italy}}
\email{leonardo@leonardofiore.it}

\author{Jeffrey Yelton}
\address{\parbox{\linewidth}{Department of Mathematics and Computer Science, Wesleyan University \\ 265 Church Street, Middletown, CT 06459-0128}}
\email{jyelton@wesleyan.edu}

\begin{document}

\maketitle

\begin{abstract}
    Given a Galois cover $Y \to X$ of smooth projective geometrically connected curves over a complete discrete valuation field $K$ with algebraically closed residue field, we define a semistable model of $Y$ over the ring of integers of a finite extension of $K$ which we call the \emph{relatively stable model} $\Yrst$ of $Y$, and we discuss its properties.  We focus on the case when $Y : y^2 = f(x)$ is a hyperelliptic curve, viewed as a degree-$2$ cover of the projective line $X := \proj_K^1$, and demonstrate a practical way to compute the relatively stable model.
    
    In the case of residue characteristic $p \neq 2$, the components of the special fiber $\SF{\Yrst}$ correspond precisely to the non-singleton \emph{clusters} of roots of the defining polynomial $f$, i.e.\ the subsets of roots of $f$ which are closer to each other than to the other roots of $f$ with respect to the induced discrete valuation on the splitting field.  This relationship, however, is far less straightforward in the $p=2$ case, which is our main focus (the techniques we introduce nevertheless also allow us to recover the simpler, already-known results in the $p\neq 2$ case).  We show that, when $p = 2$, for each cluster containing an \textit{even} number of roots of $f$, there are $0$, $1$, or $2$ components of $\SF{\Yrst}$ corresponding to it, and we determine a direct method of finding and describing them. We also define a polynomial $F(T) \in K[T]$ whose roots allow us to find the components of $\SF{\Yrst}$ which are not connected to even-cardinality clusters.  We finish by using our methods to find relatively stable models of hyperelliptic curves in genus $1$ and $2$, using linear inequalities among valuations of various elements of $\bar{K}$ associated to $f$ to break the situation for each genus into several cases which yield different results about the structure of the special fiber of the relatively stable model.
\end{abstract}

\tableofcontents

\section{Introduction} \label{sec introduction}

The focus of this paper is to investigate the reduction types of hyperelliptic curves over discrete valuation fields.  Given a complete discrete valuation field $K$ of characteristic different from $2$ with algebraically closed residue field, our starting point is to consider a \textit{hyperelliptic curve} $Y$ over $K$; that is, $Y / K$ is a smooth projective curve of positive genus admitting a degree-$2$ morphism onto the projective line $\proj_K^1$.

This paper is concerned with constructing a semistable model of a given hyperelliptic curve $Y / K$ and understanding the structure of the special fiber of a semistable model of $Y$.  As this problem is already entirely understood in the case that the residue characteristic is not $2$ and the procedure in that case can be described entirely in terms of the distances between the branch points with respect to the $p$-adic metric on $K$, our primary focus will be on the case where the residue characteristic is $2$.  The increased complexity of the problem for this case arises from the fact that a hyperelliptic curve comes with a degree-$2$ map to the projective line: the fact that this degree is the same as the residue characteristic implies that we are in a ``wild setting".  Problems involving reduction of curves in the ``wild case", in which one studies semistable models of curves with a degree-$p$ map to the projective line over residue characteristic $p$, have been investigated in a number of works in recent decades (see \S\ref{sec introduction comparison} below), but mainly in the situation where the branch points of the map $Y \to \proj_K^1$ are $p$-adically equidistant.  In this article, we will consider general hyperelliptic curves over residue characteristic $2$, with a particular focus on the relationship between the combinatorial data of how the branch points are ``clustered" and the structure of the special fiber of a semistable model.

\subsection{Our main problem} \label{sec introduction main problem}

It is well known that an affine chart for a hyperelliptic curve $Y / K$ of genus $g \geq 1$ is given by an equation of the form 
\begin{equation} \label{eq hyperelliptic p not 2}
    y^2 = f(x) = c\prod_{i = 1}^d (x - a_i),
\end{equation}
where $f(x) \in K[x]$ is a polynomial of degree $d \in \{2g+1, 2g+2\}$ that does not have multiple roots, $c \in K^\times$ is the leading coefficient of $f(x)$, and the elements $a_i \in \bar{K}$ are the roots of $f$.  We call $f$ the \emph{defining polynomial} of (this chart of) the hyperelliptic curve $Y$.  The degree-$2$ morphism of $Y$ onto the projective line is given simply by the coordinate function $x$; this morphism is branched precisely at each of the roots of $f$ as well as, in the case that $d = 2g + 1$ (in other words, when $f$ has odd degree), at the point $\infty$.  After applying an appropriate automorphism of the projective line (i.e.\ a suitable change of coordinate) which moves one of the branch points to $\infty$, we obtain an equation of the form in (\ref{eq hyperelliptic p not 2}) with $d = 2g + 1$; we will adhere to this assumption about $f$ throughout most of the paper (see \S\ref{sec models hyperelliptic} for more details).  Our aim will be showing how to explicitly form semistable models of $Y$ over finite extensions of $K$.  We more fully explain various aspects of the problem below.

\subsubsection{Semistable models of curves} \label{sec introduction main problem semistable}

Given any smooth projective geometrically connected curve $C$ over a complete discrete valuation field $K$ with ring of integers $R \subset K$ and algebraically closed residue field $k$, a \emph{model} of $C$ over $R'$, where $R'$ is the ring of integers of some finite extension $K'\supseteq K$, is a normal projective flat $R'$-scheme $\CC$ whose generic fiber is isomorphic to $C$ over $K'$. We say that a model $\CC$ is \emph{semistable} if its special fiber $\CC_s$ is a reduced $k$-curve with at worst nodes as singularities.
The following groundbreaking theorem was proved by Deligne and Mumford in \cite{deligne1969irreducibility} and then through independent arguments by Artin and Winters in \cite{artin1971degenerate} (see also \cite[Section 10.4]{liu2002algebraic} for a detailed explanation of the arguments in Artin-Winters).

\begin{thm} \label{thm semistable reduction}
    Every smooth projective geometrically connected curve $C$ over $K$ achieves semistable reduction over a finite extension $K' \supseteq K$, i.e.\ $C$ admits a semistable model $\CC^{\sst}$ over $R'$, where $R'$ is the ring of integers in $K'$.
\end{thm}

The above result is not constructive and does not tell us how to find a semistable model $\CC^{\sst}$ or exactly how large an extension $K' \supseteq K$ is needed in order to define it.  It moreover does not specify, for a given curve $C / K$, anything about the structure of the special fiber $(\mathcal{C}^{\sst})_s$.  It is therefore natural to ask whether there is any general method by which we may construct a semistable model $\YY^{\sst}$ of a hyperelliptic curve $Y / K$ defined by an equation of the form in (\ref{eq hyperelliptic p not 2}).

\subsubsection{Special fibers of semistable models of curves} \label{sec introduction main problem special fibers}

In this paper, we are interested not only in how to construct a semistable model $\YY^\sst$ of a hyperelliptic curve $Y$, but also in how certain characteristics of the defining polynomial may determine the \textit{structure} of the special fiber of such a semistable model.  The special fiber $\SF{\YY^\sst}$ of a semistable model $\YY^\sst$ of a curve $Y / K$ by definition consists of reduced components which meet each other only at nodes.  Each node, viewed as a point in $\YY^{\sst}$, has a \emph{thickness} (see the initial discussion in \S\ref{sec preliminaries semistable}) which is a positive integer.  The structure of the special fiber $\SF{\YY^\sst}$ can be described entirely in terms of the set of its irreducible components, the genus of the normalization of each of these components, the data of which components intersect which others at how many nodes, and the thicknesses of the nodes.  The sum of the genera of the normalizations of the irreducible components is known as the \emph{abelian rank} of $\SF{\YY^\sst}$, while the number of loops in the configuration of components and their intersections (i.e.\ the number of loops in the dual graph of $\SF{\YY^\sst}$) is known as the \emph{toric rank} of $\SF{\YY^\sst}$.  The property of being semistable implies that the sum of these two ranks equals the genus of $Y$.  See \S\ref{sec preliminaries abelian toric unipotent} below for more details.

Replacing a semistable model $\YY^\sst$ of $Y$ over $R'$ with another semistable model of $Y$ over $R''$ (where $R'$ and $R''$ are the ring of integers of possibly different extensions of $K$) does not affect its abelian or toric rank (see \Cref{prop invariance of ranks} below), and therefore these ranks are intrinsic to the curve $Y$ itself and particularly interesting to determine (meanwhile, the thicknesses of the nodes change in a predictable manner between semistable models over different extensions of $R$; see the discussions in \S\ref{sec preliminaries semistable} and \S\ref{sec preliminaries field extensions}).

\subsubsection{The reduction of a curve given by \texorpdfstring{$y^2 = f(x)$}{y2=f(x)}} \label{sec introduction main problem y^2=f(x)}

Our first na\"{i}ve attempt to produce a semistable model for $Y$ is to perform simple changes of variables (if necessary) over a low-degree field extension $K' \supset K$ so that the coefficients appearing in the equation in (\ref{eq hyperelliptic p not 2}) are all integral and then to simply use this equation to define a scheme $\YY$ over the corresponding ring of integers $R'$. More precisely, it is clear that after possibly scaling $x$ and $y$ by appropriate elements of $\bar{K}^\times$, we may assume that $f$ is monic (i.e., $c=1$), and that the roots $a_i$ are all integral, with $\min_{i,j}v(a_i-a_j)=0$.  In particular, $f$ has integral coefficients, and so we may extend $Y$ to a scheme $\YY / R'$ whose generic fiber is $Y$ and whose special fiber $\YY_s$ is given (over the affine chart $x\neq \infty$ of $\mathbb{P}^1_k$) by the equation 
\begin{equation} \label{eq hyperelliptic p not 2 reduction}
    y^2 = \bar{f}(x) := \prod_{i = 1}^{2g + 1} (x - \bar{a}_i),
\end{equation}
where each element $\bar{a}_i$ is the reduction of $a_i \in \mathcal{O}_{\bar{K}}$ in the residue field $k$.

Suppose that the residue characteristic of $K$ is different from $2$.  Then the curve $\mathcal{Y}_s / k$ is generically an \'{e}tale double cover of the projective line $\proj_k^1$, and its only possible singularities are produced by multiple roots of the reduced polynomial $\bar{f}$; consequently, the reduced curve $\mathcal{Y}_s$ is smooth if and only if the roots of $f$ are all distinct modulo the prime ideal of the splitting field.

Suppose on the other hand that the residue characteristic of $K$ is $2$.  Then the curve $\mathcal{Y}_s / k$ is an inseparable cover of $\proj_k^1$, and it \textit{always} has non-nodal singularities whether or not the reduction of the polynomial $f$ has multiple roots.  We summarize these (fairly elementary) facts in the following proposition.

\begin{prop} \label{prop naive model}
    Let $Y / K$ and $\mathcal{Y} / R'$ be defined as in the discussion above.
    \begin{enumerate}[(a)]
        \item Suppose that the residue characteristic of $K$ is not $2$.  Then each singular point of the special fiber $\mathcal{Y}_s$ is of the form $(x, y) = (\bar{a}, 0)$, where $\bar{a} \in k$ is a multiple root of the reduced polynomial $\bar{f}$.  Given a singular point $(\bar{a}, 0)$ of $\mathcal{Y}_s$, let $\mathfrak{s} \subset \bar{K}$ be the subset of roots of $f$ which each reduce to $\bar{a}$.  Then, 
        \begin{enumerate}[(i)]
            \item if $\mathfrak{s}$ has cardinality $2$, the singular point $(\bar{a}, 0)$ is a node; and 
            \item if $\mathfrak{s}$ has cardinality at least $3$, the singular point $(\bar{a}, 0)$ is not a node.
        \end{enumerate}
        \item Suppose that the residue characteristic of $K$ is $2$.  Then the special fiber $\mathcal{Y}_s$ has a non-nodal singularity at each point whose $x$-coordinate is a root of the derivative polynomial $\bar{f}'$ (and these are the only singularities of $\mathcal{Y}_s$).
    \end{enumerate}
\end{prop}

\begin{proof}
    It is straightforward to verify, using a standard equation for another affine open subset of $\mathcal{Y}_s$ (which is given in \S\ref{sec models hyperelliptic}) which contains the points over $x = \infty$, that there is no singular point over $x = \infty$ due to the fact that $f$ is monic so that its reduction $\bar{f}$ has maximal degree.  To prove both parts of the proposition, it therefore suffices to consider singular points on the affine part of $\mathcal{Y}_s$ defined by the equation $y^2 = \bar{f}(x)$.

    Assume first that the residue characteristic is not $2$.  Then, applying the Jacobian criterion and setting both partial derivatives of $y^2 - \bar{f}(x)$ to $0$, we get that a singular point can only occur where $y = 0$ (which implies that $x$ is a root of $\bar{f}$) and $x$ is a root of $\bar{f}'$.  These conditions imply that the $x$-coordinate of a singular point must be a multiple root of $\bar{f}$.  After appropriately translating the $x$-coordinate, we may assume that a given singular point is $(0, 0)$, which implies that $\bar{f}(x)$ is exactly divisible by $x^n$ for some integer $n \geq 2$ which is the multiplicity of the root $0$.  The singular point $(0, 0)$ is a node if and only if the polynomial consisting of the terms of degree $\leq 2$ in the defining polynomial $y^2 - \bar{f}(x)$ factors into distinct linear polynomials over $k$ (this is indeed how \emph{singular point} and \emph{node} are defined in \cite[\S I.1.2]{shafarevivch1994basic}).  This clearly happens if and only if $n = 2$, which finishes the proof of part (a).
    
    Now assume that the residue characteristic is $2$.  This time, applying the Jacobian criterion tells us that there is a singular point wherever we have $\bar{f}'(x) = 0$.  After translating both coordinate variables $x$ and $y$ suitably, we may assume that a given singular point is $(0, 0)$.  Now it is clear that the polynomial consisting of the terms of degree $\leq 2$ in the defining polynomial $y^2 - \bar{f}(x)$ does not factor into distinct linear polynomials over $k$ since it does not include an $xy$-term and is therefore the square of a linear polynomial instead.  This implies that the singular point is not a node, and part (b) is proved.
\end{proof}

\subsubsection{Cluster data} \label{sec introduction main problem cluster data}

We have just seen that the na\"{i}ve attempt to construct a semistable model of a hyperelliptic curve $Y$ as in \S\ref{sec introduction main problem y^2=f(x)} always fails over residue characteristic $2$.  Meanwhile, in the case that the residue characteristic is not $2$, Proposition \ref{prop naive model} more or less implies that the na\"{i}ve model $\mathcal{Y} / R'$ is semistable if and only if (1) the roots of $f$ are \textit{equidistant} (i.e.\ the valuations of the difference between the roots are all equal) so that $\mathcal{Y}_s$ is smooth, or (2) the roots of $f$ are equidistant except for certain pairs of roots of $f$ which are closer to each other with respect to the discrete valuation of $K$ (so that each pair maps to a root of multiplicity $2$ of the reduced polynomial $\bar{f}$ and produces a node of $\mathcal{Y}_s$).  This suggests that when the residue characteristic of $K$ is not $2$, the data of the valuations of differences between roots of $f$ may be directly crucial for constructing a semistable model of $Y$ and for understanding the structure of the special fiber of such a semistable model.

This notion is made precise in \cite{dokchitser2022arithmetic} by defining the \emph{cluster data} associated to a hyperelliptic curve $Y$ over a discrete valuation field $K$: roughly speaking, if $Y$ is defined by an equation of the form in (\ref{eq hyperelliptic p not 2}), its associated cluster data consists of subsets $\mathfrak{s}$ of roots of the defining polynomial $f$, called \textit{clusters}, which are closer to each other with respect to the discrete valuation of $K$ than they are to the roots of $f$ which are not contained in $\mathfrak{s}$, along with, for each non-singleton cluster $\mathfrak{s}$, the minimum valuation of differences between roots in $\mathfrak{s}$, called the \textit{depth} of $\mathfrak{s}$.  For precise definitions, see Definition \ref{dfn cluster} below or \cite[Definition 1.1]{dokchitser2022arithmetic}.

When the residue characteristic is different from $2$, the process of construction of a semistable model of $Y$ as well as the structure of its special fiber is governed entirely by the cluster data associated to $Y$.  This can be deduced from the explicit constructions given in \cite[\S4, 5]{dokchitser2022arithmetic} in any case, but we will present a variant of this construction in \S\ref{sec clusters}.  The rough idea is summarized as follows, under the simplifying assumption that $f$ has degree $2g + 1$.

\begin{enumerate}
    \item There is a one-to-one correspondence between discs $D \subset \bar{K}$ (with respect to the induced valuation on $\bar{K}$) and smooth models of $\proj_K^1$ over finite extensions of $R$, and each model of $\proj_K^1$ over a finite extension of $R$ with reduced special fiber is the compositum of a finite number of smooth models and thus corresponds to a finite collection of discs $D \subset \bar{K}$ (see \S\ref{sec models hyperelliptic line} for more details).
    \item We define $\XX^{(\sst)}$ to be the model of $\proj_K^1$ over a finite extension of $R$ corresponding in the above way to the set of discs $D_\mathfrak{s}$ for all non-singleton clusters $\mathfrak{s}$, where each $D_\mathfrak{s} \subset \bar{K}$ denotes the minimal disc whose intersection with the set of roots of $f$ coincides with $\mathfrak{s}$.
    \item There is a semistable model $\mathcal{Y}^{\sst}$ with a degree-$2$ map to $\mathcal{X}^{(\sst)}$ which is constructed simply by normalizing $\mathcal{X}^{(\sst)}$ in the function field $K(Y)$ after possibly replacing $K$ with a finite extension, which is at most the unique quadratic extension of the splitting field of $f$.
\end{enumerate}

The discs $D_i$ mentioned in Step (2) correspond  to changes in coordinate of the form $x = \alpha_i + \beta_i x_i$ for some $\alpha_i \in \bar{K}$ and $\beta_i \in \bar{K}^{\times}$; for each such change in coordinate, we may perform appropriate substitutions into the equation $y^2 = f(x)$ and transform $y$ appropriately to get a new equation of the form $y_i^2 = f_{i}(x_i) \in R'[x_i]$, where $R'$ is the ring of integers of an appropriate finite extension $K'\supseteq K$; this new equation defines a model of $Y$, which is the normalization of the model of $\mathbb{P}^1_K$ corresponding to the disc $D_i$ in the function field $K'(Y)$. The desired semistable model $\YY^{\sst}$ is comprised of these normalizations.  The idea is illustrated by the following example.

\begin{ex} \label{ex introduction p not 2}
    Let $K = \qq_p^{\unr}$ for some $p \geq 5$ and 
    \begin{equation}
        f(x) = x(x - p^3)(x - p)(x - 1)(x - 1 + p^4)(x - 2)(x - 3).
    \end{equation}
    The set of roots of $f$ is $\mathcal{R} := \{0, p^3, p, 1, 1 - p^4, 2, 3\}$.  The \emph{clusters} of these roots (i.e.\ the subsets $\mathfrak{s}$ consisting of roots which are closer to each other than they are to the roots in $\mathcal{R} \smallsetminus \mathfrak{s}$) are 
    $$\mathfrak{s}_0 := \mathcal{R}, \ \mathfrak{s}_1 := \{0, p^3, p\}, \ \mathfrak{s}_2 := \{0, p^3\}, \ \mathfrak{s}_3 := \{1, 1 - p^4\},$$
    as well as each of the singleton subsets of $\mathcal{R}$ (which we ignore).  The data of these clusters is represented by the following diagram.
    
    \begin{equation*}
	\mathrm{cluster} \ \mathrm{picture} \ \mathrm{of} \ \mathcal{R}:\
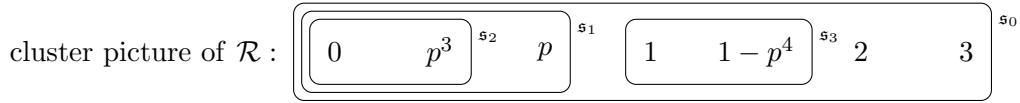
	\begin{clusterpicture}
		\comincia;
		\punto{1}{0}{$0$};
		\punto{2}{1}{$p^3$};
		\punto{3}{2}{$p$};
		\punto{4}{3}{$1$};
		\punto{5}{4}{$1-p^4$};
		\punto{6}{5}{$2$};
		\punto{7}{6}{$3$};
		\cluster{12}{\pto{1}\pto{2}}{$\mathfrak{s}_2$};
		\cluster{123}{\clu{12}\pto{3}}{$\mathfrak{s}_1$};
		\cluster{45}{\pto{4}\pto{5}}{$\mathfrak{s}_3$};
		\cluster{tot}{\clu{123}\clu{45}\pto{6}\pto{7}}{$\mathfrak{s}_0$};
	\end{clusterpicture}
    \end{equation*}
    
    The discs $D_i \subset \bar{K}$ minimally containing each of the clusters $\mathfrak{s}_i$ are then given by 
    \begin{equation}
        \begin{aligned}
            D_0 &:= \overline{\zz_p},& D_1 &:= p \overline{\zz_p} = \{0 + pz \ | \ z \in \overline{\zz_p}\},\\
            D_2 &:= p^3 \overline{\zz_p} = \{0 + p^3 z \ | \ z \in \overline{\zz_p}\},& D_3 &:= 1 + p^4\overline{\zz_p} = \{1 + p^4 z \ | \ z \in \overline{\zz_p}\},\\
        \end{aligned}
    \end{equation}
    where $\overline{\zz_p}$ denotes the ring of integers of the algebraic closure $\overline{\qq_p}$ of $\qq_p$.  The changes in coordinates corresponding to each of these discs are given by 
    $$x = x_0 = p x_1 = p^3 x_2 = p^4 (x_3 - 1),$$ where each $x_i$ corresponds to the disc $D_i$ in an obvious way, and we define corresponding coordinates $y_i$ by scaling $y$ by suitable elements of $\qq_p(\sqrt{p})$ as 
    $$y = y_0 = p^{3/2} y_1 = p^{7/2} y_2 = p^4 y_3.$$
    We now define corresponding models $\mathcal{Y}_i / \zz_p^{\unr}[\sqrt{p}]$ of $Y / \qq_p^{\unr}(\sqrt{p})$ for $i = 0, 1, 2, 3$, given by the below equations.
    \begin{equation}
    \begin{split}
        \mathcal{Y}_0: y_0^2 &= f(x) = f(x_0) \\
        \mathcal{Y}_1: y_1^2 &= p^{-3} f(x) = x_1 (x_1 - p^2)(x_1 - 1)(px_1 - 1)(px_1 - 1 + p^4)(px_1 - 2)(px_1 - 3) \\
        \mathcal{Y}_2: y_2^2 &= p^{-7} f(x) = x_2 (x_2 - 1)(p^2 x_2 - 1)(p^3 x_2 - 1)(p^3 x_2 - 1 + p^4)(p^3 x_2 - 1)(p^3 x_2 - 2) \\
        \mathcal{Y}_3: y_3^2 &= p^{-8} f(x) = (p^4 x_3 - 1)(p^4 x_3 - 1 - p^3)(p^4 x_3 - 1 - p)(x_3)(x_3 - 1)(p^4 x_3 - 2)
    \end{split}
    \end{equation}
    Their respective reductions (that is, their special fibers $\SF{\YY_i}$) over the residue field $\overline{\ff_p}$ are as follows.
    \begin{equation}
    \begin{split}
        \SF{\YY_0}: y_0^2 &= x_0^3 (x_0 - 1)^2 (x_0 - 2)(x_0 - 3) \\
        \SF{\YY_1}: y_1^2 &= 6x_1^2 (x_1 - 1) \\
        \SF{\YY_2}: y_2^2 &= -6x_2 (x_2 - 1) \\
        \SF{\YY_3}: y_3^2 &= 2x_3 (x_3 - 1)
    \end{split}
    \end{equation}
    The desingularizations of each of these special fibers give rise to the components of the special fiber of the desired semistable model $\YY^{\sst}$: here $\SF{\YY_0}$ contributes a smooth component $V_0$ of genus 1; $\SF{\YY_1}$ contributes a line $V_1$ which intersects $V_0$ at a single node; $\SF{\YY_2}$ contributes a line $V_2$ which intersects $V_1$ at $2$ nodes; and $\SF{\YY_3}$ contributes a line $V_3$ which intersects $V_0$ at $2$ nodes.  The configuration is shown in \Cref{fig pnot2 g2 example2}.
    
    One can see from the configuration of components displayed in \Cref{fig pnot2 g2 example2} that the toric rank of $\SF{\YY^{\sst}}$ is 2; if one adds this to the sum of the genera of the components $V_i$, the genus $g=3=2+1$ of $Y$ is recovered.
\end{ex}

\begin{figure}
    \centering
    \vspace{-2cm}
    \tikzset{every picture/.style={line width=0.75pt}} 

\begin{tikzpicture}[x=0.75pt,y=0.75pt,yscale=-1,xscale=1]

\draw    (14,80) -- (238,80) ;
\draw    (32,122) .. controls (79,120.33) and (45,-76) .. (128,63) ;
\draw    (131,125.33) .. controls (167,2.67) and (199,-12.67) .. (216,123.33) ;
\draw    (363,130) -- (397,19) ;
\draw    (480,81) -- (343.47,81) ;
\draw    (263.95,81) -- (321.47,81) ;
\draw [shift={(324.47,81)}, rotate = 180] [fill={rgb, 255:red, 0; green, 0; blue, 0 }  ][line width=0.08]  [draw opacity=0] (10.72,-5.15) -- (0,0) -- (10.72,5.15) -- (7.12,0) -- cycle    ;
\draw    (408,135.33) -- (442,24.33) ;
\draw    (14,33) -- (144,33) ;
\draw    (353,33) -- (417,33) ;

\draw (16,55.4) node [anchor=north west][inner sep=0.75pt]    {$V_{1}$};
\draw (60,104.4) node [anchor=north west][inner sep=0.75pt]    {$V_{0}$};
\draw (220,105.4) node [anchor=north west][inner sep=0.75pt]    {$V_{2}$};
\draw (459,59.4) node [anchor=north west][inner sep=0.75pt]    {$L_{1}$};
\draw (370,108.4) node [anchor=north west][inner sep=0.75pt]    {$L_{0}$};
\draw (417,113.73) node [anchor=north west][inner sep=0.75pt]    {$L_{2}$};
\draw (20,12.4) node [anchor=north west][inner sep=0.75pt]    {$V_{3}$};
\draw (356,13.4) node [anchor=north west][inner sep=0.75pt]    {$L_{3}$};

\end{tikzpicture}
    \caption{The special fiber $\SF{\YY^\sst}$, shown on the left, mapping to the special fiber $\SF{\XX^\sst}$; each component $V_i$ of $\SF{\YY^\sst}$ maps to each component $L_i := \SF{\XX_{D_i}}$ of $\SF{\XX^\sst}$.}
    \label{fig pnot2 g2 example2}
\end{figure}

\begin{rmk} \label{rmk example introduction p not 2}
    In the case that $Y / K$ is an \emph{elliptic curve} (i.e.\ $g = 1$) over residue characteristic $p \neq 2$, where the polynomial $f$ has degree $3$, there are at most $2$ non-singleton clusters of roots of $f$, and a similar procedure can be performed to get a semistable model of $Y$ over the (unique) quadratic ramified extension of the splitting field of $f$, which will be smooth if and only if $Y / K$ has potentially good reduction.  This is more or less the process outlined in the proof of \cite[III.1.7(a)]{silverman2009arithmetic} combined with the proof of \cite[VII.5.4(c)]{silverman2009arithmetic}, except that Silverman does not construct a separate component of the semistable model corresponding to a cardinality-$2$ cluster of roots (in the case that there is one).  So, following Silveman's method, the special fiber of the semistable model always consists of only $1$ component which has a node if and only if there is a cardinality-$2$ cluster of roots of $f$ (this is the case of \emph{multiplicative reduction}).
\end{rmk}

When the residue characteristic of $K$ is $2$, it is natural to ask whether a semistable model of $Y$ can be constructed by a procedure governed entirely by the associated cluster data in this way.  In short, the answer is ``no", but in this paper we develop methods of finding a particular collection of discs in $\bar{K}$ which corresponds to a model $\XX^{(\sst)}$ of $\mathbb{P}^1_K$ over a finite extension of $R$, such that the model $\YY^\sst$ of $Y$ which is constructed directly from $\XX^{(\sst)}$ in a similar manner to Steps (2)-(3) above is guaranteed to be semistable (and to satisfy several other nice properties discussed in \S\ref{sec relatively stable}).  We will present and prove results relating such a set of discs to the set of clusters $\mathfrak{s}$ appearing in the cluster data associated to $Y$.

\subsection{A summary of our main results for residue characteristic 2} \label{sec introduction main results}

Although the arguments used in this paper will recover what is already known about the construction of semistable models of hyperelliptic curves in characteristic different from $2$, our primary aim is to understand how to construct a semistable model as well as the structure of its special fiber when the residue characteristic is $2$.  This is addressed by our main results.

\subsubsection{Constructing equations for models with semistable reduction} \label{sec introduction main results constructing}

It is clear from \Cref{prop naive model} that if $K$ has residue characteristic $2$, a model of $Y$ given by an equation of the form $y^2 = f(x)\in R[x]$ cannot possibly have semistable reduction.  We must therefore find a model given by one or more equations of the more general form 
\begin{equation} \label{eq hyperelliptic affine model general case}
    y_i^2 + q_i(x_i)y_i = r_i(x_i),
\end{equation}
where $q_i(x_i), r_i(x_i) \in R'[x_i]$ are polynomials of degree less than or equal to $g + 1$ and $2g + 1$ respectively (see \S\ref{sec models hyperelliptic equations} below for more details on this form of equation).  This is generally accomplished in the following manner.  First (as in the case of residue characteristic not $2$) we make a substitution of the form $x = \alpha_i + \beta_i x_i$ with $\alpha_i \in \bar{K}$ and $\beta_i \in \bar{K}^{\times}$ and scale $y$ by a suitable element of $\bar{K}^{\times}$ to get a coordinate $\tilde{y}_i$ and an equation of the form $\tilde{y}_i^2 = f_i(x_i) \in \bar{K}[x_i]$, where $f_i$ has integral coefficients and nonzero reduction.  Then, in order to turn this into an equation of the form in (\ref{eq hyperelliptic affine model general case}), we find a decomposition $f_i = q_i^2 + 4r_i$, where $q_i(x_i), r_i(x_i) \in \bar{K}[x_i]$ are polynomials of degree less than or equal to $g + 1$ and $2g + 1$ respectively (this is a \emph{part-square decomposition} as we define it below in Definition \ref{dfn qrho}) and set $\tilde{y}_i = 2y_i + q_i(x)$; note that this is essentially performing the standard operation of ``completing the square" in reverse.

There are two points of delicacy that must be taken into account when choosing the elements $\alpha_i, \beta_i$ and the decomposition $f_i = q_i^2 + 4r_i$.  One is that $\alpha_i$ and $\beta_i$ must be chosen carefully so that all terms in the resulting equation of the form (\ref{eq hyperelliptic affine model general case}) have integral coefficients, 
so that these equations may be defined over the ring of integers $R' \supseteq R$ of the finite extension of $K$ given by adjoining all elements $\alpha_i, \beta_i$ and coefficients of the polynomials $q_i, r_i$.  Secondly, one must be sure that all of the components of $\SF{\XX^\sst}$ (each corresponding to a choice of $\alpha_i$ and $\beta_i$) have really been found; otherwise, the model of $Y$ corresponding to an incomplete set of coordinates $x_i$ will contain non-nodal singularities in its special fiber and will therefore fail to be semistable.

\subsubsection{Our main results} \label{sec introduction main results main}

As discussed above, a semistable model $\mathcal{Y}^\sst / R'$ of $Y$ (where $R'$ is the ring of integers of a finite extension $K'/K$) may be constructed, more or less, as the normalization of a suitable model $\XX^{(\sst)}$ of $\proj_{K}^1$ in the function field of $Y$, and $\XX^{(\sst)}$, in turn, corresponds to a finite collection of changes of coordinate $x = \beta_i x_i + \alpha_i$ (so that its special fiber is composed of copies of the projective line over the residue field $k$ corresponding to each coordinate $x_i$; see \S\ref{sec models hyperelliptic line} below).  Finding a collection of appropriate substitutions $x = \beta_i x_i + \alpha_i$ is therefore in some sense the most essential step in finding our desired semistable model $\mathcal{Y}^\sst$, just as it is in the case of residue characteristic not $2$.  As in our discussion in \S\ref{sec introduction main problem cluster data}, each new coordinate $x_i$ obtained from $x$ in this way by translation and homothety corresponds to a disc $D_i := \{\beta_i z +\alpha_i\ | \ z \in \mathcal{O}_{\bar{K}}\}$, so finding a semistable model of $Y$ again largely amounts to choosing an appropriate collection of discs in $\bar{K}$.  The difference now is that, unlike in the case of residue characteristic not $2$, these discs generally do not correspond in a one-to-one manner to non-singleton clusters of roots of $f$.

In \S\ref{sec relatively stable definition} of this paper, we define a particularly nice (unique up to unique isomorphism) semistable model $\Yrst$ of a given hyperelliptic curve $Y$ which we call the \emph{relatively stable model} (see Definition \ref{dfn relatively stable} below).  We will define a \emph{valid disc} (\Cref{dfn valid disc} below) to be a disc $D \subset \bar{K}$ among the collection of discs used the manner discussed above to construct the semistable model $\Yrst$ (excluding such discs which correspond to components of $\SF{\Xrst}$ over which the cover $\SF{\Yrst}\to \SF{\Xrst}$ is inseparable).  The central results we present in this paper are on how to find valid discs.  While the exact procedure provided by these results cannot be described succinctly in this introduction, we give a partial summary of the general outcome in the following theorem.

\begin{thm} \label{thm introduction main}
    Assume all of the above set-up for a hyperelliptic curve $Y / K$ of genus $g$ given by an equation of the form $y^2 = f(x) \in K[x]$, where the polynomial $f$ has degree $2g + 1$, and assume that the residue characteristic of $K$ is $2$.  Let $\Yrst/R'$ be the relatively stable model of $Y$, where $R'$ is the ring of integers of an appropriate finite field extension $K'\supseteq K$. Let $\RR \subset \bar{K}$ denote the set of roots of $f$. For any cluster of roots $\mathfrak{s} \subsetneq \RR$, we write $\mathfrak{s}'$ for the minimal cluster which properly contains $\mathfrak{s}$.
    
    The clusters of roots in $\mathcal{R}$ and the valid discs associated to $Y$ are related in the following manner.
    \begin{enumerate}[(a)]
        \item Given a valid disc $D \subseteq \bar{K}$, the cardinality of $D \cap \mathcal{R}$ is even (and we may have $D \cap \mathcal{R} = \varnothing$).
        \item If a cluster $\mathfrak{s}$ has even cardinality, there are either $0$, $1$, or $2$ valid discs $D \subseteq R'$ such that either $D \cap \mathcal{R} = \mathfrak{s}$ or $D$ is the smallest disc containing $\mathfrak{s}'$.
        \item Let $\mathfrak{s}$ be an even-cardinality cluster of \emph{relative depth} $m := \min\{v(a - a') \ | \ a, a' \in \mathfrak{s}\} - \min\{v(a - a') \ | \ a, a' \in \mathfrak{s}'\}$ (see Definition \ref{dfn depth}), and write $f_0(x) = \prod_{a \in \mathfrak{s}} (x - a)$ and $f_\infty(x) = f(x) / f_0(x)$.  There exists a rational number $B_{f, \mathfrak{s}} \in \qq_{\geq 0}$ which is independent of the relative depth of $\mathfrak{s}$ in the sense of Remark \ref{rmk B independence}, such that 
        \begin{enumerate}[(i)]
            \item if $m > B_{f,\mathfrak{s}}$, the number of valid discs as in part (b) is ``2";
            \item if $m = B_{f,\mathfrak{s}}$, the number of valid discs as in part (b) is ``1"; and 
            \item if $m < B_{f,\mathfrak{s}}$, the number of valid discs as in part (b) is ``0".
        \end{enumerate}
        Moreover, in the case of (i), the $2$ guaranteed valid discs containing $\mathfrak{s}$ each give rise to $1$ component or to $2$ non-intersecting components of $\SF{\Yrst}$.  In the case that each gives rise to a single component of $\SF{\Yrst}$, the resulting pair of components intersects at $2$ nodes, whereas in the case that one of the valid discs gives rise to $2$ (non-intersecting) components $V_1$ and $V_2$ of $\SF{\Yrst}$, the other valid disc gives rise to a single component of $\SF{\Yrst}$ which intersects each of $V_1$ and $V_2$ at a single node.  In either case, each of these nodes has thickness equal to $(m - B_{f,\mathfrak{s}})/v(\pi)$, where $\pi$ is a uniformizer of $K'$.
        \item Given an even-cardinality cluster $\mathfrak{s}$, the bound $B_{f,\mathfrak{s}}$ from part (d) satisfies $B_{f,\mathfrak{s}} \leq 4v(2)$.  If we furthermore assume that $\mathfrak{s}$ and $\mathfrak{s}'$ each have a maximal subcluster of odd cardinality (e.g.\ a maximal subcluster which is a singleton), we have the inequality
        \begin{equation}
            B_{f,\mathfrak{s}} \geq \Big( \frac{2}{|\mathfrak{s}| - 1} + \frac{2}{2g + 1 - |\mathfrak{s}|} \Big) v(2).
        \end{equation}
        \item The toric rank of some (any) semistable model of $Y$ is equal to the number of even-cardinality clusters satisfying item (i) above which themselves cannot be written as a disjoint union of such even-cardinality clusters.
        \item Let $\mathfrak{s}$ be a cluster of odd cardinality not equal to $1$ or $2g + 1$.  Then $\SF{\Yrst}$ consists of two curves $C_0$ and $C_\infty$ meeting as a single node in $\SF{\Yrst}$; their arithmetic genera are $\frac{1}{2}(|\mathfrak{s}| - 1)$ and $g - \frac{1}{2}(|\mathfrak{s}| - 1)$ respectively.
    \end{enumerate}
\end{thm}

The statements in the above theorem are a combination of a (sometimes simplified version of) statements of the main results presented and proved in this paper.  Parts (a)--(c), apart from the final statement in (c), are adapted from \Cref{thm summary depths valid discs} and \Cref{prop depth threshold} below (see also \Cref{thm cluster p=2}); part (d) is adapted from \Cref{prop estimating threshold}(c); the final statement in (c) comes directly from \Cref{prop viable correspondence}; part (e) is a rephrasing of \Cref{thm toric rank}; and part (f) is a rephrasing of \Cref{cor odd-cardinality clusters} (we note that this statement actually also holds when the residue characteristic is different from $2$).  Formulas for thicknesses are not explicitly given in the above-mentioned results but in general can easily be computed using \Cref{prop vanishing persistent}(b) combined with \Cref{prop thickness}; we get the assertion about thicknesses in part (c) from applying \Cref{prop viable thicknesses} to our results in \S\ref{sec depths separating roots} (see \Cref{prop formulas for b_pm}) which tell us explicitly what the depths of the $2$ guaranteed valid discs are in the situation of \Cref{thm introduction main}(c)(i).

The results in this paper can be viewed as a vast generalization of the results in \cite{yelton2021semistable}, where the second author explicitly constructed semistable models of elliptic curves with a cluster of cardinality $2$ and depth $m$ (as well as elliptic curves with no even-cardinality clusters).  The threshold for $m$ above which there are $1$ or $2$ valid discs containing that cardinality-$2$ cluster which is found in \cite{yelton2021semistable} comes as the following easy corollary to the above theorem; we remark that this corollary can be deduced also from standard formulas for the $j$-invariant of an elliptic curve (specifically, the particular choice of power of $2$ multiplied to the rest of the formula, which influences the valuation of the $j$-invariant in residue characteristic $2$; see \Cref{rmk Legendre}(a) below).

\begin{cor} \label{cor g=1 B=4v(2)}
    Suppose that we are in the $g = 1$ case of the situation in \Cref{thm introduction main} and that $\mathfrak{s}$ is a cluster of cardinality $2$.  Then we have $B_{f,\mathfrak{s}} = 4v(2)$.
\end{cor}

\begin{proof}
    The parent cluster of $\mathfrak{s}$ (i.e., the minimal cluster strictly containing $\mathfrak{s}$) is $\mathfrak{s}'=\RR$, which has cardinality 3. It is clear that both $\mathfrak{s}$ and $\mathfrak{s}'$ have a singleton child cluster (i.e., a maximal subcluster consisting of only one root). Now, \Cref{thm introduction main}(d) gives that $B_{f,\mathfrak{s}}\leq 4v(2)$ and 
    \begin{equation}
        B_{f,\mathfrak{s}} \geq \big( \frac{2}{1} + \frac{2}{1} \big)v(2) = 4v(2).
    \end{equation}
    The equality $B_{f,\mathfrak{s}} = 4v(2)$ follows.
\end{proof}

For examples of semistable models of hyperelliptic curves over residue characteristic $2$ which are explicitly computed in the manner discussed above, see Examples \ref{ex g=1 p=2} and \ref{ex g=2 p=2} below, which are worked out directly from the results and processes developed in \S\ref{sec depths} and \S\ref{sec centers}.  Note that in both of these examples, the set of clusters consists of a single cardinality-$2$ cluster $\mathfrak{s}$ as well as the full set $\mathcal{R}$ of the roots, so that following what happens in the case of residue characteristic not $2$, we would expect that $\SF{\Yrst}$ contains exactly $2$ components obtained by centering at an element of the $\mathfrak{s}$ and scaling according to how close the $2$ elements in $\mathfrak{s}$ are.  However, in this case, the choices of scaling factors $\beta_i$ are not so ``obvious" as in the situation of residue characteristic not $2$ (as in \Cref{ex introduction p not 2}), and moreover, in \Cref{ex g=2 p=2} we get a further component. 

\Cref{thm introduction main} above describes the overall relationship between clusters and valid discs associated to a hyperelliptic curve over residue characteristic $2$, which is one of our main points of focus, but in our more broad investigation we come up with a general method of finding all valid discs.  The process of finding all valid discs having a given center (in particular, those containing a given cluster) is developed in \S\ref{sec depths} (the actual computations that are necessary are aided by \Cref{algo sufficiently odd}), while for residue characteristic $2$, the process of finding centers of all valid discs (in particular the ones which do not contain roots of $f$) is developed in \Cref{sec centers}, relying on the computation of a certain polynomial $F(T) \in K[T]$; \Cref{cor centers}(a) states in particular that each valid disc not containing roots of $f$ is centered at a root of $F$.

\subsection{Comparison to other works} \label{sec introduction comparison}

A hyperelliptic curve is a special case of a \emph{superelliptic curve}, i.e.\ a curve defined by an equation of the form $y^n = f(x)$ for some $n \geq 2$.  There have been a number of works discussing semistable models of superelliptic curves.  When the exponent $n$ in the equation for a superelliptic curve is not divisible by the residue characteristic $p$, the process of constructing a semistable model is relatively straightforward and is provided in \cite[\S3]{bouw2015semistable}, \cite[\S4]{bouw2017computing}, \cite[\S4, 5]{dokchitser2022arithmetic} (for hyperelliptic curves, using the language of clusters), and \cite{gehrunger2021reduction} (for hyperelliptic curves, using the language of stable marked curves), as well as earlier works.  We recover our own variant of their results in the hyperelliptic case (i.e.\ when $n = 2$ and $p \neq 2$) based on \Cref{thm cluster p odd} below, in the process of investigating the situation when $p = 2$.

The existing results for the wild case of semistable reduction of superelliptic curves, i.e.\ when the defining equation is of the form $y^p = f(x)$ where $p$ is the residue characteristic, have been far more limited.  To the best of our knowledge, investigations into this case began with Coleman, who in \cite{coleman1987computing} outlined an algorithm for changing coordinates in such a way that the defining equation is converted to a form whose reduction over the residue field does not describe a curve which is an inseparable degree-$p$ cover of the line; when $p = 2$, this is more or less equivalent to our notion of \emph{part-square decompositions} which will be introduced in \S\ref{sec models hyperelliptic part-square}.  This idea is further developed by Lehr and Matignon in \cite{matignon2003vers} and later in \cite{lehr2006wild} (among several other works).  Their results apply only to the very particular case of \textit{equidistant geometry}, meaning that the valuations of differences between each pair of distinct roots of the defining polynomial $f$ are all equal, which in the language of clusters means that there are no proper, non-singleton clusters of roots.  Much of their focus is on the (finite) extension of the ground field over which semistable reduction is obtained and the action of the (finite) Galois group of this extension on the special fiber of their semistable model.  The wild case is also discussed in \cite[\S4]{bouw2017computing}, in which several examples are computed and interpreted in terms of rigid analytic geometry; the working of these examples is mainly done through clever guessing rather than a direct algorithm, however.

There are further similarities between the ideas presented in the work of Lehr and Matignon and some of our results, which are applied to hyperelliptic curves whose branch points are not necessarily geometrically equidistant.  The notion of \emph{$p$-d\'{e}veloppements de Taylor} (Taylor $p$-expansions) introduced in \cite[\S2]{matignon2003vers}, while defined completely differently, is alike in motivation and applications to our notion of \emph{sufficiently odd decompositions} (see \S\ref{sec depths sufficiently odd} below), and our algorithm for computing sufficiently odd decompositions is a mild variation of \cite[Proposition 2.2.1]{matignon2003vers}, which is used to show that \textit{$p$-d\'{e}veloppements de Taylor} exist.  Moreover, in each of \cite{matignon2003vers} and \cite{lehr2006wild}, a polynomial over the ground field is defined whose roots are the centers of all discs which give rise to components of the special fiber; these polynomials (the \emph{$p$-d\'{e}riv\'{e}e} in \cite[D\'{e}finition 2.4.1]{matignon2003vers} and the \emph{monodromy polynomial} in \cite[Definition 3.4]{lehr2006wild}) are quite distinct but each is defined similarly and plays a similar role to our polynomial $F(T) \in K[T]$ given in \Cref{dfn F} below, whose roots in the geometrically equidistant case certainly provide centers of all the valid discs.

Our work differs from the prior research discussed above in that our major focus is on the relationship between clusters of roots and the structure of the special fiber of a semistable model of a hyperelliptic curve when the residue characteristic is $2$; to the best of our knowledge, the only specific case in terms of cluster data which has been investigated where equidistant geometry is not assumed is in the recent article \cite{dokchitser2023note}, which treats a case involving an even number of roots clustering in pairs.

We finish this subsection by remarking that our paper does not prioritize much focus towards describing the finite extension of $K$ over which we are constructing our semistable (relatively stable) model of $Y$ or determining the minimal extension of $K$ over which $Y$ achieves semistable reduction (although in building $\Yrst$, we try to be economical in the extension of $K$ required).  However, as our results are constructive, it is fairly straightforward to compute the (necessarily totally ramified) extension $K' / K$ over which $\Yrst$ is defined.  In general, the extension $K' / K$ is obtained from (possibly) a sequence of quadratic extensions of the subfield $K'' \subset K'$ over which the associated model $\Xrst$ of the projective line is defined using changes of coordinates $x = \alpha_i + \beta_i x_i$ as discussed above; then $K''$ is clearly just the smallest field over which the discs of $\Drst$ are defined, where a disc $D$ of $\bar{K}$ is said to be \emph{defined} over a field extention $K'/K$ if there exist $\alpha$ and $\beta$ in $K'$ such that $D = D_{\alpha,v(\beta)}$.  In practice, each scaling element $\beta_i$ may be chosen to be any element of a prescribed valuation, while a given translating element $\alpha_i$ may be chosen to be a root of $f$ (and thus already in the splitting field) when the corresponding valid disc contains roots of $f$; it is only in the case where there are valid discs not containing roots of $f$ that one may have to choose $\alpha_i$ to be a root of the (generally high-degree) polynomial $F(T) \in K[T]$ defined in \S\ref{sec centers def}.  It would be interesting to pursue results that specify the minimal extension $K' / K$ over which $\Yrst$ (or some semistable model of $Y$) is defined under various hypotheses on $f$ (or specify only its degree or its maximal tame subextension) and apply such results to other arithmetic questions (for instance, involving division fields of the Jacobian variety of $Y$).

\subsection{Outline of the paper} \label{sec introduction outline}

While our priority in this article is considering the case where the residue characteristic $p$ of our ground field $K$ is $2$, we try to be as general as possible so that we may at times compare and contrast the situation of $p = 2$ with the situation of $p \neq 2$, often considering the latter as a special case which yields more simply-stated results.  We shall state the results for $p \neq 2$ using our own set-up and terminology which arises naturally from our method of recovering them but, as we have already mentioned, equivalent results do already appear in the literature (particularly in \cite{dokchitser2022arithmetic}).  Beginning in \S\ref{sec models hyperelliptic} and throughout the rest of the paper, we often refer to the two cases as ``the $p = 2$ setting" and ``the $p \neq 2$ setting".

The rest of our paper is organized as follows. First, we establishing the algebro-geometric setting that we need in \S\ref{sec semistable models}, which begins with briefly providing the basic background definitions and facts relating to models of curves over local rings, and then proceeds to look more closely at the properties of the special fiber of such a model and how to compare two models of the same curve by considering $(-1)$-lines and $(-2)$-curves (see Definitions \ref{MinusLines2} and \ref{dfn minus two curve} below).  All of this set-up allows us in the following section to define a particular ``nice" semistable model of a curve $Y$ which is a Galois cover of another curve $X$, which we call the \textit{relatively stable model} $\Yrst$ of $Y$ (see \Cref{dfn relatively stable} below) and which is the main topic of \S\ref{sec relatively stable}.  Viewing a hyperelliptic curve $Y / K$ as a degree-$2$ (Galois) cover of the projective line $\proj_K^1 =: X$, the relatively stable model of $Y$ is the one directly treated in the rest of this paper.

After this rather general set-up, we specialize to considering models of hyperelliptic curves over discrete valuation rings.  As a hyperelliptic curve is (by definition) a double cover of a projective line, we first look at models of projective lines over discrete valuation rings; the well-known characterization of such models is summarized in \S\ref{sec models hyperelliptic equations}.  Then in the rest of \S\ref{sec models hyperelliptic}, we look at models of hyperelliptic curves from the point of view of algebraic equations which define them.  More precisely, we derive equations which define normalizations of smooth models of the projective line (possibly looking over finite extensions of $K$) in the function field of the hyperelliptic curve $Y$.  In the $p = 2$ setting, we describe how we use \emph{part-square decompositions} (see \Cref{dfn qrho} below) of the defining polynomial of $Y$ to find these normalizations.

We next turn our attention to clusters in \S\ref{sec clusters}, laying out the definitions of \emph{clusters} and \emph{cluster data} as in \cite{dokchitser2022arithmetic} (and other subsequent works) as well as introducing \emph{valid discs} (see \Cref{dfn valid disc} below), which by definition correspond more or less to the smooth models of the projective line comprising the model $\Xrst$ of the projective line of which the relatively stable model $\Yrst$ is the normalization in $K(Y)$.  In this section, we essentially recover (as \Cref{thm cluster p odd}) the method of constructing a semistable model of $Y$ according to cluster data in the $p \neq 2$ setting by showing that in this case, there is a one-to-one correspondence between valid discs and non-singleton clusters.  The closest that we can come to an analogous statement for the $p = 2$ setting is then presented as \Cref{thm cluster p=2} (which provides some of the statements of \Cref{thm introduction main}), but we defer the proof this theorem to \S\ref{sec depths}.

The next two sections of our paper focus on developing a method of finding valid discs for any particular hyperelliptic curve $Y$.  The objective of \S\ref{sec depths} is an investigation of how to determine the existence and find the radius of a valid disc with a given center, whereas the goal in \S\ref{sec centers} is to show how to find those elements of $\bar{K}$ which are centers of valid discs.  The main focus in \S\ref{sec depths} is on finding valid discs containing a given cluster of roots.  In the course of developing the methods presented in this section, given a polynomial $f$, we define lower-degree polynomials $f^{\mathfrak{s}}_+$ and $f^{\mathfrak{s}}_-$ determined by a particular even-cardinality cluster $\mathfrak{s}$ of roots of $f$ such that part-square decompositions of $f^{\mathfrak{s}}_\pm$ can be used to determine the existence and depths of valid discs containing $\mathfrak{s}$.  One of the main findings is that an even-cardinality cluster $\mathfrak{s}$ has $2$ associated valid discs if and only if the \emph{depth} of $\mathfrak{s}$ exceeds a certain ``threshold" $B_{f,\mathfrak{s}} \in \qq$ as in \Cref{thm introduction main}(c).  In \S\ref{sec depths thresholds}, we present and prove a number of results which give exact formulas or estimates of $B_{f,\mathfrak{s}}$ that apply to various situations, in particular proving the inequalities in \Cref{thm introduction main}(d).  In \S\ref{sec depths algorithm} we present an algorithm for finding useful part-square decompositions of $f^{\mathfrak{s}}_\pm$ (\Cref{algo sufficiently odd}).  Meanwhile, in \S\ref{sec centers}, we characterize the centers of valid discs by defining a polynomial (\Cref{dfn F} below) whose roots are centers of all valid discs, with certain exceptions, as described by \Cref{thm centers}.  The results of \S\ref{sec depths} and \S\ref{sec centers} together show how all valid discs may be found; in particular, \Cref{thm centers} is guaranteed (by \Cref{cor centers}) to find centers of all valid discs which do not contain any roots of the defining polynomial $f$.

In \S\ref{sec structure}, we proceed to examine the structure of the special fiber of our desired semistable model $\Yrst$ given knowledge of the valid discs containing particular clusters of roots.  In this section, we show that in the situation of \Cref{thm introduction main}(c)(i) above, the guaranteed pair of valid discs, under certain circumstances, produces a loop in the graph of components of the special fiber of $\Yrst$, or in other words, increases the toric rank of the hyperelliptic curve by $1$.  This allows us to present (as \Cref{thm toric rank}) and prove a formula for the toric rank in terms of viable valid discs, as seen in \Cref{thm introduction main}(d).

Finally, we devote \S\ref{sec computations} to providing more direct formulas for the aforementioned polynomial $F$ as well as the bounds $B_{f,\mathfrak{s}}$ for low-genus hyperelliptic curves, classified according to their associated cluster data (for the special case of genus $1$, that is, for elliptic curves, this recovers the results which were presented and proved in a more concretely elementary way in \cite{yelton2021semistable}).  In particular, \Cref{thm g=2} describes the possible structures of the special fiber of $\Yrst$ for genus-$2$ hyperelliptic curves classified according to their cluster data and broken into cases depending on valuations of certain elements of $\bar{K}$ associated to the defining polynomial.

\subsection{Notations and conventions} \label{sec introduction notation}

Below we outline our notation and conventions for this paper.

Firstly, whenever we use interval notation (e.g.\ $[a, b]$, $(a, b)$, $(a, +\infty)$, etc.), the bounds will always be elements of $\qq \cup \{\pm\infty\}$, and the interval will be understood to consist of all \textit{rational} numbers (rather than all real numbers) between the bounds; i.e.\ we have $[a, b] = [a, b] \cap \qq$; we have $[a, +\infty] = [a, +\infty) = [a, +\infty) \cap \qq$; etc.

\subsubsection{Rings, fields, and valuations}

We will adhere to the following assuptions:
\begin{itemize}
    \item $K$ is a field endowed with a discrete valuation $v : K \to \qq \cup \lbrace +\infty \rbrace$, complete with respect to $v$; when studying hyperelliptic curves over $K$ (i.e., from \S\ref{sec models hyperelliptic} on), we will also always assume that the characteristic of $K$ is $\neq 2$;
    \item $R = \OO_K = \{z \in K \ | \ v(z) \geq 0\}$ is the ring of integers of $K$;
    \item $k$ is the residue field of $R$ (and of $K$), which we assume to be algebraically closed;
    \item $p$ is the characteristic of $k$ (that is, $p$ is the residue characteristic of $K$);
    \item thanks to the completeness of $K$, given any algebraic extension $K' \supseteq K$, the valuation $v: K\to \qq  \cup \lbrace +\infty \rbrace$ extends uniquely to a valuation on $K'$ which we also denote by $v: K'\to \qq  \cup \lbrace +\infty \rbrace$: this turns $K'$ into a non-archimedean field with residue field $k$, whose ring of integers will be denoted $R':=\OO_{K'}$; when the extension $K'/K$ is finite, $K'$ is actually a complete discretely-valued field, and $R'$ is hence a complete DVR; and 
    \item $\bar{K}$ is an algebraic closure of $K$.
\end{itemize}

\subsubsection{Lines, hyperelliptic curves, and models} \label{sec notation he}

Beginning in \S\ref{sec models hyperelliptic}, the symbol $X$ will normally denote the projective line $\mathbb{P}^1_K$, and $x$ will be its standard coordinate.  Similarly, beginning in \S\ref{sec models hyperelliptic}, the symbol $Y$ will in general be used to denote a hyperelliptic curve of any genus $g \ge 1$ over $K$ and ramified over $\infty\in X(K)$ and endowed with a 2-to-1 ramified cover map $Y\to X$; over the affine chart $x\neq \infty$, $Y$ can be described by an equation of the form $y^2=f(x)$, with $f(x)\in K[x]$ a polynomial of odd degree $2g+1$. The set of the $2g+1$ roots of $f(x)$ will be denoted $\RR\subseteq \bar{K}$.  We  will use the notation $\Rinfty$ to mean the set of \emph{all} $2g+2$ branch points of $Y\to X$, including $\infty$.

In \S\ref{sec relatively stable}, we work with Galois covers in greater generality, and in that section $Y\to X$ indicates any Galois cover of smooth projective geometrically connected $K$-curves.

For convenience, we list the notation we will use relating to curves and models in the table below.

\begin{longtable}{ p{3.25cm} p{10.5cm} p{1.25cm} }
  \caption{Notation relating to a given curve $C / K$} \\
  Notation & Description & Section \\
  \hline \\ \vspace{.25em}
  $\CC / R$, ($\XX / R$, $\YY / R$) & a model of $C$ (or $X$ or $Y$) over the ring of integers of $K$ & \S\ref{sec preliminaries curves and models} \\ \vspace{.25em}
  $g(C)$ & the genus of $C$ & \S\ref{sec preliminaries special fibers} \\ \vspace{.25em}
  $\CC_s / k$ & the special fiber of a model $\CC / R$ & \S\ref{sec preliminaries special fibers} \\ \vspace{.25em}
  $a(V), m(V), w(V)$ & several integers attached to a component $V$ of $\CC_s$ & \S\ref{sec special fibers invariants} \\ \vspace{.25em}
  $\underline{w}(V)$ & a partition of $w(V)$ for a component $V \in \Irred(\YY_s)$, coming from a $G$-action on $\YY$ & \S\ref{sec relatively stable -2-curves} \\ \vspace{.25em}
  $a(\CC_s), t(\CC_s), u(\CC_s)$ & abelian, toric, and unipotent ranks of the special fiber $\CC_s / k$ & \S\ref{sec preliminaries abelian toric unipotent} \\ \vspace{.25em}
  $\Ctr(\CC, \CC')$ & the set of points of $\CC_s$ to which the irreducible components of $\CC'_s$ that do not appear in $\CC_s$ are contracted & \S\ref{sec preliminaries comparing} \\ \vspace{.25em}
  $\Irred(\CC_s)$ & set of irreducible components of the special fiber $\CC_s / k$ & \S\ref{sec preliminaries special fibers} \\ \vspace{.25em}
  $\Sing(\CC_s)$ & set of singular points of the special fiber $\CC_s / k$ & \S\ref{sec special fibers invariants} \\ \vspace{.25em}
  $\Cmin$, $\Cst$ & the minimal regular model and the stable model of $C$ & \S\ref{sec preliminaries regular}, \S\ref{sec preliminaries semistable} \\ \vspace{.25em}
  $\CC^{\mathrm{rst}}$ & the relatively stable model of $C$, given in \Cref{dfn relatively stable} & \S\ref{sec relatively stable definition} \\ \vspace{.25em}
  $\Xmini$, $\Xst$ & the quotients $\Ymini / G$ and $\Yst / G$ given a $G$-Galois cover $Y \to X$ & \S\ref{sec relatively stable galois covers} \\ \vspace{.25em}
  $\Xrst$ & the quotient $\Yrst / G$ given a $G$-Galois cover $Y \to X$ & \S\ref{sec relatively stable definition} \\ \vspace{.25em}
  $\Gamma(\CC_s)$ & the dual graph of the special fiber & \S\ref{sec preliminaries semistable}
\end{longtable}

\subsubsection{Polynomials, discs, and clusters}

Let $h \in \bar{K}[z]$ be a polynomial; we denote its degree by $\deg(h)$.  Then we extend the valuation $v : \bar{K} \to \qq \cup \lbrace +\infty\rbrace$ to the \emph{Gauss valuation} $v: \bar{K}[z]\to \qq \cup \lbrace +\infty\rbrace$; that is, for any polynomial $h(z) := \sum_{i = 0}^{\deg(h)} H_i z^i \in \bar{K}[z]$, we set 
$$v(h) = v\left(\sum_{i = 0}^{\deg(h)} H_i z^i\right) = \min_{1 \leq i \leq {\deg(h)}} \{v(H_i)\}.$$

In many situations, we will also invoke the operation of taking a \textit{normalized reduction} of a polynomial over $\bar{K}$, defined as follows.

\begin{dfn} \label{dfn normalized reduction}
    A \emph{normalized reduction} of a nonzero polynomial $h(z) \in \bar{K}[z]$ is the reduction in $k[z]$ of $\gamma^{-1}h$, where $\gamma \in \bar{K}^{\times}$ is some scalar satisfying $v(\gamma) = v(h)$.
\end{dfn}

\begin{rmk} \label{rmk normalized reduction}

Clearly a normalized reduction of a polynomial $h(z)$ is a nonzero polynomial in $k[x]$ and is unique up to scaling; thus, the degrees of the terms appearing in the normalized reduction (which is what we will be chiefly interested in for our purposes) do not depend on the particular choice of $\gamma \in K^{\times}$ in \Cref{dfn normalized reduction}.

\end{rmk}

By a \emph{disc} (of $\bar{K}$), we mean any subset of $\bar{K}$ of the form $D_{\alpha,b}:=\{x\in \bar{K}: v(x-\alpha)\ge b\}$ for some $\alpha\in \bar{K}$ and $b\in \qq$.  A \emph{cluster} of the set of roots $\mathcal{R} \subset \bar{K}$ is simply a non-empty intersection of $\mathcal{R}$ with a disc of $\bar{K}$: see \Cref{dfn cluster} and \Cref{rmk cluster} below.

In this article, we will often speak of the \textit{depths} of clusters and of discs in $\bar{K}$; in the case of clusters, our definition of \textit{depth} is the one used throughout \cite{dokchitser2022arithmetic}.  More general, we can define depth for a subset of elements of $\bar{K}$ to measure how close the elements in the subset are to each other, as follows.
\begin{dfn}
    Given a subset $S \subset \bar{K}$, if a minimum of the valuations $v(\zeta- \zeta') \in \qq \cup \{+\infty\}$ among all elements $\zeta, \zeta' \in S$ exists, we call it the \emph{depth} of $S$.
\end{dfn}
In this article, all depths will be rational numbers so that there will always exist an element of $\bar{K}$ whose valuation is equal to any given depth.  Note that the depth of a disc is essentially minus a logarithm of its radius under the $p$-adic metric, and so a \textit{greater} depth corresponds to a \textit{smaller} disc.

For convenience, list the special notation for this paper that we will use relating to polynomials, discs, and clusters in the table below.

\begin{longtable}[htb]{ p{3cm} p{10cm} p{1.25cm} }
\caption{Notation relating to polynomials, discs and clusters} \\
Notation & Description & Section \\
\hline \\ \vspace{.25em}
$z_{\alpha, \beta}$ & the coordinate obtained from $z$ (resp. $h$) through translation by $\alpha \in \bar{K}$ and scaling by $\beta \in \bar{K}^\times$ & \S\ref{sec models hyperelliptic line} \\ \vspace{.25em}
$h_{\alpha, \beta}$ & the polynomial obtained from $h$ such that $h_{\alpha, \beta}(x_{\alpha, \beta}) = h(x)$ & \S\ref{sec models hyperelliptic line} \\ \vspace{.25em}
$\mathcal{X}_{\alpha, \beta}$ & the model of $X = \proj_K^1$ with coordinate $x_{\alpha, \beta}$ & \S\ref{sec models hyperelliptic line} \\ \vspace{.25em}
$D_{\alpha, b}$ & the disc centered at $\alpha$ with radius $b$ & \S\ref{sec models hyperelliptic line} \\ \vspace{.25em}
$D_{\mathfrak{s}, b}$ & the disc containing a cluster $\mathfrak{s}$ with radius $b$, given in \Cref{dfn linked} & \S\ref{sec cluster cluster definition} \\ \vspace{.25em}
$\mathcal{X}_D$ & the model of $X = \proj_K^1$ corresponding to a disc $D$ & \S\ref{sec models hyperelliptic line} \\ \vspace{.25em}
$\mathcal{X}_{\mathcal{D}}$ & the minimal model of $X = \proj_K^1$ dominating $\mathcal{X}_D$ for all discs $D$ in a collection $\mathcal{D}$ & \S\ref{sec models hyperelliptic line} \\ \vspace{.25em}
$\Drst$ & the collection of discs corresponding to $\Xrst$ & \S\ref{sec structure partitioning} \\ \vspace{.25em}
$\mathcal{D}_P$ & the set of discs $D' \in \Drst$ such that $\Ctr(\XX_D, \XX_{D'}) = \{P\}$ given a point $P \in \SF{\XX_D}$ & \S\ref{sec structure partitioning} \\ \vspace{.25em}
$\ell(\mathcal{X}_D, P)$ & given in \Cref{dfn ell ramification index} & \S\ref{sec models hyperelliptic separable} \\ \vspace{.25em}
$\mu(\mathcal{X}_D, P)$ & given in \Cref{dfn mu} & \S\ref{sec models hyperelliptic inseparable} \\ \vspace{.25em}
$\underline{v}_h(D)$ & the valuation of the polynomial $h_{\alpha, \beta}$ for any $\alpha$ and $\beta$ such that $D = D_{\alpha, v(\beta)}$ & \S\ref{sec depths piecewise-linear} \\ \vspace{.25em}
$t_{q, \rho}$ & a rational number associated to a part-square decomposition $h = q^2 + \rho$ & \S\ref{sec models hyperelliptic part-square} \\ \vspace{.25em}
$\underline{t}_{q, \rho}(D)$ & the difference $\underline{v}_\rho(D) - \underline{v}_f(D)$ given a part-square decomposition $f = q^2 + \rho$ & \S\ref{sec depths piecewise-linear} \\ \vspace{.25em}
$\mathfrak{t}^{\mathfrak{s}}(D)$, $\mathfrak{t}^{\mathcal{R}}(D)$ & given in \Cref{dfn t fun}, applied to a cluster $\mathfrak{s}$ or to $\mathcal{R}$ & \S\ref{sec depths piecewise-linear} \\ \vspace{.25em}
$\mathfrak{s}'$ & the parent cluster of a cluster $\mathfrak{s}$, given in \Cref{dfn parent} & \S\ref{sec clusters} \\ \vspace{.25em}
$d_\pm(\mathfrak{s}), \delta(\mathfrak{s})$ & rational numbers relating to the depth of a cluster $\mathfrak{s}$, given in \Cref{dfn depth} & \S\ref{sec cluster cluster definition} \\ \vspace{.25em}
$d_\pm(\mathfrak{s}, \alpha)$, $\delta(\mathfrak{s}, \alpha)$ & variants of $d_\pm(\mathfrak{s})$, $\delta(\mathfrak{s})$ given in \Cref{dfn interval I alpha} & \S\ref{sec depths construction valid discs} \\ \vspace{.25em}
$I(\mathfrak{s})$ & a subset of $\qq$ associated to a cluster $\mathfrak{s}$, given in \Cref{dfn depth} & \S\ref{sec cluster cluster definition} \\ \vspace{.25em}
$I(\mathfrak{s}, \alpha)$ & variant of $I(\mathfrak{s})$ given in \Cref{dfn interval I alpha} & \S\ref{sec depths construction valid discs} \\ \vspace{.25em}
$J(\mathfrak{s}, \alpha)$ & a certain sub-interval of $I(\mathfrak{s}, \alpha)$ containing rational numbers $b$ such that $\mathfrak{t}^{\mathcal{R}}(D_{\alpha, b}) = 2v(2)$ & \S\ref{sec depths construction valid discs} \\ \vspace{.25em}
$b_\pm(\mathfrak{s}, \alpha)$ & the endpoints of the interval $J(\mathfrak{s}, \alpha)$ & \S\ref{sec depths construction valid discs} \\ \vspace{.25em}
$\partial^\pm \mathfrak{t}^\mathcal{R}(D)$ & given in \Cref{lemma ell and t function} & \S\ref{sec depths construction valid discs} \\ \vspace{.25em}
$\lambda_\pm(\mathfrak{s}, \alpha)$ & the slopes $\mp\partial^\pm \mathfrak{t}^{\mathcal{R}}(D_{\alpha, b_\pm})$ & \S\ref{sec depths construction valid discs} \\ \vspace{.25em}
$f^{\mathfrak{s}}$, $f^{\mathcal{R} \smallsetminus \mathfrak{s}}$ & given by the formulas in (\ref{eq factorization}) & \S\ref{sec depths separating roots factorizing} \\ \vspace{.25em}
$f_\pm^{\mathfrak{s}, \alpha}$ & given by the formulas in (\ref{eq standard form}) & \S\ref{sec depths separating roots std form} \\ \vspace{.25em}
$\mathfrak{t}_\pm^{\mathfrak{s}, \alpha}(b)$ & given by the formulas in (\ref{eq mathfrak t pm}) & \S\ref{sec depths separating roots factorizing} \\ \vspace{.25em}
$b_0(\mathfrak{t}_\pm^{\mathfrak{s}, \alpha})$ & the least value of $b$ at which $\mathfrak{t}_\pm^{\mathfrak{s}, \alpha}$ attains $2v(2)$ & \S\ref{sec depths separating roots reconstructing invariants} \\ \vspace{.25em}
$B_{f, \mathfrak{s}}$ & the ``threshold depth" given in \Cref{prop depth threshold} & \S\ref{sec depths thresholds} \\ \vspace{.25em}
$F(T)$ & given in \Cref{dfn F} & \S\ref{sec centers def}
\end{longtable}

\subsection{Acknowledgements} \label{sec introduction acknowledgements}

The authors would like to thank Fabrizio Andreatta for proposing that the first author, as work for his Masters thesis, join the early stages of the research project of the second author, as well as for providing guidance and helpful discussions to the first author throughout his research work in the Masters program.

\section{Semistable models of curves and their special fibers} \label{sec semistable models}

The purpose of this section is to recall and develop definitions and results on semistable models of general curves over discretely-valued fields, which will later be applied to hyperelliptic curves.

\subsection{Preliminaries on semistable models} \label{sec preliminaries}

In this subsection, we briefly recall a number of background results we will need about models of curves, for which our main reference will be \cite{liu2002algebraic}. In this section, $C$ is a smooth, geometrically connected, projective curve over a complete discretely-valued field $K$, whose ring of integers is denoted $R$, and whose residue field $k$ is assumed to be algebraically closed (see \S\ref{sec introduction notation}).

\subsubsection{Curves and models} \label{sec preliminaries curves and models}

A \emph{model} of $C$ (over $R$) is a normal, flat, projective $R$-scheme $\CC$ whose generic fiber is identified with $C$. The models of $C$ form a preordered set $\Models(C)$, the order relation being given by \emph{dominance}: given two models $\CC$ and $\CC'$ of $C$, we will write $\CC\le \CC'$ to mean that $\CC'$ dominates $\CC$, i.e.\ that the identity $\id: C\to C$ extends to a birational morphism $\CC'\to \CC$. The preordered set $\Models(C)$ is \emph{filtered}, meaning that given two models $\CC_1$ and $\CC_2$, it is always possible to find a model $\CC$ dominating them both.

\subsubsection{Special fibers of models and birational morphisms} \label{sec preliminaries special fibers}

The special fiber $\CC_s$ of a model $\CC$ of $C$ is (geometrically) connected and consists of a number of irreducible components $V_1, \ldots, V_n$; these components are projective, possibly singular curves over the residue field $k$, each one appearing in $\CC_s$ with a certain multiplicity (which is defined as the length of the local ring of $\CC_s$ at the generic point of the component). We will denote by $\Irred(\CC_s)=\{ V_1,\ldots, V_n\}$ the set of such components. Since $\CC \to \Spec(R)$ is proper and flat, the Euler-Poincaré characteristic of the generic fiber $C$ and that of the special fiber $\CC_s$ coincide; in other words, the genus $g(C)$ of the smooth $K$-curve $C$ coincides with the arithmetic genus $p_a(C_s)$ of the $k$-curve $\CC_s$ (see, for example, \cite[Proposition 8.3.28]{liu2002algebraic}).

When $\CC$ and $\CC'$ are two models such that $\CC\le \CC'$, the image of a component $V'$ of $\CC'_s$ through the birational morphism $\CC'\to \CC$ is either a component $V$ of $\CC_s$ or a single point $P$ of $\CC_s$; in this second case, we say that the birational morphism $\CC' \to \CC$ \emph{contracts} $V'$. The rule $V'\mapsto V$ defines a one-to-one correspondence between the irreducible components of $\CC'_s$ that are not contracted by $\CC'\to \CC$ and the irreducible components of $\CC_s$; we say that $V$ is the \emph{image} of $V'$ in $\CC_s$, and $V'$ the \emph{strict transform} of $V$ in $\CC'_s$. In other words, taking strict transforms defines an injection $\Irred(\CC_s)\hookrightarrow \Irred(\CC'_s)$, and $\Irred(\CC'_s)\setminus \Irred(\CC_s)$ is the set of the irreducible components of $\SF{\CC'}$ that the birational morphism contracts.

The birational morphism $\CC'\to \CC$ is an isomorphism precisely over the open subscheme $\CC\setminus \lbrace P_1, \ldots, P_n \rbrace$, where the $P_i$'s are the points of the special fiber of $\CC$ to which some $V'\in \Irred(\CC'_s)\setminus \Irred(\CC_s)$ contracts. The fiber of $\CC'_s$ above each $P_i$ is connected of pure dimension 1, and its irreducible components are those components of $\CC'_s$ that contract to $P_i$. If $V$ is a component of $\CC_s$ and $V'$ is its strict transform in $\CC'_s$, then the birational morphism of models $\CC'\to \CC$ restricts to a birational morphism of $k$-curves $V'\to V$.

\subsubsection{Comparing models} \label{sec preliminaries comparing}

Suppose that $\CC$ and $\CC'$ are two models of $C$. We can compare them by looking at the components of their special fibers. To this aim, let us make the auxiliary choice of a model $\CC''$  dominating them both, so that we can think of $\Irred(\CC_s)$ and $\Irred(\CC'_s)$ as two subsets of a common larger set, namely $\Irred(\CC''_s)$. We will denote by $\Ctr(\CC,\CC')\subseteq \CC_s(k)$ the set of points $P\in \CC_s$ such that there exists an irreducible component $V''\in \Irred(\CC''_s)$ that is the strict transform of some component of $V'\in \Irred(\CC'_s)$ and contracts to $P$.
It is clear that the formation of $\Ctr(\CC,\CC')$ does not depend on the choice of $\CC''$.
\begin{rmk}
    In the language of \cite[Subsection 8.3.2]{liu2002algebraic}, $\Ctr(\CC,\CC')$ is just the set of the centers in $\CC$ of the $R$-valuations of $K(C)$ that are of the first kind in $\CC'$ but not in $\CC$.
\end{rmk}

Roughly speaking, this is the way that one should think of $\Ctr(\CC,\CC')$: any given component $V\in \Irred(\CC'_s)$ is either also present in the special fiber of $\CC$ (i.e., $V\in \Irred(\CC_s)$) or it is not, in which case it is \emph{contracted} to some point $P_V\in \CC_s(k)$. The set $\Ctr(\CC,\CC')$ is simply the set of all such $P_V$'s, as $V$ varies in $\Irred(\CC'_s)\setminus \Irred(\CC_s)$. 
We clearly have that $\Ctr(\CC,\CC')=\varnothing$ (i.e., all the irreducible components of $\CC'_s$ are also present in $\CC_s$) if and only if $\CC$ dominates $\CC'$.

\subsubsection{Contracting components}
Given a model $\CC'$ of $C$ and any proper subset $\{ V_1, \ldots, V_n\}\subsetneq \Irred(\CC'_s)$, it is always possible to form a model $\CC$ of $C$ dominated by $\CC'$ such that the birational morphism $\CC'\to \CC$ contracts precisely the components $V_1, \ldots, V_n\in \Irred(\CC'_s)$; as a consequence, we have  $\Irred(\CC_s)=\Irred(\CC'_s)\setminus \{ V_1, \ldots, V_n\}$.

Given a finite number of models $\CC_1, \ldots, \CC_n$, one can form a minimal model $\CC$ dominating them all: it is enough to take any model $\CC'$ dominating them all, and then contract each $V\in \Irred(\CC'_s)$ that is not the strict transform of an irreducible component of $\SF{\CC_i}$ for some $i$. It is clear that $\Irred(\CC_s)$ coincides with the (non necessarily disjoint) union $\bigcup_i \Irred(\CC_i)$.

\subsubsection{Regular models} \label{sec preliminaries regular}

If we consider a model $\CC$ of $C$ that is regular (i.e.\ regular as an $R$-scheme), then it is possible to define the \emph{intersection number} of any two components of $\CC_s$; the resulting intersection matrix is negative semi-definite (see \cite[Chapter 9]{liu2002algebraic}). 
\begin{dfn}
    \label{MinusLines}
    A component $V$ of the special fiber of a regular model is said to be a \emph{(-1)-line} if it is a line (i.e., $V\cong \mathbb{P}^1_k$) and its self-intersection number is $-1$. Similarly, it is said to be a (-2)-line if it is a line whose self-intersection number equals $-2$.
\end{dfn}

If one contracts any set of (-1)-lines in a regular model, it remains regular.

Given any model $\CC$ of $C$, it is possible to find a regular model $\CC'$ dominating $\CC$. Moreover, among all regular models $\CC'$ dominating $\CC$, there is a minimum one (with respect to dominance), which is named the \emph{minimal desingularization} of $\CC$; it can be characterized as the unique regular model $\CC'$ dominating $\CC$ such that $\CC' \to \CC$ does not contract any (-1)-line of $\CC'_s$. If $\CC'$ is the minimal desingularization of $\CC$, then the birational morphism $\CC'\to \CC$ fails to be an isomorphism precisely above the points of ${\CC}_s$ at which $\CC$ is not regular.

If the genus of $C$ is positive, then, among all regular models of $\CC$, there is a minimum one (with respect to dominance).  This model is named the \emph{minimal regular model} and will be denoted by $\Cmin$; it can be characterized as the unique regular model of $C$ whose special fiber does not contain (-1)-lines.

\subsubsection{Semistable models} \label{sec preliminaries semistable}

A model $\CC$ of $C$ is said to be \emph{semistable} if its special fiber is reduced and its singularities (if there are any) are all nodes (i.e.\ ordinary double points). More generally, we say that a model $\CC$ of $C$ is \emph{semistable} at a point $P\in \CC_s$ if $\CC_s$ is reduced at $P$ and if $P$ is either a smooth point or a node of $\CC_s$. Given a point $P\in \CC_s$, if the model $\CC$ is semistable at $P$ then its completed local ring at $P$ has the form $R[[t]]$, if $P\in \CC_s$ is a smooth point, or $R[[t_1,t_2]]/(t_1 t_2 - a)$ for some $a\in R$, with $v(a)>0$ if $P\in \CC_s$ is a node. The integer $v(a)/v(\pi)\ge 1$, where $\pi$ is a uniformizer of $R$, is known as the \emph{thickness} of the node. A semistable model is regular precisely when all of its nodes have thickness equal to 1.

To describe the combinatorics of a semistable model $\CC$ of a curve $C$, one can form the \emph{dual graph} $\Gamma(\CC_s)$ of its special fiber, whose set of vertices is $\Irred(\CC_s)$ and whose edges correspond to the nodes connecting them.

The notions of (-1)-line and (-2)-line given in \Cref{MinusLines} for regular models can be extended to semistable ones as follows.

\begin{dfn}
    \label{MinusLines2}
    If  $\CC$ is a model, $V\in \Irred(\CC_s)$, and $\CC$ is semistable at the points of $V$, then $V$ is said to be a (-1)-line (resp.\ a (-2)-line) if it is a line (i.e., $V\cong \mathbb{P}^1_k$) and the number of nodes of $\CC_s$ lying on it is equal to 1 (resp.\ 2).
\end{dfn}
\begin{rmk}
    It is possible to show that, if $\CC$ is regular model that is semistable at the points of a component $V\in \Irred(\CC_s)$, then the definition above is consistent with the one given in \Cref{MinusLines}: this follows, for example, from the formula for self-intersection numbers given in \cite[Proposition 9.1.21(b)]{liu2002algebraic}.
\end{rmk}

\begin{lemma}
    \label{lemma no intersecting minus one lines}
    Suppose that $\CC$ is a model of $C$ that is semistable at the points of two components $V_1, V_2\in \Irred(\CC_s)$. If $V_1$ and $V_2$ are (-1)-lines, then they cannot intersect each other unless $g(C)=0$.
    \begin{proof}
        Since $V_1$ and $V_2$ are (-1)-lines, if they intersect each other at a node, they cannot intersect any other irreducible components of $\CC_s$. Since $\CC_s$ is connected, this implies that $\CC_s$ only consists of the two lines $V_1$ and $V_2$ crossing each other at a node, implying that $p_a(\CC_s)=0$ and consequently that $g(C)=0$.
    \end{proof}
\end{lemma}

Contracting (-1) and (-2)-lines does not ever disrupt semistability: more precisely, if $\CC'$ is a model that is semistable at the points of some components $V_1, \ldots, V_n$ of $\CC_s$, and the $V_i$'s happen to all be (-1) and (-2)-lines, then the model $\CC$ that is obtained from $\CC$ by contracting all the $V_i$'s is semistable at the points where the $V_i$'s contract. Desingularizing is also an operation that preserves semistability: if $\CC$ is semistable at a point $P\in \CC_s$, and $\CC'$ is its minimal desingularization, then $\CC'$ is still semistable at all points lying above $P$; moreover, the desingularization $\CC'$ is easy to describe:
\begin{enumerate}[(a)]
    \item if $P$ is a smooth point of $\CC_s$, we have that $\CC$ is regular at $P$, and the birational map $\CC'\to \CC$ is consequently an isomorphism above $P$; and 
    \item if $P$ is a node of thickness $t$, the inverse image of $\CC'_s$ at $P$ consist of a chain of $t$ nodes of thickness 1, joined by $t-1$ (-2)-lines (see \cite[Section 5.3]{liu2002algebraic}).
\end{enumerate}

More generally, if $\CC$ is a model that is semistable at a point $P$ and if $\CC'$ is any model dominating $\CC$ but dominated by its minimal desingularization, then $\CC'$ is semistable at the points above $P$, and we have the following:
\begin{enumerate}[(a)]
    \item if $P$ is a smooth point of $\CC_s$, then the birational map $\CC'\to \CC$ is an isomorphism above $P$;
    \item if $P$ is a node of thickness $t$, then the inverse image of $\CC'_s$ at $P$ consist of a chain of $m\ge 1$ nodes whose thicknesses add up to $t$, joined by $m-1$ (-2)-lines (see \cite[Section 5.3]{liu2002algebraic}).
\end{enumerate}

When a semistable model exists, we say that $C$ has \emph{semistable reduction} over $R$. By Theorem \ref{thm semistable reduction} above, any curve $C$ is guaranteed to have semistable reduction after replacing $R$ with a large enough finite extension.

If $C$ has semistable reduction and positive genus, its minimal regular model is semistable (\cite[Theorem 10.3.34]{liu2002algebraic}). If $C$ has semistable reduction and genus at least $2$, then the set of its semistable models has a minimum (with respect to dominance), which is named the \emph{stable model} of $C$; it is denoted by $\Cst$ and can be characterized as the unique semistable model of $C$ whose special fiber contains neither (-1)-lines nor (-2)-lines. The stable model $\Cst$ can be obtained from $\Cmin$ by contracting all the (-2)-lines appearing in its special fiber.

\subsubsection{Abelian, toric, and unipotent ranks} \label{sec preliminaries abelian toric unipotent}

Given a model $\CC$ of $C$, the Jacobian $\Pic^0(\CC_s)$ of the (possibly singular) $k$-curve $\CC_s$ is an extension of an abelian variety $A$ by a linear algebraic group, which in turn is an extension of a torus $T$ by a smooth unipotent algebraic group $U$. The ranks of $A$, $T$, and $U$ are respectively known as the abelian, toric, and unipotent rank of the special fiber $\CC_s$ and will be denoted by $a(\CC_s)$, $t(\CC_s)$, and $u(\CC_s)$ respectively: they are three non-negative integers adding up to the genus $p_a(\CC_s)=g(C)$: see, for example, \cite[Section 7.5]{liu2002algebraic} or \cite[Chapters 8 and 9]{bosch2012neron}. We have the following.
\begin{enumerate}[(a)]
    \item For all models $\CC$, the abelian rank $a(\CC_s)$ coincides with the sum $a(\CC_s)=\sum_{V\in \Irred(\CC_s)} a(V)$, where $a(V):=g(\widetilde{V})$ is the genus of the normalization $\widetilde{V}$ of $V$.
    \item If $\CC$ is a semistable model, the toric rank can be computed as $t(\CC_s)=\dim_k H^1(\Gamma(\CC_s))$, where $\Gamma(\CC_s)$ is the dual graph of $\CC_s$ (see \cite[Example 9.8]{bosch2012neron} for a proof); in other words, we have $t(\CC_s)=\Nnodes(\CC_s)-\Nirreducible(\CC_s)+1$, where $\Nnodes(\CC_s)$ denotes the number of nodes, and $\Nirreducible(\CC_s)$ is the number of irreducible components (i.e., the cardinality of $\Irred(\CC_s)$).
    \item The unipotent rank $u(\CC_s)$ is 0 if $\CC$ is a semistable model.
\end{enumerate}
Under certain hypotheses, the abelian, unipotent and toric rank do not depend on the chosen model. In particular, we have the following.
\begin{prop}
    \label{prop invariance of ranks}
    If $\CC$ and $\CC'$ are two models of $C$ over $R$, and if each of these models is either regular or semistable, then we have $a(\CC_s)=a(\CC'_s)$, $t(\CC_s)=t(\CC'_s)$, and $u(\CC_s)=u(\CC'_s)$.
    \begin{proof}
        It is clearly enough to prove the result (a) when $\CC$ and $\CC'$ are both regular, and (b) when $\CC$ is a semistable model and $\CC'$ is its minimal desingularization. In case (a), one can form a regular model $\CC''$ dominating them both and apply \cite[Lemma 10.3.40]{liu2002algebraic}. In case (b), it follows from the description given in \S\ref{sec preliminaries semistable} of the desingularization of a semistable model that $\CC'$ is also semistable, and we get $\Nnodes(\CC_s)-\Nirreducible(\CC_s)=\Nnodes(\CC'_s)-\Nirreducible(\CC'_s)$, which is to say $t(\CC_s)=t(\CC'_s)$  Moreover, since any new components introduced by the desingularization process are lines, we have $a(\CC_s)=a(\CC'_s)$, and finally, since both $\CC$ and $\CC'$ are semistable, we have $u(\CC_s)=u(\CC'_s)=0$.
    \end{proof}
\end{prop}

\subsubsection{Field extensions} \label{sec preliminaries field extensions}

Suppose we are given a finite extension $K'/K$ (which, under our assumptions, will necessarily be totally ramified, since the residue field $k$ of $K$ is algebraically closed); let $e_{K'/K}$ be the ramification index (which coincides, in our setting, with the degree of the extension), and let $R'\supseteq R$ denote the ring of integers of $K'$. We will freely say a \emph{model of $C$ over $R'$} to mean a model of $C':=C\otimes_K K'$ over $R'$ as defined in \S\ref{sec preliminaries curves and models}. Given a model $\CC$ of $C$ over $R$, it is possible to construct a corresponding model $\CC'$ of $C$ over $R'$, which is defined as the normalization of the base-change $\CC\otimes_R R'$. When $\CC$ has reduced special fiber (and hence, in particular, when $\CC$ is semistable), the scheme $\CC\otimes_R R'$ is already normal (for example, by Serre's criterion for normality), and so we have $\CC'=\CC\otimes_R R'$: in this last case, the special fibers $\CC'_s$ and $\CC_s$ are canonically isomorphic.

We remark that regularity is not preserved in general when $R$ gets extended: if $\CC$ is a regular model over $R$, the corresponding model $\CC'$ over an extension $R'$ may no longer be regular. Semistability, however, is preserved: whenever $\CC$ is semistable, $\CC'$ is semistable too; however, the thickness of each node of $\CC$ gets multiplied by the ramification index $e_{K'/K}$ in $\CC'$.

In this section we look more closely at the special fibers of models (over $R$) of a smooth projective geometrically connected $K$-curve $C$. In \S\ref{sec special fibers invariants}, in particular, we define a number of invariants attached to each component of the special fiber of a model, while in \S\ref{sec special fibers criterion}, we use them to state a criterion that allows us to identify those models of $C$ that are part of the minimal regular model when $C$ has semistable reduction.

\subsection{Invariants attached to a component of the special fiber}
\label{sec special fibers invariants}
Given a model $\CC$ of $C$ and a component $V\in \Irred(\CC_s)$, we consider several invariants attached to $V$, listed as follows:
\begin{enumerate}
    \item $m(V)$ denotes the multiplicity of $V$ in $\CC_s$;
    \item $a(V)$ denotes the \emph{abelian rank} of $V$, i.e.\ the genus of the normalization $\widetilde{V}$ of $V$;
    \item $w(V)$ is defined only when $m(V)=1$, and it denotes the number of singular points of $\CC_s$ that belong to $V$, each one counted as many times as the number of branches of $V$ at that point; in other words, if $\tilde{V}$ is the normalization of $V$, then $w(V)$ is the number of points of $\tilde{V}$ that lie over $V\cap \Sing(\CC_s)$, where $\Sing(\CC_s)$ is the set of singular points of $\CC_s$.
\end{enumerate}

We now show that, under appropriate assumptions, the integers $m$, $a$, and $w$ are left invariant when the model is changed.
\begin{lemma}
    \label{lemma inv m a w}
    Let $\CC'$ be another model of $C$ which dominates $\CC$, and let $V'$ denote the strict transform of $V$ in ${\CC'}_s$. Then we have $m(V')=m(V)$ and $a(V') = a(V)$. Moreover, if $\CC'$ is dominated by the minimal desingularization of $\CC$, we also have $w(V')=w(V)$.
    \begin{proof}
        For $m$ and $a$, the lemma immediately follows from the consideration that $\CC'_s\to \CC_s$ is an isomorphism away from a finite set of points of $\CC_s$. We will now prove the result for $w$.
        
        Let $Q$ be a point of $V'$ which lies over some $P\in V$. We claim that $\CC'_s$ is smooth (resp.\ singular) at $Q$ if and only if $\CC_s$ is smooth (resp.\ singular) at $P$. This is obvious whenever $\CC'\to \CC$ is an isomorphism above $P$. If $\CC'\to \CC$ is not an isomorphism above $P$, the claim follows from the two following observations.  Firstly, since we are assuming that $\CC'$ is dominated by the minimal desingularization of $\CC$, it must be the case that $\CC$ is not regular at $P$, and consequently that $\CC_s$ is singular at $P$.  Secondly, the fiber of $\CC'\to \CC$ above $P$ is pure of dimension 1, and it consists of those components $E_i$ of $\CC'_s$ that contract to $P$; the point $Q$ will thus belong not only to $V'$, but also to one of the $E_i$'s, so that $\CC'_s$ will certainly be singular at $Q$. This completes the proof of the claim.
        
        Now, since $\CC'\to \CC$ restricts to a birational morphism $V'\to V$ of $k$-curves, the set of branches of $V$ at a point $P\in \CC_s$ equals the set of branches of $V'$ at the points of $\CC'_s$ lying above $P$. If we combine this consideration with the claim we have just proved, we have that $V'\to V$ induces a bijection between the set $\mathfrak{B}$ of the branches of $V$ at the singular points of $\CC_s$ and the set $\mathfrak{B}'$ of the branches of $V'$ at the singular points of $\CC'_s$. The equality $w(V')=w(V)$ follows.
    \end{proof}
\end{lemma}

We now describe how the invariants we have defined allow us to detect (-1)-lines and (-2)-lines.
\begin{lemma}
    \label{Lemma12lines}
    Let $\CC$ be any model of $C$, and let $V$ be an irreducible component of $\CC_s$. Then,
    \begin{enumerate}
        \item[(a)] if $\CC$ is regular and $V$ is a (-1)-line of multiplicity 1, then $a(V)=0$ and $w(V)=1$;
        \item[(b)] if $\CC$ is regular and $V$ is a (-2)-line of multiplicity 1, then $a(V)=0$ and $w(V)\in \lbrace 1, 2\rbrace$;
        \item[(c)] if $\CC$ is semistable at the points of $V$, then $V$ is a (-1)-line if and only if $a(V)=0$ and $w(V)=1$; and 
        \item[(d)] if $\CC$ is semistable at the points of $V$, then $V$ is a (-2)-line if and only if $a(V)=0$ and $w(V)=2$ (the reverse implication only holds if $g(C)\geq 2$).
    \end{enumerate}
    \begin{proof}
        If $V$ is a component of multiplicity 1 in the special fiber $\CC_s$ of a regular model $\CC$, then it follows from the intersection theory of regular models (see \cite[Chapter 9]{liu2002algebraic}) that its self-intersection number of $V$ is equal to minus the number of points at which $V$ intersects the other components of $\CC_s$, each counted with a certain (positive) multiplicity. Once this has been observed, parts (a) and (b) follow immediately from the definition of (-1)-lines and (-2)-lines for regular models (\Cref{MinusLines}).
        
        Suppose now that $V$ is a component of the special fiber $\CC_s$ of a model $\CC$ that is semistable at the points of $V$ (which, in particular, implies $m(V)=1$). From the definition of the invariant $w$ and the structure of semistable models, it is clear that $w(V)$ equals the sum $2w_{\text{self}}(V) + w_{\text{other}}(V)$, where $w_{\text{self}}(V)$ is the number of self-intersections of $V$, while $w_{\text{other}}(V)$ is the number of intersections of $V$ with other components of $\CC_s$; moreover, we have $w_{\text{self}}(V)=0$ if and only if $V$ is smooth.
        But the line $\mathbb{P}^1_k$ is the unique smooth $k$-curve with abelian rank 0, so the component $V$ is a line if and only if $a(V)=0$ and $w_{\text{self}}(V)=0$; according to \Cref{MinusLines2}, the component $V$ is thus a (-1)-line or a (-2)-curve if and only if $a(V)=0$, $w_{\text{self}}(V)=0$, and $w_{\text{other}}(V)$ equals $1$ or 2 respectively.
        
        From the considerations above, both implications of (c), as well the forward implication of (d), immediately follow. To prove the reverse implication of (d), one has to exclude the possibility that $a(V)=0$, $w_{\text{self}}(V)=1$, and $w_{\text{other}}(V)=0$. But if this were the case, the unique irreducible component of $\CC_s$ would be $V$ (because $w_{\text{other}}(V)=0$, but $\CC_s$ is connected), and the special fiber $\CC_s$ would consequently be a reduced $k$-curve having arithmetic genus equal to that of $V$, which is $a(V)+w_{\text{self}}(V)=1$. Since the arithmetic genus of $\CC_s$ coincides with $g(C)$, we would get $g(C)=1$; we therefore get the reverse implication of (d) as long as $g(C) \neq 1$.
    \end{proof}
\end{lemma}

Inspired by the above lemma, we make the following definition.
\begin{dfn} \label{dfn minus two curve}
    Given a model $\CC$ and an irreducible component $V$ of $\CC_s$ at whose points $\CC$ is semistable, the component $V$ is said to be a \emph{(-2)-curve} of $\CC_s$ if $m(V)=1$, $a(V)=0$ and $w(V)=2$.
\end{dfn}
\begin{rmk}
    \label{Minus2LinesCurves}
    \Cref{Lemma12lines} ensures that, if $V$ is a component of $\CC_s$ and $\CC$ is semistable at the points of $V$, then, when $V$ is (-2)-line, it is a (-2)-curve, and the converse also holds provided that $g(C)\neq 1$. If $g(C)=1$, the proof of \Cref{Lemma12lines} shows that $V$ may be a (-2)-curve without being a (-2)-line, and this happens precisely when $V$ is the unique component of $\CC_s$ and it is a $k$-curve of abelian rank 0 intersecting itself once (which is to say, a projective line with two points identified).
\end{rmk} 

The properties of being a (-1)-line or a (-2)-curve are preserved and reflected under desingularization in the semistable case; this is the reason why the notion of a (-2)-curve (rather than a (-2)-line) will turn out to be more convenient for us.
\begin{prop}
    \label{Minus12LinesStrictTransform}
    Let $\CC$ be a model; let $V$ be a component of $\CC_s$; and let $\CC'$ be a model dominating $\CC$ but dominated by the minimal desingularization of $\CC$. Assume that $\CC$ is semistable at the points of $V$. Let $V'$ denote the strict transform of $V$ in $\CC'_s$.  We have that $\CC'$ is semistable at the points of $V'$, and $V'$ is a (-1)-line (resp.\ a (-2)-curve) if and only if $V$ is.
    \begin{proof}
        The fact that $\CC'$ is semistable at the points of $V'$ has been discussed in \S\ref{sec preliminaries semistable}. We have seen how, in the present setting, being a (-1)-line or a (-2)-curve is something that can be characterized by means of the invariants $a$ and $w$. Hence, the result follows from \Cref{lemma inv m a w}.
    \end{proof}
\end{prop}

\subsection{A criterion for being part of the minimal regular model}
\label{sec special fibers criterion}
As initial evidence of the usefulness of the invariants $a$, $m$, $w$ introduced before, we provide a criterion for a model $\CC$ to be part of the minimal regular model $\Cmin$ in the case that $C$ has semistable reduction.
\begin{prop}
    \label{prop part of min}
    Assume that $g(C)\ge 1$ and that $C$ has semistable reduction. Let $\CC$ be any model. Then $\CC\le \Cmin$ if and only if for each component $V$ of $\CC_s$, we have
    \begin{enumerate}[(i)]
        \item $m(V)=1$, and 
        \item either $a(V)\ge 1$ or $w(V)\ge 2$.
    \end{enumerate}
    \begin{proof}
        First assume that we have $\CC\le \Cmin$.  Then the invariants $a$, $m$, and $w$ of a vertical component $V$ of $\CC$ must be equal to those of its strict transform $V^{\mathrm{min}}$ in $\Cmin$, thanks to \Cref{lemma inv m a w} (more generally, they remain the same in any model lying between $\CC\le \Cmin$). Since $\Cmin$ is semistable, its special fiber is reduced; thus, we get $m(V)=1$. Let us now assume that $a(V)=0$. If it were the case that $w(V)=0$, then $\CC_s=V$ would be a line; since the arithmetic genus of $\CC_s$ coincides with $g(C)$, this contradicts the condition that $g(C) \geq 1$. If we had $w(V)=1$, then, via \Cref{Lemma12lines}(c), $V^{\mathrm{min}}$ would be a (-1)-line, which is impossible, since the minimal regular model does not contain (-1)-lines. Thus, the quantity $w(V)$ is necessarily at least $2$.
        
        Now assume that for each component $V$ of $\CC_s$, the conditions (i) and (ii) given in the statement hold.  Let $\CC'$ be the minimal desingularization of $\CC$. Assume by way contradiction that $\CC'_s$ contains a (-1)-line. Since the desingularization $\CC'$ is minimal, such a (-1)-line must necessarily be the strict transform $V'$ of some component $V\in \Irred(\CC_s)$. By \Cref{lemma inv m a w}, the quantities $a(V')$, $m(V')$ and $w(V')$ are equal to $a(V)$, $m(V)$ and $w(V)$ respectively. Thus, from the condition $m(V)=1$, we deduce $m(V')=1$; since $V'$ is a (-1)-line of multiplicity 1,  Lemma \ref{Lemma12lines}(a) ensures that $a(V') = 0$ and $w(V') = 1$, hence $a(V)=0$ and $w(V)=1$. But this contradicts our hypothesis, so we conclude that $\CC'_s$ cannot contain a (-1)-line.  It follows that $\CC'=\Cmin$, and we get $\CC\le \Cmin$ as desired.
    \end{proof}
\end{prop}

\section{The relatively stable model} \label{sec relatively stable}

In this section, we assume that $Y \to X$ is a Galois cover of smooth projective geometrically connected curves over $K$; let $G:=\Aut_{X}(Y)=\Aut_{K(X)}(K(Y))$ denote the Galois group. We remark that, by the Riemann-Hurwitz formula, we always have that $g(Y)\ge g(X)$.

In \S\ref{sec relatively stable galois covers} we collect some well-known background results about semistable models of Galois covers, for which good references are \cite{liu2002algebraic} and \cite{liu1999lorenzini}. In \S\ref{sec relatively stable -2-curves} and in \S\ref{sec relatively stable vanishing persistent}, we study nodes and (-2)-curves of a semistable model of $Y$ with respect to the cover $Y\to X$: these preliminaries allow us to define, in \S\ref{sec relatively stable definition}, a particular semistable model of $Y$ (relative to the Galois cover $Y\to X$) that we name the \emph{relatively stable model} of $Y$; we denote it $\Yrst$, and it arises as the normalization in $K(Y)$ of a semistable model $\Xrst$ of the line $X$. Existence and uniqueness results for $\Yrst$  hold provided that we have $g(Y)\ge 2$ or that we have $g(Y)=1$ and $g(X)=0$ (existence is only guaranteed if one allows replacing $R$ with a large enough extension).  In \S\ref{sec relatively stable finding} we will explain some methods for detecting the components of $\Yrst$.

\subsection{Models of Galois covers} \label{sec preliminaries models}
\label{sec relatively stable galois covers}
In the setting describe above, one can produce models of $Y$ from models of $X$. More precisely, to each model $\XX$ of $X$ we can attach a corresponding model $\YY$ of $Y$ by taking the normalization of $\XX$ in the function field $K(Y)$ -- we will say that $\YY$ \emph{comes from} $\XX$ (or that $\YY$ is the model of $Y$ \emph{corresponding to} $\XX$). Given two models $\XX$ and $\XX'$ of $X$, if $\YY$ and $\YY'$ are the corresponding models of $Y$, then it is easy to show that $\XX\le \XX'$ if and only if $\YY\le \YY'$: in other words, normalizing in $K(Y)$ defines an embedding of preordered sets $\Models(X) \mono \Models(Y)$. The essential image of the embedding consists of those models $\YY$ of $Y$ on which $G$ acts, i.e.\ those for which the action of $G$ on the generic fiber $Y$ extends (in a necessarily unique way) to an action on the $R$-scheme $\YY$. Given a model $\YY$ of $Y$ on which $G$ acts, the model of $X$ from which $\YY$ comes can be recovered as the quotient $\YY/G$.

If $\XX$ is a model of $X$ and $\YY$ is the corresponding model of $Y$, the set of irreducible components $\Irred(\YY_s)$ is a $G$-set, and we have $\Irred(\XX_s)=\Irred(\YY_s)/G$.  Given $\XX, \XX'\in \Models(X)$ and letting $\YY, \YY'$ be the corresponding models of $Y$, it is also not difficult to see that $\Ctr(\YY,\YY')=f^{-1}(\Ctr(\XX,\XX'))$ (see \S\ref{sec preliminaries comparing} for notation), where $f: \YY\to \XX$ is the cover map.

The minimal regular model $\Ymini$ and the stable model $\Yst$, when defined, are always acted upon by $G$; we will use the notation $\Xmini=\Ymini/G$ and $\Xst=\Yst/G$ to denote the models of $X$ from which they come.

\subsection{Vertical and horizontal (-2)-curves} \label{sec relatively stable -2-curves}

Given a model $\YY$ of $Y$ coming from a model $\XX$ of $X$ and a component $V\in \Irred(\YY_s)$ of multiplicity 1, we replace the invariant $w:=w(V)$ introduced in \S\ref{sec special fibers invariants} with a richer datum $\underline{w}:=\underline{w}(V)$ that takes into account the action of $G$. We have that $G$ acts on $\Irred(\YY_s)$, and we denote by $G_V$ the stabilizer of $V$ with respect to this action. If $\mathfrak{B}:=\lbrace b_1, \ldots, b_w\rbrace$ are the branches of $V$ passing through the singular points of $\YY_s$, the stabilizer $G_V$ clearly acts on $\mathfrak{B}$, and we define $\underline{w}(V)$ to be the partition of the integer $w(V)=|\mathfrak{B}|$ given by the cardinality of the orbits of the $G_V$-set $\mathfrak{B}$.
\begin{lemma}
    \label{lemma inv ww}
    If $\YY'$ is a model acted upon by $G$ which dominates $\YY$ and is dominated by the minimal desingularization of $\YY$, then we have $\underline{w}(V') = \underline{w}(V)$, where $V'$ is the strict transform of $V$ in $\SF{\YY'}$.
    \begin{proof}
        The set $\mathfrak{B}$ of the branches of $V$ passing through the singular points of $\YY_s$ does not change as we replace $\YY$ with $\YY'$, and $V$ with its strict transform, as was shown in the proof of Lemma \ref{lemma inv m a w}; since the birational map $\YY'\to \YY$ is $G$-equivariant, the stabilizers $G_V$ and $G_{V'}$ coincide, and $\mathfrak{B}$ is preserved not only as a set, but also as a $G_V$-set.
    \end{proof}
\end{lemma}

We recall that, when $\YY$ is semistable at the points of some component $V\in \Irred(\YY_s)$, we say that $V$ is a (-2)-curve whenever $a(V)=0$ and $w(V)=2$ (see \Cref{dfn minus two curve}). In our setting, the presence of a $G$-action on $\YY$ allow us to distinguish between \emph{vertical} and \emph{horizontal} (-2)-curves, according to the two possible values for the invariant $\underline{w}(V)$.
\begin{dfn} \label{defVerticalHorizontalMinusTwoLines}
    Given a component $V\in \Irred(\YY_s)$ such that $\YY$ is semistable at the points of $V$, if $V$ is a (-2)-curve we say that it is \emph{vertical} or \emph{horizontal} depending on whether $\underline{w}(V)=(2)$ or $\underline{w}(V)=(1,1)$.
\end{dfn}

The property of being a horizontal or vertical (-2)-line is preserved and reflected under desingularization of semistable models.
\begin{prop}
    \label{Minus12LinesStrictTransformRelative}
    Let $\YY'$ another model acted upon by $G$ which dominates $\YY$ but is dominated by the minimal desingularization of $\YY$. Given $V$ a component of $\YY_s$ such that $\YY$ is semistable at the points of $V$, if $V'$ denotes the strict transform of $V$ in $\YY'_s$, we have that $\YY'$ is semistable at the points of $V'$; moreover, the transform $V'$ is a (-1)-line (resp.\ a horizontal (-2)-curve, resp.\ a vertical (-2)-curve) if and only if $V$ is.
    \begin{proof}
        The proof is analogous to that of \Cref{Minus12LinesStrictTransform}, taking into account \Cref{lemma inv ww}.
    \end{proof}
\end{prop}

\subsection{Vanishing and persistent nodes} \label{sec relatively stable vanishing persistent}

Given a model $\YY$ of $Y$ corresponding to some model $\XX$ of $X$, we can ask ourselves how the properties of $\XX$ and $\YY$ are related to each other. If we write $f: \YY\to \XX=\YY/G$ for the quotient map, we present an important result concerning the semistability of $\XX$ and $\YY$.
\begin{prop}
    \label{prop vanishing persistent}
    In the setting above, we have that $\XX$ is semistable at $f(Q)$ whenever $\YY$ is semistable at some $Q\in \YY_s$. More precisely, we have the following.
    \begin{enumerate}[(a)]
        \item If $Q$ is a smooth point of $\YY_s$, then $f(Q)$ is a smooth point of $\XX_s$;
        \item If $Q$ is a node of thickness $t$ of $\YY_s$, we have two possibilities:
        \begin{enumerate}[(i)]
            \item if the stabilizer $G_Q\le G$ of $Q$ permutes the two branches of ${\YY}_s$ passing through $Q$, then $f(Q)$ is a smooth point of ${\XX}_s$;
            \item if, instead, the stabilizer $G_Q\le G$ of $Q$ does not flip the two branches of ${\YY}_s$ passing through $Q$, then $f(Q)$ is a node of $\XX_s$, and its thickness is $t|G_Q|$.
        \end{enumerate}
    \end{enumerate}
    In particular, if the model $\YY$ is semistable, then so is the model $\XX$.
    \begin{proof}
        The proof consists of an explicit local study of the quotient map $f: \YY\to \XX$, which can be found in \cite[Proposition 10.3.48]{liu2002algebraic}.
    \end{proof}
\end{prop}
 
\begin{dfn}
    \label{dfn vanishing persistent node}
    A node $Q$ of $\YY$ is said to be \emph{vanishing} or \emph{persistent} with respect to the Galois cover $Y\to X$ depending on whether it falls under case (i) or (ii) of \Cref{prop vanishing persistent}(b), i.e.\ depending on whether it lies above a smooth point or a node of $\XX_s$.
\end{dfn}

\begin{rmk}
    \label{rmk quotient dual graph}
    We have already remarked in \S\ref{sec relatively stable galois covers} that $f: \YY\to \XX$ induces a one-to-one correspondence between the irreducible components of $\XX_s$ and the $G$-orbits of irreducible components of $\YY_s$. \Cref{prop vanishing persistent} and \Cref{dfn vanishing persistent node} tell us that the nodes of $\XX$ correspond to the $G$-orbits of \emph{persistent} nodes of $\YY$.
\end{rmk}

We have seen in \S\ref{sec preliminaries semistable} that, if $Q$ is a node of thickness $t$ of $\YY$, its inverse image in the special fiber of the minimal desingularization of $\YY$ consists of a chain of $t$ nodes of thickness 1, connected by $t-1$ (-2)-curves. More generally, if $\YY'$ is any model acted upon by $G$ that dominates $\YY$ but is dominated by its minimal desingularization, then $\YY'$ is semistable at the points lying above $Q$, and the inverse image of $Q$ in $\YY'_s$ consists of a chain of $m$ nodes $Q_1, \ldots, Q_m$, whose thicknesses add up to $t$, and $m-1$ (-2)-curves $L_1, \ldots, L_{m-1}$ connecting them. It is an interesting question to ask whether the $Q_i$'s are persistent or vanishing, and whether the $L_i$'s are horizontal or vertical.
\begin{prop}
    \label{prop galois cover desingularizing}
    In the setting above, we have the following:
    \begin{enumerate}
            \item[(a)] if $Q\in \YY_s$ is a persistent node, then the $Q_i$'s also are, and the $L_i$'s are all horizontal;
             \item[(b)] if $Q\in \YY_s$ is vanishing and $m$ is odd, then $G_Q$ permutes $L_i$ and $L_{m-i}$; the $L_i$'s are all horizontal, while the $Q_i$'s are all persistent, apart from the middle one $Q_{(m+1)/2}$ which is vanishing; and 
            \item[(c)] if $Q\in \YY_s$ is vanishing and $m$ is even, then $G_Q$ permutes each $L_i$ with $L_{m-i}$; the $L_i$'s are all horizontal, apart from the middle one $L_{m/2}$ which is vertical, while the $Q_i$'s are all persistent.
    \end{enumerate}
    \begin{proof}
        We remark that, since the $Q_i$'s and the $L_i$'s have image $Q$ in $\YY_s$, we have $G_{Q_i}\le G_Q$ and $G_{L_i}\le G_Q$, and every $g\in G_Q$ acts on the set of the sets $\lbrace L_i\rbrace_{1 \leq i \leq m - 1}$ and $\lbrace Q_i\rbrace_{1 \leq i \leq m}$.  We denote by $\Lambda_-$ and $\Lambda_+$ the strict transforms in $\YY'_s$ of the two branches of $\YY_s$ passing through $Q$, so that the node $Q_1$ connects $\Lambda_-$ with $L_1$ and the node $Q_m$ connects $L_{m-1}$ with $\Lambda_+$.
           
        Choose an element $g\in G_Q$ which fixes $\Lambda_+$ and $\Lambda_-$. Since $Q_1$ is the unique point of $\Lambda_-$ lying above $Q$, the point $Q_1$ is also fixed by $g$; since $g$ fixes $Q_1$ and $\Lambda_-$, it must also fix the only other branch of $\YY_s$ passing through $Q_1$; therefore, it fixes $L_1$. Since $g$ fixes $Q_1$ and $L_1$, it must also fix the only other node that lies on $L_1$, namely $Q_2$.  Iterating the argument, one gets that the Galois element $g$ stabilizes each of the $L_i$'s and the $Q_i$'s. 
            
        Now choose an element $g\in G_Q$ that flips $\Lambda_+$ and $\Lambda_-$. Since $Q_1$ is a point of $\Lambda_-$, the point $g \cdot Q_1$ must belong to $\Lambda_+$ and lie above $Q$; we therefore have $g\cdot Q_1 = Q_m$. From the fact that $g\cdot Q_1=Q_m$ and $g\cdot \Lambda_-=\Lambda_+$, one deduces that $g\cdot L_1 = L_{m-1}$, and so on.  From this kind of iterative argument, the results follow (how it ends clearly depends on whether $m$ is even or odd).
    \end{proof}
\end{prop}

\subsection{Defining the relatively stable model} \label{sec relatively stable definition}

We are now ready to define the \emph{relatively stable model} of $Y$ with respect to the Galois cover $Y\to X$.
\begin{dfn} \label{dfn relatively stable}
    A model of $\YY$ of $Y$ is said to be \emph{relatively stable} with respect to the Galois cover $Y \to X$  if it is semistable, it is acted upon by $G$, and its special fiber does not contain vanishing nodes, (-1)-lines, or horizontal (-2)-curves. If a relatively stable model exists, the curve $Y$ is said to have \emph{relatively stable reduction} with respect to the cover $Y\to X$.
\end{dfn}

\begin{rmk}
    If $\YY$ is relatively stable and $\XX=\YY/G$, since $\YY$ cannot contain vanishing nodes, we have $\Sing(\YY_s)=f^{-1}(\Sing(\XX_s))$, where $f: \YY\to \XX$ is the cover map, while $\Sing(\XX_s)$ and $\Sing(\YY_s)$ are the respective sets of nodes of the semistable models $\XX$ and $\YY$. 
\end{rmk}

A relatively stable model $\Yrst$ can only exist if the curve $Y$ has semistable reduction; moreover, since $\Yrst$ is semistable and contains no (-1)-lines, it is clear that $\Yrst \leq \Ymini$ (provided that $\Ymini$ exists, i.e.\ $g(Y)\ge 1$). It is also clear from the definition that the property of being relatively stable is preserved and reflected under arbitrary extensions of $R$. Finally, we observe that, if the cover $Y\to X$ is trivial, a relatively stable model of $Y$ is nothing but a stable model.

\begin{prop}
    Assume that $g(Y)\ge 1$. The relatively stable model, if it exists, is unique.
    \begin{proof}
        Let $\YY_1$ and $\YY_2$ two relatively stable models, and let $\YY_3$ be the minimal model dominating them both; by minimality, each vertical component of $\SF{\YY_3}$ is the strict transform of a component of $\SF{\YY_1}$ or of a component of $\SF{\YY_2}$.
        Since $\YY_1$ and $\YY_2$ are $\le \Ymini$, we also have $\YY_3\le \Ymini$; moreover, the model $\YY_3$ is semistable since it lies between the semistable model $\YY_1$ and its minimal desingularization $\Ymini$. Since the special fibers of $\YY_1$ and $\YY_2$ only contain vertical (-2)-curves, the same is true for $\YY_3$, thanks to \Cref{Minus12LinesStrictTransformRelative}. Suppose by way of contradiction that $\YY_3 \gneq \YY_1$: this means that some node $Q$ of $\SF{\YY_1}$ is replaced, in $\SF{\YY_3}$, by a chain of $m$ nodes and $m-1$ (-2)-curves (with $m>1$). But since $\SF{\YY_3}$ does not contain horizontal (-2)-curves, \Cref{prop galois cover desingularizing} forces $Q$ to be vanishing of thickness 2, which is a contradiction, since $\SF{\YY_1}$ does not contain vanishing nodes.
        Hence, we have $\YY_1=\YY_3$, which is to say that $\YY_1\ge \YY_2$; now we get $\YY_1= \YY_2$ by symmetry.
    \end{proof}
\end{prop}

From now on, we use the symbol $\Yrst$ to denote the relatively stable model of $Y$ (whenever it exists), while $\Xrst=\Yrst/G$ denotes the model of the line $X$ to which it corresponds in the sense of \S\ref{sec relatively stable galois covers}.

\begin{lemma}
    \label{LemmaPolygonMinusTwoLines}
    Assume that $g(Y)\ge 1$ and that $Y$ has a semistable model $\YY$ acted upon by $G$ whose special fiber only consists of horizontal (-2)-curves connected by persistent nodes. Then we have $g(Y)=g(X)=1$. 
    \begin{proof}
            Let $\XX=\YY/G$ be the semistable model of $X$ from which $\YY$ comes. Since all irreducible components of $\YY_s$ have abelian rank 0, the same is also be true for all irreducible components of $\XX_s=\YY_s/G$; if $a$ denotes the abelian rank, we thus have that $a(\YY_s)=a(\XX_s)=0$. 
            
            Since $\YY_s$ only consists of (-2)-curves, we have that its dual graph $\Gamma(\YY_s)$ is a polygon with $N\ge 1$ sides (when $N=1$, it consists of a single vertex, and a loop around it). We clearly have an action of $G$ on $\Gamma(\YY_s)$; moreover, the absence of vanishing nodes ensures that $\Gamma(\XX_s)=\Gamma(\YY_s)/G$ (see \Cref{rmk quotient dual graph}).
            
            Let $\sigma_g$ be the automorphism of the polygon $\Gamma(\YY_s)$ induced by an element $g\in G$. Then, $\sigma_g$ cannot be a reflection: this is because a reflection either fixes a vertex and flips the two edges it connects, or it fixes an edge flipping the two vertices lying on it; in the first case, $\YY_s$ would contain a vanishing node, and in the second one it would contain a vertical (-2)-curve. Hence, the automorphism $\sigma_g$ is necessarily a rotation; the image of $G$ in $\Aut(\mathcal{G})$ is consequently a cyclic subgroup consisting of $d$ rotations for some $d|N$, and $\Gamma(\XX_s)=\Gamma(\YY_s)/G$ is consequently a polygon with $N/d$ sides. If $t$ denotes the toric rank, we consequently have $t(\XX_s)=t(\YY_s)=1$.
            
            By the results in \S\ref{sec preliminaries abelian toric unipotent}, we can now conclude that $g(Y)=g(X)=1$.   
    \end{proof}
\end{lemma}

\begin{prop} \label{prop rst exists}
    Assume that $g(Y)\ge 2$, or that $g(Y)=1$ and $g(X)=0$. The relatively stable model exists if and only if $Y$ has semistable reduction and $\Ymini$ contains no vanishing nodes.  If it does exist, the relatively stable model $\Yrst$ can be formed by contracting all horizontal (-2)-curves of the minimal regular model $\Ymini$.
    
    \begin{proof}
        Suppose the curve $Y$ has semistable reduction and that $\Ymini$ contains no vanishing nodes.  By \Cref{LemmaPolygonMinusTwoLines}, the special fiber $\SF{\Ymini}$ cannot consist only of horizontal (-2)-curves, as that would contradict our hypothesis on genera.
        
        We are consequently allowed to form a model $\Yrst$ by contracting all horizontal (-2)-curves of $\Ymini$, and it will still be semistable.
        It is clear that $G$ acts on $\Yrst$. Suppose that  $\SF{\Yrst}$ contains a horizontal (-2)-curve.  Then its strict transform $\SF{\Ymini}$ is still a horizontal (-2)-curve in light of  \Cref{Minus12LinesStrictTransformRelative}, which contradicts the fact that all (-2)-curves of $\SF{\Ymini}$, by construction, get contracted in $\SF{\Yrst}$.
        
        Suppose now that $\SF{\Yrst}$ contains a vanishing node.  Then, in light of \Cref{prop galois cover desingularizing}, its preimage in $\SF{\Ymini}$ must contain a vanishing node or a vertical (-2)-curve, which is impossible since $\SF{\Ymini}$ does not contain vanishing nodes by assumption, and its vertical (-2)-curves do not get contracted in $\SF{\Yrst}$ by construction. Finally, since we have $\Yrst\le \Ymini$, the special fiber $\SF{\Yrst}$ contains no (-1)-line. Henceforth, the model $\Yrst$ is actually relatively stable.
    
        Let us now prove the converse implication. Suppose that the relatively stable model $\Yrst$ exists. By definition, its specil fiber $\SF{\Yrst}$ only contains persistent nodes; hence, by \Cref{prop galois cover desingularizing}, the special fiber $\SF{\Ymini}$ is obtained from $\SF{\Yrst}$ by replacing each of its node with an appropriate chain of horizontal (-2)-curves and persistent nodes and thus retains the property of not containing vanishing nodes.
    \end{proof}
\end{prop}

\Cref{prop rst exists} establishes a criterion to determine whether $Y$ has relatively stable reduction or not by looking at its minimal regular model. Hereafter we propose a refined version of such a criterion.

\begin{lemma}
    \label{LemmaMinusOneLinesVanishing}
    Assume that $g(Y)\ge 1$. Then, given a regular or a semistable model $\YY$ of $Y$, no (-1)-line of $\YY_s$ can pass through a vanishing node of $\YY_s$.
    \begin{proof}
        Suppose that $Q$ is a vanishing node, and let $L$ be a (-1)-line of $\YY_s$ passing through it. Let $g\in G$ be an element stabilizing $Q$ and flipping the two branches that pass through it. It is clear that $gL$ will be another (-1)-line of $\YY_s$ passing through $Q$. Since $\YY_s$ contains two intersecting (-1)-lines, we have $g(C)=0$ by \Cref{lemma no intersecting minus one lines}, which contradicts our hypothesis.
    \end{proof}
\end{lemma}

\begin{prop}
    \label{BigCriterionJeffReduction}
    Assume that $g(Y)\ge 2$, or that $g(Y)=1$ and $g(X)=0$. The following are equivalent:
    \begin{itemize}
        \item[(a)] $Y$ has relatevely stable reduction;
        \item[(b)] $Y$ has semistable reduction, and the vanishing nodes of all models $\YY$ of $Y$ acted upon by $G$ all have even thickness;
        \item[(c)] $Y$ has semistable reduction, and the vanishing nodes of some semistable model $\YY$ of $Y$ acted upon by $G$ all have even thickness;
        \item[(d)] $Y$ has semistable reduction, and some semistable model $\YY$ of $Y$ acted upon by $G$ does not contain any vanishing node.
    \end{itemize}
    \begin{proof}
        Let us prove (a)$\implies$(b). We will proceed by way of contradiction: we assume that there exists a model $\YY$ of $Y$ acted upon by $G$ which has a vanishing node $Q$ of odd thickness $t$; we need to prove that $\Ymini$ contains a vanishing node (see \Cref{prop rst exists}). We have that the minimal desingularization of $\YY$ is still semistable, and it also contains a vanishing node of odd thickness (see \Cref{prop galois cover desingularizing}), hence we lose no generality if we assume that $\YY$ is regular. Let us further assume, without loss of generality, that $\YY$ is minimal (with respect to dominance) among the regular semistable models of $Y$ acted upon by $G$ and carrying a vanishing node $Q$. If $\YY=\Ymini$, then we are done. If instead we have $\YY\gneq \Ymini$, then $\SF{\YY}$ contains a $G$-orbit of (-1)-lines $\lbrace gL: g\in G \rbrace$; let $\YY_1$ be the semistable model that is obtained by contracting it. Since none of the $gL$ can pass through the vanishing node $Q$ by \Cref{LemmaMinusOneLinesVanishing}, the birational map $f: \YY\to \YY_1$ is an isomorphism above $Q_1:=f(Q)$; in particular, the node $Q_1$ is still a vanishing node of odd thickness of the model $\YY_1$, which, by construction, is again regular, semistable and acted upon by $G$; this violates the minimality of $\YY$.
        
        The implication (b)$\implies$(c) is obvious; let us therefore prove (c)$\implies$(d). Let $\YY$ be some semistable model of $Y$ whose vanishing nodes all have even thickness; by \Cref{prop galois cover desingularizing}, the minimal desingularization $\YY'$ of $\YY$ will not contain a vanishing node, whence (d) follows. 
        
        Let us finally prove (d)$\implies$(a). If $\YY$ is a model acted upon by $G$ which contains no vanishing nodes, its minimal desingularization has the same property by \Cref{prop galois cover desingularizing}; hence we can assume without losing generality that $\YY$ is regular, and that it is moreover minimal (with respect to dominance) in the set of the regular semistable models of $Y$ acted upon by $G$ that do not contain vanishing nodes. If $\YY=\Ymini$, we are done by \Cref{prop rst exists}. If instead we have $\YY\gneq \Ymini$, then the special fiber $\YY_s$ contains a $G$-orbit of (-1)-lines $\lbrace gL: g\in G \rbrace$; if $\YY_1$ is the model we obtain by contracting them all, the birational map $\YY\to \YY_1$ is clearly an isomorphism above all nodes of $\YY_1$.  The model $\YY_1$ will also not contain any vanishing node, and by construction it is still semistable and regular; this violates the minimality of $\YY$.
    \end{proof}
\end{prop}

\begin{cor}
    Assume that $g(Y)\ge 2$, or that $g(Y)=1$ and $g(X)=0$. The curve $Y$ always has relatively stable reduction over a large enough finite extension of $R$.
    \begin{proof}
        After possibly extending $R$, we may assume that $Y$ has semistable reduction over $R$ by \Cref{thm semistable reduction}. Let $\YY$ be any semistable model of $Y$ acted upon by $G$. If the model $\YY$ does not contain a vanishing node of odd thickness, then the curve $Y$ has relatively stable reduction over $R$ by \Cref{BigCriterionJeffReduction}. Otherwise, let $R'$ be any extension of $R$ with even ramification index $e$. If we base-change $\YY$ to $R'$, we still have a semistable model, and the thicknesses of the nodes all get multiplied by $e$. Hence, all nodes of $\YY_{R'}$ have even thickness, and $Y$ has relatively stable reduction over $R'$ thanks to \Cref{BigCriterionJeffReduction}.
    \end{proof}
\end{cor}

We end this subsection by pointing out a simple but important property of the relatively stable model.
\begin{prop}
    \label{prop smoothness components rst}
    Suppose that $Y$ has relatively stable reduction. If $W\in \Irred(\SF{\Xrst})$ is a smooth $k$-curve, then its inverse image in $\SF{\Yrst}$ also is. In particular, if all irreducible components of $\SF{\Xrst}$ are smooth $k$-curves, then the same is true of the irreducible components of $\SF{\Yrst}$.
    \begin{proof}
        Let $Q$ be a node of $\SF{\Yrst}$ lying over a point $P\in W$. Since $\Yrst$ does not contain vanishing nodes by definition, we have that $P$ is a node of $\Xrst$; moreover, since the component $W$ is smooth at $P$, we have that $P$ must connect $W$ with another irreducible component $W'\in \Irred(\SF{\Xrst})$ distinct from $W$. We deduce from this that $Q$ connects two distinct components $V$ and $V'$ of $\SF{\Yrst}$, one lying over $W$, and the other lying over $W'$. We conclude that, if $V_1$ and $V_2$ are two (possibly coinciding) irreducible components of $\SF{\Yrst}$ lying over $W$, they cannot be connected by a node; hence, the inverse image of $W$ in $\SF{\Yrst}$ is a smooth $k$-curve.
    \end{proof}
\end{prop}

\subsection{Finding the relatively stable model} \label{sec relatively stable finding}

Let us assume, for this subsection, that $g(Y)\ge 2$, or that $g(Y)=1$ and $g(X)=0$. The following result is the analog of \Cref{prop part of min} for the relatively stable model (instead of the minimal regular one).

\begin{prop} \label{prop part of rst}
    Assume that $Y$ has relatively stable reduction, let $\XX$ be any model of $X$, and let $\YY$ be the corresponding model of $Y$. Then, 
    $\XX\le \Xrst$ if and only if, for all components $V$ of $\YY_s$, we have $m(V)=1$ and one of the following holds:
        \begin{enumerate}[(i)]
            \item $a(V)\ge 1$;
            \item $a(V)=0$ and $w(V)\ge 3$; or
            \item $a(V)=0$ and  $\underline{w}(V)=(2)$.
        \end{enumerate}
        \begin{proof}
            Suppose first that we have $\XX\le \Xrst$.  Given a component $V$ of $\YY_s$, its strict transform $V^{\mathrm{rst}}$ in $\Yrst$ will have the same the same invariants $m$, $a$, $w$, and $\underline{w}$ as $V$ by \Cref{lemma inv m a w,lemma inv ww}. Now, since $\SF{\Yrst}$ is reduced, we have $m(V)=1$. Since $\SF{\Yrst}$ does not contain (-1)-lines or horizontal (-2)-curves, we deduce from \Cref{Lemma12lines}(c) and \Cref{dfn minus two curve} (applied to the model $\Yrst$) that one of the three conditions (i), (ii) and (iii) above must occur.
            
            Now assume that $m(v) = 1$ and that either (i), (ii), or (iii) holds.  By \Cref{prop part of min}, we deduce that $\YY\le \Ymini$. Let $V$ be any component of $\CC_s$, and let $V^{\mathrm{min}}$ be the strict transform of $V$ in $\SF{\Ymini}$, which has the same invariants $m$, $a$, $w$, and $\underline{w}$ as $V$ (\Cref{lemma inv m a w,lemma inv ww}). Since we are excluding the case that $a(V)=0$ and $\underline{w}(V)=(1,1)$, the transform $V^{\mathrm{min}}$ is not a horizontal (-2)-curve (by  \Cref{dfn minus two curve}). Hence, all horizontal (-2)-curves of $\SF{\Ymini}$ get contracted in $\YY_s$, which, in light of \Cref{prop rst exists}, is equivalent to saying that $\YY\le \Yrst$.
        \end{proof}
\end{prop}

Given any model $\XX$ of $X$, \Cref{prop part of rst} allows us to determine whether or not it is part of $\Xrst$. Meanwhile, the following proposition, allows us to determine the position of the components of $\SF{\Xrst}$ relative to the given model $\XX$, under the assumption that $\XX\le \Xmini$.
\begin{prop} \label{prop rst isomorphism failure}
    Assume that $Y$ has relatively stable reduction, let $\XX$ be any model of $X$, and let $\YY$ be the corresponding model of $Y$. Assume that $\XX\le \Xmini$. Then we have that $\Ctr(\XX,\Xrst)$ coincides with the set of those points $P$ of $\XX_s$ above which $\YY_s$ has non-nodal singularities or vanishing nodes.
    
    \begin{proof}
        Let us first remark that the minimal regular model $\Ymini$ dominates both $\YY$ and $\Yrst$. Let $Q$ be a point of $\YY_s$, and let $P$ be its image in $\XX_s$; as discussed in \S\ref{sec relatively stable galois covers}, we have $P\in \Ctr(\XX,\Xrst)$ if and only if $Q\in \Ctr(\YY,\Yrst)$.
        
        If $Q$ is a non-singular point of $\YY_s$, then in particular it is a regular point of $\YY$, and hence the minimal desingularization morphism $\Ymini\to \YY$ is an isomorphism above $Q$ (i.e.\ no component of $\SF{\Ymini}$ contracts to $Q$), and hence we have $Q\not\in \Ctr(\YY,\Yrst)$.
        
        Suppose that $Q$ is a persistent node of $\YY_s$. Then, its inverse image in $\SF{\Ymini}$, by \Cref{prop galois cover desingularizing}, consists of a chain of horizontal (-2)-curves, which will all be contracted in $\SF{\Yrst}$ (see \Cref{prop rst exists}). Hence, we have $Q\not\in \Ctr(\YY,\Yrst)$.
        
        Suppose that $Q$ is a vanishing node of $\YY_s$. Since $Y$ has relatively stable reduction, it must have even thickness (see \Cref{BigCriterionJeffReduction}), and thus its inverse image in $\SF{\Ymini}$ contains a vertical (-2)-curve by \Cref{prop galois cover desingularizing}, which does not get contracted in $\SF{\Yrst}$ (see \Cref{prop rst exists}).  Hence, we have $Q\in \Ctr(\YY,\Yrst)$.
        
        Suppose that $Q$ is a non-nodal singularity of $\YY_s$; then the morphism $\Ymini\to\YY$ cannot be an isomorphism above $Q$ (because $\Ymini$ is semistable), and the inverse image of $Q$ in the semistable model $\Ymini$ cannot only contain (-2)-curves; otherwise $Q$ would be a node of $\YY_s$. Hence, there exists a component $V$ of $\SF{\Ymini}$ that contracts to $Q$ and is not a (-2)-curve; by \Cref{prop rst exists}, it is clear that $V$ does not get contracted in $\SF{\Yrst}$.  Hence, we have $Q\in \Ctr(\YY,\Yrst)$.
    \end{proof}
\end{prop}

The following statement is the analog of \Cref{prop rst isomorphism failure} in the case $\XX\not\le \Xmini$.
\begin{prop} \label{prop rst isomorphism failure not in mini}
    Assume that $Y$ has relatively stable reduction, let $\XX$ be any model of $X$, and let $\YY$ be the corresponding model of $Y$. Assume that $\XX\not\le \Xmini$, that $\XX_s$ is irreducible and that $\YY_s$ is reduced. Then, $\YY_s$ has a unique singular point $Q$, which is a non-nodal singularity, and we have $\Ctr(\XX,\Xrst)=\{ f(Q) \}$, where $f: \YY\to \XX$ is the cover map.
    \begin{proof}
        Since $\XX$ has an irreducible special fiber, the components $\lbrace V_i\rbrace_{i=1}^N$ of $\YY_s$ are transitively permuted by $G$, and hence the invariants $m(V_i)$, $a(V_i)$ and $\underline{w}(V_i)$ do not depend on $i$.  Since $\YY_s$ is reduced, we have $m(V_i)=1$, and since $\YY\not\le \Ymini$, we have $a(V_i)=0$ and $w(V_i)=1$ by \Cref{prop part of min}. This implies in particular that there is a unique singular point $Q_i$ of $\YY_s$ that lies on $V_i$; moreover, since $\YY_s$ is connected, the $Q_i$'s must all coincide. We conclude that $\YY_s$ contains a unique singular point $Q$. If $Q$ were a node, then $\YY_s$ would be semistable, and its vertical components would all be (-1)-lines by \Cref{Lemma12lines}(c), which is impossible since $g(Y)\ge 1$. Hence, $Q$ is a non-nodal singularity of $\YY_s$. Let $\YY'$ be the minimal desingularization of $\YY$; we have dominance relations $\YY\le\YY'\gneq \Ymini\ge \Yrst$.  Since $\YY\not\le \Ymini$, we have $\YY\not\le\Yrst$, which implies that $\Ctr(\YY,\Yrst)\neq \varnothing$.  Since the desingularization $\YY'\to \YY$ is necessarily an isomorphism above the smooth points of $\YY_s$, the set $\Ctr(\YY,\Yrst)$ cannot contain any point of $\YY_s$ other then $Q$.
    \end{proof}
\end{prop}

When $\XX$ is a smooth model, the (other) components of $\Xrst$ are contracted precisely to the points of $\XX_s$ above which $\YY_s$ is singular, as the following corollary emphasizes.
\begin{cor}
    \label{cor rst isomorphism failure}
    Assume that $Y$ has relatively stable reduction, let $\XX$ be any model of $X$, and let $\YY$ be the corresponding model of $Y$. Assume that $\YY$ has reduced special fiber. Then we have $\Ctr(\XX,\Xrst)=f(\Sing(\YY_s))$, where $f: \YY\to \XX$ is the covering map and $\Sing(\YY_s)$ is the set of all singular points of $\YY_s$.
\end{cor}

\begin{proof}
    This is immediate from Propositions \ref{prop rst isomorphism failure} and \ref{prop rst isomorphism failure not in mini}.
\end{proof}

\section{Models of hyperelliptic curves} \label{sec models hyperelliptic}

In this section, we specialize to the case in which the Galois cover $Y\to X$ that we considered in \S\ref{sec relatively stable} is the degree-$2$ map from a hyperelliptic curve $Y$ of genus $g\ge 1$ to the projective line $X$ (see \S\ref{sec notation he}). Our main aim, for this and the following sections, is computing the relatively stable model $\Yrst$ of $Y$ that we defined in \S\ref{sec relatively stable definition}. After recalling some basic general facts about hyperelliptic curves in \S\ref{sec models hyperelliptic equations}, we describe the semistable models of the line $X$ in \S\ref{sec models hyperelliptic line}: we will see how smooth models of the line correspond to discs $D\subseteq \bar{K}$, while the semistable ones correspond to certain finite collections $\mathfrak{D}$ of discs. After introducing the notion of a \emph{part-square decomposition} of a polynomial in \S\ref{sec models hyperelliptic part-square}, we exploit it in \S\ref{sec models hyperelliptic forming} to describe the model of the hyperelliptic curve $Y$ corresponding to a given smooth model of the line $X$ (i.e., to a given disc $D$). The special fiber of such models of $Y$ will be more thoroughly studied in \S\ref{sec models hyperelliptic separable} and \S\ref{sec models hyperelliptic inseparable}, and each of the two subsections provides a criterion to decide whether a disc $D$ belongs to the collection $\mathfrak{D}$ that defines the semistable model of the line $\Xrst$ from which $\Yrst$ comes (\Cref{thm part of rst separable} for the separable case, and \Cref{prop part of rst inseparable} for the inseparable case). 

\subsection{Equations for hyperelliptic curves} \label{sec models hyperelliptic equations}

We let $F$ be any field and write $X := \proj_F^1$ for the projective line over $F$.  In this subsection, we review basic facts and definitions relating to hyperelliptic curves over $F$ which can be found in \cite[\S7.4.3]{liu2002algebraic}.

\begin{dfn}
	A \emph{hyperelliptic curve} over $F$ is a smooth curve $Y / F$ of genus $g \geq 1$ along with a separable (branched) covering morphism $Y \to X$ of degree 2, which we call the \emph{hyperelliptic map}.
\end{dfn}

It is possible through repeated applications of the Riemann-Roch Theorem to show the well-known fact that the affine chart $x\neq \infty$ of any hyperelliptic curve $Y / F$ can be described by an equation of the form 
\begin{equation} \label{eq hyperelliptic general affine1}
y^2 + q(x)y = r(x),
\end{equation}
where $\deg(q) \leq g + 1$ and $\deg(r) \leq 2g + 2$ and the hyperelliptic map is given by the coordinate $x : Y \to X$.  If $F$ has characteristic different from $2$, a suitable change of the coordinate $y$ allows us to convert this equation into the simpler form $y^2 = r(x) + \frac{1}{4}q^2(x)$, from which it is clear that the smoothness condition implies that the polynomial $f(x) := r(x) + \frac{1}{4}q^2(x)$ must be separable.
Over the complementary affine chart of $\proj_F^1$ where $x \neq 0$, the hyperelliptic curve $Y$ can be described by the equation
\begin{equation} \label{eq hyperelliptic general affine2}
    \check{y}^2 + \check{q}(\check{x})\check{y} = \check{r}(\check{x}),
\end{equation}
where $\check{x} = x^{-1}$, $\check{y} = x^{-(g+1)}y$, $\check{q}(\check{x}) = x^{-(g+1)}q(x)$, and $\check{r}(\check{x}) = x^{-(2g+2)}r(x)$.  Note that the polynomial $\check{q}(z) \in F[z]$ (resp.\ $\check{r}(z) \in F[z]$) differs from the polynomial $q(z) \in F[z]$ (resp.\ $r(z) \in F[z]$) only in that each power $z^i$ which appears in the polynomial is replaced by $z^{g+1-i}$ (resp.\ $z^{2g+2-i}$) while the coefficients remain the same.

If $F$ has characteristic different from $2$ (i.e.\, if we consider \emph{tame} hyperelliptic curves), the Riemann-Hurwitz formula ensures that the ramification locus of $Y_{\bar{F}}\to X_{\bar{F}}$ consists of $2g+2$ points of $Y_{\bar{F}}$, lying over $2g+2$ distinct branch points of $X_{\overline{F}}$.  In fact, the branch locus determines a hyperelliptic curve almost completely, as we see from the following proposition.

\begin{prop} \label{prop hyperelliptic}
	Given a field $F$ of characteristic different from $2$, and letting $X$ be the projective line $\mathbb{P}^1_F$, the following data are equivalent:
	\begin{enumerate}[(i),nolistsep]
		\item a hyperelliptic curve $Y$ of genus $g$ having rational branch locus, endowed with a distinguished hyperelliptic map $Y \to X$; 
		\item a separable polynomial $f(x) \in F[x]$ of degree $2g + 1$ or $2g + 2$ all of whose roots lie in $F$, modulo multiplication by a scalar in $(F^{\times})^2$; and 
		\item a cardinality-$(2g+2)$ subset $\mathcal{B} \subset X(F)$ together with an element $c \in F^\times/(F^\times)^2$.
	\end{enumerate}
	Moreover, in (ii) above, the polynomial $f$ will have degree $2g + 1$ (resp.\ $2g + 2$) if in the context of (iii) above the coordinate of $X$ is chosen such that $\infty$ is (resp.\ is not) an element of $\mathcal{B}$.
	
\end{prop}

\begin{proof}
    We construct the above equivalences as follows.  Given a hyperelliptic curve $Y$ as in (i), we denote the distinguished hyperelliptic map by $x : Y \to X$.  Clearly, the morphism $x$ can be viewed as an element of the function field $F(Y)$; in fact, as the hyperelliptic map is not constant, we must have $F(X) = F(x) \hookrightarrow F(Y)$.  Since the hyperelliptic map has degree $2$, the extension $F(Y) \supset F(X)$ must be generated by a single element $y \in F(Y) \smallsetminus F(X)$ with $y^2 \in F(X)$; after multiplying $y$ by a suitable polynomial in $x$, we may assume that $f(x) := y^2 \in F[x]$.  Then it is straightforward to see that the affine open subset of $Y$ given by the inverse image of $\aff_F^1$ under the hyperelliptic map is described by the equation $y^2 = f(x)$.  The roots of $f$ clearly coincide with the points on $\aff_{\bar{F}}^1$ over which the hyperelliptic map is ramified; meanwhile, one sees by applying the Riemann-Hurwitz formula that the map $Y \to X$ must have exactly $2g + 2$ ramification points.  Therefore, the polynomial $f$ has $2g + 1$ (resp.\ $2g + 2$) roots all lying in $F$ if $\infty \in X(F)$ is (resp.\ is not) a ramification point.  This polynomial $f$ (modulo multiplication by elements in $(F^{\times})^2$) gives us the data in (ii).
    
    Given the data in (ii), let $\mathcal{R} \subset X(F) \smallsetminus \{\infty\}$ be the subset of roots of $f$ and let $c \in F^{\times} / (F^{\times})^2$ be the leading coefficient of $f$ modulo squares of elements in $F^{\times}$.  Setting $\mathcal{B} = \mathcal{R}$ (resp.\ $\mathcal{B} = \mathcal{R} \cup \{\infty\}$) if the degree of $f$ is $2g + 1$ (resp.\ $2g + 2$), we have that the set $\mathcal{B}$ has cardinality $2g + 2$ and we get the data of (iii).
    
    Finally, given a cardinality-$(2g+2)$ subset $\mathcal{B} \subset X(F)$ and a scalar $c \in F^{\times} / (F^{\times})^2$, write $\tilde{c} \in F^{\times}$ for a representative of $c$ and let $Y / F$ be the smooth completion of the affine curve over $F$ described by the equation 
    \begin{equation} \label{eq hyperelliptic}
        y^2 = f(x) := \tilde{c} \prod_{a \in \mathcal{B} \smallsetminus \{\infty\}} (x - a).
    \end{equation}
    Then it is easy to check that $Y$ satisfies the criteria given in (i), with the hyperelliptic map being given by the function $x \in F(Y)$.  If a different representative $\tilde{c}' \in F^{\times}$ is chosen for $c$ in order to define $f$, then we must have $\tilde{c}' = \gamma^2 \tilde{c}$ for some $\gamma \in F^{\times}$, and replacing the coordinate $y$ by $\gamma y$ gives us the same equation (\ref{eq hyperelliptic}) and therefore the same curve $Y$, so the data in (i) is uniquely determined by (iii).
\end{proof}

\begin{rmk} \label{rmk assumptions}
    Suppose that in the context of the above proposition, none of branch points of $Y \to X$ is $\infty$.  We may then find an isomorphic hyperelliptic curve over $F$ whose branch points over $X$ include the point $\infty \in X(F)$ by applying an automorphism of the projective line $X$ which moves one of the branch points to $\infty$.  More precisely, if $a_0$ is the $x$-coordinate of a branch point of $Y \to X$ that does not coincide with $\infty$, we perform the substitution $(x, y) \mapsto ((x-a_0)^{-1}, \check{x}^{g+1}\check{y})$ and get a curve (isomorphic over $F$ to our original one) ramified over $\infty \in X(F)$ defined by an equation of the form ${y}^2 = {f}(x)$, where ${f}(x) \in F[x]$ is a polynomial of degree $2g+1$.
\end{rmk}

From now on, we assume that $F$ is the discretely-valued field $K$ satisfying the conditions given in \S\ref{sec introduction notation}.  In light of the remark above, up to possibly replacing $K$ with a finite extension (so that at least one of the branch points of the cover $Y\to X$ is rational), we can and will make the following assumption throughout the rest of the paper.

\begin{hyp} \label{hyp properties of f}
The hyperelliptic curve $Y$ is defined over $K$ by the equation $y^2 = f(x)$, where $x$ is the standard coordinate of $X=\mathbb{P}^1_K$, and $f(x) \in K[x]$ is a polynomial of (odd) degree $2g+1$, where $g$ is the genus of $Y$.
\end{hyp}

Proposition \ref{prop hyperelliptic} allows us to treat the hyperelliptic curve $Y$ essentially as a marked line.  This is peculiar to the hyperelliptic case: if we were to deal with a tame Galois covering of the line of degree greater than $2$, the same branch locus would in general be shared by multiple branched coverings corresponding to various possible monodromy actions.

Our aim will be constructing semistable models of a given hyperelliptic curve $Y \to X$ by normalizing some carefully chosen semistable models of the line $X$ in the quadratic extension $K(X)\subseteq K(Y)$. We will start by analyzing smooth and semistable models of the line $X$ in the next subsection; in the subsequent ones, we will turn our attention to the corresponding models of $Y$.

\subsection{Models of the projective line} \label{sec models hyperelliptic line}

As before, let $X:=\mathbb{P}^1_K$ be the projective line, and let $x$ denote its standard coordinate. Given $\alpha\in \bar{K}$ and  $\beta\in \bar{K}^\times$, one can define a smooth model $\XX_{\alpha,\beta}$ of $X$ over the ring of integers $R'$ of $K':=K(\alpha,\beta)\subseteq \bar{K}$) by declaring $\XX_{\alpha,\beta}:=\mathbb{P}^1_{R'}$, with coordinate $x_{\alpha,\beta} := \beta^{-1}(x - \alpha)$, as an $R'$-scheme, and identifying the generic fiber $\XX_\eta$ with $X$ via the linear transformation $x_{\alpha,\beta}=\beta^{-1}(x-\alpha)$. If $(\alpha_1,\beta_1)$ and $(\alpha_2,\beta_2)$ are such that $v(\alpha_1-\alpha_2)\ge v(\beta_1)=v(\beta_2)$, then $\XX_{\alpha_1,\beta_1}$ and $\XX_{\alpha_2,\beta_2}$ are isomorphic as models of $X$, the isomorphism being given by the change of variable $x_{\alpha_2,\beta_2}=u x_{\alpha_1,\beta_1}+ \delta$, where $u$ is the unit $\beta_1(\beta_2)^{-1}$ and $\delta$ is the integral element $\beta_2^{-1}(\alpha_1-\alpha_2)$. In other words, the model $\XX_{\alpha,\beta}$ only depends, up to isomorphism, on the disc $D = D_{\alpha,b}:=\lbrace x\in \bar{K}: v(x-\alpha)\le b\rbrace$ of center $\alpha$ and depth $b:=v(\beta)$; for this reason, we will often denote $\XX_{\alpha,\beta}$ by $\XX_D$.
\begin{prop}
    \label{prop smooth models line discs}
    The construction $D\mapsto \XX_D$ described above defines a bijection between the discs of $\bar{K}$ and the smooth models of $X$ defined over finite extensions of $R$ considered up to isomorphism (two models $\XX_1/R'_1$ and $\XX_2/R'_2$ are considered isomorphic if they become so over some common finite extension $R''\supseteq R'_1, R'_2$).
    \begin{proof}
        Given a smooth model of the line $\XX$ over the ring of integers $R'$ of some finite extension $K'/K$, one may prove that there exists an isomorphism of $R'$-schemes $\XX\cong \mathbb{P}^1_{R'}$ (see \cite[Exercise 8.3.5]{liu2002algebraic}), which immediately implies that $\XX$ is isomorphic, as a model, to $\XX_D$ for some uniquely determined disc $D=D_{\alpha,b}$ with $\alpha\in K'$ and $b\in v((K')^\times)$.
    \end{proof}
\end{prop}

Given two discs $D=D_{\alpha,b}$ and $D'=D_{\alpha',b'}$, we want to compare the smooth models $\XX_{D}$ and $\XX_{D'}$: using the notation introduced in \S\ref{sec preliminaries comparing}, we have the following proposition, which can be verified by an immediate computation. 
\begin{prop}
    \label{prop relative position smooth models line}
    With the notation above, assume $D\neq D'$, and let $P\in \SF{\XX_D}(k)$ and $P'\in \SF{\XX_{D'}}(k)$ be the points such that $\Ctr(\XX_D,\XX_{D'})=\{ P \}$ and $\Ctr(\XX_{D'},\XX_{D})=\{ P' \}$. Then there are the following three possibilities (illustrated in \Cref{fig two models line}):
    \begin{enumerate}[(a)]
        \item when $D\subsetneq D'$, $P$ is the point $\overline{x_{\alpha,\beta}}=\infty$ and $P'$ is the point $\overline{x_{\alpha',\beta'}}=\overline{(\beta')^{-1}(\alpha-\alpha')}\neq \infty$;
        \item when $D'\subsetneq D$, $P$ is the point $\overline{x_{\alpha,\beta}}=\overline{\beta^{-1}(\alpha'-\alpha)}\neq \infty$, and $P'$ is the point $\overline{x_{\alpha',\beta'}}=\infty$; or 
        \item when $D\cap D'=\varnothing$, $P$ is the point $\overline{x_{\alpha,\beta}}=\infty$ and $P'$ is the point $\overline{x_{\alpha',\beta'}}=\infty$.
    \end{enumerate}
\end{prop}

\begin{figure}
    \centering
    \tikzset{every picture/.style={line width=0.75pt}} 

\begin{tikzpicture}[x=0.75pt,y=0.75pt,yscale=-1,xscale=1]

\draw [color={rgb, 255:red, 0; green, 0; blue, 0 }  ,draw opacity=1 ]   (100,126) -- (37,201) ;
\draw    (19,120.5) -- (101,193) ;
\draw    (68.5,163.5) .. controls (66.79,161.87) and (66.75,160.21) .. (68.38,158.5) .. controls (70.01,156.79) and (69.97,155.13) .. (68.26,153.5) .. controls (66.55,151.87) and (66.51,150.21) .. (68.14,148.5) .. controls (69.76,146.79) and (69.72,145.13) .. (68.01,143.51) .. controls (66.3,141.88) and (66.26,140.22) .. (67.89,138.51) .. controls (69.52,136.8) and (69.48,135.14) .. (67.77,133.51) .. controls (66.06,131.88) and (66.02,130.22) .. (67.65,128.51) .. controls (69.28,126.8) and (69.24,125.14) .. (67.53,123.51) .. controls (65.82,121.88) and (65.78,120.22) .. (67.41,118.51) .. controls (69.04,116.8) and (69,115.14) .. (67.29,113.51) -- (67.24,111.66) -- (67.05,103.67) ;
\draw [shift={(67,101.67)}, rotate = 88.61] [color={rgb, 255:red, 0; green, 0; blue, 0 }  ][line width=0.75]    (10.93,-3.29) .. controls (6.95,-1.4) and (3.31,-0.3) .. (0,0) .. controls (3.31,0.3) and (6.95,1.4) .. (10.93,3.29)   ;
\draw [color={rgb, 255:red, 0; green, 0; blue, 0 }  ,draw opacity=1 ]   (311,132) -- (248,207) ;
\draw    (230,126.5) -- (312,199) ;
\draw    (279.5,169.5) .. controls (277.75,167.92) and (277.67,166.26) .. (279.25,164.51) .. controls (280.83,162.76) and (280.75,161.1) .. (279,159.51) .. controls (277.25,157.93) and (277.17,156.27) .. (278.75,154.52) .. controls (280.33,152.77) and (280.25,151.11) .. (278.5,149.53) .. controls (276.75,147.94) and (276.67,146.28) .. (278.25,144.53) .. controls (279.83,142.78) and (279.75,141.12) .. (278,139.54) .. controls (276.25,137.95) and (276.17,136.29) .. (277.75,134.54) .. controls (279.33,132.79) and (279.25,131.13) .. (277.5,129.55) .. controls (275.75,127.97) and (275.67,126.31) .. (277.25,124.56) .. controls (278.83,122.81) and (278.75,121.15) .. (277,119.56) .. controls (275.25,117.98) and (275.17,116.32) .. (276.75,114.57) -- (276.5,109.65) -- (276.1,101.66) ;
\draw [shift={(276,99.67)}, rotate = 87.13] [color={rgb, 255:red, 0; green, 0; blue, 0 }  ][line width=0.75]    (10.93,-3.29) .. controls (6.95,-1.4) and (3.31,-0.3) .. (0,0) .. controls (3.31,0.3) and (6.95,1.4) .. (10.93,3.29)   ;
\draw [color={rgb, 255:red, 0; green, 0; blue, 0 }  ,draw opacity=1 ]   (531,133) -- (468,208) ;
\draw    (450,127.5) -- (532,200) ;
\draw    (499.5,170.5) .. controls (497.8,168.87) and (497.76,167.2) .. (499.39,165.5) .. controls (501.02,163.8) and (500.99,162.13) .. (499.29,160.5) .. controls (497.59,158.87) and (497.55,157.2) .. (499.18,155.5) .. controls (500.81,153.8) and (500.78,152.13) .. (499.08,150.5) .. controls (497.38,148.87) and (497.34,147.21) .. (498.97,145.51) .. controls (500.6,143.81) and (500.56,142.14) .. (498.86,140.51) .. controls (497.16,138.88) and (497.13,137.21) .. (498.76,135.51) .. controls (500.39,133.81) and (500.35,132.14) .. (498.65,130.51) .. controls (496.95,128.88) and (496.92,127.21) .. (498.55,125.51) .. controls (500.18,123.81) and (500.14,122.14) .. (498.44,120.51) .. controls (496.74,118.88) and (496.71,117.21) .. (498.34,115.51) .. controls (499.97,113.81) and (499.93,112.14) .. (498.23,110.51) -- (498.21,109.66) -- (498.04,101.67) ;
\draw [shift={(498,99.67)}, rotate = 88.79] [color={rgb, 255:red, 0; green, 0; blue, 0 }  ][line width=0.75]    (10.93,-3.29) .. controls (6.95,-1.4) and (3.31,-0.3) .. (0,0) .. controls (3.31,0.3) and (6.95,1.4) .. (10.93,3.29)   ;
\draw    (31.5,132.5) .. controls (32.31,134.71) and (31.6,136.22) .. (29.39,137.03) .. controls (27.18,137.84) and (26.47,139.35) .. (27.27,141.56) -- (25.23,145.94) -- (21.85,153.19) ;
\draw [shift={(21,155)}, rotate = 295.02] [color={rgb, 255:red, 0; green, 0; blue, 0 }  ][line width=0.75]    (10.93,-3.29) .. controls (6.95,-1.4) and (3.31,-0.3) .. (0,0) .. controls (3.31,0.3) and (6.95,1.4) .. (10.93,3.29)   ;
\draw    (296.5,148.5) .. controls (298.81,148.99) and (299.71,150.39) .. (299.21,152.7) .. controls (298.71,155.01) and (299.61,156.41) .. (301.92,156.91) .. controls (304.23,157.4) and (305.13,158.8) .. (304.63,161.11) -- (305.58,162.59) -- (309.92,169.32) ;
\draw [shift={(311,171)}, rotate = 237.2] [color={rgb, 255:red, 0; green, 0; blue, 0 }  ][line width=0.75]    (10.93,-3.29) .. controls (6.95,-1.4) and (3.31,-0.3) .. (0,0) .. controls (3.31,0.3) and (6.95,1.4) .. (10.93,3.29)   ;
\draw    (499.5,170.5) .. controls (497.99,172.31) and (496.33,172.46) .. (494.52,170.95) .. controls (492.71,169.44) and (491.05,169.6) .. (489.54,171.41) .. controls (488.03,173.22) and (486.37,173.37) .. (484.56,171.86) -- (481.96,172.09) -- (473.99,172.82) ;
\draw [shift={(472,173)}, rotate = 354.81] [color={rgb, 255:red, 0; green, 0; blue, 0 }  ][line width=0.75]    (10.93,-3.29) .. controls (6.95,-1.4) and (3.31,-0.3) .. (0,0) .. controls (3.31,0.3) and (6.95,1.4) .. (10.93,3.29)   ;

\draw (122,202) node    {$L'$};
\draw (23,204) node  [color={rgb, 255:red, 0; green, 0; blue, 0 }  ,opacity=1 ]  {$L$};
\draw (12,159.4) node [anchor=north west][inner sep=0.75pt]    {$\infty $};
\draw (333,208) node    {$L'$};
\draw (234,209) node  [color={rgb, 255:red, 0; green, 0; blue, 0 }  ,opacity=1 ]  {$L$};
\draw (553,209) node    {$L'$};
\draw (454,211) node  [color={rgb, 255:red, 0; green, 0; blue, 0 }  ,opacity=1 ]  {$L$};
\draw (450,168.4) node [anchor=north west][inner sep=0.75pt]    {$\infty $};
\draw (31,234) node [anchor=north west][inner sep=0.75pt]   [align=left] {Case (a): $\displaystyle D\varsubsetneq D'$};
\draw (233,234) node [anchor=north west][inner sep=0.75pt]   [align=left] {Case (b): $\displaystyle D'\varsubsetneq D$};
\draw (450,234) node [anchor=north west][inner sep=0.75pt]   [align=left] {Case (c): $\displaystyle D\cap D'=\emptyset $};
\draw (313,170.4) node [anchor=north west][inner sep=0.75pt]    {$\infty $};
\draw (17,54.4) node [anchor=north west][inner sep=0.75pt]  [font=\normalsize]  {$\overline{x_{\alpha ,\beta }} =\infty $};
\draw (18,71.4) node [anchor=north west][inner sep=0.75pt]  [font=\normalsize]  {$\overline{x_{\alpha ',\beta '}} =\overline{(\beta ')^{-1} (\alpha -\alpha ')} \neq \infty $};
\draw (227,53.4) node [anchor=north west][inner sep=0.75pt]  [font=\normalsize]  {$\overline{x_{\alpha ,\beta }} =\overline{\beta ^{-1} (\alpha '-\alpha )} \neq \infty $};
\draw (228,72.4) node [anchor=north west][inner sep=0.75pt]  [font=\normalsize]  {$\overline{x_{\alpha ',\beta '}} =\infty $};
\draw (453,53.4) node [anchor=north west][inner sep=0.75pt]  [font=\normalsize]  {$\overline{x_{\alpha ,\beta }} =\infty $};
\draw (454,72.4) node [anchor=north west][inner sep=0.75pt]  [font=\normalsize]  {$\overline{x_{\alpha ',\beta '}} =\infty $};

\end{tikzpicture}
    \caption{The special fiber of the minimal model $\XX_{\{D,D'\}}$ dominating $\XX_D$ and $\XX_D'$. Here, $L$ and $L'$ are the lines corresponding to the discs $D$ and $D'$ respectively.}
    \label{fig two models line}
\end{figure}
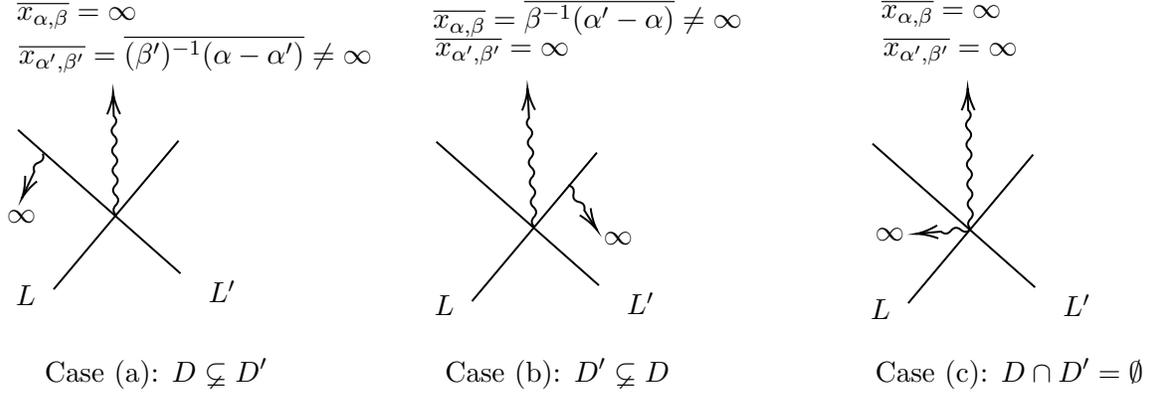

If $\mathfrak{D}=\{D_1,\ldots, D_n\}$ is a non-empty, finite collection of discs of $\bar{K}$, one can form a corresponding model $\XX_{\mathfrak{D}}$, which is defined as the minimal model dominating all the smooth models $\lbrace \XX_D: D\in \mathfrak{D}\rbrace$. Its special fiber is a reduced $k$-curve of arithmetic genus $p_a(\SF{\XX_{\mathfrak{D}}})=g(X)=0,$ i.e.\ it consists of $n$ lines $L_1, \ldots, L_n$ corresponding to the discs $D_i$'s meeting each other at ordinary multiple points, without forming loops.

\begin{prop} \label{prop collections of discs}
    The construction $\mathfrak{D}\mapsto \XX_\mathfrak{D}$ described above defines a bijection between the finite non-empty collection of discs $\mathfrak{D}$ of $\bar{K}$ and the models of $X$ having reduced special fiber defined over finite extensions of $R$, considered up to isomorphism (two models $\XX_1/R'_1$ and $\XX_2/R'_2$ are considered isomorphic if they become so over some common finite extension $R''\supseteq R'_1, R'_2$).
    \begin{proof}
        Suppose that $\XX$ is a model of the line $X$ with reduced special fiber, and let $\lbrace L_1, \ldots, L_n \rbrace$ be the components of its special fiber $\XX_s$. For each $i$, let $\XX_i$ be the model of obtained from $\XX$ by contracting all lines in $\Irred(\XX_s)$, except for $L_i$ it is easy to prove that $\XX_i$ is smooth (see, for example, \cite[Exercise 8.3.5]{liu2002algebraic}), so that, by \Cref{prop smooth models line discs}, $\XX_i=\XX_{D_i}$ for some uniquely determined disc $D_i \subset \bar{K}$. Now the model $\XX$ can be described as the minimal model dominating all the $\XX_i$'s, i.e.\ we have $\XX\cong \XX_{\mathfrak{D}}$ with $\mathfrak{D}=\{D_1, \ldots, D_n\}$.
    \end{proof}
\end{prop}

\begin{prop} \label{prop thickness}
    Suppose that $\XX / R'$ is a semistable model of the line $X$ for some finite extension $R' / R$, such that there are discs $D_{\alpha,b} \subsetneq D_{\alpha',b'} \subset \bar{K}$ corresponding to two intersecting components of $\SF{\XX}$.  Then the thickness of the node where they intersect is given by the formula $(b' - b) / v(\pi)$, where $\pi \in \bar{K}$ is a uniformizer of $R'$.
\end{prop}

\begin{proof}
    We can clearly replace the center $\alpha'$ with $\alpha \in D_{\alpha,b} \subsetneq D_{\alpha',b'}$; now choosing $\beta, \beta' \in \bar{K}^{\times}$ be scalars such that $v(\beta) = b$ and $v(\beta') = b'$.  Then, with the notation above, we have coordinates $x_{\alpha,\beta}$ and $x_{\alpha,\beta'}$ corresponding to each of these components of $\SF{\XX}$, and these coordinates are related by the equation $x_{\alpha,\beta} = \beta' \beta^{-1} x_{\alpha,\beta'}$.  Locally around the point of intersection, a defining equation is $x_{\alpha,\beta} x_{\alpha,\beta'}^\vee = \beta' \beta^{-1}$, where $x_{\alpha,\beta'}^{\vee} = x_{\alpha,\beta'}^{-1}$, and so the thickness by definition is equal to $v(\beta' \beta^{-1}) / v(\pi) = (b' - b) / v(\pi)$.
\end{proof}

\Cref{prop collections of discs} certainly implies that a semistable model of the line always has the form $\XX_{\mathfrak{D}}$ for some finite non-empty family of discs $\mathfrak{D}$; however, it is not always true that, given a collection of discs $\mathfrak{D}$, the corresponding model $\XX_{\mathfrak{D}}$ of the line is semistable. In fact, its special fiber $\SF{\XX_{\mathfrak{D}}}$ is always a $k$-curve with at worst ordinary singularities, but it is possible that more than two lines $L_i\in \Irred(\SF{\XX_{\mathfrak{D}}})$ intersect at the same ordinary multiple point, violating semistability. However, it is not difficult to give a combinatorial necessary and sufficient condition for a collection $\mathfrak{D}$ of discs to give rise to a semistable model.

\begin{prop}
    \label{prop semistable family of discs}
    The model $\XX_\mathfrak{D}$ of $X$ corresponding to a finite non-empty collection of discs $\mathfrak{D}$ is semistable if and only if it satisfies the following property: if three discs $D_1, D_2, D_3\in \mathfrak{D}$ satisfy any of the three conditions
    \begin{enumerate}[(a)]
        \item $D_1, D_2, D_3\in \mathfrak{D}$ are mutually disjoint, and any disc in $\bar{K}$ containing two of them also contains the third one;
        \item $D_1, D_2, D_3\in \mathfrak{D}$ are mutually disjoint, and there exists a disc in $\bar{K}$ containing $D_1$ and $D_2$ that is disjoint from $D_3$; or 
        \item $D_3 \supseteq D_1 \cup D_2$, and $D_3$ is not minimal among the discs of $\bar{K}$ satisfying this property,
    \end{enumerate}
    then, letting $D$ be the minimal disc of $\bar{K}$ containing both $D_1$ and $D_2$, we have $D\in \mathfrak{D}$.
\end{prop}

The proof of the proposition relies on the following elementary lemma.
\begin{lemma}
    \label{lemma semistable family of discs}
    We have the following.
    \begin{enumerate}[(a)]
        \item Given three discs $D_1, D_2, D_3 \subset \bar{K}$, some permutation of them satisfies the assumptions (a), (b) or (c) of  \Cref{prop semistable family of discs} if and only if $\SF{\XX_{\{D_1,D_2,D_3\}}}$ consists of three lines $L_1, L_2$ and $L_3$ meeting at an ordinary triple point.
        \item Given three discs $D_1, D_2$ and $D_3$ satisfying the assumptions (a), (b) or (c) of \Cref{prop semistable family of discs}, and letting $D$ be the minimal disc containing $D_1$ and $D_2$, we have that the three lines $L_1, L_2$ and $L_3$ corresponding to $D_1, D_2,$ and $D_3$ do not intersect each other in the special fiber of $\XX_{\{D_1,D_2,D_3,D\}}$, and they intersect the line $L$ corresponding to the disc $D$ at three distinct points (see \Cref{fig three models line}).  
    \end{enumerate}
    \begin{proof}
        The lemma can be proved by means of straightforward computations, which we omit.
    \end{proof}
\end{lemma}

\begin{figure}
    \centering
    \tikzset{every picture/.style={line width=0.75pt}} 

\begin{tikzpicture}[x=0.75pt,y=0.75pt,yscale=-1,xscale=1]

\draw [color={rgb, 255:red, 0; green, 0; blue, 0 }  ,draw opacity=1 ]   (180,46) -- (43,40) ;
\draw    (56,84.33) -- (68,23) ;
\draw    (86,90.33) -- (102,23) ;
\draw    (122,90.33) -- (135,29) ;
\draw [color={rgb, 255:red, 0; green, 0; blue, 0 }  ,draw opacity=1 ]   (379,45) -- (242,39) ;
\draw    (255,83.33) -- (267,22) ;
\draw    (285,89.33) -- (301,22) ;
\draw    (321,89.33) -- (334,28) ;
\draw [color={rgb, 255:red, 0; green, 0; blue, 0 }  ,draw opacity=1 ]   (583,42) -- (446,36) ;
\draw    (459,80.33) -- (471,19) ;
\draw    (489,86.33) -- (505,19) ;
\draw    (525,86.33) -- (538,25) ;
\draw    (332.5,43.5) .. controls (333.17,41.24) and (334.63,40.44) .. (336.89,41.1) .. controls (339.15,41.76) and (340.61,40.96) .. (341.27,38.7) .. controls (341.94,36.44) and (343.4,35.64) .. (345.66,36.3) .. controls (347.92,36.96) and (349.38,36.16) .. (350.05,33.9) -- (350.23,33.8) -- (357.25,29.96) ;
\draw [shift={(359,29)}, rotate = 151.31] [color={rgb, 255:red, 0; green, 0; blue, 0 }  ][line width=0.75]    (10.93,-3.29) .. controls (6.95,-1.4) and (3.31,-0.3) .. (0,0) .. controls (3.31,0.3) and (6.95,1.4) .. (10.93,3.29)   ;
\draw    (153.5,44.5) .. controls (154.17,42.24) and (155.63,41.44) .. (157.89,42.1) .. controls (160.15,42.76) and (161.61,41.96) .. (162.27,39.7) .. controls (162.94,37.44) and (164.4,36.64) .. (166.66,37.3) .. controls (168.92,37.96) and (170.38,37.16) .. (171.05,34.9) -- (171.23,34.8) -- (178.25,30.96) ;
\draw [shift={(180,30)}, rotate = 151.31] [color={rgb, 255:red, 0; green, 0; blue, 0 }  ][line width=0.75]    (10.93,-3.29) .. controls (6.95,-1.4) and (3.31,-0.3) .. (0,0) .. controls (3.31,0.3) and (6.95,1.4) .. (10.93,3.29)   ;
\draw    (525.5,79.5) .. controls (526.17,77.24) and (527.63,76.44) .. (529.89,77.1) .. controls (532.15,77.76) and (533.61,76.96) .. (534.27,74.7) .. controls (534.94,72.44) and (536.4,71.64) .. (538.66,72.3) .. controls (540.92,72.96) and (542.38,72.16) .. (543.05,69.9) -- (543.23,69.8) -- (550.25,65.96) ;
\draw [shift={(552,65)}, rotate = 151.31] [color={rgb, 255:red, 0; green, 0; blue, 0 }  ][line width=0.75]    (10.93,-3.29) .. controls (6.95,-1.4) and (3.31,-0.3) .. (0,0) .. controls (3.31,0.3) and (6.95,1.4) .. (10.93,3.29)   ;

\draw (46,91.4) node [anchor=north west][inner sep=0.75pt]    {$L_{1}$};
\draw (79,95.4) node [anchor=north west][inner sep=0.75pt]    {$L_{2}$};
\draw (113,94.4) node [anchor=north west][inner sep=0.75pt]    {$L_{3}$};
\draw (189,38.4) node [anchor=north west][inner sep=0.75pt]    {$L$};
\draw (76,132) node [anchor=north west][inner sep=0.75pt]   [align=left] {Case (a)};
\draw (360,18.4) node [anchor=north west][inner sep=0.75pt]    {$\infty $};
\draw (245,90.4) node [anchor=north west][inner sep=0.75pt]    {$L_{1}$};
\draw (278,94.4) node [anchor=north west][inner sep=0.75pt]    {$L_{2}$};
\draw (312,93.4) node [anchor=north west][inner sep=0.75pt]    {$L_{3}$};
\draw (388,37.4) node [anchor=north west][inner sep=0.75pt]    {$L$};
\draw (275,131) node [anchor=north west][inner sep=0.75pt]   [align=left] {Case (b)};
\draw (449,87.4) node [anchor=north west][inner sep=0.75pt]    {$L_{1}$};
\draw (482,91.4) node [anchor=north west][inner sep=0.75pt]    {$L_{2}$};
\draw (516,90.4) node [anchor=north west][inner sep=0.75pt]    {$L_{3}$};
\draw (592,34.4) node [anchor=north west][inner sep=0.75pt]    {$L$};
\draw (485,131) node [anchor=north west][inner sep=0.75pt]   [align=left] {Case (c)};
\draw (179,14.4) node [anchor=north west][inner sep=0.75pt]    {$\infty $};
\draw (556,55.4) node [anchor=north west][inner sep=0.75pt]    {$\infty $};

\end{tikzpicture}
    \caption{When three discs $D_1$, $D_2$, $D_3$ satisfy the assumptions of points (a), (b) or (c) of \Cref{prop semistable family of discs}. and $D$ is the minimal disc containing $D_1$ and $D_2$, then the special fiber of $\XX_{\{D_1,D_2,D_3,D\}}$, has the shape depicted above (the converse is actually also true). In the picture, $L_i$ is the line corresponding to the disc $D_i$, and $L$ is the line corresponding to the disc $D$.}
    \label{fig three models line}
\end{figure}
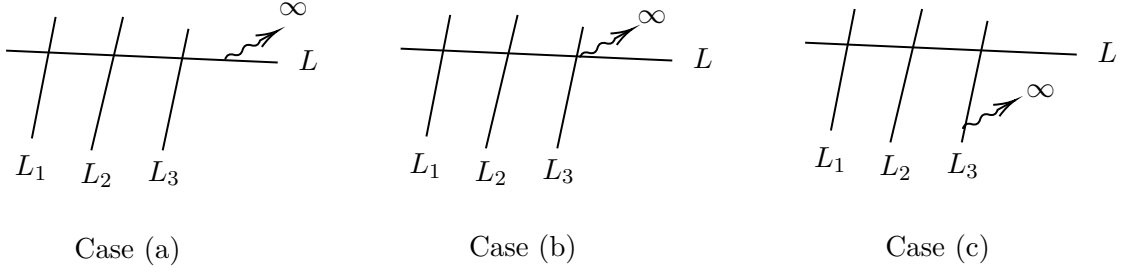

\begin{proof}[Proof of \Cref{prop semistable family of discs}]
    Both implications will be proved by way of contradiction.

    Assume that $\XX_{\mathfrak{D}}$ is not semistable, so that there exists an ordinary singular point $P\in \SF{\XX_{\mathfrak{D}}}$ through which three distinct lines $L_1, L_2, L_3\in \Irred(\SF{\XX_{\mathfrak{D}}})$ pass; letting $D_1, D_2$ and $D_3$ the corresponding three discs, \Cref{lemma semistable family of discs}(a) ensures that they satisfy (possibly after performing a permutation) condition (a), (b) or (c) of \Cref{prop semistable family of discs}. Now if $D$ is the minimal disc containing $D_1$ and $D_2$, then we certainly have $D\not\in \mathfrak{D}$: otherwise, we would have $\XX_{\mathfrak{D}}\ge \XX_{\{D_1, D_2, D_3, D\}}$, and this would prevent $L_1$, $L_2$ and $L_3$ from intersecting each other in $\XX_{\mathfrak{D}}$ by \Cref{lemma semistable family of discs}(b).
    
    Conversely, assume that $D_1, D_2, D_3\in \mathfrak{D}$ are three discs satisfying either condition (a), (b) or (c) of \Cref{prop semistable family of discs}, and such that $D\not \in\mathfrak{D}$, where $D$ is the minimal disc containing $D_1$ and $D_2$. Let $P$ be the point of $\SF{\XX_\mathfrak{D}}$ such that $\Ctr(\XX_{\mathfrak{D}},\XX_D)=\{ P \}$; in other words, $P$ is the point of $\SF{\XX_{\mathfrak{D}}}$ to which the unique line $L$ of which the special fiber of $\XX_D$ consists is contracted. We observe that, in the model $\XX_{\{D_1, D_2, D_3, D\}}$, the line $L$ corresponding to the disc $D$ intersects the rest of the special fiber at more than $2$ points (this follows from \Cref{lemma semistable family of discs}(b)); the same will consequently also be true in the model $\XX_{\mathfrak{D}\cup \{D\}}\ge \XX_{\{D_1, D_2, D_3, D\}}$. Hence, at least $3$ lines will pass through the point $P\in \SF{\XX_{\mathfrak{D}}}$ to which the line $L\in \Irred(\SF{\XX_{\mathfrak{D}\cup \{D\}}})$ gets contracted, which implies that $\XX_{\mathfrak{D}}$ is not semistable.
\end{proof}

The following way of rephrasing the conditions (a), (b) and (c) of \Cref{prop semistable family of discs} will also be useful later.
\begin{prop}
     \label{prop discs positions and ctr}
     Given a collection of discs $\mathfrak{D}$ and a disc $D$ of $\bar{K}$, the following are equivalent:
     \begin{enumerate}[(i)]
         \item there exist discs $D_1, D_2, D_3\in \mathfrak{D}$ which satisfy conditions (a), (b) or (c) of \Cref{prop semistable family of discs} and the disc $D$ is the minimal disc of $\bar{K}$ containing $D_1$ and $D_2$;
         \item we have $|\Ctr(\XX_D,\XX_{\mathfrak{D}})|\ge 3$.
     \end{enumerate}
     \begin{proof}
        First assume that (i) holds.  It follows from \Cref{lemma semistable family of discs}(b) that $|\Ctr(\XX_D,\XX_{\{D_1, D_2, D_3\}})|= 3$, which implies (ii), since we clearly have $\Ctr(\XX_D,\XX_{\{D_1, D_2, D_3\}})\subseteq \Ctr(\XX_D,\XX_{\mathfrak{D}})$.  Now assume that (ii) holds; from $|\Ctr(\XX_D,\XX_{\mathfrak{D}})|\ge 3$ one deduces that there exist discs $D_1, D_2,$ and $D_3\in \mathfrak{D}$ such that $\Ctr(\XX_D,\linebreak[0] \XX_{\lbrace D_1, D_2, D_3\rbrace})= \bigcup_{i=1}^3 \Ctr(\XX_D,\allowbreak \XX_{D_i})$ consists of three distinct points of $\SF{\XX_D}$; now (i) follows from straightforward calculations, taking into account \Cref{prop relative position smooth models line}.
     \end{proof}
\end{prop}

Given any non-empty collection of discs $\mathfrak{D}$, we can complete it to a family $\mathfrak{D}^\sst$ of discs corresponding to a semistable model: it is enough that, for every three discs $D_1, D_2, D_3\in \mathfrak{D}$ satisfying the conditions (a), (b) or (c) of \Cref{prop semistable family of discs}, the minimal disc containing $D_1$ and $D_2$ is added to $\mathfrak{D}$. It is not difficult to see that the resulting family of discs $\mathfrak{D}^\sst\supseteq \mathfrak{D}$ satisfies the hypothesis of \Cref{prop semistable family of discs} and consequently corresponds to the minimal semistable model $\XX_{\mathfrak{D}^\sst}$ of the line $X$ that dominates $\XX_\mathfrak{D}$. 

\begin{rmk}
    \label{rmk semistabilization reversed}
    Suppose a non-empty collection of discs $\mathfrak{D}$ corresponding to a semistable model $\XX_{\mathfrak{D}}$ is given, and let $\mathfrak{D}'\subseteq \mathfrak{D}$ be a non-empty subfamily. Suppose that, for all $D\in \mathfrak{D}\setminus \mathfrak{D}'$, there exists three discs $D_1, D_2, D_3\in \mathfrak{D}$ satisfying the conditions (a), (b) or (c) of \Cref{prop semistable family of discs}, and such that $D$ is the minimal disc of $\bar{K}$ containing $D_1$ and $D_2$ -- by \Cref{prop discs positions and ctr}, this condition can be equivalently expressed by saying that, for all $D\in \mathfrak{D}\setminus \mathfrak{D}'$, the set $\Ctr(\XX_D,\XX_{\mathfrak{D}})$ consists of three or more points. Then, $\mathfrak{D}$ can be reconstructed from $\mathfrak{D}'$ by applying the completion procedure described above, i.e.\ we have $\mathfrak{D}=(\mathfrak{D}')^\sst$.
\end{rmk}

\subsection{Part-square decompositions} \label{sec models hyperelliptic part-square}

We begin this subsection by defining a \textit{part-square decomposition}, and then we study part-square decompositions with certain properties.

\begin{dfn} \label{dfn qrho}
    Given a nonzero polynomial $h(x) \in \bar{K}[z]$, a \emph{part-square decomposition} of $h$ is a way of writing $h = q^2 + \rho$ for some $q(x), \rho(x) \in \bar{K}[x]$, with $\deg(q)\le \lceil\deg(h)/2\rceil$.
\end{dfn}

\begin{rmk}
    \label{rmk degree of rho}
    The definition forces $\deg(\rho)\le \deg(h)$ when $h$ has even degree and $\deg(\rho)\le \deg(h)+1$ when $h$ has odd degree. The definition allows $q$ to be equal to zero.
\end{rmk}

Given a part-square decomposition $h = q^2 + \rho$, we define the rational number $t_{q, \rho}:=v(\rho)-v(h) \in \qqinfty$.

\begin{dfn} \label{dfn good totally odd}

We define the following properties of a part-square decomposition $h = q^2 + \rho$.

\begin{enumerate}[(a)]
    \item The decomposition is said to be \emph{good} either if we have $t_{q,\rho} \geq 2v(2)$ or if we have $t_{q,\rho} < 2v(2)$ and there is no decomposition $h = \tilde{q}^2 + \tilde{\rho}$ such that $t_{\tilde{q}, \tilde{\rho}} > t_{q, \rho}$.
    \item The decomposition is said to be \emph{totally odd} if $\rho$ only consists of odd-degree terms.
\end{enumerate}

\end{dfn}

\begin{rmk} \label{rmk good decompositions}
    The trivial part-square decomposition $h=0^2+h$ has $t_{0,h} = 0$; this immediately implies that all good decompositions $h = q^2 + \rho$ satisfy $t_{q,\rho} \ge 0$. When $p\neq 2$, the converse also holds because we have $2v(2)=0$.
\end{rmk}

\begin{rmk} \label{rmk same t for good}
    If $h=q^2+\rho=(q')^2+\rho'$ are two good part-square decompositions for the same nonzero polynomial $h$, then we have $\truncate{t_{q,\rho}}=\truncate{t_{q',\rho'}}$ directly from \Cref{dfn good totally odd}.
\end{rmk}

\begin{prop} \label{prop good decomposition}
    Let $h = q^2 + \rho$ be a part-square decomposition satisfying $t_{q,\rho} < 2v(2)$.  Then we have the following.
    \begin{enumerate}[(a)]
        \item The decomposition $h = q^2 + \rho$ is good if and only if the normalized reduction of $\rho$ is not the square of a polynomial with coefficients in $k$.
        \item Suppose that the decomposition $h = q^2 + \rho$ is good and that $h = \tilde{q}^2 + \tilde{\rho}$ is another good decomposition.  Then given any normalized reductions of $\rho$ and $\tilde{\rho}$ respectively, the same odd degrees appear among terms in these normalized reductions, and their derivatives are equal up to scaling.
    \end{enumerate}
    
    \begin{proof}
        We begin by proving part (a).  If $t_{q,\rho}<0$, the decomposition is not good (see \Cref{rmk good decompositions}), and any normalized reduction of $\rho$ is a square, since it is a scalar multiple of a normalized reduction of $q^2$. We now have to prove the two implications when $t_{q,\rho}\ge 0$.
    
        Suppose that $h = q^2 + \rho$ satisfies $0\le t_{q,\rho} < 2v(2)$ but is not good, so that another decomposition $h = \tilde{q}^2(x) + \tilde{\rho}(x)$ with $t_{\tilde{q}, \tilde{\rho}} > t_{q, \rho}$ can be found.
        Let us now consider $q + \tilde{q}$ and $q - \tilde{q}$: their product has valuation $v(q^2-\tilde{q}^2) = v(\tilde{\rho}-\rho) = v(\rho)$, while their difference has valuation
        \begin{equation} \label{eq v(2 tilde q)}
            v(2\tilde{q})=v(2)+\frac{1}{2}v(h-\tilde{\rho}) \geq v(2)+\frac{1}{2}v(h) > \frac{1}{2}v(\rho).
        \end{equation}
        From this, it is immediate to deduce that they must both have valuation equal to $\frac{1}{2}v(\rho)$.  We may now write 
        \begin{equation} \label{eq good decomposition}
        	\rho = \tilde{\rho} + 2\tilde{q}(\tilde{q}-q) -(\tilde{q}-q)^2.
        \end{equation}
        But we observe that the first two summands both have valuation $>v(\rho)$. This implies that the normalized reduction of $\rho$ is a square.
        		
        Conversely, suppose that the decomposition satisfies $0\le t_{q,\rho} < 2v(2)$ and that the normalized reduction of $\rho$ is a square; this is clearly equivalent to saying that we can form a part-square decomposition $\rho = q_1^2 + \rho_1$ of the polynomial $\rho$ that satisfies $t_{q_1,\rho_1}>0$; hence, we have $v(q_1) = \frac{1}{2}v(\rho)$ and $v(\rho_1) > v(\rho)$.
        
        Let us now consider the part-square decomposition $h = \tilde{q}^2 + \tilde{\rho}$, where $\tilde{q}:=q + q_1$ and $\tilde{\rho}=\rho_1 - 2q q_1$. Notice that the assumption $t_{q,\rho}\ge 0$ implies that $v(q)\ge v(h)/2$; we therefore have $v(2q q_1) \geq v(2) + \frac{1}{2} v(h) + \frac{1}{2}v(\rho) > v(\rho)$. We conclude that $v(\tilde{\rho})=v(\rho_1-2qq_1)>v(\rho)$, i.e.\  $t_{\tilde{q}, \tilde{\rho}}> t_{q, \rho}$; therefore, the original part-square decomposition $h = q^2 + \rho$ was not good.  Thus, both directions of part (a) are proved.
        
        We now turn to part (b) and assume that $h = q^2 + \rho = \tilde{q}^2 + \tilde{\rho}$ are both good decompositions.  Then both (\ref{eq v(2 tilde q)}) and (\ref{eq good decomposition}) are still valid, and the fact that $v(\rho) = v(\tilde{\rho})$ implies that $v(2\tilde{q}(\tilde{q}-q) -(q-\tilde{q})^2) \geq v(\rho)$.  Then if $v(q - \tilde{q}) < \frac{1}{2}v(\rho)$, from (\ref{eq good decomposition}) we must have $v(2\tilde{q}(\tilde{q} - q)) = v((q - \tilde{q})^2) < v(\rho)$, which contradicts (\ref{eq v(2 tilde q)}).  We therefore have $v(\tilde{q} - q) \geq \frac{1}{2}v(\rho)$, from which $v(2\tilde{q}(\tilde{q} - q)) > \frac{1}{2}v(\rho)$ follows from (\ref{eq v(2 tilde q)}).
        
        Let $\gamma \in \bar{K}$ be a scalar with $v(\gamma) = v(\rho) = v(\tilde{\rho})$.  If $v(\tilde{q} - q) > \frac{1}{2}v(\rho)$, then (\ref{eq good decomposition}) shows that $v(\tilde{\rho} - \rho) > v(\rho)$ and so $\gamma^{-1}(\tilde{\rho} - \rho)$ has positive valuation; therefore, the reductions of $\gamma^{-1}\rho$ and $\gamma^{-1}\tilde{\rho}$ are equal, and we are done.  If $v(\tilde{q} - q) = \frac{1}{2}v(\rho)$, then $\gamma^{-1}(\tilde{\rho} - \rho)$ reduces to a square (namely a normalized reduction of $\tilde{q} - q$ squared); the square of a polynomial in $k[x]$ has only even-degree terms, and its derivative vanishes, which shows that the reductions of $\gamma^{-1}\rho$ and $\gamma^{-1}\tilde{\rho}$ have the same odd degrees appearing and have the same derivative.  Thus again we are done, and part (b) is proved.
    \end{proof}
\end{prop}

\begin{cor} \label{cor totally odd is good}
Every totally odd part-square decomposition of a polynomial is good.
    \begin{proof}
    Suppose that the decomposition $h = q^2 + \rho$ is totally odd.  If $t_{q, \rho} \geq 2v(2)$, then we are already done, so assume that $t_{q, \rho} < 2v(2)$.  Then since $\rho$ consists only of odd-degree terms, the same is true of any normalized reduction of $\rho$, which consequently cannot be the square of any polynomial in $k[z]$. Then \Cref{prop good decomposition} implies that the decomposition is good.
    \end{proof}
\end{cor}

We now want to show that a good part-square decomposition of a polynomial always exists, for which, thanks to Corollary \ref{cor totally odd is good}, it suffices to show that a polynomial always has a totally odd part-square decomposition.

\begin{prop} \label{prop totally odd existence}

Given a nonzero polynomial $h(z) \in \bar{K}[z]$, there always exists a totally odd part-square decomposition $h = q^2 + \rho$ with $q(z), \rho(z) \in \bar{K}[z]$.

\end{prop}

\begin{proof}

We write $h_e(z)$ and $h_o(z)$ for the sums of the even- and odd- degree terms of $h(z)$ respectively, so that $h = h_e + h_o$.  We denote the degree of $h_e$ by $2m$. As all terms of $h_e$ have even degree, we may write $h_e(z) = \hat{h}(z^2)$ for some uniquely determined $\hat{h}(z) \in {K}[z]$ of degree $m$. Let $\alpha_1, \ldots, \alpha_m\in \bar{K}$ be the roots of $\hat{h}$, and let us denote by $c \in K$ its leading coefficient. Let us also choose a square root $\sqrt{c} \in \bar{K}$ of $c$ and a square root $\sqrt{\alpha_i} \in \bar{K}$ of each root of $\hat{h}$, and let us define
\begin{equation*}
	\begin{split}
		\hat{h}_+(z):=&\sqrt{c}\prod_i (z+\sqrt{\alpha_i}) = c_0 z^m + c_1 z^{m-1} + \ldots + c_m,\\
		\hat{h}_-(z):=&\sqrt{c}\prod_i (z-\sqrt{\alpha_i}) = c_0 z^m - c_1 z^{m-1} + \ldots + (-1)^m c_m.\\
	\end{split}
\end{equation*}
It is clear that we have $h_e(z) = \hat{h}_+(z) \hat{h}_-(z)$; exploiting this factorization of $h_e$, we may write the $k$th-order coefficient of $h$, whenever $k$ is even, as 
\begin{equation}
	\sum_{i+j=2m-k} (-1)^i c_i c_j.
	\label{Coefficientk}
\end{equation}
If we now choose a square root $\sqrt{-1} \in \bar{K}$ of $-1$, and we define 
\begin{equation*}
	c'_i := \begin{cases}
		c_i & \text{if $i$ is even},\\
		\sqrt{-1}\cdot c_i & \text{if $i$ is odd},
	\end{cases}
\end{equation*}
we may rewrite the expression (\ref{Coefficientk}), for all even values of $k$, in the more symmetric form
\begin{equation}
	\sum_{i+j=2m-k} c'_i c'_j.
\end{equation}
If we now set $q(z) = c'_0 z^m + \ldots + c'_m$, the even-degree terms of $q^2$ reproduce $h_e$. Hence, the polynomial $\rho(z) := h(z) - q^2(z)$ only consists of odd-degree terms: in other words, $h = q^2 + \rho$ is a totally odd part-square decomposition for $h$.
\end{proof}

We note that, if a nonzero polynomial $h\in \bar{K}[z]$ is written as a product of factors $h=\prod_{i=1}^N h_i$ with $h_i\in \bar{K}[z]$, then, given part-square decompositions $h_i=q_i^2+\rho_i$, with $q_i, \rho_i\in \bar{K}[z]$, one can use them to form a part-square decomposition $h = q^2+\rho$, where $q = \prod_{i=1}^N q_i$ and $\rho = h-q^2$. We have the following.
\begin{prop}
    \label{prop product part-square}
    In the setting above, let $t_i:=t_{q_i,\rho_i}$ and $t:=t_{q,\rho}$.
    \begin{enumerate}[(a)]
        \item If $t_i\ge 0$ for all $i$, then we have $t\ge \min\{t_1,\ldots t_N\}$.
        \item If $t_i\ge 0$ for all $i$, and  the minimum $\min\{t_1,\ldots t_N\}$ is achieved by only one of the $t_i$'s, then we have $t= \min\{t_1,\ldots t_N\}$; moreover, in this case, if $i_0$ is the index such that $t_{i_0}<t_i$ for all $i$, then the part-square decomposition of $h$ is good if and only if that of $h_{i_0}$ is.
        \item Assume that $N=2$, and suppose that, for all roots $s_1$ in $\bar{K}$ of $h_1$ and for all roots $s_2$ of $h_2$, we have $v(s_1)>0$ but $v(s_2)<0$; assume, moreover, that both decompositions $h_i = q_i^2 + \rho_i$ are good. Then, if $\min\{t_1,t_2\}<2v(2)$, we have $t=\min\{t_1,t_2\}$, and the corresponding decomposition of $h$ is also good.
    \end{enumerate}
    \begin{proof}
        Let us first address points (a) and (b). It is clearly enough to prove these results for $N=2$. In this case, we have
        \begin{equation}
            \rho=h-q^2 = (q_1^2+\rho_1)(q_2^2+\rho_2)-(q_1 q_2)^2 = g_1+g_2+g_3,
        \end{equation}
        where $g_1, g_2, g_3\in \bar{K}[z]$ are the polynomials
        \begin{equation}
            g_1=\rho_1 q_2^2,\qquad g_2=\rho_2 q_1^2,\qquad g_3=\rho_1 \rho_2.
        \end{equation}
        Since $t_i\ge 0$, i.e.\ $v(\rho_i)\ge v(h_i)$, we have $v(q_i^2)=v(h_i-\rho_i)\ge v(h_i)$ for $i=1, 2$. Moreover, we have $v(\rho_i)=v(h_i)+t_i$.  We therefore get $
            v(g_1)\ge t_1+v(h), 
            v(g_2)\ge t_2+v(h), \text{ and }
            v(g_3)\ge t_1+t_2+v(h);$
        all three thresholds are clearly $\ge \min\{t_1,t_2\}+ v(h)$, from which we deduce $v(\rho)\ge \min\{t_1,t_2\}+ v(h)$, and thus $t\ge \min\{t_1,t_2\}$. 
        
        To prove part (b), let us now further assume that $t_1<t_2$. This implies, in particular, $t_2>0$; hence we have $v(\rho_2)>v(h_2)$ and $v(q_2^2)=v(h_2-\rho_2)=v(h_2)$, and we consequently get $v(g_1)= t_1+v(h)$. We deduce from this that $v(g_2)\ge t_2+v(h)>v(g_1)$, and $v(g_3)\ge t_1+t_2+v(h)>v(g_1)$. It follows that $v(\rho)=v(g_1)=t_1+h$, implying that $t=t_1=\min\{t_1,t_2\}$. Moreover, the normalized reduction of $\rho$ equals that of $g_1=\rho_1 q_2^2$; as a consequence, the normalized reduction of $\rho$ is a square if and only if that of $\rho_1$ is.  From this, together with the fact that $t_1=t$, we deduce that $h=q^2+\rho$ is a good decomposition if and only if $h_1=q_1^2+\rho_1$ is (see \Cref{prop good decomposition}).
        
        Let us now address part (c). Let $d_i$ be the degree of $h_i$, and let $\gamma_i\in \bar{K}^\times$ be an element of valuation $v(h_i)$ for $i=1,2$. Since we are assuming that $v(s_1)>0$, for all root $s_1$ of $h_1$ we have $\overline{\gamma_1^{-1}h_1}(z)=c_1 z^{d_1}$ for some $c_1\in k^\times$; similarly, since we have $v(s_2)<0$ for all roots $s_2$ of $h_2$ we have that $\overline{\gamma_2^{-1}h_2}(z)$ is a constant $c_2\in k^\times$.
        Since the decompositions of $h_1$ and $h_2$ are assumed to be good, we have $t_1, t_2\ge 0$; moreover, from the fact that the normalized reduction of $h_2$ is a square we deduce, via \Cref{prop good decomposition}, that $t_2>0$.
        Now, when $t_1\neq t_2$ the conclusion follows from (b). We are consequently only left to address the case where $0<t_1=t_2<2v(2)$.  Here we already know that $t>0$ from part (a), and that $d_1$ must necessarily be even, because, since $t_1>0$, $\overline{\gamma_1^{-1}h_1}=\overline{\gamma_1^{-1}q_1^2}$ must be a square; let us write $d_1=2m$.
        
        In the case we are considering, we clearly have that $v(g_1)=v(g_2)=t_1+v(h)$, while $v(g_3)>t_1+v(h)$. If we let $\gamma\in \bar{K}^\times$ be any element of valuation $t$, we consequently have that
        \begin{equation}
            \overline{\gamma_1^{-1}\gamma_2^{-1}\gamma^{-1}\rho} =  (\overline{\gamma_2^{-1}h_2}) \overline{r_1} + (\overline{\gamma_1^{-1}h_1}) \overline{r_2} = c_2 \overline{r_1} + c_1 z^{2m} \overline{r_2}
        \end{equation}
        where ${r_1}:={\gamma^{-1}\gamma_1^{-1}\rho_1}$ and ${r_2}:={\gamma^{-1}\gamma_2^{-1}\rho_2}$, so that $\overline{r}_1$ and $\overline{r}_2$ are normalized reduction of $\rho_1$ and $\rho_2$, respectively.  We remark $\overline{r_1}$ has degree $\deg(\overline{r}_1)\le 2m$; hence, if an odd-degree term of degree $s$ appears in the normalized reduction $\overline{r_1}$ of $\rho_1$ (resp.\ in the normalized reduction $\overline{r_2}$ of $\rho_2$), then an odd-degree term of degree $s$ (resp.\ $s+2m$) will also show up in $\overline{\gamma_1^{-1}\gamma_2^{-1}\gamma^{-1}\rho}$: roughly speaking, in the expression for $\overline{\gamma_1^{-1}\gamma_2^{-1}\gamma^{-1}\rho}$ no cancellation occurs between the odd-degree monomials of $\overline{r_1}$ and those of $\overline{r_2}$. We conclude that, since $\overline{r_1}$ and $\overline{r_2}$ are not squares by \Cref{prop good decomposition}, the reduced polynomial $\overline{\gamma_1^{-1}\gamma_2^{-1}\gamma^{-1}\rho}$ is not a square, so that $v(\rho)=t_1+v(h)$ (i.e., $t=t_1=t_2$), and the decomposition of $h$ is good by \Cref{prop good decomposition}.
    \end{proof}
\end{prop}

\subsection{Forming models of \texorpdfstring{$Y$}{Y} using part-square decompositions} \label{sec models hyperelliptic forming}
In this subsection, we compute the model of the hyperelliptic curve $Y: y^2=f(x)$ corresponding to any given smooth model of the projective line $X$, in the sense of \S\ref{sec relatively stable galois covers}.

More precisely, choose elements $\alpha \in \bar{K}$ and $\beta \in \bar{K}^\times$.  Given any polynomial $h(x) \in \bar{K}[x]$, define the translated and scaled coordinate $x_{\alpha,\beta} = \beta^{-1}(x - \alpha)$ (as defined in \S\ref{sec models hyperelliptic line}), and let $h_{\alpha,\beta}$ be the polynomial such that $h_{\alpha,\beta}(x_{\alpha,\beta}) = h(x)$.  Let $D := D_{\alpha,b}$ be a disc in $\bar{K}$ with $\alpha\in \bar{K}$ and $b=v(\beta)$ for some $\beta\in \bar{K}^\times$. To this disc we can attach (see \S\ref{sec models hyperelliptic line}) a smooth model $\XX_D$ of the line $X$, defined over some extension of $R$. We will show that, after possibly replacing this extension with a further extension $R'$, which in particular will be large enough so that $f_{\alpha,\beta}$ admits a good part-square decomposition $f_{\alpha,\beta}=q_{\alpha,\beta}^2+\rho_{\alpha,\beta}$ over the fraction field $K'$ of $\Frac(R')$, the model of $Y$ corresponding to $\XX_D/R'$ has reduced special fiber, and its equation can explicitly be written using $q_{\alpha,\beta}$ and $\rho_{\alpha,\beta}$; we will denote this model by $\YY_{D}$.

The strategy will be the following one: after a suitable change of the coordinate $y$, we will rewrite the equation $y^2=f_{\alpha,\beta}(x_{\alpha,\beta})$ of the hyperelliptic curve $Y$ in the form
\begin{equation}
    \label{equation front chart YD}
    y^2 + q_0(x_{\alpha,\beta})y - \rho_0(x_{\alpha,\beta})=0,\quad\text{with }\deg(q_0)\le g+1,\ \deg(\rho_0)\le 2g+2,
\end{equation}
such that the following conditions are satisfied:
\begin{enumerate}[(a)]
    \item $\rho_0$ and $q_0$ have integral coefficients (i.e., we have $\rho_0(x_{\alpha,\beta}), q_0(x_{\alpha,\beta}) \in R'[x_{\alpha,\beta}]$);
    \item the $k$-curve given by the reduction of the equation in (\ref{equation front chart YD}) is reduced.
\end{enumerate}

Then, the model $\YY_D$ is constructed as follows.  The equation in (\ref{equation front chart YD}) above defines a scheme $W$ over $R'$ whose generic fiber is isomorphic to the affine chart $x_{\alpha,\beta}\neq \infty$ of the hyperelliptic curve $Y$. The coordinate $x_{\alpha,\beta}$ defines a map $W\to \XX_D$, whose image is the affine chart $x_{\alpha,\beta}\neq \infty$ of $\XX_D$. Over the affine chart $x_{\alpha,\beta}\neq 0$ of $\XX_D$, we can correspondigly form the $R'$-scheme $W^\vee$ defined by the equation \begin{equation}
\begin{split}
    \label{equation rear chart YD}
    \check{y}^2+q_0^\vee(\check{x}_{\alpha,\beta})\check{y}-\rho_0^\vee(\check{x}_{\alpha,\beta})&=0, \qquad\text{where }\  \check{x}_{\alpha,\beta}=x_{\alpha,\beta}^{-1},\  \check{y}=x_{\alpha,\beta}^{-(g+1)}y, \\
    q_0^\vee(\check{x}_{\alpha,\beta})&={x_{\alpha,\beta}}^{-(g+1)} q_0(x_{\alpha,\beta}), \ \text{and} \ \rho_0^\vee(\check{x}_{\alpha,\beta})={x_{\alpha,\beta}}^{-(2g+2)} \rho_0(x_{\alpha,\beta}).
\end{split}
\end{equation}

We can now \emph{define} $\YY_D$ to be the scheme obtained by gluing the affine charts $W$ and $W^\vee$ together in the obvious way: it is endowed with a degree-$2$ covering map $\YY_D\to \XX_D$, and its generic fiber is identified with the hyperelliptic curve $Y\to X$.

\begin{prop}
    The scheme $\YY_D$ constructed above, which is defined over an appropriate extension $R'$ of $R$, coincides with the normalization of $\XX_D/R'$ in the function field of the hyperelliptic curve $Y$, and it is a model of $Y$ whose special fiber is reduced.
    \begin{proof}
        We have to show that the scheme $\YY_D$ we have constructed is normal. The $R'$-schemes $W$ and $W^\vee$ are complete intersections, and hence they are Cohen-Macaulay; as a consequence, to check that $\YY_D$ is normal, it is enough to prove that it is regular at its codimension-1 points. Since the generic fiber of $\YY_D$ coincides with $Y$, it is certainly regular; hence, all that is left is to check that $\YY_D$ is regular at the generic point $\eta_{V_i}$ of each irreducible component $V_i$ of the special fiber $\SF{\YY_D}$. Since we are assuming that the $k$-curve $W_s$ is reduced, the $k$-curve $\SF{\YY_D}$ is also clearly reduced, which implies that $\YY_D$ is certainly regular at the points $\eta_{V_i}$.  Thus, the scheme $\YY_D$ is actually normal.
        
        It is now completely clear that $\YY_D$ is the model of $Y$ obtained by normalizing $\XX_D/R'$ in the function field of $Y$.
    \end{proof}
\end{prop}

All we have to do now is determine a change of the coordinate $y$ such that conditions (a) and (b) above are satisfied. To do this, suppose that we are given a good part-square decomposition $f_{\alpha, \beta} = q_{\alpha,\beta}^2 + \rho_{\alpha,\beta}$ (which certainly exists over some extension of $K$, thanks to \Cref{prop totally odd existence} and \Cref{cor totally odd is good}).  Let $\gamma\in \bar{K}^\times$ be an element whose valuation is $v(\gamma)=\truncate{t}+v(f_{\alpha,\beta})$, where $t=t_{q_{\alpha,\beta},\rho_{\alpha,\beta}}=v(\rho_{\alpha,\beta})-v(f_{\alpha,\beta})$. We remark that we necessarily have $t\ge 0$, since the part-square decomposition is assumed to be good (see \Cref{rmk good decompositions}). The change of variable we perform is $y \mapsto \gamma^{1/2} y + q_{\alpha,\beta}(x_{\alpha,\beta})$, and it leads to an equation of the form (\ref{equation front chart YD}) with 
\begin{equation}
    \label{equation q_0 rho_0 formulas YD}
        q_0 = 2\gamma^{-1/2}q_{\alpha,\beta} \qquad \mathrm{and} \qquad
        \rho_0 = \gamma^{-1}\rho_{\alpha,\beta}.
\end{equation}

The valuations of $q_0$ and $\rho_0$ can be computed as follows.
\begin{enumerate}
    \item For $q_0$, we have $2v(q_0)= 2v(2)-\truncate{t}+2v(q_{\alpha,\beta})-v(f_{\alpha,\beta})$. Let us remark that, since $t\ge 0$, we have $2v(q_{\alpha,\beta})\ge v(f_{\alpha,\beta})$, and moreover equality holds whenever $t>0$. We deduce that:
    \begin{enumerate}[(i)]
        \item $v(q_0)\ge 2v(2)-\truncate{t}$ for all $t$, so that $q_0$ is consequently always integral;
        \item $v(q_0)=2v(2)-\truncate{t}$ whenever $t>0$;
        \item $v(q_0)>0$ if $0\le t< 2v(2)$ (which can only happen in the $p=2$ setting); and 
        \item $v(q_0)=0$ if $t\ge 2v(2)$ in the $p=2$ setting.
    \end{enumerate}
    \item For $\rho_0$, we have $v(\rho_0) = t - \truncate{t}$; in particular,
    \begin{enumerate}[(i)]
        \item $\rho_0$ is always integral;
        \item $v(\rho_0)=0$ if $0\le t\le 2v(2)$; and 
        \item $v(\rho_0)>0$ if $t>2v(2)$.
    \end{enumerate}
\end{enumerate}
These computations guarantee that condition (a) is satisfied. We now verify that also condition (b) is satisfied.
\begin{lemma}
    In the context above, condition (b) is also satisfied, i.e.\ the reduction of equation (\ref{equation front chart YD}) defines a reduced $k$-curve. Moreover, this curve is a separable (resp.\ inseparable) quadratic cover of the $k$-line of coordinate $x_{\alpha,\beta}$ if and only if $t\ge 2v(2)$ (resp.\ $0\le t<2v(2)$).
    \begin{proof}
        Suppose by way of contradiction that the $k$-curve defined by the reduction of  (\ref{equation front chart YD}) is non-reduced. This is clearly equivalent to saying that the polynomial $g(x_{\alpha,\beta},y)\in k[x_{\alpha,\beta},y]$ given by the reduction of (\ref{equation front chart YD}) (i.e.\  $g(x_{\alpha,\beta},y):=y^2+\overline{q_0(x_{\alpha,\beta})}y - \overline{\rho_0(x_{\alpha,\beta})}$) is a square. If we treat $g(x_{\alpha,\beta},y)$ as a monic quadratic polynomial in the variable $y$, we can say that it is a square if and only if its constant term $\overline{\rho_0(x_{\alpha,\beta})}\in k[x_{\alpha,\beta}]$ is a square, and its discriminant $\Delta=\overline{{q}_0}^2+4\overline{\rho_0}=\overline{4\gamma^{-1} f_{\alpha,\beta}}\in k[x_{\alpha,\beta}]$ is zero. However, when $t\ge 2v(2)$, we have $v(\gamma)=v(4f_{\alpha,\beta})$ and therefore $\Delta\neq 0$; when $0\le t<2v(2)$, the reduced polynomial $\overline{\rho_0}$ is a normalized reduction of $\rho_{\alpha,\beta}$, which is not a square by \Cref{prop good decomposition}. We conclude that the $k$-curve $g(x_{\alpha,\beta},y)=0$ is always reduced.
        
        Now the coordinate $x_{\alpha,\beta}$ defines a quadratic cover from the $k$-curve $g(x_{\alpha,\beta},y)=0$ to the affine $k$-line, and it is immediate to realize that this cover is inseparable only when $p=2$ and the linear term $\overline{q_0(x_{\alpha,\beta})}y$ vanishes, which happens if and only if $0<t\le 2v(2)$.
    \end{proof}
\end{lemma}

The following proposition summarizes the results we have obtained.
\begin{prop}
	\label{prop normalization model}
	
	Let $\XX_D$ be the smooth model of the line corresponding to the disc $D:=D_{\alpha,v(\beta)}$, with $\alpha \in \bar{K}$ and $\beta \in \bar{K}^\times$.  Then, after replacing $K$ with an appropriate finite extension, the normalization $\YY_D$ of $\XX_D$ in $K(Y)$ has reduced special fiber. Given a good part-square decomposition $f_{\alpha, \beta} = q_{\alpha,\beta}^2 + \rho_{\alpha,\beta}$, and letting $t=t_{q_{\alpha,\beta}, \rho_{\alpha,\beta}}$, the model $\YY_D$ falls under (exactly) one of the following two cases:
	\begin{enumerate}
		\item $t\geq 2v(2)$; in this case, $\SF{\YY_D}$ is a separable degree-2 cover of $\SF{\XX_D}$; and 
		\item $0\le t < 2v(2)$; in this case, $\SF{\YY_D}$ is an inseparable degree-2 cover of $\SF{\XX_D}$.
	\end{enumerate}
	The equations describing the affine charts $x_{\alpha,\beta}\neq \infty$ and $x_{\alpha,\beta}\neq 0$ of the model $\YY_D$ have the form (\ref{equation front chart YD}) and (\ref{equation rear chart YD}) respectively, and they can be explicitly computed from $q_{\alpha,\beta}$ and $\rho_{\alpha,\beta}$ using the formulas in (\ref{equation q_0 rho_0 formulas YD}).
\end{prop}

\subsection{The special fiber \texorpdfstring{$\SF{\YY_D}$}{(YD)s} in the separable case}
\label{sec models hyperelliptic separable}
We will now study the special fiber of the model $\YY_D$ associated to a given disc $D:=D_{\alpha,v(\beta)}$ be a disc with $a\in \bar{K}$ and $\beta \in \bar{K}^\times$, which was computed in the previous subsection.  This subsection will consider the case in which $\SF{\YY_D}\to \SF{\XX_D}$ is separable: this means that it is possible to find a part-square decomposition $f_{\alpha,\beta}=q_{\alpha,\beta}^2+\rho_{\alpha,\beta}$ satisfying $t:=t_{q_{\alpha,\beta}, \rho_{\alpha,\beta}} \geq 2v(2)$, and the equation of $\SF{\YY_D}$ has the form $y^2 + \overline{q_0(x_{\alpha,\beta})}y = \overline{\rho_0(x_{\alpha,\beta})}$, where $q_0:=2\gamma^{-1/2}q_{\alpha,\beta}$, and $\rho_0=\gamma^{-1}\rho_{\alpha,\beta}$, where $\gamma \in \bar{K}^{\times}$ is an element of valuation $v(f_{\alpha,\beta})+2v(2)$. In the $p\neq 2$ case, the equation of $\SF{\YY_D}$ can also be written in the simpler form $y^2=\overline{f_0(x_{\alpha,\beta})}$, where \begin{equation}
    \label{equation f0}
    f_0 = 4\gamma^{-1}f_{\alpha,\beta} = q_0^2+4\rho_0.
\end{equation}

We remark that the separable quadratic cover $\SF{\YY_D}\to \SF{\XX_D}$ is branched precisely above the points $P_1, \ldots, P_N$ of $\SF{\XX_D}$ at which the roots $\Rinfty$ reduce and is étale elsewhere: this can be seen directly from the equation of $\SF{\YY_D}$, or can be deduced from the fact the branch locus of $\YY_D\to \XX_D$ has pure dimension 1 by Zariski–Nagata purity theorem. In order to state and prove the results in this subsection, we partition the branch locus $R=\lbrace P_1, \ldots, P_N\rbrace\subseteq \SF{\XX_D}(k)$ in three subsets as $R=R_0\sqcup R_1\sqcup R_2$, in the following way.
\begin{equation*}
    \begin{split}
    R_0&=\lbrace P\in \SF{\XX_D}: \text{$\SF{\YY_D}$ exhibits a unique smooth point $Q$ above $P$}\rbrace;\\
    R_1&=\lbrace P\in \SF{\XX_D}: \text{$\SF{\YY_D}$ has a (unique) singular point $Q$ above $P$ and has one branch at $Q$}\rbrace;\\
    R_2&=\lbrace P\in \SF{\XX_D}: \text{$\SF{\YY_D}$ has a (unique) singular point $Q$ above $P$ and has two branches at $Q$}\rbrace.
    \end{split}
\end{equation*}
We denote the cardinality of each subset $R_i \subseteq R$ by $N_i$ for $i = 0, 1, 2$.
\begin{rmk} \label{rmk R_0 R_1 R_2}
    The following statements are clear from the definitions above.
    \begin{enumerate}[(a)]
        \item The set $R_0\cup R_1$ is precisely the branch locus of the quadratic cover $\NSF{\YY_D}\to \SF{\XX_D}$, where $\NSF{\YY_D}$ is the normalization of the $k$-curve $\SF{\YY_D}$.
        \item The curve $\SF{\YY_D}$ has exactly $N_1+N_2$ singular points, which lie over the $N_1+N_2$ points of $R_1\cup R_2$.
        \item The unique point $Q\in \SF{\YY_D}$ lying over some given $P\in R$ is fixed by the action of the hyperelliptic involution. If $P\in R_2$, the two branches of $\SF{\YY_D}$ passing through $Q$ get flipped by the hyperelliptic involution.
        \item The special fiber $\SF{\YY_D}$ either consists of two components flipped by the hyperelliptic involution, or it is irreducible. In the first case (which always occurs, for example, if $\overline{\rho_0(x_{\alpha,\beta})}$ is the zero polynomial, i.e.\ if $t>2v(2)$), the two components are necessarily two lines that trivially cover $\SF{\XX_D}$, while $\NSF{\YY_D}$ is their disjoint union, and we have $R_0\cup R_1 = \varnothing$.  If $\SF{\YY_D}$ is irreducible, however, the quadratic cover $\NSF{\YY_D}\to \SF{\XX_D}$ is necessarily ramified, because $\mathbb{P}^1_k$ does not have non-trivial finite étale connected covers: hence, and we have $R_0\cup R_1 \neq \varnothing$.
    \end{enumerate}
\end{rmk}

We want to better understand the ramification behaviour of $\NSF{\YY_D}\to \SF{\XX_D}$ above the points of $R_0\cup R_1$; to this aim, we can measure, above each point, the length of the module of relative K\"{a}hler differentials of the cover.
\begin{dfn}
    \label{dfn ell ramification index}
    Given $P\in \SF{\XX_D}(k)$, we set \begin{equation*}\ell(\XX_D,P) = \length_{\OO_{\SF{\XX_D},P}}\left(\Omega_{\NSF{\YY_D}/\SF{\XX_D}}\otimes \OO_{\SF{\XX_D},P}\right).\end{equation*}
\end{dfn}

\begin{rmk}
    \label{rmk ell}
    For any $P\in \SF{\XX_D}(k)$, the integer $\ell(\XX_D,P)$ satisfies the following properties.
    \begin{enumerate}[(a)]
    \item If $P\not\in R_0\cup R_1$, then $\NSF{\YY_D}\to \SF{\XX_D}$ is unramified over $P$, and we thus have $\ell(\XX_D,P)=0$.
    \item If $P\in R_0\cup R_1$, and we denote by $Q$ its unique preimage $\NSF{\YY_D}$, the ramification index of the cover $\NSF{\YY_D}\to \SF{\XX_D}$ at $Q$ is $e_Q=2$, and \cite[Proposition 7.4.13]{liu2002algebraic} ensures that $\ell(\XX_D,P)\ge e_Q-1$, with equality if and only if the cover is tame. This means that $\ell(\XX_D,P)=1$ if $p\neq 2$, and $\ell(\XX_D,P)\ge 2$ if $p=2$.
    \end{enumerate}
\end{rmk}

The knowledge of $\ell(\XX_D,P)$ at the points $P \in \SF{\XX_D}$ gives us information about the abelian rank of $\SF{\YY_D}$.

\begin{prop}
    \label{prop riemann hurwitz}
    The genus of $\NSF{\YY_D}$ is given by
    \begin{equation}\label{equation rh}
        g\left(\NSF{\YY_D}\right)=-1+\frac{1}{2}\sum_{P\in \SF{\XX_D}(k)} \ell(\XX_D,P),
    \end{equation}
    with the convention that the genus of the disjoint union of two lines is $-1$.  In the $p\neq 2$ setting, this can be rewritten as
    \begin{equation}
        \label{equation rh odd p}
       g\left(\NSF{\YY_D}\right)=-1+\frac{1}{2}(N_0+N_1),
    \end{equation}
    while, in the $p=2$ setting, the formula in (\ref{equation rh}) implies the inequality
    \begin{equation}
        \label{equation rh p=2}
       g\left(\NSF{\YY_D}\right)\ge -1+(N_0+N_1).
    \end{equation}
    \begin{proof}
        Equation (\ref{equation rh})  is just the Riemann-Hurwitz formula (see, for example, \cite[Theorem 7.4.16]{liu2002algebraic}), while (\ref{equation rh odd p}) and (\ref{equation rh p=2}) follow from (\ref{equation rh}) via \Cref{rmk ell}.
    \end{proof}
\end{prop}

\begin{rmk}
    \label{rmk N0N1 even}
    In particular, the formula in (\ref{equation rh odd p}) implies that, in the $p\neq 2$ setting, the integer $N_0+N_1$ is necessarily even.
\end{rmk}

We now see how to compute $\ell(\XX_D,P)$ for a given point $P\in \SF{\XX_D}$ from the good part-square decomposition $f_{\alpha,\beta}=q_{\alpha,\beta}^2+\rho_{\alpha,\beta}$ given.

\begin{lemma}
    \label{lemma computation ell ramification}
    Choose $P\in \SF{\XX_D}(k)$. Let us denote by $n_q(P):=\ord_P(\overline{q_0})$, $n_\rho(P):=\ord_P(\overline{\rho_0})$, $n_f(P):=\ord_P(\overline{f_0})$ the respective orders of vanishing at the point $P$ of the reductions of the polynomials $q_0$, $\rho_0$ and $f_0$ defined in (\ref{equation q_0 rho_0 formulas YD}) and (\ref{equation f0}), with the convention that the zero polynomial has vanishing order $\infty$, and that, if $P=\infty$, the vanishing orders of $\overline{f_0}$, $\overline{\rho_0}$ and $\overline{q_0}$ at $\infty$ are respectively those of $\overline{f_0^\vee}$, $\overline{\rho_0^\vee}$ and $\overline{q_0^\vee}$ at $0$, i.e.\  $n_q=g+1 - \deg(\overline{q_0})$,  $n_\rho=2g+2 - \deg(\overline{\rho_0})$ and $n_f=2g+2 - \deg(\overline{f_0})$. Then,
    \begin{enumerate}[(a)]
        \item if $p\neq 2$, then $\ell(\XX_D,P)$ is $0$ (resp.\ $1$) if the integer $n_f$ is even (resp.\ odd);
        \item if $p=2$ and $2n_q(P)\le n_\rho(P)$, then we have $\ell(\XX_D,P)=0$; and
        \item if $p=2$ and if $2n_q(P)>n_\rho(P)$ and $n_\rho(P)$ is odd, then we have $\ell(\XX_D,P)=2n_q(P)-n_\rho(P)+1$.
    \end{enumerate}
    \begin{proof}
        We lose no generality in assuming $P$ has coordinate $\overline{x_{\alpha,\beta}}=0$.  For brevity we write $z$ for the variable $x_{\alpha,\beta}$. We proceed by desingularizing $\SF{\YY_D}$ above $P$ by means of a sequence of blowups.  Let us first work in the $p \neq 2$ setting.  The equation of $\SF{\YY_D}$, in this case, has the form $y^2 = \overline{f_0}(z) = z^{n_f}f_1(z)$, with $f_1(z)\in k[z]$ and $f_1(0) \neq 0$. If $n_f = 0$, then we are already done; otherwise, the curve becomes nonsignular above $z=0$ after blowing it up $\lfloor n_f/2 \rfloor$ times at $(0, 0)$; at each blowup, the right-hand side of the equation is divided by $z^2$, so that the desingularized equation becomes $y^2 = z^{e}f_1(z)$, where $e$ is $0$ or $1$, depending on whether $n_f$ is even or odd; moreover, when $e=1$ this curve is ramified over $z=0$, whereas, when $e=0$, it is étale over $z=0$. From this, (a) follows, taking into account \Cref{rmk ell}.
        
        Let us now adopt the $p=2$ setting. The equation of $\SF{\YY_D}$ is now $y^2 + \overline{q_0}(z)y = \overline{\rho_0}(z)$, with $\overline{q_0}(z) = z^{n_q} q_1(z)$, and $\overline{\rho_0}(z) = z^{n_\rho} r_1(z)$, where $q_1(z), \rho_1(z) \in k[z]$ do not vanish at 0.
        
        Assume that $2n_q \le n_\rho$. Then, after $n_q$ blowups at $(0,0)$, we obtain $y^2 + q_1(z)y = z^{n_\rho-2n_q} \rho_1(z)$. Since $q_1(0)\neq 0$, there are exactly $2$ solutions for $y$ at $z=0$, which means that the blown-up curve is étale above $P$, implying that $\ell(\XX_D,P) = 0$.  We have thus proved part (b).
        
        Assume that $2n_q > n_\rho$ and that $n_\rho$ is odd. Then, after $(n_\rho-1)/2$ blowups at $(0,0)$, we obtain the equation 
        \begin{equation} \label{eq 2n_q > n_rho}
            y^2 + z^{n_q-(n_\rho-1)/2}q_1(z)y = z \rho_1(z).
        \end{equation}
        The curve given by (\ref{eq 2n_q > n_rho}) has a unique point $(0,0)$ above $z=0$ and it is non-singular at that point; this is enough to guarantee that $\ell(\XX_D,P) > 0$.  Let $B := k[z,y]_{(z)}/(\text{equation in }(\ref{eq 2n_q > n_rho}))$ be the local ring of functions on the blown-up curve at $(0,0)$, which is a free $k[z]_{(z)}$-algebra of rank 2. Then $\ell(\XX_D,P)$ equals the length of the $k[z]_{(z)}$-module $\Omega_{B/k[z]_{(z)}}$, or, equivalently, the dimension over $k$ of $\Omega_{B/k[z]_{(z)}}$. We have an isomorphism of $k[z]_{(z)}$-modules 
        \begin{equation}
            \Omega_{B/k[z]_{(z)}} = B dy / (z^{n_q-(n_\rho-1)/2}q_1(z) dy) \stackrel{\sim}{\to} (k[z]_{(z)})[y] / (y^2 - zr_1(z), z^{n_q-(n_\rho-1)/2}q_1(z)),
        \end{equation}
        where the isomorphism is given by sending $dy$ to 1.  The latter $k[z]_{(z)}$-module, however, is a free algebra of rank 2 over the ring
        $$k[z]_{(z)}/(z^{n_q-(n_\rho-1)/2}q_1(z)) \cong k[z]/(z^{n_q-(n_\rho-1)/2}),$$
        which clearly has dimension $n_q-(n_\rho-1)/2$ over $k$.  From this, part (c) follows.
    \end{proof}
\end{lemma}
\begin{rmk} \label{rmk ramification index}
    We make the following observations about the subsets $R_i \subseteq R$.
    \begin{enumerate}[(a)]
        \item Assume that $p\neq 2$. \Cref{lemma computation ell ramification} tells us that, for all $P\in \SF{\XX_D}$, the integer $\ell(\XX_D,P)$ is 0 or 1 depending on whether an even or an odd number of the $2g+2$ points of $\Rinfty$ reduce to $P$.  In light of \Cref{rmk ell}, we conclude that, in the $p\neq 2$ case, $R_2$ (resp.\ $R_0\cup R_1$) is the set of points of $\SF{\XX_D}$ at which an \emph{even} (resp.\ \emph{odd}) number of roots of $\Rinfty$ reduce.  Actually, it is also easy to see that
        \begin{enumerate}[(i)]
            \item $P\in R_0$ if and only if exactly one root of $\Rinfty$ reduces to it;
            \item $P\in R_1$ if and only if only an odd number $\ge 3$ of roots of $\Rinfty$ reduces to it; and 
            \item $P\in R_2$ if and only if an even number $\ge 2$ of roots of $\Rinfty$ reduces to it.
        \end{enumerate}
        When we partition the even-cardinality set $\Rinfty$ according to the points of $\SF{\XX_D}$ at which its elements reduce, the number of odd cardinality classes must be even: this shows that $N_1+N_0$ is even, as we have already observed in \Cref{rmk N0N1 even}.
        
        \item Assume that $p=2$. \Cref{lemma computation ell ramification} allows us to calculate $\ell(\XX_D,P)$ from a given good part-square decomposition of $f_{\alpha,\beta}$ only in certain cases: in fact, when $2n_q(P)> n_\rho(P)$ and $n_\rho(P)$ is even, the lemma is inconclusive. At the same time, we remark that if we choose a \emph{totally odd} part-square decomposition for $f_{\alpha,\beta}$ (which can always be done by \Cref{prop totally odd existence}), the polynomial $\overline{\rho_0}$ will certainly have a zero of odd multiplicity at the points $0$ and $\infty$ of $\SF{\XX_D}$; hence, we will certainly be able to compute $\ell(\XX_D,P)$ where $P$ is the point $\overline{x_{\alpha,\beta}}=0$ or the point $\overline{x_{\alpha,\beta}}=\infty$ via the lemma. In other words, given a point $P\in \SF{\XX_D}$, by appropriately choosing the center $\alpha$ of the disc $D$ and constructing a totally odd decomposition for $f_{\alpha,\beta}$, \Cref{lemma computation ell ramification} allows us to compute $\ell(\XX_D,P)$ at the point, and the result it produces is a non-negative even integer.
    \end{enumerate}
\end{rmk}

We now give a criterion to determine whether $\XX_D\le \Xrst$ (which is equivalent to saying that $\YY_D\le \Yrst$).
\begin{thm}
    \label{thm part of rst separable}
    Assume that $D$ is a disc such that $\SF{\YY_D}\to \SF{\XX_D}$ is separable, and let $N$ denote the number of points of $\SF{\XX_D}$ to which the roots $\Rinfty$ reduce. We have $\XX_D\le \Xrst$ if and only if one of the following conditions holds:
    \begin{enumerate}
            \item $N\ge 3$;
            \item $p=2$, $N=2$ and $\SF{\YY_D}$ is irreducible; or 
            \item $p=2$, $N=1$ and $\SF{\YY_D}$ is irreducible of positive abelian rank.
    \end{enumerate}
    Moreover, whenever $\XX_D\le \Xrst$, the strict transform of the $k$-curve $\SF{\YY_D}$ in $\SF{\Yrst}$ is smooth, and it consequently coincides with its normalization $\NSF{\YY_D}$.
    \begin{proof}
        The result essentially follows from a combinatorial argument that directly makes use of the description we have given of $\SF{\YY_D}$ in this subsection, by applying the criterion we have presented in \Cref{prop part of rst}.  Let us write $N=N_0+N_1+N_2$ as we did at the beginning of this subsection; we recall that the integers $N_0$, $N_1$, and $N_2$ are respectively the number of points of $\SF{\XX_D}$ above which $\SF{\YY_D}$ is ramified and exhibits a smooth point, a singular point through which only one branch of $\SF{\YY_D}$ passes, and a singular point through which two branches of $\SF{\YY_D}$ pass.  We also recall from \Cref{rmk N0N1 even} that, in the $p\neq 2$ setting, the integer $N_0+N_1$ is necessarily even.
        
        Suppose that $\SF{\YY_D}$ is not irreducible.  As we have seen in \Cref{rmk R_0 R_1 R_2}(d), this is equivalent to the saying that $N_0=N_1=0$, and the curve $\SF{\YY_D}$ consists, in this case, of two lines $L_1$ and $L_2$ meeting each other above the $N=N_2$ points of $\SF{\XX_D}$; the number of singular points of $\SF{\YY_D}$ is $N$, and through each singular point one branch of $L_1$ and one branch of $L_2$ pass, flipped by the hyperlliptic involution. We have $m(L_i)=1$, $a(L_i)=0$, $w(L_i)=N$, and $\underline{w}(L_i)=(1, \ldots, 1)$ for $i = 1, 2$; hence, \Cref{prop part of rst} ensures that $\XX_D\le \Xrst$ if and only if $N\ge 3$.
        
        Suppose now that $\SF{\YY_D}$ is irreducible, which is to say that $N_0+N_1 \geq 1$, and let $V=\SF{\YY_D}$ denote the unique irreducible component of $\SF{\YY_D}$. We have the following: 
        \begin{itemize}
            \item $w(V)=N_1+2N_2$;
            \item $\underline{w}(V)=(1, \ldots, 1, 2, \ldots, 2)$ with $1$ appearing $N_1$ times and $2$ appearing $N_2$ times; and 
            \item $a(V)=-1+(N_0+N_1)/2$ in the $p\neq 2$ case, and $a(V)\ge -1+N_0+N_1$ in the $p=2$ setting, by \Cref{prop riemann hurwitz}.
        \end{itemize}
        Suppose that $N = 1$.  Then we have $N_0 + N_1 = 1$, which is impossible if $p \neq 2$ (as it contradicts \Cref{rmk N0N1 even}) and so we must have $p = 2$.  If $a(V) \geq 1$ then by Proposition \ref{prop part of rst} we have $\XX_D \leq \Xrst$, while if $a(V) = 0$, then we have $w(V) \leq 1$ and so Proposition \ref{prop part of rst} says that $\XX_D \not\leq \Xrst$.
        
        Suppose now that $N = 2$ and $p \neq 2$.  This forces $N_0 + N_1 = 2$ by \Cref{rmk N0N1 even}, from which it follows that $a(V) = 0$; meanwhile, we have $w(V) = N_1 + 2N_2 = N_1 \leq 2$ and $\underline{w}(V)$ consists only of $1$'s, and so by Proposition \ref{prop part of rst} we have $\XX_D \not\leq \Xrst$.  
        
        Finally, suppose that $N \geq 3$ or that $N = 2$ and $p = 2$.  If $N_2 \geq 1$, then we have $w(V) \geq 2$ and that a $2$ appears in $\underline{w}(V)$, and so $\XX_D \leq \Xrst$ by Proposition \ref{prop part of rst}.  If $N_2 = 0$, then we have $N_0 + N_1 \geq 4$ if $p \neq 2$ by \Cref{rmk N0N1 even} and $N_0 + N_1 \leq 2$ if $p = 2$; either way, we get $a(V) \geq 1$, and so again $\XX_D \leq \Xrst$ by \Cref{prop part of rst}.
        
        The statement about the strict transform of $\SF{\YY_D}$ in $\SF{\Yrst}$ is an immediate consequence of \Cref{prop smoothness components rst}, taking into account that the irreducible components of the special fiber of a semistable model of the line are always lines, and hence, in particular, smooth $k$-curves.
    \end{proof}
\end{thm}

\subsection{The special fiber \texorpdfstring{$\SF{\YY_D}$}{(YD)s} in the inseparable case} \label{sec models hyperelliptic inseparable}
We again let $D:=D_{\alpha,v(\beta)}$ be a disc with $a\in \bar{K}$ and $\beta \in \bar{K}^\times$, and let $\YY_D$ be the corresponding model of $Y$ constructed in \S\ref{sec models hyperelliptic forming}; this subsection will analyze the case in which $\SF{\YY_D}\to \SF{\XX_D}$ is inseparable (we are thus in the $p = 2$ setting). In this case, given a good part-square decomposition $f_{\alpha,\beta}=q^2_{\alpha,\beta}+\rho_{\alpha,\beta}$, we have $0 \leq t:=t_{q_{\alpha,\beta},\rho_{\alpha,\beta}} < 2v(2)$, and the special fiber $\SF{\YY_D}$ is described by an equation of the form $y^2=\overline{\rho_0}(x_{\alpha,\beta})$ over the $k$-line $\SF{\XX_D}$, where $\overline{\rho_0}$ is, in this case, a normalized reduction of $\rho_{\alpha,\beta}$, and it is not a square. 

We introduce the following notation.
\begin{dfn}
    \label{dfn mu}
    Given a point $P\in \SF{\XX_D}$, we define $\mu(\XX_D,P)$ to be the order of vanishing of the derivative $\overline{\rho_0}'$ of $\overline{\rho_0}$ at $P$; when $P=\infty$, we set $\mu(\XX_D,P) = 2g-\deg(\overline{\rho_0}')$.
\end{dfn}

\begin{rmk}
    \label{rmk mu}
    We make note of the following.
    \begin{enumerate}[(a)]
        \item The integer  $\mu(\XX_D,P)$ is independent of the chosen good part-square decomposition for $f_{\alpha,\beta}$, thanks to \Cref{prop good decomposition}(b).
        \item Since $p=2$, the derivative $\overline{\rho_0}'$ is a square, and so the integer $\mu(\XX_D,P)$ is even and non-negative for all $P\in \SF{\XX_D}$. 
        \item Since the degree of $\rho$ is $2g + 1$, we have $\sum_{P\in \SF{\XX_D}}\mu(\XX_D,P)=2g$.
    \end{enumerate}
\end{rmk}

It is immediate to verify that the singularities of $\SF{\YY_D}$ lie exactly over the finite set of points $R_{\mathrm{sing}}\subseteq \SF{\XX_D}$ at which $\overline{\rho_0}'$ vanishes, i.e.\ the points at which $\mu(\XX_D,P)>0$. Since we have $\sum_{P\in \SF{\XX_D}}\mu(\XX_D,P)\linebreak[0]=2g$, and since the integer $\mu(\XX_D,P)$ is always even, we have that $R_{\mathrm{sing}}$ has cardinality $\leq g$.

We remark that, if $t=0$, then the points of $R_\mathrm{sing}$ are just the roots of some (any) normalized reduction of $(f_{\alpha,\beta})'$, because, in this case, the trivial part-square decomposition $f_{\alpha,\beta}=0^2+f_{\alpha,\beta}$ is good; in particular, when $t=0$ we have $R_\mathrm{mult}\subseteq R_\mathrm{sing}$, where $R_{\mathrm{mult}}$ is the set of points of $\SF{\XX_D}$ to which two or more of the roots $\Rinfty$ reduce.

The normalization $\NSF{\YY_D}$ of the special fiber $\SF{\YY_D}$ is simply a projective line, and $\NSF{\YY_D}\to \SF{\XX_D}$ is the Frobenius cover of the projective line $\SF{\XX_D}$, which can be described by an equation of the form $y^2=x_{\alpha,\beta}$.
\begin{prop}
    \label{prop part of rst inseparable}
    If $D$ is a disc such that $\SF{\YY_D}\to \SF{\XX_D}$ is inseparable, then we have $\XX_D\le \Xrst$ if and only if $|R_\mathrm{sing}|\ge 3$. Moreover, whenever $\XX\le \Xrst$, the strict transform of $\SF{\YY_D}$ in $\SF{\Yrst}$ is a projective line.
    \begin{proof}
        The special fiber $\SF{\YY_D}$ is reduced, and it consists of a unique component $V$, which has $|R_\mathrm{sing}|$ unibranch singularities lying over $|R_\mathrm{sing}|$ distinct points of $\SF{\XX_D}$; moreover, the normalization $\widetilde{V}$ is a line.  We thus have $m(V)=1$, $a(V)=0$, $w(V)=|R_\mathrm{sing}|$, and $\underline{w}(V)=(1, \ldots, 1)$ and consequently deduce, via the criterion expressed in \Cref{prop part of rst}, that $\XX_D\le \Xrst$ if and only if $N\ge 3$.
        
        Moreover we have that, when $\XX_D\le \Xrst$, the strict transform of $\SF{\YY_D}$ in $\SF{\Yrst}$ coincides with the normalization $\NSF{\YY_D}$: the proof is identical to the one given in \Cref{thm part of rst separable}, and in our specific case, $\NSF{\YY_D}=\widetilde{V}$ is just a projective line.
    \end{proof}
\end{prop}

\begin{rmk}
    \label{rmk no inseparable genus 1,2}
    Since we always have $|R_{\mathrm{sing}}|\le g$ as shown in the above discussion, when  $g=1$ or $g=2$, the hypothesis of \Cref{prop part of rst inseparable} is never satisfied, hence we never have $\XX_D\le \Xrst$ if $\SF{\YY_D}\to \SF{\XX_D}$ is inseparable.
\end{rmk}

Regarding the contribution of $\XX_D$ to $\Xrst$ when $\SF{\YY_D}\to \SF{\XX_D}$ is inseparable, we have the following result.
\begin{prop}
    \label{prop evanescence inseparable components}
    Letting $\mathfrak{D}$ be the collection of discs corresponding to the model $\Xrst$ (see \S\ref{sec models hyperelliptic line}), let us write $\mathfrak{D}=\mathfrak{D}_{\mathrm{sep}}\sqcup \mathfrak{D}_{\mathrm{insep}}$, where $D\in \mathfrak{D}$ belongs to $\mathfrak{D}_{\mathrm{sep}}$ (resp.\ $\mathfrak{D}_{\mathrm{insep}}$) if the covering map $\SF{\YY_D}\to \SF{\XX_D}$ is (resp.\ is not) separable. Then the set $\mathfrak{D}$ can be reconstructed as $(\mathfrak{D}_{\mathrm{sep}})^{\sst}$ following the algorithm presented in \S\ref{sec models hyperelliptic line}.
    \begin{proof}
        For a disc $D\in \mathfrak{D}_\mathrm{insep}$, \Cref{prop part of rst inseparable} ensures that $|R_{\mathrm{sing}}|\ge 3$.  since \Cref{cor rst isomorphism failure} says that $R_{\mathrm{sing}}=\Ctr(\XX_D,\Xrst)$, we deduce that $|\Ctr(\XX_D,\linebreak[0]\Xrst)|\ge 3$.  Now the proposition follows from \Cref{rmk semistabilization reversed}.
    \end{proof}
\end{prop}

Roughly speaking, we can conclude that the role of the inseparable components in $\SF{\Yrst}$ is inessential: they are just lines that, in light of the proposition above, only get added whenever it is necessary to create room between three or more separable components that would otherwise intersect at the same point and violate semistability. We can consequently focus our attention on the \emph{separable} components of $\SF{\Yrst}$, which is to say on the discs $D$ such that $\XX_D\le \Xrst$ for which $\SF{\YY_D}\to \SF{\XX_D}$ in a separable cover.  Starting in the next section, we will refer to them by the term \emph{valid discs}, and their determination, by \Cref{prop evanescence inseparable components}, suffices to compute the whole $\Xrst$. 

\section{Clusters and valid discs} \label{sec clusters}

We begin this section by defining, in \S\ref{sec cluster cluster definition}, \emph{clusters (of roots)}, \emph{depths}, and \emph{relative depths} of clusters, and the \emph{cluster picture} associated to the odd-degree polynomial $f(x)$ defining the hyperelliptic curve $Y: y^2=f(x)$.  This notion of ``cluster" is equivalent to the one found in \cite[Definition 1.1]{dokchitser2022arithmetic}.  It is known (see for instance \cite[Theorem 1.10]{dokchitser2022arithmetic}) that the cluster picture completely determines the structure of the special fiber $\SF{\Ymini}$ of the minimal regular model as long as we are in the $p \neq 2$ setting.  Similarly, when $p \neq 2$, a minor variant of this result says that the structure of $\SF{\Yrst}$ is also determined entirely by the cluster picture associated to $f$, in such a way that each component of $\SF{\Xrst}$ corresponds to a non-singleton cluster (see \Cref{thm cluster p odd}). In the $p = 2$ setting, however, it is no longer the case that the cluster picture associated to a polynomial $f$ determines the structure of $\SF{\Ymini}$ or $\SF{\Yrst}$.  In light of this, in \S\ref{sec valid discs definition} we set up the notion of \emph{valid discs} associated to $f$, so that each one corresponds to a component of $\SF{\Xrst}$, and we explore the relationship between these valid discs and clusters associated to $f$ (\Cref{thm cluster p=2}).

\subsection{Clusters}
\label{sec cluster cluster definition}

We want to define the \emph{cluster picture} associated to the set $\RR\subseteq \bar{K}$ consisting of the $2g+1$ roots of the polynomial $f(x)$ defining the hyperelliptic curve $Y$. First, let us introduce a number of invariants attached to a subset $\mathfrak{s}\subseteq \RR$.
\begin{dfn} \label{dfn depth}
    \label{dfn interval I}
    Given a subset $\mathfrak{s}\subseteq \RR$, we set
    \begin{equation*}
    \begin{gathered}
        d_+(\mathfrak{s}) = \min_{\zeta, \zeta' \in \mathfrak{s}} v(\zeta - \zeta') \in \qqinfty; \qquad 
        d_-(\mathfrak{s}) = \max_{\zeta \in \mathfrak{s}, \ \zeta' \in \RR\setminus\mathfrak{s}} v(\zeta - \zeta') \in \qqminusinfty,
   \end{gathered}
   \end{equation*}
   where we follow the convention that $\min \varnothing = +\infty$ and $\max \varnothing = -\infty$.  The number $d_+(\mathfrak{s})$ is named the \emph{(absolute) depth} of $\mathfrak{s}$, while $\delta(\mathfrak{s}):=d_+(\mathfrak{s})-d_-(\mathfrak{s})\in \qqinfty$ will be also referred to as the \emph{relative depth} of $\mathfrak{s}$.
   We will use the notation $I(\mathfrak{s})$ to mean the closed interval $[d_-(\mathfrak{s}),d_+(\mathfrak{s})]$, with the convention that $I(\mathfrak{s})=\varnothing$ whenever $d_+(\mathfrak{s})<d_-(\mathfrak{s})$.
\end{dfn}

We are now ready to define the notion of a cluster.
\begin{dfn}
    \label{dfn cluster}
    Given a non-empty subset $\mathfrak{s}\subseteq \RR$, we say that $\mathfrak{s}$ is a \emph{cluster} (of $\mathcal{R}$) if we have $\delta(\mathfrak{s})>0$. The set of pairs $(\mathfrak{s},d_+(\mathfrak{s}))$, where $\mathfrak{s}$ varies among all clusters of $\RR$, is called the \emph{cluster picture} of $\RR$.
\end{dfn}

\begin{rmk} \label{rmk cluster}
    We note the following.
    \begin{enumerate}[(a)]
        \item It is elementary to verify that, given a non-empty subset $\mathfrak{s}\subseteq \RR$ is a cluster if and only if there exists a disc $D\subseteq \bar{K}$ such that $D\cap \RR=\mathfrak{s}$.
        \item We have that $\RR$ itself is always a cluster, with $d_-(\RR)=-\infty$, $d_+(\RR)$ finite, and $\delta(\RR)=+\infty$.
        \item For every $a \in \mathfrak{s}$, the singleton $\{a\}$ is always a cluster, with $d_-(\RR)$ finite, $\delta_+(\{s\})=+\infty$, and $\delta(\{s\})=+\infty$.
    \end{enumerate}
\end{rmk}

\begin{dfn} \label{dfn parent}
    For every cluster $\mathfrak{s}\subsetneq \RR$, the \emph{parent cluster} of $\mathfrak{s}$ is the smallest cluster $\mathfrak{s}'$ properly containing it; in this situation, we say that $\mathfrak{s}$ is a \emph{child cluster} of $\mathfrak{s}'$. Two distinct clusters having the same parent are said to be \emph{sibling clusters}.
\end{dfn}

With the notation of the above definition, it is immediate to verify that $d_+(\mathfrak{s}')=d_-(\mathfrak{s})$.

\begin{prop}
    \label{prop cluster delta characterization}
    Given a subset $\mathfrak{s}\subseteq \RR$, we have the following:
    \begin{enumerate}[(a)]
        \item we have $\delta(\mathfrak{s})> 0$ if and only if $\mathfrak{s}$ is either the empty set or a cluster; and
        \item we have $\delta(\mathfrak{s})\ge 0$ if and only if $\mathfrak{s}$ is a (possibly empty) union of sibling clusters.
    \end{enumerate}
    \begin{proof}
        This follows immediately from definitions.
    \end{proof}
\end{prop}

The term \emph{cluster picture} is inspired by the fact that the data of a cluster picture can easily be expressed visually.  To do so, we represent elements of $\mathcal{R}$ as points and represent proper clusters as loops surrounding the corresponding subsets of points with numbers next to the loops indicating the corresponding depths.  

\begin{ex}
\label{ex cluster picture}
The cluster picture associated to $\mathcal{R} := \{0, \pi^4, \pi^3, \pi, \pi(1-\pi^4)\}$, where $\pi\in K$ is an element such that $v(\pi)=1$, may be visualized using the below diagram, in which the all clusters (except the singleton ones) are displayed together with their relative depths (for the cluster $\RR$, the label indicates the absolute depth).
\begin{equation*}
	\mathrm{cluster} \ \mathrm{picture} \ \mathrm{of} \ \mathcal{R}:\
	\begin{clusterpicture}
		\comincia;
		\punto{1}{0}{$0$};
		\punto{2}{1}{$\pi^4$};
		\punto{3}{2}{$\pi^3$};
		\punto{4}{3}{$\pi$};
		\punto{5}{4}{$\pi(1-\pi^4)$};
		\cluster{12}{\pto{1}\pto{2}}{1};
		\cluster{123}{\clu{12}\pto{3}}{2};
		\cluster{45}{\pto{4}\pto{5}}{4};
		\cluster{tot}{\clu{123}\clu{45}}{1};
	\end{clusterpicture}
\end{equation*}
\end{ex}

\begin{rmk} \label{rmk reciprocal}
    Translating a subset $\mathcal{R} \subset \bar{K}$ by an element $\alpha \in \bar{K}$ clearly does not affect the cluster picture.  An important automorphism of the projective line is the reciprocal map which takes a finite point $z \neq 0$ to $z^{-1}$ and exchanges $0$ and $\infty$; composed with the translation-by-$\alpha$ map $z \mapsto z - \alpha$, we get an automorphism of the projective line given by $i_\alpha : z \mapsto (z - \alpha)^{-1}$.
    
    One can check readily (see also the proof of \cite[Proposition 14.6]{dokchitser2022arithmetic}) that the cluster picture of a set $\RR\subseteq \bar{K}$ transforms in a predictable and easily describable way under such a map $i_\alpha$. In fact, assume that $\alpha\in \RR$, and, for any subset $\mathfrak{s}\subseteq \RR$, let $\mathfrak{s}^{\vee,\alpha}$ be defined as:
    \begin{equation*}
        \mathfrak{s}^{\vee,\alpha}=\begin{cases}
            i_\alpha(\RR\setminus\mathfrak{s})\cup \lbrace 0 \rbrace, &\text{if }\alpha\in \mathfrak{s};\\
            i_\alpha(\mathfrak{s}), &\text{if }\alpha\not\in \mathfrak{s}.
        \end{cases}
    \end{equation*}
    This defines a bijection between the subsets of the set of roots of $f$ and the subsets of the set of roots of $f^{\vee,\alpha}$, where $f^{\vee,\alpha}$ is the degree-$(2g+1)$ polynomial defined as $f^{\vee,\alpha} := (z - \alpha)^{2g+2} f((z - \alpha)^{-1})$.
    Moreover, one readily computes that, when $\alpha\in \mathfrak{s}$, we have $d_\pm(\mathfrak{s}^{\vee,\alpha}) = \mp d_\pm(\mathfrak{s})$, whereas when $\alpha\not\in\mathfrak{s}$, we have $d_\pm(\mathfrak{s}^{\vee,\alpha}) = d_\pm(\mathfrak{s})-2d_+(\mathfrak{s}\cup\{ \alpha \})$. In both cases, we get $\delta(\mathfrak{s})=\delta(\mathfrak{s}^{\vee,\alpha})$; in particular, we have that $\mathfrak{s}$ is a cluster for $f$ if and only if $\mathfrak{s}^{\vee,\alpha}$ is a cluster for $f^{\vee,\alpha}$.
    
\end{rmk}

We want now to introduce some further definitions for later use that relate the clusters to the discs that cut them out of $\mathcal{R}$. We begin with the following remark.

\begin{dfn}
    \label{dfn linked}
    Given a cluster $\mathfrak{s}\subseteq \RR$, we say that a disc $D\subseteq \bar{K}$ is \emph{linked to} $\mathfrak{s}\subseteq \mathcal{R}$ if we have $D = D_{\mathfrak{s},b}$ for some $b\in I(\mathfrak{s})$ (where $D_{\mathfrak{s},b}$ denotes the disc of depth $b$ centered at any point of $\mathfrak{s}$).
\end{dfn}

To clarify this notion, \Cref{fig cluster graph} illustrates all discs linked to each cluster when $\RR$ is the cardinality-$5$ set of roots described in \Cref{ex cluster picture}.

\begin{rmk}
    A disc $D$ is linked to a cluster $\mathfrak{s}$ if and only if we have either $D\cap \RR=\mathfrak{s}$ or that $D$ is the minimal disc such that $D\cap \RR\supsetneq \mathfrak{s}$.  More precisely, if $D=D_{\mathfrak{s},b}$ for some $b\in I(\mathfrak{s})$, we have that $D\cap \RR=\mathfrak{s}$ whenever $b\in (d_-(\mathfrak{s}),d_+(\mathfrak{s})]$, whereas $D$ is the smallest disc such that $D\cap\RR\supsetneq \mathfrak{s}$ when $b=d_-(\mathfrak{s})$; moreover, in this case the subset $D\cap\RR \subseteq \RR$ is the parent cluster of $\mathfrak{s}$.
\end{rmk}

\begin{figure}
    \centering
    \tikzset{every picture/.style={line width=0.75pt}} 

\begin{tikzpicture}[x=0.75pt,y=0.75pt,yscale=-1,xscale=1]

\draw  [dash pattern={on 0.84pt off 2.51pt}]  (50,30) -- (50,290) ;
\draw    (10,30) -- (488,30) ;
\draw [shift={(490,30)}, rotate = 180] [color={rgb, 255:red, 0; green, 0; blue, 0 }  ][line width=0.75]    (10.93,-3.29) .. controls (6.95,-1.4) and (3.31,-0.3) .. (0,0) .. controls (3.31,0.3) and (6.95,1.4) .. (10.93,3.29)   ;
\draw    (50,30) -- (50,40.67) ;
\draw    (100,30) -- (100,40.67) ;
\draw    (150,30) -- (150,40.67) ;
\draw    (200,30) -- (200,40.67) ;
\draw    (20,180) -- (100,180) ;
\draw    (150,230) -- (300,230) ;
\draw    (350,200) -- (450,200) ;
\draw    (100,180) .. controls (114,203.67) and (122,228.67) .. (150,230) ;
\draw    (300,230) .. controls (310,251.67) and (303,258.67) .. (350,260) ;
\draw    (300,230) .. controls (306,198.67) and (302,201.67) .. (350,200) ;
\draw    (350,260) -- (450,260) ;
\draw    (100,180) .. controls (114,141.67) and (125,128.67) .. (150,130) ;
\draw    (150,130) -- (200,130) ;
\draw    (200,130) .. controls (206,98.67) and (202,101.67) .. (250,100) ;
\draw    (200,130) .. controls (210,151.67) and (203,158.67) .. (250,160) ;
\draw    (300,70) -- (450,70) ;
\draw    (250,100) .. controls (260,121.67) and (253,128.67) .. (300,130) ;
\draw    (250,100) .. controls (256,68.67) and (252,71.67) .. (300,70) ;
\draw    (300,130) -- (450,130) ;
\draw    (250,160) -- (450,160) ;
\draw    (250,30) -- (250,40.67) ;
\draw    (300,30) -- (300,40.67) ;
\draw  [fill={rgb, 255:red, 0; green, 0; blue, 0 }  ,fill opacity=1 ] (95,180) .. controls (95,177.24) and (97.24,175) .. (100,175) .. controls (102.76,175) and (105,177.24) .. (105,180) .. controls (105,182.76) and (102.76,185) .. (100,185) .. controls (97.24,185) and (95,182.76) .. (95,180) -- cycle ;
\draw  [fill={rgb, 255:red, 0; green, 0; blue, 0 }  ,fill opacity=1 ] (195,130) .. controls (195,127.24) and (197.24,125) .. (200,125) .. controls (202.76,125) and (205,127.24) .. (205,130) .. controls (205,132.76) and (202.76,135) .. (200,135) .. controls (197.24,135) and (195,132.76) .. (195,130) -- cycle ;
\draw  [fill={rgb, 255:red, 0; green, 0; blue, 0 }  ,fill opacity=1 ] (245,100) .. controls (245,97.24) and (247.24,95) .. (250,95) .. controls (252.76,95) and (255,97.24) .. (255,100) .. controls (255,102.76) and (252.76,105) .. (250,105) .. controls (247.24,105) and (245,102.76) .. (245,100) -- cycle ;
\draw  [fill={rgb, 255:red, 0; green, 0; blue, 0 }  ,fill opacity=1 ] (295,230) .. controls (295,227.24) and (297.24,225) .. (300,225) .. controls (302.76,225) and (305,227.24) .. (305,230) .. controls (305,232.76) and (302.76,235) .. (300,235) .. controls (297.24,235) and (295,232.76) .. (295,230) -- cycle ;
\draw  [dash pattern={on 0.84pt off 2.51pt}]  (100,30) -- (100,290) ;
\draw  [dash pattern={on 0.84pt off 2.51pt}]  (150,30) -- (150,290) ;
\draw  [dash pattern={on 0.84pt off 2.51pt}]  (200,30) -- (200,290) ;
\draw  [dash pattern={on 0.84pt off 2.51pt}]  (250,30) -- (250,290) ;
\draw  [dash pattern={on 0.84pt off 2.51pt}]  (300,30) -- (300,290) ;

\draw (45,12.4) node [anchor=north west][inner sep=0.75pt]    {$0$};
\draw (95,12.4) node [anchor=north west][inner sep=0.75pt]    {$1$};
\draw (2,169.4) node [anchor=north west][inner sep=0.75pt]    {$\infty $};
\draw (461,62.4) node [anchor=north west][inner sep=0.75pt]    {$0$};
\draw (461,121.4) node [anchor=north west][inner sep=0.75pt]    {$\pi ^{4}$};
\draw (459,151.4) node [anchor=north west][inner sep=0.75pt]    {$\pi ^{3}$};
\draw (459,191.4) node [anchor=north west][inner sep=0.75pt]    {$\pi $};
\draw (461,242.4) node [anchor=north west][inner sep=0.75pt]    {$\pi \left( 1-\pi ^{4}\right)$};
\draw (145,12.4) node [anchor=north west][inner sep=0.75pt]    {$2$};
\draw (196,12.4) node [anchor=north west][inner sep=0.75pt]    {$3$};
\draw (245,12.4) node [anchor=north west][inner sep=0.75pt]    {$4$};
\draw (295,12.4) node [anchor=north west][inner sep=0.75pt]    {$5$};
\draw (351,12) node [anchor=north west][inner sep=0.75pt]   [align=left] {Depth};
\draw (41,162.4) node [anchor=north west][inner sep=0.75pt]    {$1-5$};
\draw (379,242.4) node [anchor=north west][inner sep=0.75pt]    {$5-1$};
\draw (379,182.4) node [anchor=north west][inner sep=0.75pt]    {$5-1$};
\draw (379,142.4) node [anchor=north west][inner sep=0.75pt]    {$5-1$};
\draw (379,112.4) node [anchor=north west][inner sep=0.75pt]    {$5-1$};
\draw (379,52.4) node [anchor=north west][inner sep=0.75pt]    {$5-1$};
\draw (191,82.4) node [anchor=north west][inner sep=0.75pt]    {$4-2$};
\draw (121,112.4) node [anchor=north west][inner sep=0.75pt]    {$3-3$};
\draw (161,212.4) node [anchor=north west][inner sep=0.75pt]    {$4-2$};

\end{tikzpicture}
    \caption{This tree describes, when $\RR$ is the set of roots of \Cref{ex cluster picture}, all discs linked to at least one cluster of $\RR$, i.e.\ all discs $D$ such that $D\cap \RR\neq \varnothing$. Each edge corresponds to a cluster $\mathfrak{s}\subseteq \RR$ (the labels denote the cardinalities $2g+2-|\mathfrak{s}|$ and $|\mathfrak{s}|$); the initial and final depth of the edge correspond to $d_-(\mathfrak{s})$ and $d_+(\mathfrak{s})$ respectively; the points in the edge correspond to all discs linked to $\mathfrak{s}$. The 4 vertices correspond to those discs that are linked to more than one cluster of $\RR$.}
    \label{fig cluster graph}
\end{figure}

We observe that, given a disc $D$, the points of $\Rinfty$ reduce to $N$ distinct points $P_1, \ldots,\linebreak[0] P_{N-1},\linebreak[0] \infty \in \SF{\XX_D}(k)$; we can accordingly write $\RR=\mathfrak{s}_1\sqcup \ldots\sqcup \mathfrak{s}_{N-1}\sqcup \mathfrak{s}_\infty$, where $\mathfrak{s}_i$ consists of the roots of $\RR$ reducing to $P_i$, and $\mathfrak{s}_\infty$ (which is possibly empty) consists of the roots of $\RR$ reducing to $\infty$. We clearly have $D\cap \RR=\mathfrak{s}_1\sqcup \ldots \sqcup \mathfrak{s}_{N-1}=\RR\setminus \mathfrak{s}_\infty$. With \Cref{fig cluster graph} in mind, it is easy to verify that the following holds.

\begin{lemma}
    \label{lemma discs linked to clusters}
    With notation as above, we have the following.
    \begin{enumerate}[(a)]
        \item If $N=1$, which is equivalent to $D\cap \RR=\varnothing$, the disc $D$ is not linked to any cluster.
        \item If $N=2$, the disc $D$ is linked to exactly one cluster, namely $\mathfrak{s}_1=\mathcal{R}\setminus \mathfrak{s}_\infty=D\cap \mathcal{R}$.
        \item[(c)] If $N \ge 3$, the disc $D$ is linked to exactly $N$ clusters, namely $\mathfrak{s}_1, \ldots, \mathfrak{s}_{N-1}$, and $\mathfrak{s}_N:=D\cap \RR=\RR\setminus \mathfrak{s}_\infty = \mathfrak{s}_1\sqcup \ldots \sqcup \mathfrak{s}_{N-1}$.
    \end{enumerate}

    Moreover, case (c) occurs if and only if we have $D=D_{\mathfrak{s},d_+(\mathfrak{s})}$ for some non-singleton cluster $\mathfrak{s}\subseteq \mathcal{R}$, in which case we have $\mathfrak{s}_N=\mathfrak{s}$ and that $\mathfrak{s}_1, \ldots, \mathfrak{s}_{N-1}$ are precisely the children clusters of $\mathfrak{s}$, and we have $D=D_{\mathfrak{s}_i,d_-(\mathfrak{s}_i)}$ for $i=1,\ldots, N-1$.
\end{lemma}

\subsection{Valid discs} \label{sec valid discs definition}

We now define a term which we will use throughout the rest of the paper in order to refer to components of the relatively stable model of a hyperelliptic curve.

\begin{dfn} \label{dfn valid disc}
     A disc $D \subseteq \bar{K}$ is a \emph{valid disc} if it satisfies $\XX_D \leq \Xrst$ and if the quadratic cover $\SF{\YY_D}\to \SF{\XX_D}$ is separable.
\end{dfn}

We note that our notion of \emph{valid disc} differs from the one in \cite{dokchitser2022arithmetic}, although in both cases valid discs are used to build a particular semistable model of $Y$ with desired properties.

The cluster picture allows us to completely identify the valid discs in the tame case.
\begin{thm} \label{thm cluster p odd}
    In the $p\neq 2$ setting, there is a one-to-one correspondence between non-singleton clusters and valid discs, which is given by $\mathfrak{s}\mapsto D_{\mathfrak{s},d_+(\mathfrak{s})}$. In other words, the valid discs are precisely those discs that minimally cut out the clusters in $\mathcal{R}$.
    \begin{proof}
        We have already observed (\Cref{lemma discs linked to clusters}) that a disc $D$ is of the form $D_{\mathfrak{s},d_+(\mathfrak{s})}$ for some non-singleton cluster $\mathfrak{s}$ if and only if the number $N$ of points to which $\Rinfty\subseteq X(K)$ reduces in $\SF{\XX_D}$ is $\ge 3$. Hence, the theorem immediately follows from \Cref{thm part of rst separable}.
    \end{proof}
\end{thm}

We want some analog of Theorem \ref{thm cluster p odd} for working over residue characteristic $2$; however, in the $p = 2$ setting, valid discs do not correspond in this way in a one-on-one manner with clusters, as is shown by the following theorem.
\begin{thm} \label{thm cluster p=2}
In the $p = 2$ setting, we have the following.
\begin{enumerate}[(a)]
    \item Given an odd-cardinality cluster $\mathfrak{s} \subseteq \RR$, there is no valid disc $D$ linked to $\mathfrak{s}$.
    \item Given an even-cardinality cluster $\mathfrak{s} \subset \mathcal{R}$, there may be 0, 1, or 2 distinct valid discs $D$ linked to $\mathfrak{s}$.
\end{enumerate}
\end{thm}

\begin{rmk}\label{rmk crushed discs}
    In the $p=2$ setting, we will see that it is possible that a valid discs $D$ is linked to no cluster, which is to say $D\cap \RR=\varnothing$ (see \Cref{lemma discs linked to clusters}), whereas \Cref{thm cluster p odd} ensures that this never happens when $p\neq 2$.
\end{rmk}

\begin{rmk} \label{rmk cluster p=2}
    It follows directly from \Cref{thm cluster p odd} and \Cref{lemma discs linked to clusters} that the statement in part (b) of the above theorem holds in the $p \neq 2$ setting as well.  In fact, when $p \neq 2$, that statement holds even after removing the hypothesis that $\mathfrak{s}$ has even cardinality, and it can be made more precise by saying that there there are \emph{exactly} $2$ distinct valid discs $D$ linked to $\mathfrak{s}$, namely $D_{\mathfrak{s},d_+(\mathfrak{s})}$ and $D_{\mathfrak{s},d_-(\mathfrak{s})}$, except in the case $\mathfrak{s}=\mathcal{R}$, when there is exactly $1$, namely $D_{\mathfrak{s},d_+(\mathfrak{s})}$.
\end{rmk}

The proof of \Cref{thm cluster p=2} is deferred to the following section, in which we set up a framework for considering the models $\YY_{D}$ corresponding to families of discs $D:=D_{\alpha,b}$ which share a common center $\alpha$ and which all contain the same subset $\mathfrak{s}\subseteq \RR$ of roots.

\section{Finding valid discs with a given center} \label{sec depths}

In this section, we will fix a center $\alpha\in \bar{K}$, and we will investigate for which depths $b\in I\subseteq \qq$, where $I=[d_-,d_+]$ is some closed interval, the disc $D_{\alpha,b}$ is valid for the hyperelliptic curve $Y: y^2 = f(x)$. The interval $I$ will be chosen so that, when $b$ ranges in the internal of $I$, the intersection $\mathfrak{s}:=D_{\alpha,b}\cap \RR$ is constant. More precisely, we are interested in the following two scenarios.
\begin{enumerate}[(a)]
    \item We choose $\mathfrak{s}$ to be a cluster of $\RR$, we fix $\alpha\in \bar{K}$ to be any point of the disc $D_{\mathfrak{s},d_+(\mathfrak{s})}$, and we set $I = I(\mathfrak{s}) = [d_-(\mathfrak{s}), d_+(\mathfrak{s})]$. In this case, as we vary $b\in I$, we have that $D_{\alpha,b}$ ranges among all discs linked to $\mathfrak{s}$.
    \item We choose $\mathfrak{s}=\varnothing$, we fix $\alpha\in \bar{K}\setminus \RR$, and we let $I = [d_-,+\infty)$, where $d_-$ is the maximum depth such that $D_{\alpha,d_-}\cap \RR\neq \varnothing$. This means that, as $b$ ranges in the interior of $I$, we have that $D_{\alpha,b}$ ranges among all discs centered at $\alpha$ that are linked to no cluster of $\RR$.
\end{enumerate}
It is important to consider case (b) as well as case (a), because, in the $p=2$ setting, there may exist valid discs that are linked to no cluster in $\RR$.

The section is organized as follows. In \S\ref{sec depths piecewise-linear} we introduce the language of translated and scaled part-square decompositions, which will be useful for dealing with the problem. In \S\ref{sec depths construction valid discs} we identify the depths $b\in I$ for which $D_{\alpha,b}$ is a valid disc as the endpoints $b_{\pm}$ of a sub-interval $J\subseteq I$ (see \Cref{thm summary depths valid discs}). When $p\neq 2$, we always have $J=I$; however, when $p=2$, we have that $J$ may be strictly smaller than $I$; for example, we will see that $J = \varnothing$ whenever $|\mathfrak{s}|$ is odd. Subsections \S\ref{sec depths separating roots} and \S\ref{sec depths thresholds} develop our strategy for determining $J$ when $\mathfrak{s}$ has even cardinality, provided that, for each of the two factors $f^{\mathfrak{s}}$ and $f^{\RR\setminus \mathfrak{s}}$ of $f$ corresponding to the roots lying in $\mathfrak{s}$ and $\RR\setminus\mathfrak{s}$ respectively, we know a part-square decomposition that is totally odd with respect to the center $\alpha$.  Finally, in \S\ref{sec depths sufficiently odd} we show that, when  $\mathfrak{s}$ has even cardinality, the necessary computations can also be performed by replacing totally odd part-square decompositions with \emph{sufficiently odd} ones, which are in general easier to find. In \S\ref{sec depths algorithm}, we present an algorithm to compute sufficiently odd decompositions, while in \S\ref{sec depths computations low degree} we present through elementary computations its application to low-degree polynomials.


\subsection{Translated and scaled part-square decompositions} 
\label{sec depths piecewise-linear}



Given any polynomial $h(z) \in \bar{K}[z]$ and any choice of elements $\alpha\in \bar{K}, \beta\in\bar{K}^\times$, we can compute the (Gauss) valuations of the translated and scaled polynomial $h_{\alpha,\beta}$ (as defined in  obtained from $h$ by translating by $\alpha$ and scaling by $\beta$ (as defined in \S\ref{sec models hyperelliptic forming}).  The following lemma will allow us to treat the Gauss valuation of a certain translation and scaling of $h$ as a function on discs.

\begin{lemma}
    \label{lemma vfun of disc}
    As we vary $\alpha \in \bar{K}$ and $\beta \in \bar{K}^\times$, the valuation $v(h_{\alpha,\beta})$ depends only on the disc $D:=D_{\alpha,v(\beta)}$.
    \begin{proof}
        Let $\alpha, \alpha'\in \bar{K}$ and $\beta, \beta'\in \bar{K}^\times$ be such that $D_{\alpha,v(\beta)}=D_{\alpha',v(\beta')}$; we must show that $v(h_{\alpha,\beta})=v(h_{\alpha',\beta'})$. It is clearly sufficient to prove the result when $\alpha=0$ and $\beta=1$, in which case $h_{\alpha,\beta}=h$, the assumption $D_{\alpha,v(\beta)}=D_{\alpha',v(\beta')}$ means that $v(\beta')=0$ and $v(\alpha')\ge 0$, and the claim is verified straightforwardly.
    \end{proof}
\end{lemma}

Given a disc $D$, we will consequently denote $\vfun_h(D)$ the valuation of $h_{\alpha,\beta}$ for any $\alpha$ and $\beta$ such that $D=D_{\alpha,v(\beta)}$.
When a center $\alpha\in \bar{K}$ is fixed, we may consider the function $b \mapsto \vfun_h(D_{\alpha,b})\in\qqinfty$ defined for all $b \in \qq$.

\begin{lemma} \label{lemma mathfrakh}
    Suppose a center $\alpha\in \bar{K}$ is fixed. With the above set-up, we have the following.
    \begin{enumerate}[(a)]
        \item The function $b \mapsto \vfun_h(D_{\alpha,b})\in\qqinfty$ satisfies the property of being a continuous, non-decreasing piecewise linear function with integer slopes and whose slopes decrease as the input increases.
        \item For any $b \in \mathbb{Q}$ and $\beta \in \bar{K}^{\times}$ such that $v(\beta) = b$, the left (resp.\ right) derivative of the function $c \mapsto \vfun_h(D_{\alpha,c})\in\qqinfty$ at $c = b$ coincides with the highest (resp.\ lowest) degree of the variable $x_{\alpha,\beta}$ appearing in the normalized reduction of $h_{\alpha,\beta}$, i.e.\ with the number of roots $\zeta$ of $h$ in $\bar{K}$ (counted with multiplicity) such that $v(\zeta-\alpha)\ge b$ (resp.\ $v(\zeta-\alpha)>b$).
        
    \end{enumerate}
\end{lemma}

\begin{proof}

Write $H_i$ for the $z^i$-coefficient of $h_{\alpha,1}$, and note that $\beta^i H_i$ is the $z^i$-coefficient of $h_{\alpha,\beta}$ for any scalar $\beta$.  Now given any $b \in \qq$ and $\beta \in \bar{K}^{\times}$ with $v(\beta) = b$, by definition we have 
\begin{equation} \label{eq mathfrakv}
\vfun_h(D_{\alpha,b}) = \min_{0 \leq i \leq \deg(h)} \{v(\beta^i H_i)\} = \min_{0 \leq i \leq \deg(h)} \{v(H_i) + ib\}.
\end{equation}
All the properties of the function $b\mapsto \vfun_h(D_{\alpha,b})$ stated in the lemma immediately follow from the explicit expression given above.
\end{proof}

Given a part-square decomposition $h = q^2 + \rho$ of a nonzero polynomial $h$, by translating and scaling we can clearly form part-square decompositions $h_{\alpha,\beta}=q_{\alpha,\beta}^2+\rho_{\alpha,\beta}$ for all $\alpha\in \bar{K}, \beta\in \bar{K}^\times$.

\begin{lemma}
    \label{prop translated part-square decompositions}
    Let $h=q^2+\rho$ be a part-square decomposition.
    \begin{enumerate}[(a)]
        \item The property of the induced part-square decomposition $h_{\alpha,\beta} = q_{\alpha,\beta}^2 + \rho_{\alpha,\beta}$ being good or not only depends on the disc $D:=D_{\alpha,v(\beta)}$ and not on the particular choices of $\alpha$ and $\beta$.
        \item The property of the induced part-square decomposition $h_{\alpha,\beta} = q_{\alpha,\beta}^2 + \rho_{\alpha,\beta}$ being totally odd or not only depends on our choice of $\alpha$ and not on $\beta$.
    \end{enumerate}
    
    \begin{proof}
        Part (a)  is an immediate consequence of \Cref{lemma vfun of disc}, while part (b) is immediate. 
    \end{proof}
\end{lemma}

We can consequently make the following definitions, which are the variants of those given in \Cref{dfn good totally odd} relative to the choice of a disc.
\begin{dfn}
    \label{dfn translated scaled good totally odd}
    Let $h=q^2+\rho$ be a part-square decomposition. We make the following definitions:
    \begin{enumerate}[(a)]
        \item the decomposition is \emph{good} at a disc $D$ whenever $h_{\alpha,\beta} = q_{\alpha,\beta}^2 + \rho_{\alpha,\beta}$ is a good part-square decomposition for some (any) $\alpha\in \bar{K}$, $\beta\in \bar{K}^\times$ such that $D=D_{\alpha,v(\beta)}$; and 
        \item the decomposition is \emph{totally odd} with respect to a center $\alpha\in \bar{K}$ if $h_{\alpha,\beta}=q_{\alpha,\beta}^2+\rho_{\alpha,\beta}$ is a totally odd part-square decomposition for some (any) $\beta\in \bar{K}^\times$.
    \end{enumerate}
\end{dfn}

\begin{rmk}
    \label{rmk totally odd good}
    If $h=q^2+\rho$ is a totally odd part-square decomposition with respect to a center $\alpha$, then, by \Cref{cor totally odd is good}, it is good at the discs $D_{\alpha,b}$, for all $b\in \qq$. 
\end{rmk}

Recalling the number $t_{q,\rho} := v(\rho) - v(h)\in \qqinfty$ from \S\ref{sec models hyperelliptic part-square}, we define the related function
$$\tfun_{q, \rho} := \vfun_\rho - \vfun_f$$
so that $\tfun_{q,\rho}(D) = t_{q_{\alpha,\beta}, \rho_{\alpha,\beta}}$ for any $\alpha\in \bar{K}$, $\beta \in \bar{K}^\times$ such that $D=D_{\alpha,v(\beta)}$.  When a center $\alpha\in \bar{K}$ is fixed, we can study the function $b \mapsto \tfun_{q,\rho}(D_{\alpha,b}) : \qq \to \qq \cup \{+\infty\}$, which is the difference between two continuous piecewise-linear functions and so is itself a continuous piecewise-linear function. Taking into account \Cref{rmk same t for good}, we can give the following definition.
\begin{dfn}
    \label{dfn t fun}
    Given a (multi-)set of elements $\mathfrak{s}\subseteq \bar{K}$ and a disc $D$, we define $\tbest{\mathfrak{s}}{D}\in \zerotwo$ to be $\truncate{\tfun_{q,\rho}(D)}$ for any part-square decomposition $h=q^2+\rho$ which is good at the disc $D$, where $h(z)\in \bar{K}$ is any polynomial whose set of roots is $\mathfrak{s}$ (counted with multiplicity).
\end{dfn}
\begin{rmk}
    \label{rmk t fun computing}
    Fix a center $\alpha$. If $h\in \bar{K}[z]$ is a nonzero polynomial and $\mathfrak{s}$ is its (multi-)set of roots, the knowledge of a part-square decomposition $h=q^2+\rho$ that is totally odd with respect to the center $\alpha$ makes it possible to compute $\tbest{\mathfrak{s}}{D_{\alpha,b}}\in \zerotwo$ for all depths $b\in \qq$: this follows immediately from \Cref{dfn t fun} together with \Cref{rmk totally odd good}.
\end{rmk}

\begin{prop}
    \label{prop mathfrak t minium}
    If we have a disjoint union $\mathfrak{s}=\mathfrak{s}_1\sqcup \ldots \sqcup \mathfrak{s}_N$, then the following hold:
    \begin{enumerate}[(a)]
        \item we have $\tbest{\mathfrak{s}}{D}\ge \min\{ \tbest{\mathfrak{s}_1}{D}, \ldots, \tbest{\mathfrak{s}_N}{D}\}$ for all $D$; and 
        \item the conclusion of (a) is an equality in the following cases:
        \begin{enumerate}[(i)]
            \item whenever the minimum is attained by a unique $\tbest{\mathfrak{s}_i}{D}$; and 
            \item if $N=2$, $D\cap \mathfrak{s}_1=\varnothing$, and there exists a disc $D'\subsetneq D$ such that $\mathfrak{s}_2\subseteq D'$.
        \end{enumerate}
    \end{enumerate} 
    
    \begin{proof}
        For $1 \leq i \leq N$, choose polynomials $h_i$ having $\mathfrak{s}_i$ as their sets of roots, and let $h_i=q_i^2+\rho_i$ be part-square decompositions that are good at the disc $D$. Then, by setting $q = \prod_i q_i$, we obtain a part-square decomposition for $h:=\prod_i h_i$ satisfying $\tfun_{q,\rho}(D)\ge \min_i \{\tfun_{q_i,\rho_i}(D)\}$ by \Cref{prop product part-square}(a). From this part (a) follows. Similarly, points (i) and (ii) of (b) follow straightforwardly from parts (b) and (c) of \Cref{prop product part-square} respectively.
    \end{proof}
\end{prop}

\subsection{Identifying the valid discs}
\label{sec depths construction valid discs}

Let us now consider again the hyperelliptic curve $Y: y^2 = f(x)$, and let us now fix a center $\alpha\in \bar{K}$. In \S\ref{sec cluster cluster definition}, we defined, for each subset $\mathfrak{s}\subseteq \RR$, the invariants $d_\pm(\mathfrak{s})$ and $\delta(\mathfrak{s})$, as well as the interval $I(\mathfrak{s})$ (see \Cref{dfn interval I}). We now aim to give analogous definitions \emph{relative to} the center $\alpha$.
\begin{dfn}
    \label{dfn interval I alpha}
    Given a subset $\mathfrak{s}\subseteq \RR$ and a center $\alpha\in \bar{K}$, we set
    \begin{equation*}
    \begin{gathered}
        d_+(\mathfrak{s},\alpha) = \min_{a \in \mathfrak{s}} v(a - \alpha) \in \qqinfty;\quad
        d_-(\mathfrak{s},\alpha) = \max_{a \in \RR\setminus\mathfrak{s}} v(a - \alpha) \in \qqminusinfty.
   \end{gathered}
   \end{equation*}
   We also introduce $\delta(\mathfrak{s},\alpha):=d_+(\mathfrak{s},\alpha)-d_-(\mathfrak{s},\alpha)\in \qqinfty$, and we use the notation $I(\mathfrak{s},\alpha)$ to mean the closed interval $[d_-(\mathfrak{s},\alpha),d_+(\mathfrak{s},\alpha)]$, with the convention that $I(\mathfrak{s},\alpha)=\varnothing$ whenever $d_+(\mathfrak{s},\alpha)<d_-(\mathfrak{s},\alpha)$.
\end{dfn}

\begin{rmk}
    \label{rmk bpm s alpha}
    When $\mathfrak{s}$ is a cluster and $\alpha\in D_{\mathfrak{s},d_+(\mathfrak{s})}$, we have $d_\pm(\mathfrak{s},\alpha)=d_\pm(\mathfrak{s})$, and $I(\mathfrak{s},\alpha)=I(\mathfrak{s})$.
\end{rmk}

Given $\alpha\in \bar{K}$ and $\mathfrak{s}\subseteq \RR$, assuming the interval $I(\mathfrak{s},\alpha)$ has positive length, our aim is to establish for which $b\in I(\mathfrak{s},\alpha)$ the disc $D_{\alpha,b}$ is a valid disc. When $p\neq 2$, \Cref{thm cluster p odd} gives an exhaustive answer; we now find a way to address the general case, in which $p$ is arbitrary, i.e.\ also possibly equal to $2$: we will introduce a (possibly empty) closed sub-interval $J(\mathfrak{s},\alpha)$, whose endpoints, roughly, will correspond to the depths $b\in I(\mathfrak{s},\alpha)$ for which $D_{\alpha,b}$ is a valid disc, except possibly when $\mathfrak{s}=\varnothing$ (the precise statement is given in \Cref{thm summary depths valid discs}).

Let us begin by studying the function $I(\mathfrak{s},\alpha) \ni b \mapsto \tbest{\RR}{D_{\alpha,b}}\in \zerotwo$, which enjoys the following properties.
\begin{lemma} \label{lemma slopes of t}
    The function $ I(\mathfrak{s},\alpha) \ni b \mapsto \tbest{\RR}{D_{\alpha,b}}$ is a continuous piecewise-linear function with decreasing slopes. It is identically zero if $|\mathfrak{s}|$ is odd.  On the other hand, when $|\mathfrak{s}|$ is even, its slopes are odd integers ranging from $1-|\mathfrak{s}|$ to $2g+1-|\mathfrak{s}|$, except over the subset of $I(\mathfrak{s},\alpha)$ where $\tbest{\RR}{D_{\alpha,b}}=2v(2)$, over which the slope is zero (if this subset contains an open interval).
    \begin{proof}
        Choose an interior point $b \in I(\mathfrak{s},\alpha)$, i.e.\ $b\in (d_-(\mathfrak{s},\alpha),d_+(\mathfrak{s},\alpha))$, and choose $\beta\in \bar{K}^\times$ such that $v(\beta)=b$. Any normalized reduction of $f_{\alpha,\beta}$ is a scalar times $x_{\alpha,\beta}^{|\mathfrak{s}|}$. We deduce from \Cref{prop good decomposition} that, if $|\mathfrak{s}|$ is odd, the part-square decomposition $f_{\alpha,\beta}=0^2+f_{\alpha,\beta}$ is good and $\tbest{\RR}{D_{\alpha,b}}=0$, whereas, when $|\mathfrak{s}|$ is even, this decomposition is not good, and we therefore have $\tbest{\RR}{D_{\alpha,b}}>0$. In this case, let us take a part-square decomposition $f=q^2+\rho$ which is totally odd with respect to the center $\alpha$, so that $\tbest{\RR}{D_{\alpha,b}}=\truncate{\tfun_{q,\rho}(D_{\alpha,b})}$. Since $\deg(f)=2g+1$, by \Cref{dfn qrho} the odd-degree polynomial $\rho$ has degree at most $2g+1$. Now, $b \mapsto \tfun_{q,\rho}(D_{\alpha,b})$ is, by definition, the difference between the functions $b \mapsto \vfun_{\rho}(D_{\alpha,b})$ and $b \mapsto \vfun_{f}(D_{\alpha,b})$; by \Cref{lemma vfun of disc}, the former is a piecewise linear function with decreasing odd integer slopes between $1$ and $2g+1$, while the latter is linear with slope $|\mathfrak{s}|$ over $I(\mathfrak{s}, \alpha)$.
    \end{proof}
\end{lemma}

In light of the above lemma, either $b \mapsto \tbest{\RR}{D_{\alpha,b}}$ is always $<2v(2)$ over $I(\mathfrak{s}, \alpha)$, or else it attains the output $2v(2)$ over some closed sub-interval of $I(\mathfrak{s},\alpha)$ and is $<2v(2)$ elsewhere. Let $J(\mathfrak{s},\alpha)=[b_-(\mathfrak{s},\alpha),b_+(\mathfrak{s},\alpha)]$ denote the sub-interval of $I(\mathfrak{s}, \alpha)=[d_-(\mathfrak{s}, \alpha),d_+(\mathfrak{s}, \alpha)]$ over which the output of $b \mapsto \tbest{\RR}{D_{\alpha,b}}$ equals $2v(2)$; in the former case just mentioned, we have $J(\mathfrak{s}, \alpha)=\varnothing$, while in the latter case, the interval will have the form $J(\mathfrak{s}, \alpha)=[b_-(\mathfrak{s}, \alpha),b_+(\mathfrak{s}, \alpha)]$ for some endpoints $b_\pm(\mathfrak{s},\alpha)$.
\begin{rmk} \label{rmk structure of J}
    We make the following immediate observations about the subinterval $J(\mathfrak{s}, \alpha) \subseteq I(\mathfrak{s}, \alpha)$.

    \begin{enumerate}[(a)]
        \item In the $p\neq 2$ setting, we have $2v(2) = 0$ and so the subinterval $J(\mathfrak{s}, \alpha) \subseteq I(\mathfrak{s}, \alpha)$ coincides with all of $I(\mathfrak{s}, \alpha)$.
        \item In the $p=2$ setting, if the cluster $\mathfrak{s}$ has odd cardinality, then we have $\tbest{\RR}{D_{\alpha,b}} = 0$ for all $b\in I(\mathfrak{s},\alpha)$, and therefore we have $J(\mathfrak{s}, \alpha)=\varnothing$.
        \item If $\mathfrak{s}=\varnothing$ (which can only happen in the $p = 2$ setting), then we have $J(\mathfrak{s}, \alpha)\neq \varnothing$ and $b_+(\mathfrak{s}, \alpha)=+\infty$. In fact, the piecewise linear function $b \mapsto \tbest{\RR}{D_{\alpha,b}}$ has only positive slopes by \Cref{lemma slopes of t}, so that $\tbest{\RR}{D_{\alpha,b}}$ becomes equal to $2v(2)$ as $b\to +\infty$.
    \end{enumerate}
\end{rmk}

By \Cref{prop normalization model}, it is clear that, given $D=D_{\alpha,b}$ with $b\in I(\mathfrak{s},\alpha)$, the cover $\SF{\YY_D}\to \SF{\XX_D}$ is separable if and only if $b\in J(\mathfrak{s},\alpha)$; in particular, for $b\in I(\mathfrak{s},\alpha)$, the disc $D:=D_{\alpha,b}$ can only be valid if $b\in J(\mathfrak{s},\alpha)$. To establish for which $b\in J(\mathfrak{s},\alpha)$ the disc $D=D_{\alpha,b}$ is valid, we need the following general lemma, which will allow us to compute the ramification of the cover $\SF{\YY_{D}}\to \SF{\XX_{D}}$ above $0$ and $\infty$.
\begin{lemma} \label{lemma ell and t function}
    Fix a center $\alpha\in \bar{K}$, and choose $b\in \qq$ such that $\tbest{\RR}{D_{\alpha,b}}=2v(2)$, and consider the model $\YY_D$ corresponding to the disc $D:=D_{\alpha,b}$.  Let $\ell(\XX_D,P)$ be the integer defined in \Cref{dfn ell ramification index} for any point $P$ of $\SF{\XX_D}$.  Write $\partial^+\mathfrak{t}^\RR$ (resp.\ $\partial^-\mathfrak{t}^\RR$) for the right (resp.\ left) derivative of the function $c \mapsto \tbest{\RR}{D_{\alpha,c}}$.  
    Then in the $p=2$ setting, we have the following.
    \begin{enumerate}[(a)]
        \item If $\partial^+\tbest{\RR}{D_{\alpha,b}}\ge 0$, then we have $\ell(\XX_D,0)=0$.
        \item If $\partial^+\tbest{\RR}{D_{\alpha,b}}$ is odd and negative, then we have $\ell(\XX_D,0)=1-\partial^+\tbest{\RR}{D_{\alpha,b}}$.
        \item If $\partial^-\tbest{\RR}{D_{\alpha,b}} \leq 0$, then we have $\ell(\XX_D,\infty)=0$.
        \item If $\partial^-\tbest{\RR}{D_{\alpha,b}}$ is odd and positive, then we have $\ell(\XX_D,\infty)=1+\partial^-\tbest{\RR}{D_{\alpha,b}}$.
    \end{enumerate}
    In the $p\neq 2$ setting, we instead have the following.
    \begin{enumerate}
        \item[(e)] If $\partial^+\vfun_f(D_{\alpha,b})$ is even, then we have $\ell(\XX_D,0)=0$.
        \item[(f)] If $\partial^+\vfun_f(D_{\alpha,b})$ is odd, then we have $\ell(\XX_D,0)=1$.
        \item[(g)] If $\partial^-\vfun_f(D_{\alpha,b})$ is even, then we have $\ell(\XX_D,\infty)=0$.
        \item[(h)] If $\partial^-\vfun_f(D_{\alpha,b})$ is odd, then we have $\ell(\XX_D,\infty)=1$.
    \end{enumerate}
     
    \begin{proof}
        This is just a rephrasing of  \Cref{lemma computation ell ramification} using the language introduced in \S\ref{sec depths piecewise-linear}.
        To see this, let us fix a part-square decomposition $f=q^2+\rho$ that is totally odd with respect to the center $\alpha$, and let ${f_0}$, ${q_0}$, and ${\rho_0}$ be the polynomials involved in the statement of \Cref{lemma computation ell ramification}: they are defined as appropriate scalings of $f_{\alpha,\beta}$, $q_{\alpha,\beta}$ and $\rho_{\alpha,\beta}$, for some chosen $\beta\in \bar{K}^\times$ such that $v(\beta)=b$.
        
        When $p\neq 2$, the polynomial $\overline{f_0}$ is a normalized reduction of $f_{\alpha,\beta}$, and it is easy to see that parts (e)--(h) of the lemma follow from \Cref{lemma computation ell ramification} once the left and right derivatives of $\vfun_f$ at $D_{\alpha,b}$ are interpreted in light of \Cref{lemma mathfrakh}.
        
        When $p=2$, the polynomial $\overline{q_0}$ is a normalized reduction of $q_{\alpha,\beta}$, and either the polynomial $\overline{\rho_0}$ is $0$ (when $\tfun_{q,\rho}(D_{\alpha,\beta})>2v(2)$), or it is a normalized reduction of $\rho_{\alpha,\beta}$ (when $\tfun_{q,\rho}(D_{\alpha,\beta})=2v(2)$).
        Now we have $\mathfrak{t}(D_{\alpha,c}) = \truncate{\tfun(D_{\alpha,c})}$ for all $c \in \qq$ (see \Cref{rmk t fun computing}); moreover, whenever $\tfun_{q,\rho}(D_{\alpha,c})>0$ (and hence, in particular, for all $c$ in a neighborhood of $b$), we can write $\tfun_{q,\rho}(D_{\alpha,c})=\vfun_\rho(D_{\alpha,c})-2\vfun_q(D_{\alpha,c})$, where, in light of \Cref{lemma mathfrakh}, the first summand only has odd slopes, while the second summand only has even slopes.
        Let $n_\rho$ and $n_q$ denote the orders of vanishing of $\overline{\rho_0}$ and $\overline{q_0}$ at $x_{\alpha,\beta}=0$. The assumption in (a) means that either we have $\tfun_{q,\rho}(D_{\alpha,b})>2v(2)$, or we have $\tfun_{q,\rho}(D_{\alpha,b})=2v(2)$ with $\partial^+ \tfun_{q,\rho}(D_{\alpha,b})\ge 0$; thanks to \Cref{lemma mathfrakh}, this can be translated into saying that $n_\rho\ge 2n_q$, and it is now evident that the conclusion of part (a) follows from \Cref{lemma computation ell ramification}. A similar reasoning can be followed to prove parts (b)--(d).
    \end{proof}
\end{lemma}

As a first application of the lemma above, we will show that a necessary condition for $D_{\alpha,b}$ to be a valid disc when $b\in I(\mathfrak{s},\alpha)$ is that $b$ is an endpoint of the sub-interval $J(\mathfrak{s},\alpha)\subseteq I(\mathfrak{s},\alpha)$.
\begin{lemma}
    \label{lemma not a valid disc}
    Given $b\in I(\mathfrak{s},\alpha)$ and letting $D = D_{\alpha,b}$, we have the following.
    \begin{itemize}
        \item[(a)] If $b\not\in J(\mathfrak{s},\alpha)$, the cover $\SF{\YY_D}\to \SF{\XX_D}$ is inseparable; hence, we have in particular that $D$ is not a valid disc.
        \item[(b)] If $b$ is an interior point of $J(\mathfrak{s},\alpha)$, then we have $\XX_D\not\le \Xrst$, and so $D$ is not a valid disc.
    \end{itemize}
    Consequently, $D_{\alpha,b}$ can only be a valid disc if $J(\mathfrak{s},\alpha) \neq \varnothing$ and if $b$ is an endpoint of $J(\mathfrak{s},\alpha)$, i.e.\ $b\in \{ b_-(\mathfrak{s},\alpha), b_+(\mathfrak{s},\alpha)\}$.
    
    \begin{proof}
        The statement of (a) is a direct result of \Cref{prop normalization model}, as we have already discussed.  We therefore set out to prove the statement of (b); we assume that $J(\mathfrak{s},\alpha)\neq \varnothing$ and $b_-(\mathfrak{s},\alpha)< b< b_+(\mathfrak{s},\alpha)$ and let $D = D_{\alpha,b}$. The number $N$ of distinct points of $\SF{\XX_D}$ to which the roots of $\mathcal{R}\cup \{\infty\}$ reduce is at most $2$; this is because, since $b$ does not coincide with an endpoint of the interval $I(\mathfrak{s},\alpha)$, we have that the $2g+2$ roots $\Rinfty$ each reduce either to $0$ or to $\infty$ in $\SF{\XX_D}$. Moreover, since $b$ is an interior point of $J(\mathfrak{s},\alpha)$, we have $\tbest{\RR}{D}=2v(2)$ and that the left and right derivatives of $b'\mapsto\tbest{\RR}{D_{\alpha,b'}}$ at $b'=b$ are both equal to $0$; by \Cref{lemma ell and t function}, this implies, in the $p=2$ setting, that $\SF{\YY_D}$ has two branches above $0\in \SF{\XX_D}$ and two branches above $\infty\in \SF{\XX_D}$, and the special fiber $\SF{\YY_D}$ consequently consists of two rational components (see \Cref{rmk R_0 R_1 R_2}).
        
        Now \Cref{thm part of rst separable} implies that $\XX_D\not\le \Xrst$; this is because we know that $N\le 2$, and, in the $p=2$ setting, that the special fiber $\SF{\YY_D}$ is not irreducible.
    \end{proof}
\end{lemma}

\begin{rmk}
    \label{rmk thm cluster p=2 proved}
    The lemma above, applied to the case in which $\mathfrak{s}$ is a cluster and $\alpha\in D_{\mathfrak{s},d_+(\mathfrak{s})}$, provides a proof of \Cref{thm cluster p=2}: in fact, the lemma shows that no more than $2$ valid discs can be linked to the same cluster and that no valid disc can be linked to $\mathfrak{s}$ if $J(\mathfrak{s},\alpha)=\varnothing$.  By \Cref{rmk structure of J}, this applies in particular when $p=2$ and $\mathfrak{s}$ has odd cardinality to show that there is no valid disc linked to $\mathfrak{s}$ in this case.
\end{rmk}

Now assume that $J(\mathfrak{s},\alpha)\neq \varnothing$. Among the discs $D_{\alpha,b}$ with $b\in I(\mathfrak{s},\alpha)$, the only candidate valid discs are those of depths $b_-(\mathfrak{s},\alpha)$ and $b_+(\mathfrak{s},\alpha)$, as long as these depths are not $\pm\infty$. 
Let us write $\lambda_-(\mathfrak{s},\alpha) = \partial^-\tbest{\RR}{D_{\alpha,b_-}}$ and $\lambda_+(\mathfrak{s},\alpha) = -\partial^+\tbest{\RR}{D_{\alpha,b_+}}$ (where $\partial^\pm\mathfrak{t}^\RR$ is defined as in \Cref{lemma ell and t function}). 
The integer $\lambda_-(\mathfrak{s},\alpha)$ (resp.\ $\lambda_+(\mathfrak{s},\alpha)$) is only defined if $J(\mathfrak{s},\alpha)\neq \varnothing$ and its endpoint $b_-(\mathfrak{s},\alpha)$ (resp.\ $b_+(\mathfrak{s},\alpha)$) does not coincide with $d_-(\mathfrak{s},\alpha)$ (resp.\ $d_+(\mathfrak{s},\alpha)$). In particular, $\lambda_+(\mathfrak{s},\alpha)$ and $\lambda_-(\mathfrak{s},\alpha)$ can only be defined if $p=2$ and $\mathfrak{s}$ has even cardinality (by \Cref{rmk structure of J}); when defined, they are both positive odd integers (by \Cref{lemma slopes of t}); more precisely, we have $\lambda_-(\mathfrak{s},\alpha)\in \{1, 3, \ldots, 2g+1-|\mathfrak{s}|\}$ and $\lambda_+(\mathfrak{s},\alpha)\in \{1, 3, \ldots|\mathfrak{s}|-1\}$.

\begin{rmk}
    When $\mathfrak{s}$ is a cluster and $\alpha\in D_{\mathfrak{s},d_+(\mathfrak{s})}$, we have already observed in \Cref{rmk bpm s alpha} that $I(\mathfrak{s},\alpha)=I(\mathfrak{s})$. It is also evident that, in this case, $b\in I(\mathfrak{s},\alpha)\mapsto \tbest{\RR}{D_{\alpha,b}}$ does not depend on the particular choice of $\alpha\in D_{\mathfrak{s},d_+(\mathfrak{s})}$; hence, the sub-interval $J(\mathfrak{s},\alpha)=[b_-(\mathfrak{s},\alpha), b_+(\mathfrak{s},\alpha)]$ and the slopes $\lambda_\pm(\mathfrak{s},\alpha)$ are independent of this choice as well. Hence, for $\mathfrak{s}$ a cluster we may (and often will) use without ambiguity the shorter notation $J(\mathfrak{s})$, $b_\pm(\mathfrak{s})$, and $\lambda_\pm(\mathfrak{s})$ to mean $J(\mathfrak{s},\alpha)$, $b_\pm(\mathfrak{s},\alpha)$, and $\lambda_\pm(\mathfrak{s},\alpha)$ for some (any) choice of $\alpha \in D_{\mathfrak{s},d_+(\mathfrak{s})}$.
\end{rmk}

\begin{prop}
    \label{prop lambda plus minus and ell}
    With the above notation, suppose that $J(\mathfrak{s},\alpha) \neq \varnothing$, and let $D_\pm = D_{\alpha,b_\pm(\mathfrak{s},\alpha)}$. Then we have the following.
    \begin{enumerate}
        \item[(a)] Assume that $b_-(\mathfrak{s},\alpha)<b_+(\mathfrak{s},\alpha)$ and $|\mathfrak{s}|$ is even. Then we have $\ell(\XX_{D_-},0)=\ell(\XX_{D_+};\infty)=0$.
        \item[(b)] Assume that $b_-(\mathfrak{s},\alpha)<b_+(\mathfrak{s},\alpha)$ and $|\mathfrak{s}|$ is odd. Then we have $\ell(\XX_{D_-},0)=\ell(\XX_{D_+};\infty)=1$.
        \item[(c)] Assume that $b_-(\mathfrak{s},\alpha)>d_-(\mathfrak{s},\alpha)$. Then we have $\ell(\XX_{D_-},\infty)=1+\lambda_-(\mathfrak{s},\alpha)$.
        \item[(d)] Assume that $b_+(\mathfrak{s},\alpha)<d_+(\mathfrak{s},\alpha)$. Then we have $\ell(\XX_{D_+},0)=1+\lambda_+(\mathfrak{s},\alpha)$.
    \end{enumerate}
    \begin{proof}
        This follows immediately from \Cref{lemma ell and t function}, taking into account the properties that we have already discussed of the piecewise-linear function $I(\mathfrak{s},\alpha) \ni b \mapsto\tbest{\RR}{D_{\alpha,b}}$ and of the linear function $I \ni b \mapsto \vfun_f(D_{\alpha,b})$ in our setting.
    \end{proof}
\end{prop}

We are now ready to state a necessary and sufficient condition for $D_{\alpha,b}$ to be a valid disc when $b\in I(\mathfrak{s},\alpha)$.

\begin{thm}
    \label{thm summary depths valid discs}
    Let $\mathfrak{s}$ be a cluster of $\RR$, and suppose that $\alpha$ is a point in $D_{\mathfrak{s},d_+(\mathfrak{s})}$ or that $\mathfrak{s}=\varnothing$ and $\alpha$ is any point of $\bar{K}\setminus \RR$.
    Let $D = D_{\alpha,b}$ for some $b\in I(\mathfrak{s},\alpha)=[d_-(\mathfrak{s},\alpha),d_+(\mathfrak{s},\alpha)]$; moreover, when $\mathfrak{s}=\varnothing$, let us assume that $b$ is an interior point of $I(\mathfrak{s},\alpha)$, i.e.\ that $b>d_-(\mathfrak{s},\alpha)$. In other words, when $\mathfrak{s}\neq \varnothing$ we are assuming that $D$ is any disc linked to $\mathfrak{s}$, while, when $\mathfrak{s}=\varnothing$, the disc $D$ may be any disc that contains $\alpha$ and is linked to no cluster.
    \begin{enumerate}[(a)]
        \item If $\mathfrak{s}\neq \varnothing$, then the disc $D$ is valid precisely when $b$ is an endpoint of $J(\mathfrak{s})$. Hence, there exist two (possibly coinciding) valid discs $D_{\alpha,b_-(\mathfrak{s})}$ and $D_{\alpha,b_+(\mathfrak{s})}$ linked to $\mathfrak{s}$ when $J(\mathfrak{s})\neq \varnothing$, and there does not exist a valid disc linked to $\mathfrak{s}$ when $J(\mathfrak{s})=\varnothing$.
        \item If $\mathfrak{s}=\varnothing$, in which case we have $J(\varnothing,\alpha)=[d_-(\varnothing,\alpha),+\infty)$, we have that $D$ is a valid disc precisely when $b$ coincides with the left endpoint of $J(\varnothing,\alpha)$ and $\lambda_-(\varnothing,\alpha)\ge 3$. Hence, we have two possibilities:
        \begin{enumerate}[(i)]
            \item when $J(\varnothing,\alpha)=I(\varnothing,\alpha)$, or when $J(\varnothing,\alpha)\subsetneq I(\varnothing,\alpha)$ and $\lambda_-(\varnothing,\alpha)=1$, there does not exist a valid disc centered at $\alpha$ and linked to no cluster; and 
            \item when $J(\varnothing,\alpha)\subsetneq I(\varnothing,\alpha)$ and $\lambda_-(\varnothing,\alpha)\ge 3$, there exists exactly $1$ valid disc centered at $\alpha$ and linked to no cluster.
        \end{enumerate}
    \end{enumerate}
    \begin{proof}
        The structure of $J(\mathfrak{s},\alpha)$ in the $\mathfrak{s}=\varnothing$ case is discussed in \Cref{rmk structure of J}(c). Moreover, we have already shown that $D=D_{\alpha,b}$ can only be a valid disc when $b$ is an endpoint of $J(\mathfrak{s},\alpha)$ (see \Cref{lemma not a valid disc}). So assume from now on that $J(\mathfrak{s},\alpha)\neq\varnothing$ and that $b\in \{ b_-(\mathfrak{s},\alpha),b_+(\mathfrak{s},\alpha)\}$. We remark that, since $\tbest{\RR}{D}=2v(2)$, the cover $\SF{\YY_D}\to \SF{\XX_D}$ is separable; to determine whether or not $D$ is a valid disc, we may therefore apply the criterion stated in \Cref{thm part of rst separable}. Let $N$ be the integer defined in that theorem.
        
        Assume that $b$ is an endpoint of $I(\mathfrak{s},\alpha)$. By hypothesis, this is only possible when $\mathfrak{s}\neq \varnothing$, in which case we have $I(\mathfrak{s},\alpha)=I(\mathfrak{s})$, and \Cref{lemma discs linked to clusters} implies that the roots $\Rinfty$ reduce to $\geq 3$ distinct points of $\SF{\XX_D}$ (i.e., $N\ge 3$), and $D$ is certainly a valid disc by \Cref{thm part of rst separable}.
        
        If, instead, the rational number $b$ is an interior point of $I(\mathfrak{s},\alpha)$, then we are in the $p=2$ setting; the roots of $\mathfrak{s}$ reduce to $0\in \SF{\XX_D}$, while those of $\mathcal{R}\setminus \mathfrak{s}$, together with $\infty$, reduce to $\infty\in \SF{\XX_D}$. Assume that $b=b_-(\mathfrak{s},\alpha)$: the $b=b_+(\mathfrak{s},\alpha)$ case is analogous and will thus be omitted. We know from \Cref{prop lambda plus minus and ell}(b) that $\ell(\XX_-,\infty)=1+\lambda_-(\mathfrak{s},\alpha)\ge 2$; in particular, $\SF{\YY_D}$ has only one branch above $\infty\in \SF{\XX_D}$ and is consequently irreducible.  If $\mathfrak{s}\neq \varnothing$, then we have $N=2$ and thus the criterion stated in \Cref{thm part of rst separable} ensures that $\XX_D\le \Xrst$. If $\mathfrak{s}=\varnothing$, then we have $N=1$: all roots of $\Rinfty$ reduce to $\infty\in \SF{\XX_D}$. In this case, the abelian rank of $\SF{\YY_D}$, which is the genus of its normalization, is given by $-1+\ell(\XX_-,\infty)/2=(\lambda_-(\mathfrak{s},\alpha)-1)/2$ by \Cref{prop riemann hurwitz}. Hence, \Cref{thm part of rst separable} ensures that, for $\mathfrak{s}=\varnothing$, we have $\XX_D\le\Xrst$ precisely when $\lambda_-(\mathfrak{s},\alpha)>1$.
    \end{proof}
\end{thm}


\subsection{Separating the roots (for an even-cardinality cluster \texorpdfstring{$\mathfrak{s}$}{s})}
\label{sec depths separating roots}
Let us fix a center $\alpha\in \bar{K}$, and let $\mathfrak{s}\subseteq \RR$ be any even-cardinality subset.

\subsubsection{Factoring $f$}
\label{sec depths separating roots factorizing}
We write the polynomial $f(x)$ as a product $f(x)=c f^{\mathfrak{s}}(x) f^{\RR\setminus \mathfrak{s}}(x)$, where $c$ is the leading coefficient of $f$ and write 
\begin{equation} \label{eq factorization}
    f^{\mathfrak{s}}(x) = \prod_{a \in \mathfrak{s}} (x - a)
    \qquad \mathrm{and} \qquad
    f^{\RR\setminus\mathfrak{s}}(x) = \prod_{a \in \mathcal{R}\setminus\mathfrak{s}} (x- a).
\end{equation}

Now let us define $\mathfrak{t}_+^{\mathfrak{s},\alpha}$ and $\mathfrak{t}_-^{\mathfrak{s},\alpha}$ to be the functions on the domain $[0,+\infty)$ given by 
\begin{equation} \label{eq mathfrak t pm}
    \mathfrak{t}_+^{\mathfrak{s},\alpha}: b \mapsto \tbest{\mathfrak{s}}{D_{\alpha,d_+(\mathfrak{s},\alpha)-b}}\qquad \text{and}\qquad
    \mathfrak{t}_-^{\mathfrak{s},\alpha}: b \mapsto \tbest{\RR\setminus\mathfrak{s}}{D_{\alpha,b+d_-(\mathfrak{s},\alpha)}}.
\end{equation}

Essentially, the function $\mathfrak{t}_+^{\mathfrak{s},\alpha}$ is defined by evaluating $\tbestsimple{\mathfrak{s}}$ on discs that are enlargements of $D_{\alpha,d_+(\mathfrak{s},\alpha)}$; all such discs contain $\mathfrak{s}$, and  $D_{\alpha,d_+(\mathfrak{s})}$ is the minimal disc centered at $\alpha$ with this property.  Symmetrically, the function $\mathfrak{t}_-^{\mathfrak{s},\alpha}$ is defined by evaluating $\tbestsimple{\RR\setminus\mathfrak{s}}$ at contractions of $D_{\alpha,d_-(\mathfrak{s},\alpha)}$ around the center $\alpha$: all such discs are disjoint from $\RR\setminus\mathfrak{s}$, except the largest one (i.e.\ $D_{\alpha,d_-(\mathfrak{s},\alpha)})$, which is the minimal disc centered at $\alpha$ that intersects $\RR\setminus\mathfrak{s}$.

\begin{prop}
    \label{prop properties t plus minus}
    Both functions $\mathfrak{t}^{\mathfrak{s},\alpha}_{\pm}$ are strictly increasing on their domains until they reach $2v(2)$ and become constant.  Over the part of the domain where $\mathfrak{t}^{\mathfrak{s},\alpha}_+$ (resp.\ $\mathfrak{t}^{\mathfrak{s},\alpha}_-$) is not constant, its slopes are decreasing odd integers between 1 and $|\mathfrak{s}|-1$ (resp.\ between 1 and $2g+1-|\mathfrak{s}|$).
    \begin{proof}
        We will only prove the result for $\mathfrak{t}^{\mathfrak{s},\alpha}_{+}$, as the proof for $\mathfrak{t}^{\mathfrak{s},\alpha}_{-}$ is analogous. Choose a part-square decomposition for $f^{\mathfrak{s}}=(q^{\mathfrak{s}})^2+\rho^{\mathfrak{s}}$ that is totally odd with respect to the center $\alpha$, so that  $\mathfrak{t}^{\mathfrak{s},\alpha}_{+}(b)=\truncate{\tfun_{q^\mathfrak{s},\rho^\mathfrak{s}}(D_{\alpha,b_+(\mathfrak{s},\alpha)-b})}$ for all $b\in [0,+\infty)$. Since $\mathfrak{s}\subset D_{\alpha,b_+(\mathfrak{s},\alpha)-b}$, we deduce from \Cref{lemma mathfrakh} that $[0, +\infty) \ni b \mapsto \vfun_{f^\mathfrak{s}}(D_{\alpha,b_+(\mathfrak{s},\alpha)-b})$ has slope 0; on the other hand, the function $[0, +\infty) \ni b \mapsto \vfun_{\rho^\mathfrak{s}}(D_{\alpha,b_+(\mathfrak{s},\alpha)-b})$ has odd integer slopes between 1 and $|\mathfrak{s}|-1$. From this, recalling that $\tfun_{q^\mathfrak{s}, \rho^\mathfrak{s}}=\vfun_{\rho^\mathfrak{s}}-\vfun_{f^\mathfrak{s}}$ by definition, the proposition follows.
    \end{proof}
\end{prop}  

\begin{rmk}
    \label{rmk tpm indepdendent of the center}
    We remark that, when $\mathfrak{s}\neq\varnothing$ and $\alpha\in D_{\mathfrak{s},d_+(\mathfrak{s})}$, the function $\mathfrak{t}_+^{\mathfrak{s},\alpha}$ does not depend on the particular choice of $\alpha\in D_{\mathfrak{s},d_+(\mathfrak{s})}$; we will therefore use the notation $\mathfrak{t}_+^{\mathfrak{s}}$ to mean $\mathfrak{t}_+^{\mathfrak{s},\alpha}$ where $\alpha$ is some (any) point of $D_{\mathfrak{s},d_+(\mathfrak{s})}$. On the other hand, the function  $\mathfrak{t}_-^{\mathfrak{s},\alpha}(b)$ is the same for all $\alpha\in D_{\mathfrak{s},d_+(\mathfrak{s})}$ only when $b\in [0,\delta(\mathfrak{s})]\subseteq [0,+\infty)$: when evaluating at such inputs, we may safely drop the superscript $\alpha$ and simply write $\mathfrak{t}_-^{\mathfrak{s}}$ to mean $\mathfrak{t}_-^{\mathfrak{s},\alpha}$ for any $\alpha\in D_{\mathfrak{s},d_+(\mathfrak{s})}$.
\end{rmk}

We now remark that the function $I(\mathfrak{s},\alpha) \ni b \mapsto \tbest{\RR}{D_{\alpha,b}} \in \zerotwo$ we have studied in the previous subsection can be completely recovered from $\mathfrak{t}_\pm^{\mathfrak{s},\alpha}$. In fact, we have the following.
\begin{prop} \label{prop t^R is min of t^s and t^(R-s)}
    Assume that $I(\mathfrak{s},\alpha)$ has positive length (which is always true, for example, when $\alpha$ and $\mathfrak{s}$ are as in the statement of \Cref{thm summary depths valid discs}). Then, we have
    \begin{equation*}
        \tbest{\RR}{D_{\alpha,b}}=\min\{\tbest{\mathfrak{s}}{D_{\alpha,b}},\tbest{\RR\setminus\mathfrak{s}}{D_{\alpha,b}}\}=\min\{\mathfrak{t}_+^{\mathfrak{s},\alpha}(d_+(\mathfrak{s},\alpha)-b),\mathfrak{t}_-^{\mathfrak{s},\alpha}(b-d_-(\mathfrak{s},\alpha))\}\  \text{ for all }b\in I(\mathfrak{s},\alpha).
    \end{equation*}
    \begin{proof}
        It is clearly enough to prove the result for $b$ an interior point of $I(\mathfrak{s},\alpha)$, which will extend by continuity to the endpoints of $I(\mathfrak{s},\alpha)$. For such an input $b$, we note that the roots $s\in \mathfrak{s}$ satisfy $v(s-\alpha)<b$, while the roots $s\in \RR\setminus \mathfrak{s}$ satisfy $v(s-\alpha)>b$. As a consequence, point (ii) of \Cref{prop mathfrak t minium}(b) applies.
    \end{proof}
\end{prop}

\subsubsection{A standard form for the two factors}
\label{sec depths separating roots std form}
Let us introduce the polynomials
\begin{equation} \label{eq standard form}
    f_+^{\mathfrak{s},\alpha}(z):=\prod_{a \in \mathfrak{s}} (1 - \beta_{d_+}^{-1}(a-\alpha) z)
    \qquad \mathrm{and} \qquad
    f_-^{\mathfrak{s},\alpha}(z):=\prod_{a \in \mathcal{R}\setminus\mathfrak{s}} (1 -  \beta_{d_-}(a-\alpha)^{-1}z),
\end{equation}
where the scalars $\beta_{d_\pm} \in \bar{K}^{\times}$ are chosen to satisfy $v(\beta_{d_\pm}) = d_\pm(\mathfrak{s}, \alpha)$.  These are just transformed versions of $f^{\mathfrak{s}}$ and $f^{\RR\setminus\mathfrak{s}}$, normalized so that their constant terms are $1$ and all coefficients are integral.  More precisely, we have the conversion formulas $f^{\mathfrak{s}} = \beta_{d_+}^{|\mathfrak{s}|}(f^{\mathfrak{s},\alpha}_+)^\vee(\beta_{d_+}^{-1} (z-\alpha))$ and $f^{\mathcal{R} \smallsetminus \mathfrak{s}} = \big(\prod_{a \in \RR \smallsetminus \mathfrak{s}} (\alpha - a)\big) f^{\mathfrak{s},\alpha}_-(\beta_{d_-}^{-1} (z-\alpha))$, where $(f^{\mathfrak{s},\alpha}_+)^\vee(z) = z^{|\mathfrak{s}|}f^{\mathfrak{s},\alpha}_+(1/z)$. Given a part-square decomposition for $f^{\mathfrak{s},\alpha}_+$ and for $f^{\mathfrak{s},\alpha}_-$, there is an obvious way of producing one for $f^{\mathfrak{s}}$ and $f^{\mathcal{R} \smallsetminus \mathfrak{s}}$, which in turn induces one for $f = c f^{\mathfrak{s}} f^{\RR \smallsetminus \mathfrak{s}}$.  More precisely, given two part-square decompositions 
\begin{equation*}
    f^{\mathfrak{s},\alpha}_+ = q_+^2 + \rho_+, \qquad
    f^{\mathfrak{s},\alpha}_- = q_-^2 + \rho_-,
\end{equation*} 
one obtains the decompositions
\begin{equation*}
    f^\mathfrak{s} = [q^{\mathfrak{s}}]^2 + \rho^{\mathfrak{s}}, \qquad
    f^{\RR \smallsetminus \mathfrak{s}} = [q^{\RR \smallsetminus \mathfrak{s}}]^2 + \rho^{\RR \smallsetminus \mathfrak{s}}, \qquad
    f = q^2 + \rho
\end{equation*}
by setting $q^{\mathfrak{s}} = \beta_{d_+}^{|\mathfrak{s}|/2} q_+^\vee(\beta_{d_+}^{-1}(z-\alpha))$, $q^{\RR \smallsetminus \mathfrak{s}} = \sqrt{\prod_{a \in \RR \smallsetminus \mathfrak{s}} (\alpha - a)}q_-(\beta_{d_-}^{-1}(z-\alpha))$, and $q = \sqrt{c} q^{\mathfrak{s}} q^{\RR \smallsetminus \mathfrak{s}}$ (after making appropriate choices of square roots), where $q_+^\vee(z) = z^{|\mathfrak{s}|/2} q_+(1/z)$.

\begin{rmk} \label{rmk tpm and fpm}
    We have the following.
    \begin{enumerate}[(a)]
        \item By construction, we have $\tfun_{q^{\mathfrak{s}},\rho^{\mathfrak{s}}}(D_{\alpha,b}) = \tfun_{q_+,\rho_+}(d_+(\mathfrak{s},\alpha)-b)$ and $\tfun_{q^{\RR \smallsetminus \mathfrak{s}},\rho^{\RR \smallsetminus \mathfrak{s}}}(D_{\alpha,b})=\tfun_{q_-,\rho_-}(b-d_-(\mathfrak{s},\alpha))$ for all $b \in \mathbb{Q}$; moreover, the above decomposition of $f^{\mathfrak{s}}$ (resp.\ $f^{\RR \smallsetminus \mathfrak{s}}$) is good at $D_{\alpha,b}$ if and only if the above decomposition of $f^{\mathfrak{s}}_+$ (resp.\ $f^{\mathfrak{s}}_-$) is good at $d_+(\mathfrak{s},\alpha)-b$ (resp.\ $b-d_-(\mathfrak{s},\alpha)$).
        \item It follows from part (a) above that the introduction of $f^{\mathfrak{s}}_\pm$ allows us to reinterpret the function $\mathfrak{t}^{\mathfrak{s},\alpha}_{\pm}$ as $[0,+\infty) \ni b \mapsto \tbest{Z_\pm}{D_{0,b}}$, where $Z_\pm$ denotes the set of roots of $f^{\mathfrak{s},\alpha}_\pm$. We remark that the translation and homotheties that define $f^{\mathfrak{s},\alpha}_\pm$ are chosen so that all elements of $Z_\pm$ have valuation $\le 0$ and some element in each of $Z_+$ and $Z_-$ has valuation $0$.
        \item Part (b) above implies that the knowledge of a totally odd part-square decomposition for $f_\pm^{\mathfrak{s},\alpha}$ allows us to compute $\mathfrak{t}_\pm^{\mathfrak{s},\alpha}$: this is just \Cref{rmk t fun computing}. More precisely, if $f_\pm^{\mathfrak{s},\alpha}=q_\pm^2+\rho_\pm$ is a totally odd part-square decomposition, we have $\mathfrak{t}_\pm^{\mathfrak{s},\alpha}(b)=\truncate{\tfun_{q_\pm,\rho_\pm}(D_{0,b})}$ for all $b\in [0,+\infty)$.
    \end{enumerate}
\end{rmk}

The following proposition will be useful in that it allows one to study the valid discs containing an even-cardinality subset $\mathfrak{s}$ by considering the image of $(\Rinfty) \smallsetminus \mathfrak{s}$ under the reciprocal map (see \Cref{rmk reciprocal} above).

\begin{prop} \label{prop reciprocal}
    Assume the notation of \Cref{rmk reciprocal}; let $\mathfrak{s}\subseteq \RR$ a subset of even cardinality; and assume that $\alpha \in \mathfrak{s}$.  Then we have $f^{\mathfrak{s}^{\vee,\alpha},0}_\pm = f^{\mathfrak{s},\alpha}_\mp$.  It follows that, given part-square decompositions $f^{\mathfrak{s},\alpha}_\pm = q_\pm^2 + \rho_\pm$, we have part-square decompositions $f^{\mathfrak{s}^{\vee,\alpha},0}_\pm = q_\mp^2 + \rho_\mp$, and in fact, we have the equality of functions $\mathfrak{t}^{\mathfrak{s}^{\vee,\alpha},0}_\pm = \mathfrak{t}^{\mathfrak{s},\alpha}_\mp$.
\end{prop}

\begin{proof}
    The first claims can be straightforwardly checked directly from the observations in \Cref{rmk reciprocal} and the defining formulas for the terms.  The final claim follows from choosing the decompositions $f^{\mathfrak{s},\alpha}_\pm = q_\pm^2 + \rho_\pm$ to be totally odd and using \Cref{rmk t fun computing}.
\end{proof}

\subsubsection{Reconstructing the invariants} 
\label{sec depths separating roots reconstructing invariants}
Let $b_0(\mathfrak{t}^{\mathfrak{s},\alpha}_\pm)$ be the least value of $b \in [0,+\infty)$ at which  $\mathfrak{t}_\pm^{\mathfrak{s},\alpha}: [0,+\infty)\to \zerotwo$ attains $2v(2)$, and let $\lambda(\mathfrak{t}^{\mathfrak{s},\alpha}_\pm)$ denote the left derivative of $\mathfrak{t}_\pm^{\mathfrak{s},\alpha}$ at $b_0(\mathfrak{t}^{\mathfrak{s},\alpha}_\pm)$, which is clearly only defined when $b_0(\mathfrak{t}^{\mathfrak{s},\alpha}_\pm)>0$. These invariants are closely related to those introduced in the previous subsection.
\begin{prop} \label{prop formulas for b_pm}
     Suppose that $\mathfrak{s}$ has even cardinality, that $I(\mathfrak{s},\alpha)$ has positive length (which always occurs, for example, if $\mathfrak{s}$ and $\alpha$ as as in \Cref{thm summary depths valid discs}), and that $J(\mathfrak{s},\alpha)\neq \varnothing$. Then, we have 
\begin{equation}
    \label{eq bpm}
    b_\pm(\mathfrak{s},\alpha)=d_\pm \mp b_0(\mathfrak{t}^{\mathfrak{s},\alpha}_\pm)\qquad \text{and}\qquad \lambda_\pm(\mathfrak{s},\alpha) = \lambda(\mathfrak{t}^{\mathfrak{s},\alpha}_\pm).
\end{equation}
\end{prop}

\begin{proof}
    We have $\tbest{\RR}{D_{\alpha,b}}=\min\{\mathfrak{t}_+^{\mathfrak{s},\alpha}(b_+-b),\mathfrak{t}_-^{\mathfrak{s},\alpha}(b-b_-)\}$ by \Cref{prop t^R is min of t^s and t^(R-s)} and that each of the functions $\mathfrak{t}_{\pm}^{\mathfrak{s},\alpha}$ is strictly increasing until it reaches $2v(2)$ by \Cref{prop properties t plus minus}.  The proposition now follows immediately.
\end{proof}

\begin{rmk}
    \label{rmk extending def bpm}
    In the context of \Cref{prop formulas for b_pm}, when $J(\mathfrak{s},\alpha)=\varnothing$ the formulas in (\ref{eq bpm}) can be taken as the \emph{definitions} of the rational numbers $b_\pm(\mathfrak{s},\alpha)$ and $\lambda_\pm(\mathfrak{s},\alpha)$, and from $\tbest{\RR}{D_{\alpha,b}}=\min\{\mathfrak{t}_+^{\mathfrak{s},\alpha}(b_+-b),\mathfrak{t}_-^{\mathfrak{s},\alpha}(b-b_-)\}$, it is easy to see that the condition $J(\mathfrak{s},\alpha)=\varnothing$ corresponds to the fact that $b_-(\mathfrak{s},\alpha)>b_+(\mathfrak{s},\alpha)$: roughly speaking, $J(\mathfrak{s},\alpha)$ is empty whenever its endpoints, which can always be defined, are in the reversed order.  Observe that $\lambda_+(\mathfrak{s},\alpha) \in \{1, \ldots, |\mathfrak{s}|-1\}$ is actually only defined when $b_+(\mathfrak{s},\alpha)>d_+(\mathfrak{s},\alpha)$ (that is, when $b_0(\mathfrak{t}^{\mathfrak{s},\alpha}_+)>0$), and $\lambda_-(\mathfrak{s},\alpha) \in \{1, \ldots, 2g + 1 -|\mathfrak{s}|\}$ is only defined when $d_-(\mathfrak{s},\alpha)<b_-(\mathfrak{s},\alpha)$ (that is, when $b_0(\mathfrak{t}^{\mathfrak{s},\alpha}_-)>0$).  When $\mathfrak{s}$ is a cluster and $\alpha\in D_{\mathfrak{s},d_+(\mathfrak{s})}$, even with these more general definitions, $\lambda_+(\mathfrak{s},\alpha)$ and $b_+(\mathfrak{s},\alpha)$ only depend on $\mathfrak{s}$ and not on the particular choice of the center $\alpha\in D_{\mathfrak{s},d_+(\mathfrak{s})}$; the same is true for $b_-(\mathfrak{s},\alpha)$ and $\lambda_-(\mathfrak{s},\alpha)$, provided that $b_-(\mathfrak{s},\alpha)\le d_+(\mathfrak{s},\alpha)$ (see \Cref{rmk tpm indepdendent of the center}).
\end{rmk}
  



The computations of the invariants $J(\mathfrak{s},\alpha), b_\pm(\mathfrak{s},\alpha)$ and $\lambda_\pm(\mathfrak{s},\alpha)$ appearing in \Cref{thm summary depths valid discs} are now reduced to determining $b_0(\mathfrak{t}_\pm^{\mathfrak{s},\alpha})$ and $\lambda(\mathfrak{t}_\pm^{\mathfrak{s},\alpha})$. A priori, this would require the knowledge of the functions $\mathfrak{t}_\pm^{\mathfrak{s},\alpha}: [0,+\infty)\to \zerotwo$, which in turn are immediate to compute once a totally odd part-square decomposition for the polynomials $f_\pm^{\mathfrak{s},\alpha}$ is known: see \Cref{rmk tpm and fpm}(c). However, determining a totally odd part-square decomposition  for  $f_\pm^{\mathfrak{s},\alpha}$ can be a difficult task, even if easier than determining one for the whole polynomial $f_{\alpha,1}$.  In \S\ref{sec depths sufficiently odd}, we will introduce a class of decompositions which we will name \emph{sufficiently odd decompositions} (see \Cref{dfn sufficiently odd} below); these are easier to compute and will still allow us to find $b_0(\mathfrak{t}_\pm^{\mathfrak{s},\alpha})$ and $\lambda(\mathfrak{t}_\pm^{\mathfrak{s},\alpha})$, as we will show in \Cref{prop b0 and lambda from suff odd}.

\subsection{Estimating thresholds for depths of even-cardinality clusters}
\label{sec depths thresholds}
The results that we have obtained in the above subsections show that given an even-cardinality cluster $\mathfrak{s}$ of roots associated to the hyperelliptic curve $Y : y^2 = f(x)$, there are $0$, $1$, or $2$ valid discs linked to it, and the results suggest how we may determine how many valid discs are linked to it via the knowledge of the rational numbers $b_\pm(\mathfrak{s})=d_{\pm}(\mathfrak{s})\mp b_0(\mathfrak{t}^{\mathfrak{s}}_\pm)$.  Roughly speaking, the results of \S\ref{sec depths construction valid discs} and \S\ref{sec depths separating roots} show that an even-cardinality cluster $\mathfrak{s}$ has $2$ (resp.\ $1$) valid discs linked to it if and only if its relative depth $\delta(\mathfrak{s}) = d_+(\mathfrak{s}) - d_-(\mathfrak{s})$ exceeds (resp.\ equals) some threshold depending on $\mathfrak{s}$, namely the rational number given by $b_0(\mathfrak{t}^{\mathfrak{s}}_+) + b_0(\mathfrak{t}^{\mathfrak{s}}_-)$.  The precise statement is the following rephrasing of \Cref{thm summary depths valid discs}(a) combined with \Cref{rmk extending def bpm}.

\begin{prop} \label{prop depth threshold}
    Write $B_{f,\mathfrak{s}} = b_0(\mathfrak{t}^{\mathfrak{s}}_+) + b_0(\mathfrak{t}^{\mathfrak{s}}_-)$.  Given an even-cardinality cluster $\mathfrak{s} \subset \mathcal{R}$ of relative depth $\delta(\mathfrak{s})$, there are exactly $2$ (resp.\ $1$; resp.\ $0$) valid discs linked to $\mathfrak{s}$ if we have $\delta(\mathfrak{s}) > B_{f,\mathfrak{s}}$ (resp.\ $\delta(\mathfrak{s}) = B_{f,\mathfrak{s}}$; resp.\ $\delta(\mathfrak{s}) < B_{f,\mathfrak{s}}$).
\end{prop}

\begin{rmk} \label{rmk B independence}
    We note that the rational number $B_{f,\mathfrak{s}}$ given in the above corollary does not depend on the depth $\delta(\mathfrak{s})$ in the following sense.  Given a center $\alpha \in D_{\mathfrak{s},d_+(\mathfrak{s})}$, let $\mathfrak{s}_{[\lambda]} = \{\lambda(a - \alpha) + \alpha \ | \ a \in \mathfrak{s}\}$ for some $\lambda \in \bar{K}^{\times}$ such that $v(\lambda) > -\delta(\mathfrak{s})$, so that $\mathfrak{s}_{[\lambda]}$ is a scaled version of $\mathfrak{s}$ and is a cluster in $\RR_{[\lambda]} := \mathfrak{s}_{[\lambda]} \sqcup (\mathcal{R} \smallsetminus \mathfrak{s})$ with relative depth $\delta(\mathfrak{s}_{[\lambda]}) = \delta(\mathfrak{s}) + v(\lambda)$. 
    Then it follows easily from the constructions in \S\ref{sec depths separating roots} that we have
    $\mathfrak{t}_+^\mathfrak{s} = \mathfrak{t}_+^{\mathfrak{s}_{[\lambda]}}$
    and $\mathfrak{t}_-^\mathfrak{s} = \mathfrak{t}_-^{\mathfrak{s}_{[\lambda]}}$,
    from which it follows that $B_{f_{[\lambda]},\mathfrak{s}_{[\lambda]}} = B_{f,\mathfrak{s}}$.
    In this sense, loosely speaking, we may view $B_{f,\mathfrak{s}}$ as a sort of ``threshold'' for the depth of $\mathfrak{s}$ at which we obtain $1$ valid disc linked to $\mathfrak{s}$ and above which we obtain $2$ valid discs linked to $\mathfrak{s}$.
\end{rmk}

In the rest of this subsection, we work towards obtaining estimates and exact formulas for the ``threshold depth" $B_{f,\mathfrak{s}}$ defined in \Cref{prop depth threshold} under various conditions on $\mathfrak{s} \subset \mathcal{R}$.

\begin{prop}  \label{prop cases of slope always 1}
    Let $\mathfrak{s}\subseteq \RR$ be an even-cardinality cluster.
    \begin{enumerate}[(a)]
        \item Suppose that $\mathfrak{s}$ can be written as a disjoint union $\mathfrak{r} \sqcup \mathfrak{c}_1 \sqcup\ldots\sqcup \mathfrak{c}_N$ for some $N \geq 0$, and where each $\mathfrak{c}_i$ is an even-cardinality child of $\mathfrak{s}$. Let $\delta = d_+(\mathfrak{r})-d_+(\mathfrak{s})$ (so that, in particular, $\delta=0$ when $N=0$, and $\delta=\delta(\mathfrak{r})$ when $N\ge 1$), and assume that $\delta(\mathfrak{c}_i)>\delta$ for all $i=1,\ldots,N$. Assume moreover that the sum $\sigma:=\sum_{a \in \mathfrak{r}} (a - \alpha_0)$ for some (any) fixed $\alpha_0 \in \mathfrak{r}$ has valuation equal to $d_+(\mathfrak{r})$. Then for all $b \in [0, +\infty)$ we have
        \begin{equation*}
            \mathfrak{t}^{\mathfrak{s}}_+(b) = \truncate{\delta + b},
        \end{equation*}
        so that $b_0(\mathfrak{t}^{\mathfrak{s}}_+) = \max\{2v(2) - \delta, 0\}$.
        
        \item Assume that the parent cluster $\mathfrak{c}_1$ of $\mathfrak{s}$ has even cardinality and can be written as a disjoint union $\mathfrak{c}_1=\mathfrak{c}_2\sqcup \ldots \sqcup \mathfrak{c}_N\sqcup \mathfrak{s}\sqcup \mathfrak{r}$ for some $N\ge 1$, where $\mathfrak{c}_2, \ldots, \mathfrak{c}_N$ are even-cardinality sibling clusters of $\mathfrak{s}$. Let $\delta = \delta(\mathfrak{r})$, and assume that $\delta(\mathfrak{c}_i)>\delta$ for all $i=1, \ldots, N$. Assume moreover that the sum $\sigma:=\sum_{a \in \mathfrak{r}} (a -\alpha_0)$ has valuation $d_+(\mathfrak{r})$ for some (any) fixed $\alpha_0 \in \mathfrak{r}$. Then for all $b \in [0, \delta(\mathfrak{s})]$ we have
        \begin{equation*}
            \mathfrak{t}^{\mathfrak{s}}_-(b) = \truncate{\delta + b},
        \end{equation*}
        so that $b_0(\mathfrak{t}^{\mathfrak{s}}_-) = \max\{2v(2) - \delta, 0\}$.
        
        \item Let $\mathfrak{r}=\mathfrak{s}\sqcup \mathfrak{c}_1\sqcup \ldots \sqcup \mathfrak{c}_N$ be a union of $\mathfrak{s}$ and some even-cardinality sibling clusters $\mathfrak{c}_1, \ldots, \mathfrak{c}_N$ of $\mathfrak{s}$, where $N\ge 0$. Let $\delta=d_-(\mathfrak{s})-d_-(\mathfrak{r})$ (so that, in particular, $\delta=0$ when $N=0$, and $\delta=\delta(\mathfrak{r})$ when $N\ge 1$),  and assume that $\delta(\mathfrak{c}_i)>\delta$ for all $i=1, \ldots, N$. Assume moreover that the sum $\sigma:=\sum_{a \in \RR\setminus\mathfrak{r}} (a-\alpha_0)^{-1}$ has valuation equal to $-d_-(\mathfrak{r})$ for some (any) fixed $\alpha_0\in \mathfrak{s}$. Then for all $b \in [0, \delta(\mathfrak{s})]$, we have
        \begin{equation*}
            \mathfrak{t}^{\mathfrak{s}}_-(b) = \truncate{\delta + b},
        \end{equation*}
        so that $b_0(\mathfrak{t}^{\mathfrak{s}}_-) = \max\{2v(2) - \delta, 0\}$.
    \end{enumerate}
\end{prop}

\begin{rmk} \label{rmk cases of slope always 1}
    The assumption on $v(\sigma)$ in (a) and (b) is automatically satisfied whenever $\mathfrak{r}$ is a disjoint union of two odd-cardinality clusters $\mathfrak{r}_1$ and  $\mathfrak{r}_2$ (in particular, it is satisfied whenever $\mathfrak{r}$ has cardinality 2). In fact, fixing choices of elements $\alpha_i \in \mathfrak{r}_i$ for $i = 1, 2$ and choosing $\alpha_0 = \alpha_1$, we may write 
        \begin{equation*}
            \sigma=\sum_{a \in \mathfrak{r}} (a - \alpha_1) = |\mathfrak{r}_2|(\alpha_2 - \alpha_1) + \sum_{a \in \mathfrak{r}_1} (a - \alpha_1) + \sum_{a \in \mathfrak{r}_2} (a - \alpha_2).
        \end{equation*}
    Clearly the valuation of $|\mathfrak{r}_2|(\alpha_2 - \alpha_1)$ equals $d_+(\mathfrak{r})$ while all other terms in the above formula have higher valuation, and therefore the entire sum has valuation $d_+(\mathfrak{r})$.
    
    Analogously, the assumption on $v(\sigma)$ in (c) is automatically satisfied whenever there exist odd-cardinality clusters $\mathfrak{r}_1, \mathfrak{r}_2\subseteq \RR$ such that $\mathfrak{r}_1=\mathfrak{r}\sqcup \mathfrak{r}_2$; this always happens, in particular, when $\mathfrak{r}$ has cardinality $2g$.
\end{rmk}

\begin{proof}[Proof (of \Cref{prop cases of slope always 1})]
    Let us first assume that $\mathfrak{s}$ satisfies the hypotheses of part (a).  We first set out to show that we have $\mathfrak{t}_+^{\mathfrak{r}}(b) = \truncate{b}$ for all $b \in [0,+\infty)$ (which in particular proves part (a) in the case that $N = 0$).  We consider the polynomial $f^{\mathfrak{r},\alpha_0}_+$ (with the notation in \S\ref{sec depths separating roots std form}): by construction, its coefficient are all integral, and its constant term is 1; moreover, the assumption on $v(\sigma)$ easily implies that its linear coefficient is a unit. From this, it is easy to deduce that $\rho_+$ is also a polynomial with integral coefficients and unit linear coefficient, where $f^{\mathfrak{r},\alpha_0}_+ = q_+^2 + \rho_+$ is a totally odd part-square decomposition.  Let us now consider the function $\mathfrak{t}^{\mathfrak{r}}_+: [0,+\infty)\to \qq$, which is piecewise-linear with odd decreasing slopes between $1$ and $|\mathfrak{r}|-1$ until it reaches $2v(2)$ (see \Cref{prop properties t plus minus}). This function can be computed as $\mathfrak{t}^{\mathfrak{r}}_+(b)=\tfun_{q_+,\rho_+}(D_{0,b})=\vfun_{\rho_+}(D_{0,b})$; since $\rho_+$ has integral coefficients and unit linear term, the function $\mathfrak{t}^{\mathfrak{r}}_+$ has initial output $0$ and initial slope $1$, from which we conclude that  $\mathfrak{t}^{\mathfrak{r}}_+(b)=\truncate{b}$ for all $b\in [0,+\infty)$.
    
    Now fix any $i \in \{1, \ldots ,N\}$, and take any disc $D:=D_{\mathfrak{s},d_+(\mathfrak{s})-b}$ with $b\in [0,+\infty)$.  Then the formula for $\mathfrak{t}^{\mathfrak{r}}_+$ that we found above is equivalent to the formula $\mathfrak{t}^{\mathfrak{r}}(D)=\mathfrak{t}_+^{\mathfrak{r}}(b+\delta)=\truncate{b+\delta}$.  Moreover, using the fact that $\mathfrak{t}^{\mathfrak{c}_i}_+$ has positive integer slopes as long as its output is $< 2v(2)$ by \Cref{prop properties t plus minus}, we get $\mathfrak{t}^{\mathfrak{c}_i}(D)=\mathfrak{t}_+^{\mathfrak{c}_i}(b+\delta(\mathfrak{c}_i))\ge \truncate{b+\delta(\mathfrak{c}_i)}$.  Now, using our assumption that $\delta(\mathfrak{c}_i) > \delta$, \Cref{prop mathfrak t minium}(b) implies that $\mathfrak{t}^{\mathfrak{s}}(D) = \mathfrak{t}^{\mathfrak{r}}(D)$, which can clearly be rewritten as $\mathfrak{t}_+^{\mathfrak{s}}(b) = \truncate{\delta + b}$.  This finishes the proof of part (a).
    
    Now, if we apply to the setting described in the hypothesis of (a) the automorphism $i_{\alpha}: z\mapsto (z-\alpha)^{-1}$ of the projective line, where $\alpha$ is an element of $\mathfrak{c}_1$ (resp.\ of $\mathfrak{r}$), we obtain exactly the setting described in the hypothesis of (b) (resp.\ of (c)): this can be readily verified by applying \Cref{rmk reciprocal}. Moreover, \Cref{prop reciprocal} ensures that, after applying $i_{\alpha}$, the conclusion of (a) turns into that of (b) (resp.\ of (c)).
\end{proof}

\begin{prop} \label{prop deep ubereven cluster}
    We have the following analogous statements.
    \begin{enumerate}[(a)]
        \item Suppose that $\mathfrak{s}$ is an even-cardinality cluster which itself is the disjoint union of even-cardinality child clusters $\mathfrak{c}_1, \ldots , \mathfrak{c}_N$ for some $N\ge 2$.  The minimum of the set     \begin{equation*}
            \{\mathfrak{t}_+^{\mathfrak{c}_i}(\delta(\mathfrak{c}_i))\}_{1 \leq i \leq N} \cup \{\mathfrak{t}_+^{\mathfrak{s}}(0)\}
        \end{equation*} of rational numbers is attained by more than one element. In particular, if we have $\mathfrak{t}_+^{\mathfrak{c}_i}(\delta(\mathfrak{c}_i)) = 2v(2)$ for $1 \leq i \leq N$, then we have $\mathfrak{t}_+^{\mathfrak{s}}(0) = 2v(2)$ also.
        \item Suppose that $\mathfrak{s}$ is an even-cardinality cluster whose parent cluster $\mathfrak{s}'$ is a disjoint union $\mathfrak{s}\sqcup \mathfrak{c}_2\sqcup \ldots \sqcup \mathfrak{c}_N$ of even-cardinality clusters for some $N\ge 2$. Then, the minimum of the set \begin{equation*}
            \{\mathfrak{t}_+^{\mathfrak{c}_i}(\delta(\mathfrak{c}_i))\}_{2 \leq i \leq N} \cup \{\mathfrak{t}_-^{{\mathfrak{s}'}}(\delta(\mathfrak{s}')), \mathfrak{t}_-^{\mathfrak{s}}(0)\}
        \end{equation*} 
        of rational numbers is attained by more than one element. In particular, if we have $\mathfrak{t}_+^{\mathfrak{c}_i}(\delta(\mathfrak{c}_i)) = 2v(2)$ for $2 \leq i \leq N$, and $\mathfrak{t}_-^{\mathfrak{s}'}(\delta(\mathfrak{s}'))=2v(2)$ then we have $\mathfrak{t}_-^{\mathfrak{s}}(0) = 2v(2)$ also.
    \end{enumerate}
    
    \begin{proof}
        Let us assume the setting of (a). If the minimum of the set $\{\mathfrak{t}_+^{\mathfrak{c}_i}(\delta(\mathfrak{c}_i)\}_{1 \leq i \leq N}$ is attained by more than one element, then we are done, so assume that this minimum is attained by a unique element.  Then, if we apply \Cref{prop mathfrak t minium} to the disc $D:=D_{\mathfrak{s},d_+(\mathfrak{s})}$, we obtain that $\mathfrak{t}_+^{\mathfrak{s}}(0) = \min_{1 \leq i \leq N} \mathfrak{t}_+^{\mathfrak{c}_i}(\delta(\mathfrak{c}_i))$, and the claim is proved. The proof of part (b) is analogous.
    \end{proof}
\end{prop}

The following pleasant corollary provides a simple result for a very special case of cluster picture.
\begin{cor} \label{cor g pairs}
    Suppose that we have a cluster $\mathfrak{S}$ of cardinality $2g$ which has $g$ children $\mathfrak{c}_1, \ldots , \mathfrak{c}_g$, each of which has cardinality $2$.  Assume that we have $\delta(\mathfrak{c}_i) \geq 2v(2)$ for $1 \leq i \leq g$ as well as $\delta(\mathfrak{S}) \geq 2v(2)$.  Then we have $b_0(\mathfrak{t}^{\mathfrak{c}_i}_+) = 2v(2)$ and $b_0(\mathfrak{t}^{\mathfrak{c}_i}_-) = 0$ for $1 \leq i \leq g$, and we have $b_0(\mathfrak{t}^{\mathfrak{S}}_+) = 0$ and $b_0(\mathfrak{t}^{\mathfrak{S}}_-) = 2v(2)$.
    \begin{proof}
        Letting $\mathfrak{s} = \mathfrak{r} = \mathfrak{c}_i$ and applying \Cref{prop cases of slope always 1}(a), we get $\mathfrak{t}^{\mathfrak{c}_i}_+(b) = \truncate{b}$ and $b_0(\mathfrak{t}^{\mathfrak{c}_i}_+) = 2v(2)$ for each $i$;
        similarly, if we let $\mathfrak{s}=\mathfrak{r}=\mathfrak{S}$ and apply \Cref{prop cases of slope always 1}(c), we get $\mathfrak{t}^{\mathfrak{S}}_-(b) = \truncate{b}$ and $b_0(\mathfrak{t}^{\mathfrak{S}}_-) = 2v(2)$.  Now, by applying \Cref{prop deep ubereven cluster} and using the assumption that $\delta(\mathfrak{c}_i)\ge 2v(2)$ and $\delta(\mathfrak{S})\ge 2v(2)$, we get $\mathfrak{t}^{\mathfrak{S}}_+(0)=2v(2)$ and  $\mathfrak{t}^{\mathfrak{c}_i}_-(0)=2v(2)$ for all $i$, which imply $b_0(\mathfrak{t}^{\mathfrak{S}}_+) = 0$ and $b_0(\mathfrak{t}^{\mathfrak{c}_i}_-) = 0$, respectively.
    \end{proof}
\end{cor}

\begin{rmk} \label{rmk g pairs}
    The special case treated by the above corollary can be viewed more symmetrically as involving a hyperelliptic curve defined by a degree-$(2g+2)$ polynomial (after applying a suitable automorphism so that the $2g + 2$ branch points of $Y$ all have $x$-coordinate different from $\infty$) whose roots are paired into $g + 1$ clusters each of cardinality $2$; this is the type of hyperelliptic curve treated in \cite{dokchitser2023note}.

    It immediately follows from \Cref{cor g pairs} combined with \Cref{thm summary depths valid discs}(a) and \Cref{prop formulas for b_pm} that under the hypotheses of \Cref{cor g pairs}, letting $d=d_+(\mathfrak{S})=d_-(\mathfrak{c}_i)$, the full set of valid discs consists of $D_{\mathfrak{S},d}$, along with $D_{\mathfrak{S},d-\delta(\mathfrak{S})+2v(2)}$ if $\delta(\mathfrak{S}) > 2v(2)$, as well as $D_{\mathfrak{s}_i,d+\delta(\mathfrak{c}_i) - 2v(2)}$ for each $i$ such that $\delta(\mathfrak{c}_i) > 2v(2)$.  In fact, it is not difficult to see that the components of $\SF{\Yrst}$ corresponding to the valid discs other than $D_{\mathfrak{S},d}$ 
    are all vertical (-2)-curves which, when contracted, yield the stable model (which in turn coincides with $\YY_{D_{\mathfrak{S},d}}$).  The special fiber of the stable model then has a node corresponding to each $i$ such that $\delta(\mathfrak{c}_i) > 2v(2)$ as well as an additional node if $\delta(\mathfrak{S}) > 2v(2)$.  This essentially recovers the statement of \cite[Proposition 1.5]{dokchitser2023note}.
\end{rmk}

\begin{lemma} \label{lemma estimates of b_0 from slopes}
    Let $\mathfrak{s}$ be a cluster of even cardinality.  We have \begin{equation}
        b_0(\mathfrak{t}^{\mathfrak{s}}_+) \geq \frac{2v(2) - \mathfrak{t}^{\mathfrak{s}}_+(0)}{|\mathfrak{s}| - 1} \ \ \mathrm{and} \ \ b_0(\mathfrak{t}^{\mathfrak{s}}_-) \geq \frac{2v(2) - \mathfrak{t}^{\mathfrak{s}}_-(0)}{2g + 1 - |\mathfrak{s}|}.
    \end{equation}
    Moreover, if $\lambda_+(\mathfrak{s}) = |\mathfrak{s}| - 1$ (resp.\ $\lambda_-(\mathfrak{s}) = 2g + 1 - |\mathfrak{s}|$), then the first (resp.\ the second) inequality above is an equality.
    \begin{proof}
        This follows immediately from the properties of $\mathfrak{t}^{\mathfrak{s}}_\pm$ presented in \Cref{prop properties t plus minus}.
    \end{proof}
\end{lemma}

\begin{lemma} \label{lemma tplus minus odd cardinality child}
    Let $\mathfrak{s}$ be a cluster of even cardinality. Then we have the following:
    \begin{enumerate}[(a)]
        \item if $\mathfrak{s}$ has an odd-cardinality child cluster, then we have $\mathfrak{t}^{\mathfrak{s}}_+(0)=0$; and 
        \item if $\mathfrak{s}$ has an odd-cardinality sibling cluster, then we have  $\mathfrak{t}^{\mathfrak{s}}_-(0)=0$.
    \end{enumerate}
    \begin{proof}
        It is immediate to see that, if $\mathfrak{s}$ has a child cluster $\mathfrak{c}$ of odd cardinality $2m+1$, then, letting $\alpha\in \mathfrak{c}$, any normalized reduction of $f^{\mathfrak{s},\alpha}_+$ has odd degree $2g+1-2m$ and thus, in particular, is not a square. This implies that $f^{\mathfrak{s},\alpha}_+=0^2+f^{\mathfrak{s},\alpha}_+$ is a good part-square decomposition (see \Cref{prop good decomposition}), and hence that $\mathfrak{t}^{\mathfrak{s}}_+(0)=0$. This proves (a); the proof of (b) is analogous.
    \end{proof}
\end{lemma}

\begin{prop} \label{prop estimating threshold}
    Let $\mathfrak{s}$ be a cluster of even cardinality, and let $\mathfrak{s}'$ be its parent cluster.  The rational number $B_{f,\mathfrak{s}}$ given in \Cref{prop depth threshold} satisfies the below inequalities.
    \begin{enumerate}[(a)]
        \item We have $B_{f,\mathfrak{s}} \leq 4v(2)$.  In particular, if $\delta(\mathfrak{s}) \ge 4v(2)$, then there exists a valid disc linked to $\mathfrak{s}$, and if $\delta(\mathfrak{s}) > 4v(2)$, then it is guaranteed that there are exactly $2$ valid discs linked to $\mathfrak{s}$.
        \item If $\mathfrak{s}$ has a child cluster (resp.\ a sibling cluster) of odd cardinality, then we have $B_{f,\mathfrak{s}} \geq \frac{2v(2)}{|\mathfrak{s}| - 1}$ (resp.\ $B_{f,\mathfrak{s}} \geq \frac{2v(2)}{2g + 1 - |\mathfrak{s}|}$). 
        In particular, if $\mathfrak{s}$ is a cardinality-$2$ or a cardinality-$2g$ cluster, there cannot be a valid disc linked to $\mathfrak{s}$ if $\delta(\mathfrak{s})<2v(2)$.
        \item If $\mathfrak{s}$ has both a child and a sibling cluster of odd cardinality, then we have 
        $$B_{f,\mathfrak{s}} \geq \left( \frac{2}{|\mathfrak{s}| - 1} + \frac{2}{2g + 1 - |\mathfrak{s}|} \right)v(2).$$
        \item Suppose that we are in one of the following settings.
        \begin{enumerate}[(i)]
            \item Assume that $\mathfrak{s}$ has exactly $2$ odd-cardinality child clusters $\mathfrak{r}_1$ and $\mathfrak{r}_2$. Let $\tilde{f}(x) \in \bar{K}[x]$ be a polynomial whose set of roots coincides with $\widetilde{\RR}:=\mathfrak{r}_1\sqcup \mathfrak{r}_2 \sqcup (\RR \smallsetminus \mathfrak{s})$, and let $\widetilde{\mathfrak{s}} = \mathfrak{r}_1\sqcup \mathfrak{r}_2$.
            \item Assume that $\mathfrak{s}$ has exactly $1$ odd-cardinality sibling cluster $\mathfrak{r}_2$, and let $\mathfrak{r}_1$ denote the (odd-cardinality) parent cluster of $\mathfrak{s}$. Let $\tilde{f}(x) \in \bar{K}[x]$ be a polynomial whose set of roots coincides with  $\widetilde{\RR}:=(\RR\setminus\mathfrak{r}_1)\sqcup \mathfrak{r}_2 \sqcup \mathfrak{s}$, and let $\widetilde{\mathfrak{s}} = \mathfrak{s}$.
        \end{enumerate}
        In each case we have $B_{f,\mathfrak{s}} = B_{\tilde{f},\widetilde{\mathfrak{s}}}$.
        \item Suppose that at least one of the following holds:
        \begin{enumerate}[(i)]
            \item each of the child clusters of $\mathfrak{s}$ has even cardinality and depth $\geq 2v(2)$; or 
            \item the parent and each of the sibling clusters of $\mathfrak{s}$ have even cardinality and depth $\geq 2v(2)$.
        \end{enumerate}
        Then we have $B_{f,\mathfrak{s}} \leq 2v(2)$.  If both (i) and (ii) above hold, then we have $B_{f,\mathfrak{s}} = 0$.
    \end{enumerate}
\end{prop}

\begin{proof}
    As the continuous piecewise-linear functions $\mathfrak{t}^{\mathfrak{s}}_\pm$ have positive integer slopes until reaching an output of $2v(2)$ by \Cref{prop properties t plus minus}, we must have $b_0(\mathfrak{t}^{\mathfrak{s}}_\pm) \leq 2v(2)$.  This implies that $B_{f,\mathfrak{s}} = b_0(\mathfrak{t}^{\mathfrak{s}}_+) + b_0(\mathfrak{t}^{\mathfrak{s}}_-) \leq 4v(2)$, proving part (a).
    
    Now if $\mathfrak{s}$ has a child cluster (resp.\ a sibling cluster) of odd cardinality, then we have $\mathfrak{t}^{\mathfrak{s}}_+(0) = 0$ (resp.\  $\mathfrak{t}^{\mathfrak{s}}_-(0) = 0$) by \Cref{lemma tplus minus odd cardinality child}.  By \Cref{lemma estimates of b_0 from slopes}, we then have $b_0(\mathfrak{t}^{\mathfrak{s}}_+) \geq \frac{2v(2)}{|\mathfrak{s}| - 1}$ (resp.\ $b_0(\mathfrak{t}^{\mathfrak{s}}_-) \geq \frac{2v(2)}{2g + 1 - |\mathfrak{s}|}$).  This proves (b) and (c) if we recall that $B_{f,\mathfrak{s}} = b_0(\mathfrak{t}^{\mathfrak{s}}_+) + b_0(\mathfrak{t}^{\mathfrak{s}}_-)$.
    
    Let us now address (d). Taking into account \Cref{rmk cases of slope always 1}, the alternate hypotheses of part (d) correspond respectively to the hypotheses of parts (a) and (c) of \Cref{prop cases of slope always 1}(a), with $\mathfrak{r}$ being $\mathfrak{r}_1\sqcup \mathfrak{r}_2$ in case (i), and with $\mathfrak{r}$ being the union of $\mathfrak{s}$ with its even-cardinality sibling clusters in case (ii). Let us first address case (i). By applying \Cref{prop cases of slope always 1}, we get $\mathfrak{t}^{\mathfrak{s}}_+(b)=\truncate{b}$ for all $b\in[0,+\infty)$. Moreover, we have that $\tilde{\mathfrak{s}}:=\mathfrak{r}$ is an even-cardinality cluster of the set $\tilde{\RR}:=\mathfrak{r}\sqcup (\RR \smallsetminus \mathfrak{s})$. Now, since $\RR\setminus \mathfrak{s}=\widetilde{\RR}\setminus \mathfrak{\tilde{s}}$, we have that the two functions $t_{-}^{\mathfrak{s}}$ and $t_{-}^{\tilde{\mathfrak{s}}}$ coincide; on the other hand, \Cref{prop cases of slope always 1}(a) can also be applied to $\tilde{\mathfrak{s}}\subseteq \widetilde{\RR}$ to get that $\mathfrak{t}^{\tilde{\mathfrak{s}}}_+(b)=\truncate{b}$. We conclude that  $B_{f,\mathfrak{s}} = b_0(\mathfrak{t}^{\mathfrak{s}}_+) + b_0(\mathfrak{t}^{\mathfrak{s}}_-) = b_0(\tilde{f}^{\tilde{\mathfrak{s}}}_+) + b_0(\tilde{f}^{\tilde{\mathfrak{s}}}_-) = B_{\tilde{f},\tilde{\mathfrak{s}}}$. The proof for case (ii) is completely analogous.
    
    Let us now address case (i) of (e). Since the child clusters $\mathfrak{c}_1, \ldots , \mathfrak{c}_N$ of $\mathfrak{s}$ have depth $\geq 2v(2)$, we get that $\mathfrak{t}^{\mathfrak{c}_i}_+(\delta(\mathfrak{c}_i)) = 2v(2)$ for $1 \leq i \leq N$, due to the fact that each function $\mathfrak{t}^{\mathfrak{c}_i}_+$ has positive integer slopes as long as its output is $< 2v(2)$ by \Cref{prop properties t plus minus}. Then by \Cref{prop deep ubereven cluster}, we have $\mathfrak{t}^{\mathfrak{s}}_+(0) = 2v(2)$  also, which directly implies that $b_0(\mathfrak{t}^{\mathfrak{s}}_+) = 0$.  In a similar manner, one proves that under assumption (ii), we get $\mathfrak{t}^{\mathfrak{s}}_-(0) = 2v(2)$, which directly implies that $b_0(\mathfrak{t}^{\mathfrak{s}}_-) = 0$. Now in general, as was observed in the proof of part (a), we have $b_0(\mathfrak{t}^{\mathfrak{s}}_\pm) \leq 2v(2)$, and thus the claims of part (e) follow from the formula $B_{f,\mathfrak{s}} = b_0(\mathfrak{t}^{\mathfrak{s}}_+) + b_0(\mathfrak{t}^{\mathfrak{s}}_-)$.
\end{proof}

\subsection{Sufficiently odd part-square decompositions}
\label{sec depths sufficiently odd}
The motivation for this subsection was presented at the end of \S\ref{sec depths separating roots reconstructing invariants}. For this subsection, we will let $h\in K[z]$ be a nonzero polynomial whose roots (in $\bar{K}$) all have valuation $\le 0$: the cases we care about are those of $h=f_\pm^{\mathfrak{s},\alpha}$, where $f_\pm^{\mathfrak{s},\alpha}$ are the two functions introduced in \S\ref{sec depths separating roots std form}, whose roots satisfy this property according to\Cref{rmk tpm and fpm}(b). Let us choose a part-square decomposition $h=q^2+\rho$, and let us consider the function $[0,+\infty) \ni b \mapsto \tfun_{q,\rho}(D_{0,b}) = \vfun_\rho(D_{0,b})-\vfun_h(D_{0,b})$.
\begin{rmk}
    \label{rmk properties t suff odd}
    The assumption on $h$ easily implies that the piecewise-linear function $[0,+\infty) \ni b \mapsto \vfun_h(D_{0,b})$ is constant (see \Cref{lemma mathfrakh}); as a consequence, the linear function $[0,+\infty) \ni b \mapsto \tfun_{q,\rho}(D_{0,b})$ is a non-decreasing piecewise-linear function with decreasing slopes and has the same slopes as $[0,+\infty) \ni b\mapsto\vfun_\rho(D_{0,b})$.
\end{rmk}

\begin{dfn} \label{dfn sufficiently odd}
A part-square decomposition $h = q^2 + \rho$ of a polynomial $h\in \bar{K}[z]$ whose roots all have valuation $\le 0$ is said to be \emph{sufficiently odd} if the function $\tfun_{q,\rho}: [0,+\infty) \ni b \mapsto \tfun_{q,\rho}(D_{0,b})$ satisfies the following.
\begin{enumerate}[(a)]
    \item We have $\tfun_{q,\rho}(D_{0,b})\ge 2v(2)$ for some $b\in [0,+\infty)$; we will denote by $b_0(\tfun_{q,\rho})$ the minimal $b\in [0,+\infty)$ with this property, so that $\tfun_{q,\rho}(D_{0,b})<2v(2)$ for $b\in [0,b_0(\tfun_{q,\rho}))$, and $\tfun_{q,\rho}(D_{0,b})\ge 2v(2)$ for $b\in [b_0(\tfun_{q,\rho}),+\infty)$. 
    \item If $b_0(\tfun_{q,\rho})>0$, the left derivative of $b\mapsto \tfun_{q,\rho}(D_{0,b})$ at $b=b_0(\tfun_{q,\rho})$ is odd.
\end{enumerate}
\end{dfn}

\begin{rmk}
    \label{rmk sufficiently odd}
    We have the following.
    \begin{enumerate}[(a)]
        \item Every part-square decomposition in which $\rho$ has no constant term satisfies condition (a) of \Cref{dfn sufficiently odd}, because when $\rho$ has no constant term, the function $b\in [0,+\infty)\mapsto \vfun_{\rho}(D_{0,b})$ cannot have slope 0 (see \Cref{lemma mathfrakh}), and hence $b\mapsto\tfun_{q,\rho}(D_{0,b})$ is a strictly increasing function.
        \item Totally odd decomposition are sufficiently odd, because for a totally odd decomposition all the slopes of $[0,+\infty) \ni b \mapsto\vfun_\rho(D_{0,b})$, and hence all the slopes of $[0,+\infty) \ni b \mapsto \tfun_{q,\rho}(D_{0,b})$ over $[0,+\infty)$, are odd. In particular, \Cref{prop totally odd existence} implies that a sufficiently odd decomposition always exists. 
    \end{enumerate}
\end{rmk}

\begin{prop}
    \label{prop suff odd good}
    If $h = q^2 + \rho$ is a sufficiently odd part-square decomposition, then it is good at all discs $D_{0,b}$, with $b\in [\max\{0,b_0(\tfun_{q,\rho})-\varepsilon\},+\infty)$, for $\epsilon>0$ small enough.
    \begin{proof}
        For $b\in [b_0(\tfun_{q,\rho}),+\infty)$, we have $\tfun_{q,\rho}(D_{0,b})\ge 2v(2)$, which clearly implies that the decomposition is good at $D_{0,b}$. Let us therefore assume that $b_0(\tfun_{q,\rho})>0$ and focus on the interval $[b_0(\tfun_{q,\rho})-\varepsilon,b_0(\tfun_{q,\rho}))$, where $\varepsilon > 0$ has been chosen small enough that the function $b\mapsto \tfun_{q,\rho}(D_{0,b})$ is $<2v(2)$ has odd slope, as prescribed by \Cref{dfn sufficiently odd}. Via \Cref{lemma mathfrakh}, we deduce that any normalized reduction of $\rho_{0,\beta}$, for $\beta\in \bar{K}^\times$ any element of valuation $b\in [b_0(\tfun_{q,\rho})-\varepsilon,b_0(\tfun_{q,\rho}))$, is not a square, hence the decomposition of $h$ is good at $D_{0,b}$ thanks to \Cref{prop good decomposition}(a).
    \end{proof}
\end{prop}

\begin{cor}
    \label{cor invariants of suff odd}
    The value of $b_0(\tfun_{q,\rho})$ is the same for all sufficiently odd decompositions $h=q^2+\rho$ of the polynomial $h$. When $b_0(\tfun_{q,\rho})>0$, the left derivative $\lambda(\tfun_{q,\rho})$ of the function $b\in [0,+\infty)\mapsto\tfun_{q,\rho}(D_{0,b})$ at $b=b_0(\tfun_{q,\rho})$ (which is an odd positive integer) is also independent of the sufficiently odd decomposition.
    \begin{proof}
        This follows immediately from the proposition above, taking into account \Cref{rmk same t for good}.
    \end{proof}
\end{cor}

Now we recall the the invariants $b_0(\mathfrak{t}^{\mathfrak{s},\alpha}_\pm)$ and $\lambda(\mathfrak{t}^{\mathfrak{s},\alpha}_\pm)$ introduced in \Cref{sec depths separating roots reconstructing invariants} for any even-cardinality-subset $\mathfrak{s}\subseteq \RR$ and $\alpha \in \bar{K}$, the knowledge of which is sufficient to determine the sub-interval $J(\mathfrak{s},\alpha)\subseteq I(\mathfrak{s},\alpha)$ and the slopes $\lambda_{\pm}(\mathfrak{s},\alpha)$ that we have introduced in \S\ref{sec depths construction valid discs} and that play a crucial role in determining which are the valid discs $D$ centered at $\alpha$ as well as the corresponding structure of $\SF{\YY_D}$. In the language of this subsection, the quantities $b_0(\mathfrak{t}^{\mathfrak{s},\alpha}_\pm)$ and $\lambda(\mathfrak{t}^{\mathfrak{s},\alpha}_\pm)$ are nothing but $b_0(\tfun_{q_\pm,\rho_\pm})$ and $\lambda(\tfun_{q_\pm,\rho_\pm})$ for two totally odd part-square decompositions $f_\pm^{\mathfrak{s},\alpha}=q_\pm^2+\rho_\pm$, where $f_\pm^{\mathfrak{s},\alpha}$ are the two polynomials introduced in \Cref{sec depths separating roots std form}: this is clear from \Cref{rmk tpm and fpm}(c). Now, the following proposition immediately follows from \Cref{cor invariants of suff odd}.
\begin{prop}
    \label{prop b0 and lambda from suff odd}
    We have $b_0(\mathfrak{t}^{\mathfrak{s},\alpha}_\pm)=b_0(\tfun_{q_\pm,\rho_\pm})$ and $\lambda(\mathfrak{t}^{\mathfrak{s},\alpha}_\pm)=\lambda(\tfun_{q_\pm,\rho_\pm})$ for any two sufficiently odd (and not necessarily totally odd) part-square decompositions $f_\pm^{\mathfrak{s},\alpha}=q_\pm^2+\rho_\pm$.
\end{prop}

We conclude by observing that, for a valid disc $D$, the equation $\YY_D\to \XX_D$ can also be written down from the knowledge of sufficiently odd decompositions only.
\begin{prop} \label{prop finding J from b0}
    Let $\alpha$ and $\mathfrak{s}$ be as in the assumptions of \Cref{thm summary depths valid discs}; assume that we have $J(\mathfrak{s},\alpha) \neq \varnothing$, and let $D := D_{\alpha,b_+(\mathfrak{s},\alpha)}$ or $D := D_{\alpha,b_-(\mathfrak{s},\alpha)}$ be a valid disc provided by the theorem. Then the part-square decomposition of $f$ we obtain from any two chosen sufficiently odd decompositions of $f_+^{\mathfrak{s},\alpha}$ and $f_-^{\mathfrak{s},\alpha}$ (see \S\ref{sec depths separating roots std form}) is good at the disc $D$.
    \begin{proof}
        Let $f_{\pm}^{\mathfrak{s},\alpha}=q_{\pm}^2+\rho_{\pm}$ be any sufficiently odd decompositions for $f_+^{\mathfrak{s},\alpha}$ and $f_-^{\mathfrak{s},\alpha}$, and let $f^{\mathfrak{s}}=(q^{\mathfrak{s}})^2+\rho^{\mathfrak{s}}$, $f^{\RR\setminus\mathfrak{s}}=(q^{\RR\setminus\mathfrak{s}})^2+\rho^{\RR\setminus\mathfrak{s}}$ and $f=q^2+\rho$ be the decompositions they induce for $f^{\mathfrak{s}}$, $f^{\RR\setminus\mathfrak{s}}$ and $f$ (see \S\ref{sec depths separating roots std form}). By \Cref{dfn sufficiently odd} and \Cref{prop b0 and lambda from suff odd}, we have $\tfun_{q_\pm, \rho_\pm}(D_{0,b})\ge 2v(2)$ when $b\ge b_0(\mathfrak{t}_{\pm}^{\mathfrak{s},\alpha})$.  Now, \Cref{rmk tpm and fpm}(a) ensures that the corresponding decompositions $f^{\mathfrak{s}}=(q^{\mathfrak{s}})^2+\rho^{\mathfrak{s}}$ and $f^{\RR\setminus \mathfrak{s}}=(q^{\RR\setminus\mathfrak{s}})^2+\rho^{\RR\setminus\mathfrak{s}}$ satisfy
        $\tfun_{q^{\mathfrak{s}},\rho^{\mathfrak{s}}}(D_{\alpha,b})\ge 2v(2)$ for $b\le b_+(\mathfrak{s},\alpha)$, and
        $\tfun_{q^{\RR\setminus\mathfrak{s}},\rho^{\RR\setminus\mathfrak{s}}}(D_{\alpha,b})\ge 2v(2)$ for  $b\ge b_-(\mathfrak{s},\alpha)$, recalling that $b_\pm(\mathfrak{s},\alpha)=d_\pm(\mathfrak{s},\alpha)\mp b_0(\mathfrak{t}_{\pm}^{\mathfrak{s},\alpha})$ by \Cref{rmk extending def bpm}. Since we have $J(\mathfrak{s},\alpha)\neq \varnothing$, which is to say that $b_-(\mathfrak{s},\alpha)\le b_+(\mathfrak{s},\alpha)$ (see \Cref{rmk extending def bpm}), we consequently have $\tfun_{q^{\mathfrak{s}},\rho^{\mathfrak{s}}}(D_{\alpha,b})\ge 2v(2)$ and $\tfun_{q^{\RR\setminus\mathfrak{s}},\rho^{\RR\setminus\mathfrak{s}}}(D_{\alpha,b})\ge 2v(2)$ for $b\in \{b_+(\mathfrak{s},\alpha),b_-(\mathfrak{s},\alpha)\}$. By \Cref{prop product part-square}(a), we have $\tfun_{q, \rho}(D_{\alpha,b})\ge 2v(2)$ for those values of $b$, i.e.\ the part-square decomposition for $f$ is good at $D$.
    \end{proof}
\end{prop}

\subsection{An algorithm for finding sufficiently odd part-square decompositions} \label{sec depths algorithm}

As previously explained, our main motivation for using sufficiently odd part-square decompositions rather than totally odd ones is that it is generally easier to compute a sufficiently odd decomposition of a polynomial.  The following is a general algorithm for finding sufficiently odd part-square decompositions of a nonzero polynomial $h(z) \in K[z]$ whose roots all have valuation $\leq 0$; more precisely, starting with some part-square decomposition $h = q^2 + \rho$ with $v(\rho) \geq v(h)$, this algorithm (when it terminates) transforms it into a sufficiently odd part-square decomposition of $h$, essentially by modifying $q$ by adding square roots of even-degree terms of $\rho$.  (We note that this algorithm is very similar to the procedure given by \cite[Proposition 2.2.1]{matignon2003vers}, which has a similar aim, although the latter involves adding the square roots of all even-degree terms simultaneously at each step and in that way is dissimilar to our method.)

\begin{algo} \label{algo sufficiently odd}

Let $h(z) \in K[z]$ be a nonzero polynomial whose roots all have valuation $\leq 0$, and let $h = q^2 + \rho$ be a part-square decomposition of $h$ satisfying $v(\rho) \geq v(h)$ (for instance, we may choose the trivial decomposition $h = 0^2 + \rho$).  In the steps below, we will change the polynomial $q(z)$ without changing $h(z)$, modifying $\rho(z)$ accordingly so that $h = q^2 + \rho$ is always a part-square decomposition of $h$.  At each stage, we write $R_i$ for the $i$th coefficient of $\rho$.

\begin{enumerate}
    \item \label{step1}
    Choose some ordering $n_0, n_1,\ldots, n_{\lfloor \frac{1}{2}\deg(h) \rfloor}$ of the set of natural numbers $\{0, 1, \ldots , \lfloor \frac{1}{2}\deg(h) \rfloor\}$.  In practice, this algorithm produces cleaner and more efficient results when we let $n_0 = 0$ and use the following ordering for the natural numbers in $\{1, \ldots ,\linebreak[0] \lfloor \frac{1}{2}\deg(h) \rfloor\}$.  Each natural number can be written uniquely as $s 2^j$ for a positive odd integer $s$ and an integer $j \geq 0$.  Then a natural number $s 2^j$ comes before another natural number $s' 2^{j'}$ in our ordering if and only if either we have $s < s'$ or we have $s = s'$ and $j' > j$; in other words, we order these numbers first according to their maximal odd factors and then in \textit{descending} order of their maximal $2$-power factors.
    
    Now for $0 \leq i \leq \lfloor \frac{1}{2}\deg(h) \rfloor$, perform the following two steps: 
    \begin{enumerate}[(i)]
        \item Replace $q(z)$ with $q(z) + \sqrt{R_{2n_i}} z^{n_i}$ and modify $\rho(z)$ accordingly.
        \item Check whether the decomposition $h = q^2 + \rho$ is a sufficiently odd part-square decomposition of $f$, and if it is, terminate the algorithm.
    \end{enumerate}
    \item \label{step2} Repeat Step (\ref{step1}).
\end{enumerate}
\end{algo}

The next results show that the above algorithm terminates after a finite number of steps under certain hypotheses.

\begin{lemma} \label{lemma steps in algorithm}

Assume the set-up and notation in Algorithm \ref{algo sufficiently odd}.  Suppose that we have completed Step (\ref{step1}) of Algorithm \ref{algo sufficiently odd} a total of $N$ times for some $N \geq 0$, ignoring Step (\ref{step1})(ii) (in other words, performing Step (\ref{step1})(i) for $a_i$ ranging through all natural numbers in $\{1, \ldots , \lfloor \frac{1}{2}\deg(h) \rfloor \}$) on the $N$th time.  For positive integers $j \leq \frac{1}{2}\deg(h)$, we have $v(R_{2j}) - v(h) \geq 2v(2)(1 - 2^{-N})$.  Moreover, if $N \geq 1$ and the suggested ordering of the $n_i$'s has been used in each rendition of Step \ref{step1}, then we have $R_0 = R_2 = 0$.

\end{lemma}

\begin{proof}

We prove this claim inductively, starting with the fact that it obviously holds for $N = 0$ as in this case we have $2v(2)(1 - 2^{-N}) = 0$ and we have $v(R_0) = v(h)$ since the roots all have valuation $\leq 0$.  Now assume that the claim holds for some $N \geq 0$ and consider how our part-square decomposition changes as we perform Step (\ref{step1}) for the $(N + 1)$th time.  Since all even-power terms of $\rho$ have valuation at least $2v(2)(1 - 2^{-N}) + v(h)$, it is easy to see from the instructions of Step (\ref{step1})(i) that the terms we are adding to $q(z)$ all have valuations at least $v(2)(1 - 2^{-N}) + \frac{1}{2}v(h)$.  Meanwhile, each power-$2j$ term of $\rho(z)$ is eliminated at the $i_0$th rendition of Step (\ref{step1})(i) where $n_{i_0} = j$ and may only reappear during a later rendition of Step (\ref{step1})(i) (the $i$th rendition for some $i > i_0$) as $2$ times the product of two terms of $q(z)$, one of which has been newly added: these are the $x^a$- and $x^b$-terms in $q(z)$ for some $a, b \geq 0$ with $a \neq b$ and $a + b = 2j$.  Note that if the suggested ordering of the $n_i$'s is followed, then this later $n_i$ cannot equal $2j$ and so we even have $a, b \geq 1$ in this case.  Therefore the coefficient of this new power-$2j$ term of $\rho(z)$ has valuation at least 
\begin{equation}
    \Big[v(2) + \frac{1}{2}v(h)\Big] + \Big[v(2)(1 - 2^{-N}) + \frac{1}{2}v(h)\Big] = 2v(2)(1 - 2^{-(N + 1)}) + v(h).
\end{equation}
This proves the claim for $N + 1$.  Meanwhile, if the suggested ordering of the $n_i$'s has been followed, the numbers $a$ and $b$ defined above, being distinct positive integers whose sum is an even number, must satisfy $a + b \geq 4$.  It follows that in this case, $\rho(z)$ has no constant or quadratic term (in other words, $R_0 = R_2 = 0$) after any number $N \geq 1$ of repetitions of Step (\ref{step1}).

\end{proof}

\begin{prop} \label{prop steps in algorithm}

Assume that the coefficient of the $x^s$-term of $h(z)$ has valuation equal to $v(h)$ for some odd integer $s \geq 1$.  Then Algorithm \ref{algo sufficiently odd} terminates after repeating Step (\ref{step1}) at most $\max\{1, \lfloor \log_2(s) \rfloor - 1\}$ times if the suggested ordering of the $n_i$'s is followed.

\end{prop}

\begin{proof}

First of all, since the square of any polynomial in $g(z) \in K[z]$ has the property that its odd-degree coefficients have valuation at least $v(2) + v(g)$, at any point while running the algorithm, it is clear that the $z^s$-coefficient $R_s$ in $\rho(z) = h(z) - q^2(z)$ still has valuation equal to $v(h)$.

By Remark \ref{rmk sufficiently odd}(a) and the last statement of \Cref{lemma steps in algorithm}, since at any point after Step (\ref{step1}) has been performed the first time the polynomial $\rho(z)$ has no constant term, the criterion in Definition \ref{dfn sufficiently odd}(a) is satisfied at any point after the first rendition of Step (\ref{step1}).  Assume that Step (\ref{step1}) has just been performed for the $(N := \max\{1, \lfloor \log_2(s) \rfloor - 1\})$th time in the course of running Algorithm \ref{algo sufficiently odd} and that the suggested ordering has always been followed; we shall show that the decomposition $h = q^2 + \rho$ that we have constructed by this point is sufficiently odd, which implies the statement of the proposition.  Note that the definition of the integer $N$ directly implies the inequality $N \geq \log_2(s + 1) - 2$.  It now follows directly from \Cref{lemma steps in algorithm} that for integers $j$ such that $1\leq j \leq \lfloor \frac{1}{2}\deg(h) \rfloor$, we have 
\begin{equation} \label{eq algorithm terminates}
v(R_{2j}) - v(h) \geq 2v(2)(1 - 2^{-N}) \geq 2v(2)(1 - 2^{-\log_2(s + 1) + 2}) = \big(2 - \frac{8}{s + 1}\big) v(2).
\end{equation}
As in Definition \ref{dfn sufficiently odd}(a), let $b_0 := b_0(\tfun_{q,\rho}) \in [0, +\infty)$ be the (unique) minimal non-negative rational number satisfying $\tfun_{q,\rho}(b) \geq 2v(2)$, and choose an element $\beta_0 \in \bar{K}^{\times}$ with $v(\beta_0) = b_0$.  As we have $N \geq 1$, we also have $R_0 = R_2 = 0$ by \Cref{lemma steps in algorithm}, so in particular the polynomial $\rho(\beta_0 z)$ does not have a constant or quadratic term whose coefficient has valuation equal to $v(\rho(\beta_0 z))$.  We therefore assume that $j \geq 2$ and proceed to show that the coefficient of the power-$2j$ term of $v(\rho(\beta_0 z))$ does not have valuation equal to $v(h(\beta_0 z))$ either; from Definition \ref{dfn sufficiently odd} and Lemma \ref{lemma mathfrakh}(b), this implies that $h = q^2 + \rho$ is sufficiently odd and the proposition will be proved.  Now from equation (\ref{eq algorithm terminates}) we have 
\begin{equation} \label{eq algorithm terminates2}
2v(2) - v(R_{2j}) + v(h) \leq \frac{8}{s + 1} v(2) \leq \frac{4j}{s + 1} v(2) < \frac{4j}{s} v(2).
\end{equation}
Our assumption that the roots of $h$ each have valuation $\leq 0$ implies that $\vfun_h(b)$ is constant for $b \in [0, +\infty)$ and equal to $v(h)$, so in fact we have $v(\rho(\beta_0 z)) = \vfun_\rho(b_0) = 2v(2) + \vfun_h(b_0) = 2v(2) + v(h)$.  Our hypothesis on the $z^s$-coefficient now tells us that $s b_0 = v(\beta_0^s R_s) - v(h) \geq v(\rho(\beta_0 z)) - v(h) = 2v(2)$ and therefore we have $b_0 \geq \frac{2}{s}v(2)$.  Now, using (\ref{eq algorithm terminates2}), we get 
\begin{equation}
    v(\beta_0^{2j} R_{2j}) = v(R_{2j}) + 2j b_0 > 2v(2) + v(h) - \frac{4j}{s}v(2) + \frac{4j}{s}v(2) = 2v(2) + v(h),
\end{equation}
which proves the desired statement.

\end{proof}

\begin{rmk} \label{rmk steps in algorithm}

It is not difficult to show that the algorithm terminates under the first hypothesis of \Cref{lemma steps in algorithm} (but not necessarily within $\max\{1, \lfloor \log_2(s) \rfloor - 1\}$ renditions of Step \ref{step1}) even when the suggested ordering of the $n_i$'s is not followed, through a proof similar to the above one but which does not rely on the conclusion of the last statement of Lemma \ref{lemma steps in algorithm}, namely that $R_2 = 0$.  Similarly, it is evident from the above proof that the algorithm terminates under the weaker condition that the coefficient of the $x^s$-term has valuation $< v(h) + v(2)$.  However, in this slightly more general situation it is much messier to write down a bound for the number of times Step (\ref{step1}) must be performed.

\end{rmk}

\begin{rmk} \label{rmk steps in algorithm2}

One may apply Algorithm \ref{algo sufficiently odd} to get sufficiently odd decomposition for the polynomials $h = f_\pm^{\mathfrak{s},\alpha}$ we have introduced in \S\ref{sec depths separating roots std form}, where $\mathfrak{s}\subseteq \RR$ is an even-cardinality subset and $\alpha\in \bar{K}$, as the roots of these polynomials all have valuation $\leq 0$ by \Cref{rmk tpm and fpm}(b).  However, we note that \Cref{prop steps in algorithm} cannot be applied in general to guarantee that the algorithm terminates, because the condition that there exists an odd integer $s$ such that the coefficient of the power-$s$ term is a unit does not necessarily hold.  In fact, it is not difficult to see that, when $\mathfrak{s}$ is a cluster and $\alpha\in D_{\mathfrak{s},d_+(\mathfrak{s})}$, this condition holds for $f_-^{\mathfrak{s},\alpha}$ (resp.\ $f_+^{\mathfrak{s},\alpha}$) if and only if the cardinality of the child cluster of $\mathfrak{s}$ which contains $\alpha$ (resp.\ the parent cluster of $\mathfrak{s}$) is an odd integer $s$.  We suspect that the algorithm still terminates under weaker conditions.

\end{rmk}

\subsection{Computations of sufficiently odd part-square decompositions for low degree} \label{sec depths computations low degree}

In this subsection we use Algorithm \ref{algo sufficiently odd} to compute general formulas for sufficiently odd part-square decompositions in the cases that the polynomial $h(z) \in K[z]$ has odd degree at most $7$ (noting that by construction, in \S\ref{sec depths separating roots std form} the polynomials $f^{\mathfrak{s},\alpha}_\pm$ for which we want sufficiently odd decompositions have odd degree provided that $\alpha \in \mathfrak{s}$), with an additional hypothesis in the case of degree $7$ that the roots are all units.  Some of the formulas we obtain will be used for the computations in \S\ref{sec computations}.

In the cases of degree $1$, $3$, and $5$, in fact the formulas computed below give us totally odd decompositions, whereas in degree $7$, \Cref{algo sufficiently odd} terminates and gives us formulas for a sufficiently (but not totally) odd decomposition.

\subsubsection{Polynomials of degree $1$} \label{sec computations sufficiently odd deg1}

It is immediate to see that given a linear polynomial $h(z) \in \bar{K}[z]$, letting $q(z)$ be a square root of $h(0)$ and $\rho(z) = h(z) - q^2(z) = h(z) - h(0)$, the decomposition $h = q^2 + \rho$ is sufficiently (and even totally) odd.

\subsubsection{Polynomials of degree $3$} \label{sec computations sufficiently odd deg3}

Let $h(z) = \sum_{i = 0}^3 H_i x^i \in \bar{K}[z]$ be a cubic polynomial.  Then it is clear that Algorithm \ref{algo sufficiently odd} terminates during the first rendition of Step \ref{step1} after performing Step \ref{step1}(i) for $i = 1$ following the suggested ordering (with $s_0 = 0$ and $s_1 = 1$), and we have $q(z) = \sqrt{H_0} + \sqrt{H_2}z$ (for some choices of square roots of $H_0$ and $H_2$) and 
\begin{equation} \label{eq rho degree 3}
\rho(z) = (H_1 - 2\sqrt{H_2}\sqrt{H_0})z + H_3 z^3.
\end{equation}
This decomposition $h = q^2 + \rho$ is therefore sufficiently (and even totally) odd.  One easily checks that this is the same totally odd decomposition obtained by using the method described in the proof of \Cref{prop totally odd existence}.

\subsubsection{Polynomials of degree $5$} \label{sec computations sufficiently odd deg5}

Let $h(z) = \sum_{i = 0}^5 H_i z^i \in \bar{K}[z]$ be a quintic polynomial.  In this case once again Algorithm \ref{algo sufficiently odd} terminates during the first rendition of Step \ref{step1} after performing Step \ref{step1}(i) for $i = 2$ following the suggested ordering (with $s_0 = 0$, $s_1 = 2$, and $s_2 = 1$); it is straightforward to check that here we have $q(z) = \sqrt{H_0} + \sqrt{H_2 - 2\sqrt{H_4}\sqrt{H_0}}z + \sqrt{H_4}z^2$ (where we have chosen square roots of $H_0$ and $H_4$ and then chosen a square root of $H_2 - 2\sqrt{H_4}\sqrt{H_0}$) and 
\begin{equation} \label{eq rho degree 5}
\rho(z) = \Big(H_1 - 2\sqrt{H_2 - 2\sqrt{H_4}\sqrt{H_0}}\sqrt{H_0}\Big) z + \Big(H_3 - 2\sqrt{H_4}\sqrt{H_2 - 2\sqrt{H_4}\sqrt{H_0}}\Big) z^3 + H_5 z^5.
\end{equation}
We have therefore again found a sufficiently (and even totally) odd decomposition $h = q^2 + \rho$.  It is in fact not too difficult to show that (similarly to the $g = 1$ case) we obtain this same totally odd decomposition by using the method described in the proof of \Cref{prop totally odd existence}.

\subsubsection{Polynomials of degree $7$ with unit roots} \label{sec computations sufficiently odd deg7}

Let $h(z) = \sum_{i = 0}^7 H_i z^i \in \bar{K}[z]$ be a septic polynomial whose roots all have valuation $0$ (so that in particular we have $v(H_7)=v(h)$).  Now according to \Cref{prop steps in algorithm}, Step \ref{step1} of Algorithm \ref{algo sufficiently odd} needs to performed only $\max\{1, \lfloor \log_2(7)- 1 \rfloor\} = 1$ time.  In our only rendition of Step \ref{step1}, following the suggested ordering (with $s_0 = 0$, $s_1 = 2$, and $s_2 = 1$, and $s_3 = 3$), we at most need to perform Step \ref{step1}(i) for $i = 0, 1, 2, 3$ in order to obtain a sufficiently odd decomposition $h = q^2 + \rho$; after doing this, it is straightforward to check that we have 
\begin{equation} \label{eq q degree 7}
q(z) = \sqrt{H_0} + \sqrt{H_2 - 2\sqrt{H_4}\sqrt{H_0}} z + \sqrt{H_4}z^2 + \sqrt{H_6}z^3
\end{equation}
(where we have first chosen square roots of $H_0$, $H_4$, and $H_6$ and then chosen a square root of $H_2 - 2\sqrt{H_4}\sqrt{H_0}$) and 
\begin{equation} \label{eq rho degree 7}
\begin{split}
\rho(z) =  &\Big(H_1 - 2\sqrt{H_2 - 2\sqrt{H_4}\sqrt{H_0}}\sqrt{H_0}\Big) z + \Big(H_3 - 2\sqrt{H_4}\sqrt{H_2 - 2\sqrt{H_4}\sqrt{H_0}} - 2\sqrt{H_6}\sqrt{H_0}\Big) z^3 \\ &- 2\sqrt{H_6}\sqrt{H_2 - 2\sqrt{H_4}\sqrt{H_0}} \ z^4 + \big(H_5 - 2\sqrt{H_6}\sqrt{H_4}\big) z^5 + H_7 z^7.
\end{split}
\end{equation}

\section{Finding centers of valid discs in the \texorpdfstring{$p = 2$}{p=2} setting} \label{sec centers}

This section will deal with the problem of determining a center for each valid disc $D$ in the $p=2$ setting. When $\mathfrak{s}:=D\cap \RR\neq \varnothing$, the problem is easily solved, since a center of $D$ can be chosen to be any root in $\mathfrak{s}\subseteq \RR$. When $\mathfrak{s}=\varnothing$, we will show in \S\ref{sec centers properties} that $D$ necessarily contain a root of a certain polynomial $F(T) \in K[T]$ that is introduced in \S\ref{sec centers def}.

\subsection{Defining the polynomial \texorpdfstring{$F$}{F}}
\label{sec centers def}
Given the hyperelliptic curve $y^2 = f(x)$, with $f(x)\in K[x]$ of odd degree $2g+1$, \Cref{prop totally odd existence} allows us to produce (for instance by using the procedure explained in the proof) a totally odd decomposition of the translated polynomial $f_{T,1}(z):=f(z+T)$, in which $T$ remains generic rather than being assigned to be particular center $\alpha\in \bar{K}$. Such a decomposition will have the form
\begin{equation*}
    f_{T,1} = q_{T,1}^2 + \rho_{T,1},
\end{equation*}
with
\begin{equation*}
    \begin{split}
        q_{T, 1}(z) &=  Q_0(T) + Q_1(T)z + \ldots + Q_g(T) z^g \qquad \mathrm{and} \\
        \rho_{T, 1}(z) &=  R_1(T)z + R_3(T)z^3 + \ldots + R_{2g+1}(T) z^{2g+1},
    \end{split}
\end{equation*}
where $Q_i(T)$ and $R_i(T)$ are elements of $\overline{K(T)}$, i.e.\ algebraic functions of the variable $T$.

\begin{prop}
    \label{prop integrality Qi Ri}
    The algebraic functions $Q_i(T)$ and $R_i(T)$ are integral over $K[T]$.
    \begin{proof}
        The proposition is a reflection of the general fact that a totally odd decomposition $h=q^2+\rho$ of any polynomial $h$ is always good (see \Cref{cor totally odd is good}) and hence, in particular, satisfies $v(\rho)\ge v(h)$ (by \Cref{rmk good decompositions}). We now give an explicit proof adapted to the specific setting in which we are working.
        All we have to show is that, for every valuation subring $\OO$ of $\overline{K(T)}$ such that $K[T]\subseteq \OO$, the polynomial $q_T(z)\in \overline{K(T)}[z]$ has coefficients in $\OO$. Let $w$ be the valuation of $\overline{K(T)}$ whose ring of integers is $\OO$; for any polynomial $h(z)\in \overline{K(T)}[z]$, let us also denote the Gauss valuation of $h$ by $w(h)$, i.e.\ $w(h)$ is the minimum of the valuations of the cofficients of $h$, and let $k_w$ denote the residue field of $w$. Suppose by way of contradiction that we have $w(q_{T,1})<0$; then, since $f_{T,1}=q_{T,1}^2+\rho_{T,1}$ and $w(f_{T,1})\ge 0$, we necessarily have that $w(\rho_{T,1})=w(q_{T,1}^2)<w(f_{T,1})$. Let $\gamma\in \overline{K(T)}$ be any element such that $\gamma^2$ has valuation equal to $w(\rho_{T,1})=w(q_{T,1}^2)$; now if we multiply the equation $f_{T,1}=q_{T,1}^2+\rho_{T,1}$  by $\gamma^{-2}$ and we reduce, we obtain the equation $\overline{\gamma^{-2}\rho_{T,1}}=-(\overline{\gamma^{-1}q_{T,1}})^2$ in $k_w[z]$. But this is impossible, since the left-hand side is a nonzero polynomial with coefficients in $k(w)$ whose monomials all have odd degree, while the right-hand side is the square of a nonzero polynomial with coefficients in $k(w)$, and hence it will contain nonzero monomials of even degree.
    \end{proof}
\end{prop}

\begin{dfn} \label{dfn F}
    Let $L\subset \overline{K(T)}$ be the smallest Galois extension of $K(T)$ to which $R_1(T)$ belongs. We define $F(T) \in K[T]$ to be the norm of $R_1(T)$ with respect to the extension $L/K(T)$.
\end{dfn}

\begin{rmk}
    \label{rmk F is a polynomial}
    Note that we can be sure that the norm $F(T)$ of $R_1(T)$ is actually a polynomial in the variable $T$ (and not just a rational function) because of the integrality result given by \Cref{prop integrality Qi Ri}.
\end{rmk}

\begin{rmk} \label{rmk formulas for F for g=1 and g=2}
    In the cases of $g \in \{1, 2\}$, assuming, for simplicity, that $f$ is monic, we may easily compute $F(T)$ as the norm of $R_1(T)$ using the formulas found in \S\ref{sec computations sufficiently odd deg3},\ref{sec computations sufficiently odd deg5}.  For $0 \leq i \leq 2g + 1$, let $P_i(T) \in K[T]$ be the $z^i$-coefficient of $f(z + T) \in K[T][z]$.  Then for $g = 1$, we have the formula 
    \begin{equation} \label{eq formula for F when g=1}
        F = P_1^2 - 4P_2 P_0,
    \end{equation}
    and for $g = 2$, we have the formula
    \begin{equation} \label{eq formula for F when g=2}
        F = (P_1^2 - 4P_2P_0)^2 - 64P_4 P_0^3.
    \end{equation}
\end{rmk}

\subsection{Using the polynomial \texorpdfstring{$F$}{F} to find centers}
\label{sec centers properties}

We will now establish some properties of $F$; in particular, we will show that each root of $F$ is the center of a valid disc, and that all valid discs $D$ such that $D\cap \RR=\varnothing$ contain a root of $F$.

\begin{prop} \label{prop roots of F are centers of clu discs}
    Let $\alpha$ be a root of $F$ in $\bar{K}$.  Then we have the following.
    \begin{enumerate}[(a)]
        \item There exists a part-square decomposition $f=q^2+\rho$ which is totally odd at the center $\alpha$ such that $\rho_{\alpha,1}$ has no linear term.
        \item The element $\alpha$ is not a root of $f$ (i.e.\, $\alpha\notin\RR$), and there exists a valid disc $D$ containing $\alpha$ and such that $\ell(\XX_D,\infty)>0$.
        \item If $D=D_{\alpha,b}$ is minimal among the valid discs satisfying the conditions described in (b), then we have $\ell(\XX_D,\overline{x_{\alpha,\beta}}=0)=0$ (for a choice of $\beta\in \bar{K}^\times$ with $v(\beta) = b$).
    \end{enumerate}
    \begin{proof}
        Statement (a) follows from the definition of the polynomial $F$ as the norm of the linear coefficient of $\rho_{T,1}$.  More precisely, let $L' \subset \overline{K(T)}$ be a Galois extension of $K(T)$ to which all the coefficients of the polynomials $q_{T,1}(z)$ and $\rho_{T,1}(z)$ belong, and let $S'$ be the integral closure of $K[T]$ in $L'$. The norm $\mathrm{Nm}_{L'/K(T)}(R_1(T))\in K[T]$ must be a power $F^d$ of $F$; hence, the element $\alpha$ is a root of it. Let us also choose an extension $\widetilde{\psi}_\alpha: S'\to \bar{K}$ of the evaluation map $\psi_\alpha: K[T]\to \bar{K}, T\mapsto \alpha$: it is clear that the polynomials $q_{\alpha,1}(z) := \widetilde{\psi}_\alpha(q_{T,1}(z))\in \bar{K}[z]$ and $\rho_{\alpha,1}(z) := \widetilde{\psi}_\alpha(\rho_{T,1}(z))\in \bar{K}[z]$ provide a totally odd part-square decomposition for $f_{\alpha,1}$ whose linear coefficient is $\widetilde{\psi}_\alpha(R_1(T))\in \bar{K}$. On the other hand, since $\alpha$ is a root of $F^d$, we have $\widetilde{\psi}_\alpha(F^d)=\psi_{\alpha}(F^d)=0$, while at the same time, we have the formula $F^d=\prod_{\sigma}\sigma(R_1(T))$ as $\sigma$ ranges among the elements of $\Gal_{L'/K(T)}$.  We conclude that $\widetilde{\psi}_\alpha(\sigma(R_1(T))=0$ for some $\sigma \in \Gal_{L'/K(T)}$; after replacing $\tilde{\psi_a}$ with $\tilde{\psi_\alpha}\circ \sigma$, we may assume that $\sigma=1$, so that $\widetilde{\psi}_\alpha(R_1(T))=0$, and part (a) is thus proved.
        
        Let us now prove part (b).  Using part (a), we have a part-square decomposition $f=q^2+\rho$ that is totally odd with respect to the center $\alpha$ and such that the linear term of $\rho_{\alpha,1}$ is zero.  Suppose that $\alpha$ is a root of $f$, so that the polynomial $f_{\alpha,1}$ has no constant term.  Then, coming from the fact that $x_{\alpha,1}^3 | \rho_{\alpha,1}(x_{\alpha,1})$, it is easy to see that we have $x_{\alpha,1} | q_{\alpha,1}(x_{\alpha,1})$; it immediately follows that we have $x_{\alpha,1}^2 | f_{\alpha,1}(x_{\alpha,1})$, which contradicts the fact that $f$ has no multiple roots.  This proves the first claim of part (b).  
        
        Now let us study the function $\qq \ni c \mapsto \tbest{\RR}{D_{\alpha,c}}$. When $c\to -\infty$, it is constantly zero, since $f$ has odd degree and $f=0^2+f$ is consequently a good part-square decomposition at large enough discs, while when $c\to +\infty$, it is constantly $2v(2)$ since $\alpha$ is not a root of $f$ (this was already mentioned in \Cref{rmk structure of J}). As a consequence, there exists $b\in \qq$ such that the output $c\mapsto\tbest{\RR}{D_{\alpha,c}}$ is $<2v(2)$ right before $c=b$ and equals $2v(2)$ at $c\ge b$.  Let $\mathfrak{s}=D_{\alpha,b}\cap \RR$, so that $d_-(\mathfrak{s},\alpha)<b=b_-(\mathfrak{s},\alpha)\le b_+(\mathfrak{s},\alpha)$ in the language of \S\ref{sec depths construction valid discs}. If $\mathfrak{s}\neq \varnothing$, \Cref{thm summary depths valid discs} ensures that the disc $D:=D_{\alpha,b}$ is valid.  If $\mathfrak{s}=\varnothing$, from the fact that $\rho_{\alpha,1}$ has no linear term we deduce that the function $I(\varnothing,\alpha)\to [0,2v(2)], c\mapsto \tbest{\RR}{D_{\alpha,c}}$, which can be computed as $c\mapsto \truncate{\underline{t}_{q,\rho}(D_{\alpha,c})}$, grows with slopes $\ge 3$ until reaching $2v(2)$ at $c=b$ (in other words, it cannot admit slope $1$); hence 
        \Cref{thm summary depths valid discs} still guarantees that the disc $D$ is valid since $\lambda_-(\varnothing,\alpha)\ge 3$.
        
        It is an immediate consequence of \Cref{lemma ell and t function} that we have $\ell(\XX_D,\overline{x_{\alpha,\beta}}=0)=0$ and $\ell(\XX_D,\overline{x_{\alpha,\beta}}=\infty)>0$, for $\beta\in \bar{K}^\times$ an element of valuation $b$. In particular, the proof of (b) is finished. Moreover, if we take $D'=D_{\alpha,b'}$ to be any other valid disc centered at $\alpha$ such that $\ell(\XX_{D'},\infty)>0$, by \Cref{lemma ell and t function} we must have that the output of $c\mapsto \tbest{\RR}{D_{\alpha,c}}$ is $<2v(2)$ for $c$ slightly smaller than $b'$ and $=2v(2)$ at $c=b'$. But this implies, by the construction of $b$, that we have $b'\le b$, i.e.\ $D'\supseteq D$. This shows that the valid disc $D$ that we found above is the minimal one satisfying the conditions given by part (b) and thus completes the proof of part (c).
    \end{proof}
\end{prop}

The following theorem provides a statement that is somehow converse to the one of \Cref{prop roots of F are centers of clu discs} above. Together with that proposition, it is essentially a generalization of \cite[Theorem 5.1]{lehr2006wild} (which treats only the geometrically equidistant case), and the underlying strategy of its proof is inspired by that of Lehr and Matignon.
\begin{thm} \label{thm centers}
    Suppose that $D=D_{\alpha,b}$ is a valid disc such that $\ell(\XX_D,\infty)>0$, i.e.\ such that $\SF{\YY_D}$ has only one branch above $\infty\in \SF{\XX_D}$. Then $D$ contains a root of $F$.
\end{thm}

\begin{rmk} \label{rmk centers}
    Let $D=D_{\alpha,b}$ be a disc, and let $\mathfrak{s} = D \cap \RR$, so that we have $b\in (d_-(\mathfrak{s},\alpha),d_+(\mathfrak{s},\alpha)]$.
    Then the disc $D$ satisfies the hypothesis in the above theorem if and only if $b=b_-(\mathfrak{s},\alpha)\le b_+(\mathfrak{s},\alpha)$: this is an easy consequence of \Cref{thm summary depths valid discs} together with \Cref{prop lambda plus minus and ell}(a). Therefore, a valid disc $D$ does not satisfy the hypothesis in the above theorem if and only if $b = b_+(\mathfrak{s},\alpha)>b_-(\mathfrak{s},\alpha)$.
    
    In particular, $D$ always satisfies the hypothesis of the theorem if it is linked to no cluster (i.e., $\mathfrak{s}=\varnothing$), or if it is linked to a unique cluster and is the only disc linked to it (i.e.\ $\mathfrak{s}\neq \varnothing$ and $d_-(\mathfrak{s})<b_-(\mathfrak{s})=b=b_+(\mathfrak{s})<d_+(\mathfrak{s})$).
\end{rmk}

\begin{cor} \label{cor centers}
    Each root of $F$ lies in a valid disc. Conversely, suppose that a valid disc $D$ satisfies one of the two assumptions below:
    \begin{enumerate}[(a)]
        \item the disc $D$ is linked to no cluster; or 
        \item the disc $D$ is linked to a unique cluster $\mathfrak{s}$, is the unique valid disc linked to $\mathfrak{s}$, and is minimal among valid discs.
    \end{enumerate}
    Then the disc $D$ contains a root $\alpha$ of $F$. Moreover, for any such $\alpha$ and for all $a \in D\cap \RR$, we have $v(a-\alpha)=b$, where $b$ is the depth of the disc $D$.
    \begin{proof}
        All roots of $F$ belong to valid discs by \Cref{prop roots of F are centers of clu discs}. Conversely, suppose that $D$ is a valid disc. If $D$ is linked to no cluster, or if it is linked to only one cluster and it is the unique valid disc linked to it, then by \Cref{rmk centers} it satisfies the assumption of \Cref{thm centers} (i.e., $\ell(\XX_D,\infty)>0$) and consequently contains a root $\alpha$ of $F$.  Now suppose that $D$ satisfies condition (b), so that we have $d_-(\mathfrak{s})<b_-(\mathfrak{s})=b=b_+(\mathfrak{s})<d_+(\mathfrak{s})$ (see the results in \S\ref{sec depths construction valid discs}), and so that moreover, if $P \neq \infty$ is the point of $\SF{\XX_D}$ to which $\mathfrak{s}$ reduces, we have $\ell(\XX_D,P)=1+\lambda_+(\mathfrak{s})>0$ by \Cref{prop lambda plus minus and ell}. But \Cref{prop roots of F are centers of clu discs}(c) ensures that, if $P'\neq \infty$ is the point of $\SF{\XX_D}$ to which $x=\alpha$ reduces, we have $\ell(\XX_D,P')=0$. Hence, we have $P\neq P'$, which means that $v(a - \alpha)=b$ for all $a \in \mathfrak{s}$.
    \end{proof}
\end{cor}

Let us now address the proof of \Cref{thm centers}. We begin with the following lemma.
\begin{lemma}
    \label{key lemma}
    Let $D=D_{\alpha,b}$ be a valid disc satifying the hypothesis in \Cref{thm centers}. Then, for small enough $\varepsilon>0$, we have the following: for all  $\alpha'$ such that $b':=v(\alpha-\alpha')\in [b-\varepsilon,b)$, and for any pair of part-square decompositions $f=q^2+\rho$ and $f=(q')^2+\rho'$ which are good at the disc $D':=D_{\alpha,b'}=D_{\alpha',b'}$, we have the comparison $v(R_1)>v(R'_1)$ between the respective linear coefficients $R_1, R_1'\in \bar{K}$ of $\rho_{\alpha,1}$ and $\rho'_{\alpha',1}$.
    \begin{proof}
        Let $\mathfrak{s}$ and $\alpha$ be as in \Cref{rmk centers}, so that we have $d_-(\mathfrak{s},\alpha)<b=b_-(\mathfrak{s},\alpha)\le b_+(\mathfrak{s},\alpha)$. In particular, by the results in \S\ref{sec depths construction valid discs} we have $\tbest{\RR}{D'}<2v(2)$ and that $\SF{\YY_{D'}}\to \SF{\XX_{D'}}$ is inseparable for all discs $D':=D_{\alpha,b'}$ with $b'\in [b-\varepsilon,b)$, for $\varepsilon>0$ small enough.  After possibly shrinking $\varepsilon$, we furthermore obtain that, for all such $b'$, we claim that
        \begin{equation}
            \label{fmla key lemma}
            \lbrace 0 \rbrace \subseteq \Ctr(\XX_{D'},\Xrst)\subseteq \lbrace 0, \infty \rbrace.
        \end{equation} 
        Indeed, the set $\Ctr(\XX_{D'},\Xrst)$ is just the finite union $\bigcup_{\tilde{\XX}}\Ctr(\XX_{D'},\tilde{\XX})$, where $\tilde{\XX}$ varies among the smooth models of the line dominated by $\Xrst$. For any such $\tilde{\XX}$, if we let $\tilde{D}$ be the disc such that $\tilde{\XX}=\XX_{\tilde{D}}$, we have three possibilities:
        \begin{enumerate}
            \item $\tilde{D}\subseteq D$; in this case, since $D\subsetneq D'$, we have $\tilde{D}\subsetneq D'$, so that $\Ctr(\XX_{D'},\tilde{\XX})=\{ 0 \}$; note that this case does actually occur at least once, for $\tilde{D}=D$;
            \item $\tilde{D}\supsetneq D$; in this case, since $D'$ is only slightly larger than $D$, we may also assume that $\tilde{D}\supsetneq D'$, so that $\Ctr(\XX_{D'},\tilde{\XX})=\{ \infty \}$; and 
            \item $\tilde{D}\cap D=\varnothing$; in this case, since $D'$ is only slightly larger than $D$, we may also assume that $\tilde{D}\cap D'=\varnothing$, so that $\Ctr(\XX_{D'},\tilde{\XX})=\{ \infty \}$.
        \end{enumerate}

        By \Cref{cor rst isomorphism failure}, the inclusions in (\ref{fmla key lemma}) imply that, for $b'\in [b-\varepsilon,b)$, the special fiber $\SF{\YY_{D'}}$ must be singular above $x_{\alpha,b'}=0$ but non-singular away from $x_{\alpha,b'}=0$ and $x_{\alpha,b'}=\infty$.
        Let us now pick an element $\alpha'\in \bar{K}$ such that $v(\alpha'-\alpha)=b'$ and choose an element $\beta'\in \bar{K}^\times$ of valuation $b'$.  We have $D'=D_{\alpha,b'}=D_{\alpha',b'}$ and that the special fiber $\SF{\YY_{D'}}$ is non-singular above $x_{\alpha',\beta'}=0$ (since $x_{\alpha',\beta'}=0$ corresponds to some point whose $x_{\alpha,\beta'}$-coordinate is neither 0 nor $\infty$). 
        
        Now let us choose two part-square decompositions $f=q^2+\rho$ and $f=(q')^2+\rho'$ that are good at the disc $D'$: our aim will be to show the comparison $v(R_1)<(R_1')$ between the linear terms of $\rho_{\alpha,1}$ and $\rho'_{\alpha',1}$ under the assumption that the valuation $b':=v(\alpha'-\alpha)$ satisfies $b'\in [b-\varepsilon,b)$.  We will actually show this inequality in three steps, by proving the below for the disc $D':=D_{\alpha,b'}$:
        \begin{enumerate}[(a)]
            \item $v(\beta' R_1)>\vfun_{\rho}(D')$;
            \item $v(\beta' R'_1)=\vfun_{\rho'}(D')$;
            \item $\vfun_{\rho}(D')=\vfun_{\rho'}(D')$.
        \end{enumerate}
        
        To prove (a), we observe that the inseparable curve $\SF{\YY_{D'}}$ has the equation \begin{equation*}y^2=\overline{\gamma^{-1}\rho_0(x_{\alpha,b'})},\end{equation*} where $\rho_0$ is a normalized reduction of $\rho_{\alpha,\beta'}$ (see \S\ref{sec models hyperelliptic inseparable}). In light of this, since $\beta'R_1$ is the linear term of $\rho_{\alpha,\beta'}$, (a) simply expresses the fact that $\SF{\YY_{D'}}$ is singular above $x_{\alpha,b'}=0$. In a completely analogous way, the equation in (b) expresses the fact that $\SF{\YY_{D'}}$ is not singular above $x_{\alpha',b'}=0$. Finally, (c) follows from \Cref{rmk same t for good}.
    \end{proof}
\end{lemma}

\begin{lemma}
    \label{lemma roots of a non arch polynomial}
    Suppose that $h\in K[z]$ is a nonzero polynomial and $D:=D_{\alpha,b}\subseteq \bar{K}$ is a disc not containing any of the roots of $h$ in $\bar{K}$. Then we have $v(h(z_0)) = v(h(z_1))$ for all $z_0, z_1\in D$.
    \begin{proof}
        Let $a_1, \ldots, a_r$ be the roots of $h$ in $\bar{K}$, so that we can write $h(z)=c \prod_{i=1}^r (z-a_i)$ for some $c\in K^\times$. For each $i$, since $z_0$ and $z_1$ are points of $D$, while $s_i$ is not, we have $v(z_0-a_i)=v(z_1-a_i)$, from which it clearly follows that $v(h(z_0))=v(h(z_1))$.
    \end{proof}
\end{lemma}

\begin{proof}[Proof of \Cref{thm centers}]
    Let $S$ be the minimal finite Galois extension of $K[T]$ to which $R_1(T)$ belongs, so that $F(T)=\mathrm{Nm}_{S/K[T]}(R_1) = \prod_{\sigma}\sigma(R_1)$, with the product taken over all $\sigma \in \Gal(S/K[T])$.  Now let $\alpha'$ be any point of the annulus $D_\varepsilon\setminus D$, where $D_\varepsilon = D_{\alpha,b-\varepsilon}$ for some $\varepsilon > 0$ chosen small enough so that the conclusion of \Cref{key lemma} holds. Let us consider the evaluation maps $\psi_\alpha, \psi_{\alpha'}: K[T]\to \bar{K}$ corresponding to $\alpha$ and $\alpha'$; for each of them, we make the choice of an extension $\widetilde{\psi}_\alpha, \widetilde{\psi}_{\alpha'}: S\to \bar{K}$; the other possible extensions can be obtained by precomposing with appropriate automorphisms $\sigma\in \Gal(S/K[T])$.
    
    We clearly have that $\widetilde{\psi}_\alpha(R_1)$ and $\widetilde{\psi}_{\alpha'}(R_1)$ are the linear terms of $\rho_{\alpha,1}$ and $\rho'_{\alpha',1}$ for two part-square decompositions $f=q^2+\rho$ and $f=(q')^2+\rho'$ which are totally odd with respect to the centers $\alpha$ and $\alpha'$ respectively; in particular, both decompositions are good at any disc containing both $\alpha$ and $\alpha'$ (see \Cref{rmk totally odd good}); hence, \Cref{key lemma} ensures that $v(\widetilde{\psi_\alpha}(R_1))>v(\widetilde{\psi_{\alpha'}}(R_1))$. Since this holds for any choices of extensions $\widetilde{\psi}_\alpha, \widetilde{\psi}_{\alpha'}$, we deduce that $ v(\widetilde{\psi_\alpha}(\prod_\sigma\sigma(R_1)))>v(\widetilde{\psi_{\alpha'}}(\prod_\sigma\sigma(R_1)))$, which is to say that $v(\psi_\alpha(F))>v(\psi_{\alpha'}(F))$, which in turn is nothing but the comparison $v(F(\alpha))>v(F(\alpha'))$.
    
    Now suppose by way of contradiction that $D$ does not contain any root of $F$. One can clearly find a disc $D'$, with $D\subsetneq D'\subseteq D_\varepsilon$, such that also $D'$ does not contain any root of $F$. Now, for $\alpha'\in D' \setminus D$, the argument above implies that $v(F(\alpha))>v(F(\alpha'))$, but, in light of \Cref{lemma roots of a non arch polynomial}, this contradicts the assumption that $D'$ does not contain any root of $F$.
\end{proof}

\section{The geometry of the special fiber} \label{sec structure}

Our purpose in this section is to use the framework we developed in \S\ref{sec depths} to glean information about the components of the special fiber of the relatively stable model of the hyperelliptic curve $Y$, based on knowledge of the relationship between its valid discs and the cluster picture associated to the defining polynomial. In particular, in \S\ref{sec structure toric rank} we will compute the toric rank of the special fiber of $\Yrst$, while in  \S\ref{sec structure abelian rank} we will discuss the abelian rank of its irreducible components.

\subsection{The toric rank}
\label{sec structure toric rank}

In this subsection, after introducing the notions of a \emph{viable} cluster and an \emph{\"{u}bereven} cluster, we will prove the following theorem that allows to compute the toric rank of $\SF{\Yrst}$, and hence, by \Cref{prop invariance of ranks}, the toric rank of the special fiber of any semistable model of $Y$ defined over any extension of $R$.
\begin{thm}
    \label{thm toric rank}
    The toric rank of $\SF{\Yrst}$ 
    equals the number of non-\"{u}bereven viable clusters.
\end{thm}

Let us begin by defining viable clusters.
\begin{dfn} \label{dfn viable}
    We say that a cluster $\mathfrak{s}$ is \emph{viable} if the following are satisfied:
    \begin{enumerate}[(a)]
        \item $\mathfrak{s}$ has even cardinality; and
        \item there exist $2$ distinct valid discs linked to $\mathfrak{s}$.
    \end{enumerate}
\end{dfn}

\begin{rmk}
    In the above definition, the results presented in \S\ref{sec valid discs definition} show that in the $p \neq 2$ setting, (a) implies (b) (see \Cref{rmk cluster p=2}), while, in the $p=2$ setting, (b) implies (a) (see \Cref{thm cluster p=2}(a)).
\end{rmk}

\begin{prop} \label{prop viable correspondence}
    Viable clusters are in one-to-one correspondence with the nodes of $\SF{\Xrst}$ over which the cover $\SF{\Yrst} \to \SF{\Xrst}$ is unramified (i.e., the nodes of  $\SF{\Xrst}$ that have two distinct inverse images in $\SF{\Yrst}$).
    \begin{proof}
        Suppose that $\mathfrak{s}$ is a viable cluster, and let $D_+\subsetneq D_-$ be the two valid discs linked to it. It follows from \Cref{lemma not a valid disc} that we have $\XX_D\not\leq \Xrst$ for all discs $D$ satisfying $D_+\subsetneq D\subsetneq D_-$; hence, the two lines $L_+$ and $L_-$ of $\SF{\Xrst}$ corresponding to the discs $D_+$ and $D_-$ intersect at a node $P\in \SF{\Xrst}$. We know by \Cref{prop lambda plus minus and ell}(a) that $\SF{\YY_{D_\pm}}$ has two branches above $P\in \SF{\XX_{D_\pm}}$, which implies that $\SF{\Yrst}\to \SF{\Xrst}$ is unramified above $P$.
    
        Let us now prove the converse implication. Let $P$ be a node of $\SF{\Xrst}$ above which $\SF{\Yrst}$ is unramified; let $L_-$ and $L_+$ the two lines of $\SF{\Xrst}$ passing through $P$; and let $D_\pm$ be the corresponding discs. Since the cover $\SF{\Yrst} \to \SF{\Xrst}$ is unramified above $P$, no element of $\Rinfty$ reduces to $P\in \SF{\Xrst}$, which is equivalent to saying that no element of $\Rinfty$ reduces to the unique node of $\SF{\XX_{\{D_+,D_-\}}}$. In particular, $\infty$ lies on one and only one of the two lines $L_+$ and $L_-$ comprising the special fiber $\SF{\XX_{\{D_+,D_-\}}}$, say $\infty\in L_-\setminus L_+$; this implies, in particular, that we have $D_+\subsetneq D_-$ by \Cref{prop relative position smooth models line}. We can now write the decomposition $\mathcal{R}=\mathfrak{s}\sqcup (\mathcal{R}\setminus \mathfrak{s})$, where $\mathfrak{s}$ (resp.\ $\RR \smallsetminus \mathfrak{s}$) consists of the roots whose reductions in $\SF{\XX_{\{D_+,D_-\}}}$ lie on $L_+\smallsetminus L_-$ (resp.\ $L_-\smallsetminus L_+$). It is now clear that $\mathfrak{s}=D_+\cap \mathcal{R}$ is a cluster, to which the two distinct valid discs $D_-$ and $D_+$ are linked. Since $\SF{\Yrst}\to \SF{\Xrst}$ is unramified above $P$, we have that $\SF{\YY_{D_-}}$ (resp.\ $\SF{\YY_{D_+}}$) has two branches above $0\in \SF{\YY_{D_-}}$ (resp.\ $\infty\in \SF{\YY_{D_+}}$), hence, by \Cref{prop lambda plus minus and ell}(b), the cluster $\mathfrak{s}$ must have even cardinality.
    \end{proof}
\end{prop}

\begin{prop} \label{prop viable thicknesses}
    If $\mathfrak{s}$ is a viable cluster corresponding to a node $P \in \SF{\Xrst}$ as in \Cref{prop viable correspondence}, then the thickness of each of the $2$ nodes lying above $P$ is equal to $(b_+(\mathfrak{s}) - b_-(\mathfrak{s})) / v(\pi)$ (with the notation of \S\ref{sec depths construction valid discs}).
\end{prop}

\begin{proof}
    This is straightforward from applying Propositions \ref{prop thickness} and \ref{prop vanishing persistent}(b) to \Cref{thm summary depths valid discs}.
\end{proof}

We now give the other main definition of this section.

\begin{dfn}
    An cluster $\mathfrak{s}$ is said to be \emph{\"{u}bereven} if it is viable and if all of its children clusters are also viable.
\end{dfn}

\begin{rmk}
    In the $p\neq 2$ setting, every even-cardinality cluster is viable, and so an \"{u}bereven cluster is just a cluster whose children are all even; this is the definition of ``\"{u}bereven" used in \cite{dokchitser2022arithmetic}.
\end{rmk}

\begin{lemma}
    \label{lemma all but one}
    Let $\mathfrak{s}$ be a cluster, and let $\mathfrak{c}_1, \ldots, \mathfrak{c}_N$ be its children.  If each child $\mathfrak{c}_i$ is viable, then we have  $b_+(\mathfrak{s})=d_+(\mathfrak{s})$ (with the notation of \S\ref{sec depths construction valid discs}).
    \begin{proof}
        Since $\mathfrak{c}_i$ is viable, we have $\delta(\mathfrak{c}_i)> B_{f,\mathfrak{s}}\ge b_0(\mathfrak{t}_+^{\mathfrak{c}_i})$, which implies that $\mathfrak{t}_+^{\mathfrak{c}_i}(\delta(\mathfrak{c}_i)) = 2v(2)$.  Now \Cref{prop deep ubereven cluster}(a) says that we have $\mathfrak{t}_+^{\mathfrak{s}}(0) = 2v(2)$, which directly implies that $b_+(\mathfrak{s}) = d_+(\mathfrak{s})$.
    \end{proof}
\end{lemma}

\begin{prop} \label{prop ubereven correspondence}
    The assignement $\mathfrak{s}\mapsto D_{\mathfrak{s},d_+(\mathfrak{s})}$ induces a one-to-one correspondence between the \"{u}bereven clusters and the valid discs $D$ such that the special fiber $\SF{\YY_D}$ is reducible (i.e., $\SF{\YY_D}$ consists of $2$ rational components).
    \begin{proof}
        Suppose first that $D$ is a valid disc such that $\SF{\YY_D}$ is reducible. Then, we know by \Cref{thm part of rst separable} that the elements of $\mathcal{R}\cup\lbrace \infty\rbrace$ reduce to $N\ge 3$ distinct points of $\SF{\XX_D}$; we consequently have $D=D_{\mathfrak{s},d_+(\mathfrak{s})}$ for some cluster $\mathfrak{s}$ and that the $N$ points are $\overline{\mathfrak{c}_1}, \ldots, \overline{\mathfrak{c}_{N-1}}, \infty=\overline{\mathcal{R}\setminus \mathfrak{s}}$, where $\mathfrak{c}_1, \ldots, \mathfrak{c}_{N-1}$ are the child clusters of $\mathfrak{s}$ (see \Cref{lemma discs linked to clusters}). From the fact that $D$ is a valid disc we deduce that $d_+(\mathfrak{s})=d_-(\mathfrak{c}_i)$ is a common endpoint of the intervals $J(\mathfrak{s})$ and $J(\mathfrak{c}_i)$ for all $i$; in particular, these intervals are non-empty and we have $b_+(\mathfrak{s})=d_+(\mathfrak{s})=d_-(\mathfrak{c}_i)=b_-(\mathfrak{c}_i)$. The fact that $\SF{\YY_D}$ consists of $2$ components means that we have $\ell(\XX_D,\overline{\mathfrak{c}_i})=0$ for all $i$ and $\ell(\XX_D,\infty)=0$, but, according to \Cref{prop lambda plus minus and ell}(c),(d), this implies that $b_-(\mathfrak{c}_i)<b_+(\mathfrak{c}_i)$ for all $i$ as well as $b_-(\mathfrak{s})<b_+(\mathfrak{s})$. Now, applying \Cref{prop lambda plus minus and ell}(b), we also deduce that $\mathfrak{s}$ as well as the $\mathfrak{c}_i$'s must all have even cardinality. We conclude that $\mathfrak{s}$ is a \"{u}bereven cluster.
        
        Let us now prove the converse implication. Assume that $\mathfrak{s}$ is a \"{u}bereven cluster, and let us denote by $\mathfrak{c}_1, \ldots, \mathfrak{c}_{N-1}$ its children (with $N\ge 3$); we remark that $d_+(\mathfrak{s})=d_-(\mathfrak{c}_i)$ for all $i$. Letting $D=D_{\mathfrak{s},d_+({\mathfrak{s}})}$, we have that the elements of $\Rinfty$ reduce in $\SF{\XX_D}$ to the $N$ distinct points, $\overline{\mathfrak{c}_1}, \ldots, \overline{\mathfrak{c}_{N-1}}, \infty=\overline{\mathcal{R}\setminus \mathfrak{s}}$. Now, since all of the $\mathfrak{c}_i$'s are viable, we have $b_+(\mathfrak{s})=d_+(\mathfrak{s})$ by \Cref{lemma all but one}; meanwhile, since $\mathfrak{s}$ is also assumed to be viable, we have $b_-(\mathfrak{s})<b_+(\mathfrak{s})$. We conclude, in particular, that the disc $D$ is valid (see \Cref{thm summary depths valid discs}).
        
        As before, from the fact that $D=D_{\mathfrak{s},d_+(\mathfrak{s})}$ is a valid disc, we deduce that $b_+(\mathfrak{s})=d_+(\mathfrak{s})=b_-(\mathfrak{c}_i)=d_-(\mathfrak{c}_i)$. Since both $\mathfrak{s}$ and the $\mathfrak{c}_i$'s are viable, they have even cardinality and satisfy $b_-(\mathfrak{c}_i)<b_+(\mathfrak{c}_i)$ and $b_-(\mathfrak{s})<b_+(\mathfrak{s})$, hence, by \Cref{prop lambda plus minus and ell}(a) we have that $\ell(\XX_D,\overline{\mathfrak{c}_i})=0$ for all $i$, and $\ell(\XX_D,\infty)=0$. But this means that $\NSF{\YY_D}\to \SF{\XX_D}$ is an étale double cover of the line, i.e.\ that $\SF{\YY_D}$ is reducible, as we discussed in \S\ref{sec models hyperelliptic separable}.
    \end{proof}
\end{prop}

\begin{proof}[Proof of \Cref{thm toric rank}]
    The toric rank of a semistable $k$-curve is just the number of nodes (which we denote by $\Nnodes$) minus the number of irreducible components (which we denote by $\Nirreducible$) plus 1 (see \S\ref{sec preliminaries abelian toric unipotent}). Now, since we have $g(X)=0$, the toric rank of $\SF{\Xrst}$ is $0$ (as is the toric rank of the special fiber of any model of the line). Hence, the toric rank of $\SF{\Yrst}$ can be computed as
    \begin{equation*}
       [\Nnodes(\SF{\Yrst})-\Nnodes(\SF{\Xrst})]-[\Nirreducible(\SF{\Yrst})-\Nirreducible(\SF{\Xrst})].
    \end{equation*}
    Now it follows from \Cref{prop viable correspondence} that the first difference equals the number of viable clusters; meanwhile, \Cref{prop ubereven correspondence} implies the second difference equals the number of \"{u}bereven clusters.
\end{proof}

\subsection{The abelian rank}
\label{sec structure abelian rank}
In this subsection we show how to compute the abelian rank of $\SF{\YY_D}$ for any valid disc $D$, provided that, for each cluster $\mathfrak{s}$ linked to $D$, we are able to compute the invariants $b_\pm(\mathfrak{s})$ and $\lambda_{\pm}(\mathfrak{s})$ introduced in \S\ref{sec depths construction valid discs}.

\begin{prop} \label{prop genus of components}
    Let $D:=D_{\alpha,b}$ be a valid disc. In the $p\neq 2$ setting, then the genus of $\NSF{\YY_D}$ equals $-1+N_{\mathrm{odd}}/2$, where $N_{\mathrm{odd}}$ is the number of odd-cardinality clusters to which $D$ is linked.  In the $p=2$ setting, instead we have the following.
    \begin{enumerate}[(a)]
        \item If $D$ is linked to no cluster, then $g(\NSF{\YY_D})=-1+(1+\lambda_-(\varnothing,\alpha))/2$.
        \item If $D$ is linked to a unique cluster $\mathfrak{s}$, then one of the three possibilities below holds:
    \begin{enumerate}[(i)]
        \item $b=b_-(\mathfrak{s})<b_+(\mathfrak{s})$, in which case $g(\NSF{\YY_D})=-1+(1+\lambda_-(\mathfrak{s}))/2$;
        \item $b=b_+(\mathfrak{s})>b_-(\mathfrak{s})$, in which case $g(\NSF{\YY_D})=-1+(1+\lambda_+(\mathfrak{s}))/2$; or 
        \item $b=b_+(\mathfrak{s})=b_-(\mathfrak{s})$, in which case $g(\NSF{\YY_D})=-1+(1+\lambda_-(\mathfrak{s}))/2 + (1+\lambda_+(\mathfrak{s}))/2$.
    \end{enumerate}
    \item If $D$ is linked to $N \ge 3$ clusters, i.e.\ to a cluster $\mathfrak{s}_N$ and all of its children $\mathfrak{s}_1, \ldots, \mathfrak{s}_{N-1}$, then 
    \begin{equation*}g(\NSF{\YY_D})=-1+\epsilon_N(1+\lambda_-(\mathfrak{s}_N))/2 + \sum_{i=1}^{N-1} \epsilon_i(1+\lambda_+(\mathfrak{s}_i))/2,\end{equation*}
    where $\epsilon_i\in \{0,1\},$ and $\epsilon_i$ is 1 (resp.\ 0) when $\mathfrak{s}_i$ is not viable (resp.\ viable), which in this setting is equivalent to the condition $b_-(\mathfrak{s}_i)=b_+(\mathfrak{s}_i)$ (resp.\ $b_-(\mathfrak{s}_i)<b_+(\mathfrak{s}_i)$).
    \end{enumerate}
    \begin{proof}
        The separable cover $\NSF{\YY_D}\to \SF{\XX_D}$ is only ramified above the $N$ points $P_1, \ldots, P_N=\infty$ of $\SF{\XX_D}$ to which the elements of $\Rinfty$ reduce, and we will apply the formula given in \Cref{prop riemann hurwitz} to compute the genus of $\NSF{\YY_D}$ based on the index $\ell(\XX_D,P_i)$ defined there.
        
        If $N=1$ (which can happen only if $p=2$), then $D$ is linked to no cluster; \Cref{thm summary depths valid discs} then ensures that $b=b_-(\varnothing,\alpha)$, and  $\ell(\XX_D,P_1)=\ell(\XX_D,\infty)=\lambda_-(\varnothing,\alpha)+1$ by \Cref{prop lambda plus minus and ell}(c).

        If $N=2$ (which can happen only if $p=2$), then $\mathfrak{s}=D\cap \mathcal{R}$ is the unique cluster to which $D$ is linked, and this implies that $b$ is an internal point of $I(\mathfrak{s})$, i.e.\ $d_-(\mathfrak{s})<b<d_+(\mathfrak{s})$; the $2$ points of $\SF{\XX_D}$ to which the elements of $\Rinfty$ reduce are $P_1=0$ and $P_2=\infty$.  If we have $b=b_-(\mathfrak{s})=b_+(\mathfrak{s})$, then parts (c) and (d) of \Cref{prop lambda plus minus and ell} give $\ell(\XX_D,P_1)=1+\lambda_+(\mathfrak{s})$ and $\ell(\XX_D,P_2)=1+\lambda_-(\mathfrak{s})$ respectively.   We now assume that $b_-(\mathfrak{s}) < b_+(\mathfrak{s})$.  If we have $b = b_-(\mathfrak{s}) < b_+(\mathfrak{s})$, then parts (a) and (c) of \Cref{prop lambda plus minus and ell} give $\ell(\XX,P_1)=0$ and $\ell(\XX,P_2)=1+\lambda_-(\mathfrak{s})$ respectively, while, if instead we have $b = b_+(\mathfrak{s}) > b_-(\mathfrak{s})$, points (a) and (d) of \Cref{prop lambda plus minus and ell} give  $\ell(\XX,P_1)=1+\lambda_+(\mathfrak{s})$ and $\ell(\XX,P_2)=0$ respectively.
        
        Now suppose that we have $N\ge 3$; this means that $D$ is linked to a cluster $\mathfrak{s}_N$ and to all of its children $\mathfrak{s}_1, \ldots, \mathfrak{s}_{N-1}$, which must all have even cardinality if $p=2$. We have $b=d_+(\mathfrak{s}_N)=d_-(\mathfrak{s}_i)$ for all $i=1, \ldots, N-1$; by \Cref{thm summary depths valid discs}(a), we moreover know that, since $D$ is a valid disc, we must have $b=b_+(\mathfrak{s}_N)=b_-(\mathfrak{s}_i)$; this also forces $b_-(\mathfrak{s}_i)\le b_+(\mathfrak{s}_i)$ for $1 \leq i \leq N$. Now choose any $i \in \{1, \ldots, N-1\}$.  If we have $b_+(\mathfrak{s}_i)=b_-(\mathfrak{s}_i)$ (which can only occur if $p=2$), then we obtain from \Cref{prop lambda plus minus and ell}(d) that $\ell(\XX_{D},P_i)=1+\lambda_+(\mathfrak{s}_i)$, where $P_i$ is the point to which the roots of $\mathfrak{s}_i$ reduce.  If on the other hand we have $b_+(\mathfrak{s}_i)>b_-(\mathfrak{s}_i)$, then we obtain from \Cref{prop lambda plus minus and ell}(a),(b) that $\ell(\XX_D,P_i) = 0$ (resp.\ $\ell(\XX_D,P_i) = 1$) if the cardinality $|\mathfrak{s}_i|$ is even (resp.\ odd) (it is always even if $p=2$).  Similarly, using  \Cref{prop lambda plus minus and ell}(a),(b),(d), we obtain that $\ell(\XX_D,P_N)=1+\lambda_-(\mathfrak{s})$ when $b_+(\mathfrak{s}_N)=b_-(\mathfrak{s}_N)$, while, if  $b_+(\mathfrak{s}_N)>b_-(\mathfrak{s}_N)$, we have $\ell(\XX_D,P_N) = 0$ (resp.\ $\ell(\XX_D,P_N) = 1$) if the cardinality $|\mathfrak{s}_N|$ is even (resp.\ odd).
        
        Now the claimed formulas for $g(\widetilde{\SF{\YY_D}})$ follows directly from applying \Cref{prop riemann hurwitz}.
    \end{proof}
\end{prop}

\begin{rmk}
    The genus $g(\NSF{\YY_D})$ coincides with the abelian rank of $\SF{\YY_D}$, unless $\SF{\YY_D}$ consists of $2$ components, in which case we have $g(\NSF{\YY_D})=-1$ while the abelian rank of $\SF{\YY_D}$ is 0.
\end{rmk}

\begin{cor}
    \label{cor crushed positive abelian rank}
    Suppose that $D$ is a valid disc that is linked to no cluster, or that it is linked to only one cluster $\mathfrak{s}$ and is the unique valid disc linked to $\mathfrak{s}$.  Then the $k$-curve $\SF{\YY_D}$ is irreducible and has abelian rank $\ge 1$.
    \begin{proof}
        This follows immediately from \Cref{prop genus of components}(a) and \Cref{prop genus of components}(b)(iii), taking into account that, when $D$ is linked to no cluster, we have $\lambda_-(\varnothing,\alpha)\ge 3$ by \Cref{thm summary depths valid discs}.
    \end{proof}
\end{cor}

\subsection{Partitioning the components of the special fiber} \label{sec structure partitioning}

Suppose that in the $p = 2$ setting we are given a disc $D:=D_{\alpha,b}$ with $\alpha\in \bar{K}$ and $b\in \qq$ and that the cover $\SF{\YY_D}\to \SF{\XX_D}$ is inseparable, i.e.\ the disc $D$ satisfies $\tbest{\RR}{D}<2v(2)$, and let $\beta\in \bar{K}^\times$ be such that $v(\beta)=b$.  We have seen in \S\ref{sec models hyperelliptic inseparable} that, in this case, given a part-square decomposition $f=q^2+\rho$ that is good at the disc $D$, the special fiber $\SF{\YY_D}$ is described by an equation of the form $y^2=\overline{\rho_0}(x_{\alpha,\beta})$, where $\overline{\rho_0}$ is a normalized reduction of $\rho_{\alpha,\beta}$. Let us also recall \Cref{dfn mu} which says that, given $P\in \SF{\XX_D}$, we denote by $\mu(\XX_D,P)$ the order of vanishing of $\overline{\rho_0}'$ at $P$, which is an even integer.

Letting $\Drst$ denote the collection of discs corresponding to $\Xrst$ (in the sense of \S\ref{sec models hyperelliptic line}), for any $P \in \SF{\XX_D}$ we write $\mathfrak{D}_P\subseteq \Drst$ for the non-empty subset consisting of those $D'\in \Drst$ such that $\Ctr(\XX_D, \XX_{D'})=\{P\}$.
\begin{prop} \label{prop inseparable tfae}
    In the setting above, the following are equivalent:
    \begin{enumerate}[(a)]
        \item $\mu(\XX_D,P)>0$;
        \item $\SF{\YY_D}$ is singular above $P$;
        \item $\mathfrak{D}_P\neq \varnothing$;
        \item $\mathfrak{D}_P$ contains a valid disc.
    \end{enumerate}
    \begin{proof}
        The equivalence between (a) and (b) was already discussed in \S\ref{sec models hyperelliptic inseparable}, while the equivalence between (a)/(b) and (c) is an immediate consequence of \Cref{cor rst isomorphism failure}. Finally, in light of \Cref{prop evanescence inseparable components}, it is easy to see that $\mathcal{D}_P$ contains a valid disc whenever it is non-empty.
    \end{proof}
\end{prop}

As in \S\ref{sec models hyperelliptic inseparable}, we let $R_{\mathrm{sing}}$ denote the set of points of $\SF{\XX_D}$ satisfying the equivalent conditions above.  Given any $P\in R_{\mathrm{sing}}$, we write $\Xrst_P$ for the model of the line $X$ corresponding to $\mathfrak{D}_P$ and let $\Yrst_P$ be the corresponding model of $Y$.  We now present the main result of this subsection.

\begin{prop}
    \label{prop partition rank}
    In the setting above, the model $\Yrst_P$ satisfies the following properties:
    \begin{enumerate}[(a)]
        \item the strict transform $C_P$ of $\SF{\Yrst_P}$ in $\SF{\Yrst}$ intersects the rest of $\SF{\Yrst}$ at a single node (when it does not coincide with the whole special fiber $\SF{\Yrst}$); and 
        \item the arithmetic genus of the $k$-curve $C_P$ (which is to say, the sum of the abelian and toric rank of the special fiber of $\Yrst_P$) is equal to $\frac{1}{2}\mu(\XX_D,P)$.
    \end{enumerate}
\end{prop}

\begin{rmk}
    \label{rmk separating tame case}
    An analogous result holds in the $p\neq 2$ setting if $D$ is taken to be any disc and $P\in \SF{\XX_D}(k)$ is a point over which $\SF{\YY_D}$ exhibits a unibranch singularity (i.e.\ $P\in R_1$ in the language of \S\ref{sec models hyperelliptic separable}): in this situation the invariant $\mu(\XX_D,P)$ is given by $N_P-1$, where $N_P$ is the (necessarily odd) number of roots of $\Rinfty$ reducing to $P$.  The proof is analogous to (and in some aspects simpler than) that of \Cref{prop partition rank}.
\end{rmk}

Before presenting the proof, let us introduce the following two lemmas.
\begin{lemma}
    \label{lemma bifurcation valid discs}
    Choose $\alpha\in \bar{K}$, and suppose that, for some rational number $b\in \qq$, we have $\tbest{\RR}{D_{\alpha,b}}<2v(2)$. Let us fix a part-square decomposition $f=q^2+\rho$ that is totally odd with respect to a center $\alpha \in \bar{K}$. Then the function $c\mapsto \vfun_{\rho}(D_{\alpha,c})$ is not differentiable at the input $c=b$ if and only if there is a valid disc $D_{\alpha',b'}$ such that $v(\alpha' - \alpha) = b$ and $b' > b$.
    \begin{proof}
        By \Cref{lemma mathfrakh}(b) the function $c\mapsto \vfun_{\rho}(D_{\alpha,c})$ is differentiable at $c=b$ if and only if some (any) normalized reduction of $(\rho_{\alpha,\beta})'$, for $\beta$ such that $v(\beta)=b$, has a root $P\in \SF{\XX_{D_{\alpha,\beta}}}$ which is neither $\overline{x_{\alpha,b}}=0$ nor $\overline{x_{\alpha,\beta}}=\infty$. By applying \Cref{prop inseparable tfae} to the disc $D_{\alpha,b}$ (taking into account \Cref{prop relative position smooth models line}), this is equivalent to saying that there exists a valid disc $D_{\alpha',b'}$ such that $v(\alpha' - \alpha) = b$ and $b' > b$. 
    \end{proof}
\end{lemma}

\begin{rmk} \label{rmk bifurcation valid discs}
    It is clear from \Cref{lemma mathfrakh} and definitions of the functions involved that the non-differentiability condition in the statement of the above lemma is satisfied whenever there is an input $c = b$ which is not equal to the depth of any cluster containing $\alpha$ and at which the function $c \mapsto \mathfrak{t}^\RR(D_{\alpha,c})$ is not differentiable.
\end{rmk}

\begin{lemma}
    \label{lemma merging of points}
    Let $D:=D_{\alpha,b}$ be a disc such that $\tbest{\RR}{D}<2v(2)$, and let $D^\varepsilon = D_{\alpha,b-\varepsilon}$ for some $\varepsilon>0$. Let $r: \SF{\XX_{D}}(k) \to \SF{\XX_{D^\varepsilon}}(k)$ be the map taking all points $P \in \SF{\XX_D}(k) \smallsetminus \{\infty\}$ to $0$ and taking $\infty \in \SF{\XX_D}(k)$ to $\infty \in \SF{\XX_{D^\varepsilon}}(k)$. Then if $\varepsilon$ is sufficiently small, we have $\mu(\XX_{D^\varepsilon},Q)=\sum_{P\in r^{-1}(Q)}\mu(\XX_D,P)$ for all $Q\in \SF{\XX_{D^\varepsilon}}(k)$, and moreover, we have the inclusions 
    \begin{equation} \label{eq merging of points}
        \bigsqcup_{P\in r^{-1}(Q)}\mathfrak{D}_P\subseteq \mathfrak{D}_Q\subseteq \big(\bigsqcup_{P\in r^{-1}(Q)}\mathfrak{D}_P \big)\cup \{D\}.
    \end{equation}
    \begin{proof}
        It is easy to see that both claims of the lemma hold when $\varepsilon$ is chosen such that there is no disc $D_{\alpha',b'} \in \Drst$ with $b - \epsilon < v(\alpha' - \alpha) < b$; the second (resp.\ first) inequality in (\ref{eq merging of points}) is an equality if $D$ is (resp.\ is not) a disc in $\Drst$.
    \end{proof}
\end{lemma}

\begin{proof}[Proof (of \Cref{prop partition rank})]
    Let $D, P, q, \rho, \Drst$ and $\mathfrak{D}_P$ be as specified at the beginning of this subsection; in particular, since $\tbest{\RR}{D}<2v(2)$, we are in the $p=2$ setting.  Let us write $D = D_{\alpha,b}$ for some $\alpha\in \bar{K}$ and $b\in \qq$, and let $D_\varepsilon = D_{\alpha,b+\varepsilon}$ for $\varepsilon$ arbitrarily small. In this proof, we will always assume for simplicity that $\overline{x_{\alpha,\beta}}(P)=0$ (which is certainly the case for an appropriate choice of $\alpha$); according to \Cref{prop relative position smooth models line} this means that, given a disc $D'$, we have $D'\in \mathfrak{D}_P$ if and only if we have $D'\in \Drst$ and $D'\subseteq D_\varepsilon$.
    
    Let us first address part (a). Let $P'$ denote the node of $\SF{\Xrst}$ at which the strict transform of $\SF{\Xrst_P}$ intersects the rest of $\SF{\Xrst}$, and let $D_1\in \mathfrak{D}_P$ and $D_2\in \Drst\setminus \mathfrak{D}_P$ be the discs corresponding to the two lines of $\SF{\Xrst}$ meeting at $P'$. Suppose by way of contradiction that $P'$ has two distinct inverse images in $\SF{\Yrst}$; then \Cref{prop viable correspondence} ensures that $D_1$ and $D_2$ are the $2$ valid discs linked to some viable cluster $\mathfrak{s}$ and so in particular are not disjoint. From this, since $D_1 \subseteq D_\varepsilon$ but $D_2 \not\subseteq D_\varepsilon$, it follows that we have $D_1\subseteq D_\varepsilon\subseteq D_2$ for all small enough $\varepsilon$. This means that $D$ itself is linked to $\mathfrak{s}$ and its depth $b$ lies in the interval $J(\mathfrak{s})$, and hence we have $\tbest{\RR}{D}=2v(2)$, which contradicts our assumption that $\SF{\YY_D}\to \SF{\XX_D}$ is inseparable. We conclude that $P'$ must have a single inverse image in $\SF{\Yrst}$, and (a) is proved.

    Let us now address part (b).  We write $\mu(\XX_D,P)=2\nu$ where $\nu$ is a positive integer, and we let $D'_1, \ldots, D'_h$ be the maximal valid discs that are contained in $D_\varepsilon$; for each $i$, let us moreover choose $\alpha'_i \in \bar{K}$ such that $D'_i=D_{\alpha'_i,b'_i}$ with $b_i \in \qq$.  Since the discs $D'_i$ are valid, we have $\tbest{\RR}{D'_i}=2v(2)$; thanks to the maximality assumption, we can be sure, in light of \Cref{thm summary depths valid discs}, that we have $\tbest{\RR}{D_{\alpha'_i,c}} < 2v(2)$ for all $c\in [b,b'_i)$.
    
    We proceed by induction on $n := \max_{1 \leq i \leq h}(|D'_i\cap \RR|)$, beginning by proving the result for $n=0$. We preliminarily observe that, for a fixed $n \geq 0$, it is enough to address the case where $h=1$. Indeed, one can verify, by repeatedly applying \Cref{lemma merging of points}, that the result of (b) is true for the disc $D$ and the point $P\in \SF{\XX_D}$ if it is true for the discs $D_{\alpha_i',b_i'-\varepsilon}$ at the points $\overline{\alpha_i'} \in \SF{\XX_{D_{\alpha_i',b_i'-\varepsilon}}}(k)$.
    
    Assume that $h=1$; for simplicity of notation we write $D' := D_{\alpha',b'}$ for $D'_1$. We moreover write $\mathfrak{s}'=D'\cap \RR$ and choose a part-square decomposition $f=\tilde{q}^2+\tilde{\rho}$ that is totally odd with respect to the center $\alpha'$. Since we have $\tbest{\RR}{D_{\alpha',c}} <2v(2)$ for all $c\in [b,b')$ and $D'$ is a valid disc, we have $b'=b_-(\mathfrak{s}',\alpha')$. Let us moreover remark that, by \Cref{lemma mathfrakh}, the right derivative of $c\mapsto \vfun_{\tilde{\rho}}(D_{\alpha',c})$ at $c=b$ is equal to $2\nu+1$ and, since $D'$ is the maximum among the valid discs contained in $D_\varepsilon$, the slope of the function $c \mapsto \vfun_{\tilde{\rho}}(D_{\alpha',c})$ actually remains equal to $2\nu+1$ for all $c\in [b,b']$ by \Cref{lemma bifurcation valid discs}.  In particular, by applying \Cref{lemma mathfrakh} to $c \mapsto \mathfrak{t}^\RR(D_{\alpha',c}) = \vfun_{\tilde{\rho}}(c) - \vfun_{f_{\alpha',1}}(c)$, we can be sure that $\lambda_-(\mathfrak{s}',\alpha')=2\nu+1-n$, recalling that $n=|\mathfrak{s}'|$ in this situation.
    
    We now set out to prove the case $n=0$ and $h=1$. In this case, the only valid disc contained in $D_\varepsilon$ is $D'$, and we have $\mathfrak{D}_P=\{ D'\}$ and that $\SF{\Yrst_P}$ has toric rank $0$ and abelian rank $-1+(\lambda_-(\varnothing,\alpha')+1)/2=\nu$ by \Cref{prop genus of components}(a), as we wanted.
    
    Now assume that $n \geq 2$ and $h=1$, and assume inductively that the conclusion of part (b) holds for all lesser values of $n$ for all $h$.  We can clearly write $\mathfrak{s}'=\mathfrak{s}_1\sqcup \ldots \sqcup\mathfrak{s}_r$, where the $\mathfrak{s}_i$'s are the (even-cardinality) clusters contained in $\mathfrak{s}'$ such that $b_-(\mathfrak{s}_i)\le b_+(\mathfrak{s}_i)<d_+(\mathfrak{s}_i)$ and such that we have $\mathfrak{t}^\RR(D_{\mathfrak{s}_i,c}) = 2v(2)$ for $c\in [b,b_+(\mathfrak{s}_i)]$ and $\mathfrak{t}^\RR(D_{\mathfrak{s}_i,c})<2v(2)$ for $c\in (b_+(\mathfrak{s}_i),d_+(\mathfrak{s}_i)]$. For each $i$, we write the following: let $f^2=q_i^2+\rho_i$ be a part-square decomposition that is totally odd with respect to $\alpha_i$; let $\nu_i$ be the integer such that the right derivative of $c\mapsto \vfun_{\rho_i}({D_{\mathfrak{s}_i,c}})$ at $c=b_+(\mathfrak{s}_i)$ equals $2\nu_i + 1$; and denote the disc $D_{\mathfrak{s}_i,b_+(\mathfrak{s}_i)+\varepsilon}$ by $D_i$.  Note that we have the formula $\lambda_+(\mathfrak{s}_i)=|\mathfrak{s}_i|-(2\nu_i+1)$.
    
    Let $\mathcal{J} \subseteq \mathcal{I}:= \{1, \ldots , r\}$ be the subset of indices such that $d_-(\mathfrak{s}_i) = b_-(\mathfrak{s}_i) = b_+(\mathfrak{s}_i)$, i.e.\ the indices for which $\mathfrak{s}_i$ is not viable; we define the partition $\mathcal{J} = \mathcal{J}_1 \sqcup \ldots \sqcup \mathcal{J}_s$ so that each $\mathcal{J}_j \subseteq \mathcal{J}$ is a maximal subset of indices corresponding to sibling clusters. We classify the valid discs contained in $D$ as follows. 
    \begin{enumerate}[(I)]
        \item For each $i\in \mathcal{I}$, we have the valid discs contained in $D_i$; by the inductive hypothesis, these give a total contribution of $\sum_{1 \leq i \leq r} \nu_i$ to the sum of the abelian and toric ranks of $\SF{\Yrst_P}$.
        \item We have the discs $D_{\mathfrak{s}_i,b_+(\mathfrak{s}_i)}$ for all $i\in \mathcal{I}\setminus \mathcal{J}$, along with the disc $D'$ when $\mathfrak{s}'$ is viable.  Now we observe that 
        \begin{enumerate}[(a)]
            \item each of the discs $D_{\mathfrak{s}_i,b_+(\mathfrak{s}_i)}$ contributes a component of $\SF{\Yrst_P}$ of abelian rank equal to $-1 + (1 + \lambda_+(\mathfrak{s}_i))/2 = -1+\frac{1}{2}|\mathfrak{s}_i| - \nu_i$ by \Cref{prop genus of components}(b)(ii); and 
            \item when $\mathfrak{s}'$ is viable, the disc $D'$ contributes a component of $\SF{\Yrst_P}$ of abelian rank equal to $-1 + (1 + \lambda_-(\mathfrak{s}'))/2 = \nu - \frac{1}{2}|\mathfrak{s}'|$ by \Cref{prop genus of components}(b)(i).
        \end{enumerate}
        \item We have the valid discs $D_{\mathfrak{s}_i,d_-(\mathfrak{s}_i)}$ for $i \in \mathcal{J}$, which are precisely the (distinct) discs $D_{\mathfrak{r}_j,d_+(\mathfrak{r}_j)}$ for $1 \leq j \leq s$ where each $\mathfrak{r}_j$ is the common parent of the clusters $\mathfrak{s}_i$ with $i \in \mathcal{J}_j$ and $\tilde{\alpha}_j = \alpha_i$ for a choice of $i \in \mathcal{J}_j$.  The abelian rank of $\SF{\YY_{D_{\mathfrak{r}_j,d_+(\mathfrak{r}_j)}}}$ can be computed using \Cref{prop genus of components}(c) as follows:
        \begin{enumerate}
            \item if $\mathfrak{r}_j\subsetneq \mathfrak{s}'$ or if $\mathfrak{r}_j=\mathfrak{s}'$ and $\mathfrak{s}'$ is viable, it is equal to 
            \begin{equation*}
                -1 + \sum_{i \in \mathcal{J}_j} \frac{1}{2}(1 + \lambda_+(\mathfrak{s}_i)) = -1 + \sum_{i \in \mathcal{J}_j} (\frac{1}{2}|\mathfrak{s}_i| - \nu_i); \ \mathrm{and}
            \end{equation*}
            \item if $\mathfrak{r}_j=\mathfrak{s}'$ and $\mathfrak{s}'$ is not viable, then it is equal to 
            \begin{equation*}
                -1 + \frac{1}{2}(1+\lambda_-(\mathfrak{s}')) + \sum_{i \in \mathcal{J}_j} \frac{1}{2}(1 + \lambda_+(\mathfrak{s}_i))) = \sum_{i \in \mathcal{J}_j} (\frac{1}{2}|\mathfrak{s}_i| - \nu_i) + (\nu-\frac{1}{2}|\mathfrak{s}'|).
            \end{equation*}
        \end{enumerate}
        Moreover, if $\mathfrak{s}'$ is not viable but there is no index $j \in \{1,\ldots, s\}$ such that $\mathfrak{r}_j=\mathfrak{s}'$, we include $D'$ in the subset of discs of type III(b); the disc $D'$ contributes $\nu-\frac{1}{2}|\mathfrak{s}'|$ to the abelian rank by \Cref{prop genus of components}(c).
        
        \item All other valid discs contained in $D$ are of the form $D_{\mathfrak{r},d_+(\mathfrak{r})}$ for $\mathfrak{r}\subseteq \mathfrak{s}$ a \"{u}bereven cluster containing one of the clusters $\mathfrak{s}_i$; these each contribute $2$ lines to $\SF{\Yrst_P}$ and thus do not increase the abelian rank.
    \end{enumerate}
    
    The discs of type II, III and IV give a contribution to the abelian rank of $C_P$ that adds up to $\nu-\sum_{1 \leq i \leq r} \nu_i - t$, where $t$ is the number of valid discs of type II(a) and III(a). Their contribution to the toric rank of $C_P$ equals $t$ by \Cref{thm toric rank}, taking into account that $t$ equals the number of viable non-\"{u}bereven clusters $\mathfrak{r}\subseteq \mathfrak{s}'$ such that $\mathfrak{s}_i\subseteq \mathfrak{r}$ for some $i\in \mathcal{I}$. Thus, the valid discs contained in $D$ give a total contribution to the abelian and toric rank of $\SF{\Yrst}$ equal to $\nu$, which is what we wanted.
\end{proof}

\begin{cor} \label{cor odd-cardinality clusters}
    Suppose that $\mathfrak{s}$ is a cluster of odd cardinality $2\nu + 1$ with $1 \leq \nu \leq g - 1$. Then the special fiber $\SF{\Yrst}$ consists of two $G$-invariant $k$-curves $C_0$ and $C_\infty$ meeting at a single node $Q_{\mathfrak{s}}\in \SF{\Yrst}$; their arithmetic genera are $\nu$ and $g-\nu$ respectively.
    \begin{proof}
        Choose $D$ to be any disc of the form $D_{\mathfrak{s},b}$ for some $b \in (d_-(\mathfrak{s}), d_+(\mathfrak{s}))$, and let $\alpha \in D_{\mathfrak{s},b}$ be a center and $\beta\in \bar{K}^\times$ be an element of valuation $b$. We have that $\nu$ roots of $\Rinfty$ reduce to $\overline{x_{\alpha,\beta}}=0$, while the remaining $2g-\nu$ roots reduce to $\overline{x_{\alpha,\beta}}=\infty$ in $\SF{\XX_D}$. 
        
        Assume that we are in the $p = 2$ setting. The normalized reduction of $\rho_{\alpha,\beta}$ has the form $x_{\alpha,\beta}^{2m+1}$; meanwhile, the part-square decomposition $f= 0^2 + f$ is good at $D$ by \Cref{prop good decomposition}.  We deduce in particular that $\mathfrak{t}^{\mathcal{R}}(D) = 0$, and that we have that $\mu(\XX_D,0)=2\nu$ and $\mu(\XX_D,\infty)=2g-2\nu$ and that $\mu(\XX_D,P)=0$ at all other points $P\in \SF{\XX_D}$ (see \S\ref{sec models hyperelliptic inseparable}).  As a consequence of \Cref{prop part of rst inseparable}, we have $\XX_D\not\le\Xrst$. Now the corollary follows as an immediate application of \Cref{prop partition rank}.  In the $p\neq 2$ setting, we also have $\XX_D\not\le \Xrst$ by \Cref{thm part of rst separable}, and the corollary follows from \Cref{rmk separating tame case}.
    \end{proof}
\end{cor}

\section{Computations for hyperelliptic curves of low genus} \label{sec computations}

In this section we apply our results from \S\ref{sec depths},\ref{sec centers},\ref{sec structure} to determine the possible structures of special fibers of relatively stable models $\Yrst$ of hyperelliptic curves $Y$ defined by equations of the form $y^2 = f(x)$ (with $\deg(f)=2g+1$) over residue characteristic $p = 2$, given the cluster data associated to $f$ along with (when the genus $g \geq 2$) the valuations of elements of $K$ coming from formulas involving the coefficients of certain factors of $f$.  For genera $g = 1,2$, we shall classify hyperelliptic curves over $K$ into several cases that depend on the aforementioned data and show how to compute each component of $(\Yrst)_s$ along with its toric rank on such a case-by-case basis.

In order to simplify notation in the formulas and conditions appearing in our statements below, for the hypotheses of each of the results of this section, we adopt the simplifying assumptions that 
\begin{enumerate}
    \item $f$ is monic;
    \item the depth of the full set of roots $\mathcal{R}$ is $0$; and
    \item one of the roots of $f$ (namely one which is contained in a particular even-cardinality cluster $\mathfrak{s}$ we are working with) is $0$.
\end{enumerate}
Assumption (1) holds after appropriately scaling the $y$-coordinate (by a scalar which lies in at most a quadratic extension of $K$).  Assumptions (2) and (3) hold after making simple changes of coordinates of the defining equation of the hyperelliptic curve which translate and scale the roots of $f$; this is done by translating and scaling the $x$-coordinate and again appropriately scaling the $y$-coordinate.

\begin{rmk}
    Assume that $f$ is monic and one of its roots is 0, i.e.\ it satisfies (1) and (3). Then, condition (2) just means that $f$ has integral coefficients, and at least one of its non-leading coefficients is a unit; equivalently, $f$ has integral roots, and one of its roots is a unit.
\end{rmk}

The following lemma will help us below to characterize those valid discs which are not linked to any cluster.

\begin{lemma}
    \label{lemma crushed components valuations}
    Suppose $f$ is monic; let $\alpha\not\in\RR$, and let $I(\varnothing,\alpha)=[d_-(\varnothing,\alpha),+\infty)$ be the interval defined in \Cref{dfn interval I alpha}. Then we have $d_-(\varnothing,\alpha)=\max_{a \in \RR} v(a-\alpha)$, and, for all discs $D:=D_{\alpha,b}$ with $b\in I(\varnothing,\alpha)$, we have
    \begin{equation}\label{equation lemma crushed}
        \vfun_{f}(D)=\sum_{a\in\RR} v(a-\alpha).
    \end{equation}
    \begin{proof}
        This is an immediate computation.
    \end{proof}
\end{lemma}

\subsection{The \texorpdfstring{$g = 1$}{g=1} case (elliptic curves)} \label{sec computations g=1}

We begin our search for concrete results for hyperelliptic curves by considering the simplest situation: the case that $g = 1$ so that $Y$ is an elliptic curve.  Suppose that $Y : y^2 = f(x)$ is an elliptic curve over $K$, i.e.\ we have $g = 1$ and $\deg(f) = 3$.  We note that this case is treated (in a more concrete and elementary fashion) as the main topic of the second author's paper \cite{yelton2021semistable}.  We label the three roots of $f$ as $a_1 := 0, a_2, a_3 \in \bar{K}$.  Apart from the full set of roots, clearly the only non-singleton cluster we may have is a cluster of cardinality $2$ which we assume coincides with $\{a_1 = 0, a_2\}$.  The second author's previous results \cite[Theorems 1 and 4]{yelton2021semistable} may be rephrased using the terminology of this paper and adapted to our particular desired semistable model $\Yrst$ as follows.

\begin{thm} \label{thm elliptic}
    With the above set-up, let $m = v(a_2)$ (so that $m = 0$ if and only if there is no cardinality-$2$ cluster and otherwise $m$ is the depth of the cardinality-$2$ cluster $\mathfrak{s} = \{0, a_2\}$).
    \begin{enumerate}[(a)]
        \item Suppose that $m > 4v(2)$.  Then there are exactly $2$ valid discs $D_+ := D_{0, m - 2v(2)}$ and $D_- := D_{0, 2v(2)}$, both linked to the cluster $\mathfrak{s}$.  Thus, the special fiber $(\Yrst)_s$ consists of $2$ components each of abelian rank $0$ which intersect at $2$ points.
        \item Suppose that $m \leq 4v(2)$.  Then there is exactly $1$ valid disc $D_{\alpha_1,b_1}$, where
        $\alpha_1$ satisfies $v(\alpha_1)=v(\alpha_1-a_2)=\frac{1}{2}m$ and $v(\alpha_1-a_3)=0$, while $b_1=\frac{1}{3}(m + 2v(2))$; moreover, the center $\alpha_1\in \bar{K}$ can be taken to be a root of the polynomial 
        $$F(T) := P_1^2(T) - 4P_2(T)P_0(T) \in K[T],$$
        where each $P_i(T) \in K[T]$ is the $z^i$-coefficient of $f(z + T) \in K[T][z]$.  The corresponding model $\YY_D$ of $Y$ has smooth special fiber; thus, in this case, the relatively stable model $\Yrst$ coincides with $\YY_D$ and $Y$ attains good reduction.
    \end{enumerate}
\end{thm}

\begin{figure}
    \centering
    \def\svgwidth{\linewidth}
    \input{fig_cases_g1}
    \caption{The shape of the function $I(\mathfrak{s})\to \zerotwo$, $b\mapsto \tbest{\RR}{D_{0,b}}$ in cases (a) and (b) of \Cref{thm elliptic} provided that $m>0$.}
    \label{fig cases g1}
\end{figure}

\begin{rmk} \label{rmk cases g=1}
    The cases (a) and (b) of \Cref{thm elliptic} (when $m>0$) correspond to the possible shapes of the function $b\mapsto \tbest{\RR}{D_{0,b}}$ as $b$ ranges in $I(\mathfrak{s})=[d_-(\mathfrak{s}),d_+(\mathfrak{s})]=[0,m]$, which are described in \Cref{fig cases g1}.
\end{rmk}

\begin{rmk} \label{rmk Legendre}
    The case that treated in \cite{yelton2021semistable} is when the elliptic curve is a member of the Legendre family, i.e.\ of the form $E_\lambda : y^2 = f(x) := x(x - 1)(x - \lambda)$ for some $\lambda \in K \smallsetminus \{0, 1\}$ with $m = v(\lambda)$ (in other words, we set $a_2 = \lambda$ and $a_3 = 1$, which can always done after appropriate translation and scaling).  In this situation, we make the following observations.
    \begin{enumerate}[(a)]
        \item The above theorem directly implies that the elliptic curve $E_\lambda$ has potentially good reduction if and only if $m \leq 4v(2)$.  This could alternately be deduced as a consequence of the following facts.  It is well known (see for instance \cite[\S IV.1.2]{serre1989abelian} or \cite[Proposition VII.5.5]{silverman2009arithmetic}) that any elliptic curve over a complete discrete valuation field has good (resp.\ multiplicative) reduction over some finite extension of that field if and only if the valuation of its $j$-invariant is nonnegative (resp.\ negative).  The formula for the $j$-invariant of the Legendre curve $E_\lambda$ is given as in \cite[Proposition III.1.7]{silverman2009arithmetic} by 
        \begin{equation} \label{eq j} 
            j(E_\lambda) = 2^8 \frac{(\lambda^2 - \lambda + 1)^3}{\lambda^2 (\lambda - 1)^2}.
        \end{equation}
        It is easily computed from this formula that we have $v(j(E)) = 8v(2) - 2v(\lambda)$, and our claim about potentially good reduction if and only if $m \leq 4v(2)$ follows.
        \item We compute the formulas $P_2(T) = 3T - (\lambda + 1)$, $P_1(T) = 3T^2 - 2(\lambda + 1)T + \lambda$, and $P_0(T) = T^3 - (\lambda + 1)T^2 + \lambda T$.  Then the polynomial $F$ given in \Cref{thm elliptic}(b) can be written in the simpler form 
        \begin{equation} \label{eq F Legendre}
            F(T) = -3T^4 + 4(1 + \lambda)T^3 - 6\lambda T^2 + \lambda^2.
        \end{equation}
        \item For the polynomial defining $E_\lambda$, we have $f^\mathfrak{s}_+(z)=f^\mathfrak{s}_-(z) = 1 - z$, and the obvious totally odd part-square decompositions for both of them induce (as in \S\ref{sec depths separating roots std form}) the decomposition $f(x) = [\sqrt{-1}x]^2 + [x^3 - \lambda x^2 + \lambda x]$ (for some choice of square root of $-1$), which according to \Cref{prop finding J from b0} is good at the discs $D_{0,m-2v(2)}$ and $D_{0,2v(2)}$.  This is helpful for explicitly constructing the components of $(E_{\lambda}^{\mathrm{rst}})_s$ in the case that $m \geq 4v(2)$.
    \end{enumerate}
\end{rmk}

Examples of computations which yield the desired model $\Yrst$ in the case that $m \leq 4v(2)$ are given as \cite[Examples 2 and 3]{yelton2021semistable}.  Below is an example for the $m > 4v(2)$ case, which is treated in \cite[Example 9]{yelton2021semistable} except that there a semistable model whose special fiber has a single (nodal) component, rather than the relatively stable model $\Yrst$, is found.

\begin{ex} \label{ex g=1 p=2}
    Let $Y$ be the elliptic curve over $\zz_2^{\unr}$ given by 
    $$y^2 = x(x - 64)(x - 1),$$
    so that we have a unique even-cardinality cluster $\mathfrak{s} = \{0, 64\}$ of relative (and absolute) depth $m = 6v(2)$.  We are therefore in the situation of \Cref{thm elliptic}(a), and the valid discs can be taken to be $D_1 := D_- = D_{0,2v(2)}$ and $D_2 := D_+ = D_{0,4v(2)}$.  Using the sufficiently odd part-square decomposition given by \Cref{rmk Legendre}(c), we obtain (see \S\ref{sec models hyperelliptic forming}) that the changes in coordinates corresponding to each of these discs may be written as  
    $$x = 4x_1 = 16x_2, \ \ \ y = 8y_1 + 4\sqrt{-1}x_1 = 32y_2 + 16\sqrt{-1}x_2.$$  We now get equations for the models $\YY_1$ and $\YY_2$ corresponding to $D_1$ and $D_2$ respectively as 
    \begin{equation}
        \YY_1: y_1^2 + \sqrt{-1} x_1 y_1 = x_1^3 - 2^{4} x_1^2 + 2^2 x_1, \qquad 
        \YY_2: y_2^2 + \sqrt{-1} x_2 y_2 = 2^{2}x_2^3 - 2^{4}x_2^2 + x_2.
    \end{equation}
    whose special fibers are the $\bar{\ff}_2$-curves described by the equations
    \begin{equation}
        \SF{\YY_1}: y_1^2 + \sqrt{-1} x_1 y_1 = x_1^3, \qquad \SF{\YY_2}: y_2^2 + \sqrt{-1} x_2 y_2 = x_2.
    \end{equation}
    Note that $\YY_1$ is already a semistable model of $Y$, its special fiber being a curve with a node, and is one that could be obtained from \cite[Theorem 4 and Remark 5]{yelton2021semistable}, but it is not the relatively stable model as the node is a vanishing node (see \Cref{dfn relatively stable}).  The desingularizations of $\SF{\YY_1}$ and $\SF{\YY_2}$ are each smooth curves of genus $0$ and give rise to the components of $\SF{\Yrst}$, the configuration of which is shown in \Cref{fig p2 g1 example2}. 
\end{ex}

\begin{figure}
    \centering
    \vspace{-5mm}
    \tikzset{every picture/.style={line width=0.75pt}} 

\begin{tikzpicture}[x=0.75pt,y=0.75pt,yscale=-1,xscale=1]

\draw    (43,18) .. controls (60.75,164.67) and (116.44,164.67) .. (127.74,18) ;
\draw    (43,128) .. controls (60.75,-18.67) and (116.44,-18.67) .. (127.74,128) ;
\draw    (273,128) -- (212.47,18) ;
\draw    (273,18) -- (212.47,128) ;
\draw    (151.95,73) -- (209.47,73) ;
\draw [shift={(212.47,73)}, rotate = 180] [fill={rgb, 255:red, 0; green, 0; blue, 0 }  ][line width=0.08]  [draw opacity=0] (10.72,-5.15) -- (0,0) -- (10.72,5.15) -- (7.12,0) -- cycle    ;

\draw (21,18.4) node [anchor=north west][inner sep=0.75pt]    {$V_{1}$};
\draw (14,110.4) node [anchor=north west][inner sep=0.75pt]    {$V_{2}$};
\draw (275,21.4) node [anchor=north west][inner sep=0.75pt]    {$L_{1}$};
\draw (278,108.4) node [anchor=north west][inner sep=0.75pt]    {$L_{2}$};

\end{tikzpicture}
    \vspace{-8mm}
    \caption{The special fiber $\SF{\Yrst}$, shown above, mapping to $\SF{\Xrst}$; each component $V_i$ of $\SF{\Yrst}$ maps to each component $L_i := \SF{\XX_{D_i}}$ of $\SF{\Xrst}$.}
    \label{fig p2 g1 example2}
\end{figure}




Throughout the rest of this subsection we prove Theorem \ref{thm elliptic}.  By Theorem \ref{thm cluster p=2}(a), we know that for any valid disc $D$, we have either $D \cap \mathcal{R} = \varnothing$ or $D \cap \mathcal{R} = \mathfrak{s}$.  Let us first treat the situation where $\mathfrak{s} := \{0, a_2\}$ is a cluster and search for all valid discs (if any) which contain it.

\subsubsection{Finding valid discs containing a cardinality-$2$ cluster} \label{sec computations g=1 containing cluster}

For the moment, let us assume that $\mathfrak{s} := \{0, a_2\}$ is a cluster; we fix $0$ as a center for any disc containing $\mathfrak{s}$.  Then \Cref{prop cases of slope always 1}(a),(c) (along with \Cref{rmk cases of slope always 1}(c)) directly implies that we have $\mathfrak{t}_{\pm}^{\mathfrak{s}}(b)=\truncate{b}$, so that $b_0(\mathfrak{t}_{\pm}^{\mathfrak{s}}) = 2v(2)$ and hence $B_{f,\mathfrak{s}} = 4v(2)$ (see also \Cref{cor g=1 B=4v(2)}). From this, we deduce immediately that $\SF{\Yrst}$ has toric rank 1 and hence abelian rank 0 (resp.\ toric rank 0 and hence abelian rank 1) if and only if $m>4v(2)$ (resp.\ $m\le 4v(2)$): this is a consequence of \Cref{thm toric rank}. Moreover, there are exactly $2$ (resp.\ $1$, resp.\ $0$) valid discs linked to $\mathfrak{s}$ if and only if we have $m > 4v(2)$ (resp.\ $m = 4v(2)$, resp.\ $m < 4v(2)$), and when $m \geq 4v(2)$ the valid disc(s) containing $\mathfrak{s}$ can be written as $D_{0, 2v(2)}$ and $D_{0, m - 2v(2)}$ (they coincide when $m = 4v(2)$): this follows from \Cref{thm summary depths valid discs} (see also \Cref{prop depth threshold}). Since we have $\lambda_\pm(\mathfrak{s}) = 1$ by \Cref{lemma slopes of t}, by applying \Cref{prop genus of components}(b) we can compute that each of the components of $\SF{\Yrst}$ corresponding to $D_+$ and $D_-$ has abelian rank $0$ if $m > 4v(2)$ (the fact that they intersect at $2$ nodes follows from \Cref{prop viable correspondence}) and that the component of $\SF{\Yrst}$ corresponding to $D_+ = D_-=D_{0,2v(2)}$ is smooth of abelian rank $1$ if $m = 4v(2)$.  This proves \Cref{thm elliptic}(a).

We also remark that, from the formulas for $\mathfrak{t}^{\mathfrak{s}}_\pm$ we have derived, one deduces, by \Cref{prop t^R is min of t^s and t^(R-s)}, that $\tbest{\RR}{D_{0,b}}=\min\{b,m-b,2v(2)\}$ for $b\in [0,m]$, as shown in \Cref{fig cases g1}.

\subsubsection{Finding a center and a depth of a valid disc not containing any roots} \label{sec computations g=1 not containing cluster}
In the previous subsection, we found all valid discs linked to $\mathfrak{s}$. Moreover, we have seen that, when $m> 4v(2)$, the special fiber $\SF{\Yrst}$ has abelian rank zero, while, if $m=4v(2)$, it has abelian rank 1, which is entirely contributed by the unique valid disc $D_+=D_-$ linked to $\mathfrak{s}$. When $m<4v(2)$, the special fiber $\SF{\Yrst}$ must have abelian rank 1 (as we have $g = 1$ but the toric rank is $0$ in the absence of viable clusters by \Cref{prop viable correspondence}), but there is no valid disc linked to $\mathfrak{s}$.

Now, \Cref{cor crushed positive abelian rank} ensures that a valid discs $D$ that is not linked to $\mathfrak{s}$, or that is the unique disc linked to $\mathfrak{s}$, gives a positive contribution to the abelian rank of $\SF{\Yrst}$: it follows that there is no such disc when $m> 4v(2)$, and exactly one when $m\le 4v(2)$. In the latter case, \Cref{cor centers} ensures that this unique valid disc $D$ contains all roots $\alpha$ of 
\begin{equation} \label{eq F elliptic}
    F(T) = P_1^2(T) - 4P_2(T)P_0(T),
\end{equation}
where $P_i(T) \in K[T]$ be defined as in the statement of \Cref{thm elliptic}(b).  Let us therefore assume $m\le 4v(2)$, let $\alpha_1$ be any of the roots of $F$, and let $D=D_{\alpha_1,b_1}$ be the unique valid disc. 
\begin{lemma}
    \label{lemma g=1 v(alpha)}
    In the setting above, we have $v(\alpha_1) = v(\alpha_1-a_2) = \frac{1}{2}m$ and $v(\alpha_1-a_3)=0$.
    \begin{proof}
         This can be proved by directly inspecting the Newton polygon of $F$. We present a more theoretical proof which separately treats the cases $m=0$, $0<m<4v(2)$, and $m=4v(2)$.
         
         When $m=0$, we study the model $\YY_{D'}$ corresponding to the disc $D'=D_{0,0}$. The (normalized) reduction of $f$ has a simple root at $\overline{a_1} = 0, \overline{a_2}, \overline{a_3}$ and $\infty$. In particular, the trivial decomposition $f=0^2+f$ is good at $D'$ by \Cref{prop good decomposition}, and we have $\tbest{\RR}{D'}=0<2v(2)$ and $\mu(\XX_{D'},\overline{a_i})=\mu(\XX_{D'},\infty)=0$ for $i=1,2,3$ (see \S\ref{sec models hyperelliptic inseparable}). Now by \Cref{prop inseparable tfae} we have $D\in \mathfrak{D}_P$ for some $P\neq \overline{a_1}, \overline{a_2}, \overline{a_3},\infty$, which implies, thanks to \Cref{prop relative position smooth models line}, that $v(\alpha_1-a_i)=0$ for all $i=1,2,3$.
        
        When $0<m<4v(2)$, we note that $b \mapsto \mathfrak{t}^\RR(D_{0,b})$ is not differentiable at the input $b = \frac{1}{2}m$ with $\mathfrak{t}^\RR(D_{0,\frac{1}{2}m}) = \frac{1}{2}m < 2v(2)$; from \Cref{lemma bifurcation valid discs} (and \Cref{rmk bifurcation valid discs}), one deduces that $v(\alpha_1)=\frac{1}{2}m$, which proves the lemma.
        
        When $m=4v(2)$, given $\beta$ an element of valuation $\frac{1}{2}m=2v(2)$ (which is the depth of $D$), we have that $D$ is the unique valid disc linked to $\mathfrak{s}$, and the conclusion follows from \Cref{cor centers}.
    \end{proof}
\end{lemma}

From this knowledge of $v(\alpha_1-a_i)$, by applying \Cref{lemma crushed components valuations}, one deduces that $\vfun_{f}(D_{\alpha_1,c})=\sum_{i=1}^3 v(\alpha_1-a_i)=m$ for all $c\ge \frac{1}{2}m$. Meanwhile, we know by \Cref{prop roots of F are centers of clu discs} that there exists a part-square decomposition $f=q^2+\rho$ that is totally odd with respect to the center $\alpha_1$ and such that $\rho_{\alpha_1,1}$ has no linear term; this means that $\rho(x)=(x-\alpha_1)^3$, and we thus get $\vfun_{\rho}(D_{\alpha_1,c})=3c$ for all $c\in \qq$. We conclude that $\tbest{\RR}{D_{\alpha_1,c}}=\truncate{3c-m}$ for $c \geq \frac{1}{2}m$; hence, the depth $b_1$ of the valid disc $D$ can now be obtained by solving the equation $3c-m=2v(2)$ in the variable $c\in [\frac{1}{2}m,+\infty)$ (see \Cref{thm summary depths valid discs}), which gives $b_1=\frac{1}{3}(m+2v(2))$. This proves \Cref{thm elliptic}(b).

\subsection{The \texorpdfstring{$g = 2$}{g=2} case} \label{sec computations g=2}

We now investigate the structure of $(\Yrst)_s$ where $Y$ is a genus-$2$ hyperelliptic curve; let $Y: y^2 = f(x)$ be the equation of $Y$, where the polynomial $f$ has degree $5$ and satisfies the simplifying assumptions (1), (2) and (3) listed at the beginning of this section; the roots of $f$ will be denoted $a_1 := 0, a_2, \ldots , a_5$.  Clearly there may be $0$, $1$, $2$, or $3$ even-cardinality clusters among the cluster data associated to $f$; except for in the last case of $3$ even-cardinality clusters, there may be a single cardinality-$3$ cluster as well.  

The below theorem describes our results on the possible structures of $\SF{\Yrst}$ depending on various arithmetic conditions, under the assumption that there exists at most one even-cardinality cluster. Actually, the theorem only addresses the case in which the even-cardinality cluster, if it exists, has cardinality 2 and its parent cluster coincides with $\RR$, but it may be adapted any other cluster picture having at most one even-cardinality cluster; see \Cref{rmk four of a kind}(a) below for more details.  To treat the case of more than one even-cardinality cluster, instead see \Cref{rmk four of a kind}(b),(c).

\begin{thm} \label{thm g=2}
    Assume that we are in the $g = 2$ situation and retain all of the above assumptions on $f$.  Assume moreover that there are no cardinality-$4$ clusters and there is at most one cardinality-$2$ cluster $\mathfrak{s} \subset \mathcal{R}$; if this cluster exists, we denote its relative depth by $m := \delta(\mathfrak{s})$, whereas if there is no even-cardinality cluster, we set $m = 0$.  It is clear that $\RR$ can contain at most one cardinality-3 cluster $\mathfrak{s}'$; if it exists, we denote its relative depth by $m':=\delta(\mathfrak{s}')$, whereas if there is no cardinality-3 cluster, we set $m'=0$. We assume that, when both $m$ and $m'$ are $>0$, we have $\mathfrak{s} \cap \mathfrak{s}'=\varnothing$.
    
    We label the roots $a_1, \ldots a_5$ of $f$ in such a way that, when $m>0$, we have $\mathfrak{s}=\{a_1=0,a_2\}$, and when $m'>0$, we have $\mathfrak{s}'=\{a_3,a_4,a_5\}$. Under the assumption that $m > 0$, we write the polynomial $f^{\mathfrak{s}}_-(z)$ (see \S\ref{sec depths separating roots std form} for its definition) as $1 + M_1 z + M_2 z^2 + M_3 z^3$ and let $w = v(M_1 - 2\sqrt{M_2})\ge 0$ for some choice of square root of $M_2$; when $m'>0$, we have $w=0$.
    
    Define the polynomial 
    \begin{equation*}
        F(T) = (P_1^2(T) - 4P_2(T)P_0(T))^2 - 64P_4(T)P_0^3(T) \in K[T],
    \end{equation*}
    where $P_i(T)$ is the $z^i$-coefficient of $f(z + T)$ for $0 \leq i \leq 5$, which we have seen in \Cref{rmk formulas for F for g=1 and g=2} is the polynomial $F$ defined in \S\ref{sec centers def}. 
    For any root $\alpha \in \bar{K}$ of $F$, let $f=q^2 + \rho$ be a part-square decomposition that is totally odd with respect to the center $\alpha$
    and such that $\rho_{\alpha,1}$ has no linear term (as is guaranteed to exist by \Cref{prop roots of F are centers of clu discs}(a)), and let $\kappa(\alpha)$ be the valuation of the cubic term of $\rho_{\alpha,1}$.
    
    In the language of \Cref{prop depth threshold}, when $m > 0$, we have $B_{f,\mathfrak{s}} = \max\{4v(2) - w, \frac{8}{3}v(2)\}$.  The set of valid discs and the structure of $(\Yrst)_s$ are fully described more precisely as follows.  All elements $\alpha_i$ mentioned in parts (b), (c), and (d) below may be chosen to be roots of $F$, so that in particular $\kappa(\alpha_i)$ is always defined.
    
    \begin{enumerate}[(a)]
        \item Suppose that $m > \frac{8}{3}v(2)$ and $w \geq \frac{4}{3}v(2)$.  Then there are exactly $2$ valid discs $D_- := D_{0,\frac{2}{3}v(2)}$ and $D_+ := D_{0,m - 2v(2)}$, both of which are linked to $\mathfrak{s}$.  The special fiber $(\Yrst)_s$ consists of $2$ components corresponding to the discs $D_-$ and $D_+$ which intersect at $2$ nodes and have abelian ranks $1$ and $0$ respectively.
        \item Suppose that $m>0$ and $4v(2) - m < w < \frac{4}{3}v(2)$.  Then there are two valid discs $D_+ := D_{0, m - 2v(2)}$ and $D_- := D_{0, 2v(2) - w}$ which are linked to $\mathfrak{s}$; their corresponding components of $\SF{\Yrst}$ each have abelian rank $0$ and intersect each other at $2$ points.  There is moreover another valid disc $D_{\alpha_1, b_1}$, which does not contain a root of $f$; we have $v(\alpha_1-a_i) = \frac{1}{2}w$ for $i=1,2$, $v(\alpha_1-a_i) = m'$ for $i=3,4,5$ and $b_1 = m'+\frac{1}{3}(w - \kappa(\alpha_1) + 2v(2))$.  The corresponding component of $\SF{\Yrst}$ has abelian rank $1$ and intersects the component corresponding to $D_-$ at $1$ node.
        \item Suppose that we have $m>0$, $w < \frac{1}{2}m$, and $w\le 4v(2) - m$. 
        Then there are valid discs $D_1 := D_{\alpha_1,b_1}$ and $D_2 := D_{\alpha_2,b_2}$ with $v(\alpha_1-a_i) = \frac{1}{2}w$ for $i=1,2$, $v(\alpha_1-a_i)=m'$ for $i=3,4,5$, $v(\alpha_2-a_i) = \frac{1}{2}(m - w)$ for $i=1,2$, and $v(\alpha_2-a_i)=0$ for $i=3,4,5$, $b_1 =m'+ \frac{1}{3}(w - \kappa(\alpha_1) + 2v(2))$, and $b_2 = \frac{1}{3}(m - w - \kappa(\alpha_2) + 2v(2))$.  The discs $D_i$ each do not contain a root of $f$ if $w<4v(2)-m$; when $w=4v(2)-m$, the disc $D_1$ does not, but the disc $D_2$ is the unique valid disc linked to $\mathfrak{s}$ and coincides with $D_{0, m - 2v(2)}$. The special fiber $\SF{\Yrst}$ consists of $2$ components corresponding to the discs $D_1$ and $D_2$, each of abelian rank $1$, which intersect at $1$ node.
        \item Finally, suppose that we have $m=0$, or $0<m \leq \min\{2w, \frac{8}{3}v(2)\}$.  Then there is a valid disc $D_1 := D_{\alpha_1,b_1}$ with $v(\alpha_1-a_i) = \frac{1}{4}m$ for $i=1,2$, and $v(\alpha_1-a_i)=0$ for $i=3,4,5$, and $b_1 \geq v(\alpha_1)$.  We have the following subcases.
        \begin{enumerate}[(i)]
            \item Suppose that $\kappa(\alpha_1) < \frac{2}{5}(\frac{1}{2}m + 2v(2))$.  Then there is a second valid disc $D_2 := D_{\alpha_2,b_2}$ where $\alpha_2$ satisfies $v(\alpha_2-a_i) = \frac{1}{4}m$ for $i=1,2$ and $v(\alpha_2-a_i) = m'$ for $i=3,4,5$, and we have $b_1 = \frac{1}{3}(\frac{1}{2}m - \kappa(\alpha_1) + 2v(2))$ and $b_2=b_1+m'$.  Neither of the discs $D_i$ contains a root of $f$.  The special fiber $(\Yrst)_s$ consists of $2$ components corresponding to the discs $D_i$, each of abelian rank $1$, which intersect at $1$ node.
            \item Suppose that $\kappa(\alpha_1) \geq \frac{2}{5}(\frac{1}{2}m + 2v(2))$.  Then the only valid disc is $D_1$; it is (the unique valid disc) linked to $\mathfrak{s}$ if $m = \frac{8}{3}v(2)$ but otherwise does not contain a root of $f$. Its depth is $b_1=\frac{1}{5}(\frac{1}{2}m+2v(2))$. The special fiber $(\Yrst)_s$ thus has exactly $1$ component, which has abelian rank $2$ (so $Y$ attains good reduction in this case).
        \end{enumerate}
    \end{enumerate}
\end{thm}

\begin{rmk} \label{rmk cases g=2}
    The cases (a), (b), (c) and (d) of \Cref{thm g=2} (when $m>0$) correspond to the possible shapes of the function $b\mapsto \tbest{\RR}{D_{0,b}}$ as $b$ ranges in $I(\mathfrak{s})=[d_-(\mathfrak{s}),d_+(\mathfrak{s})]=[0,m]$, which are described in \Cref{fig cases g2}.  Note that in cases (a) and (b) we have that $\mathfrak{s}$ is a viable cluster (i.e.\ $m>B_{f,\mathfrak{s}}$), while in cases (c) and (d) there are no viable clusters.
\end{rmk}

\begin{figure}
    \centering
    \def\svgwidth{\linewidth}
    \input{fig_cases_g2}
    \caption{The shape of the function $I(\mathfrak{s})\to \zerotwo$, $b\mapsto \tbest{\RR}{D_{0,b}}$ in cases (a), (b), (c) and (d) of \Cref{thm g=2} provided that $m>0$.}
    \label{fig cases g2}
\end{figure}

\begin{rmk} \label{rmk four of a kind}
    The theorem only treats the situation where there are no cardinality-$4$ clusters and at most one cardinality-$2$ cluster which is not contained in a cardinality-$3$ cluster; here we briefly explain how to treat cases where this hypothesis does not hold.
    \begin{enumerate}[(a)]
        \item If we consider a situation where the only even-cardinality cluster $\mathfrak{s}$ has relative depth $m$ and cardinality $4$ (instead of $2$), then on applying the automorphism $i_a$ as defined in \Cref{rmk reciprocal}, where $a \in \mathfrak{s}$ is a root that does not belong to a cardinality-$3$ cluster, we obtain a cluster picture in which there is a cardinality-$2$ cluster (and possibly a cardinality-$3$ cluster disjoint from it), and then using \Cref{prop reciprocal} one can derive analogous statements to everything in the above theorem.  The rough idea is as follows.  For this case, we write $f^{\mathfrak{s}}_+(z) = 1 + M_1 z + M_2 z^2 + M_3 z^3$ and again set $w = v(M_1 - 2\sqrt{M_2})$.  Then we again get $B_{f,\mathfrak{s}} = \max\{4v(2) - w, \frac{8}{3}v(2)\}$, and under each of the main hypotheses of parts (a), (b), (c), and (d) we get the same outcome in terms of the structure of $\SF{\Yrst}$ (the number of components corresponding to valid discs linked to or not linked to $\mathfrak{s}$, and how they intersect).  The valuations of the centers of the discs as well as their depths are given by different formulas, however.  In particular, the valid discs $D_\pm$ claimed in parts (a) and (b) each have depths $m - b$, where $b$ is the claimed depth in the statement of the theorem: for part (a), we now have valid discs $D_- := D_{a,m-\frac{2}{3}v(2)}$ and $D_+ := D_{a,2v(2)}$ linked to $\mathfrak{s}$, and so on.
        
        Similarly, if we begin with a cluster picture such that there is a cardinality-$3$ cluster $\mathfrak{s}'$ containing $0$, then by applying the automorphism $i_a$ as defined in \Cref{rmk reciprocal} where $a$ is a root in $\mathfrak{s}' \smallsetminus \mathfrak{s}$ (or is any root in $\mathfrak{s}'$ if $m = 0$) and using \Cref{prop reciprocal}, we obtain a cluster picture in which there is a cardinality-$3$ cluster which instead does not contain $0$.
        \item Suppose that there are exactly $2$ even-cardinality clusters $\mathfrak{s}_1$ and $\mathfrak{s}_2$ containing roots $a_1$ and $a_2$ respectively.  Then by applying appropriate parts of \Cref{prop estimating threshold}(d) combined with \Cref{prop estimating threshold}(a),(b), we get $B_{f,\mathfrak{s}_1} = B_{f,\mathfrak{s}_2} = 4v(2)$.  For $i = 1, 2$, the arguments used in \S\ref{sec computations g=1} give us the following statements.  If $m_i := \delta(\mathfrak{s}_i) < 4v(2)$ (resp.\ $\delta(\mathfrak{s}_i) \geq 4v(2)$), then there exist valid discs $D_-^{(1)} := D_{\alpha_i,b_i}$ where $\alpha_i$ is a root of $F$ with $v(\alpha_i - a_i) = d_-(\mathfrak{s}_i) + \frac{1}{2}m_i$ and $b_i = d_-(\mathfrak{s}_i) + \frac{1}{3}(m_i + 2v(2))$ (resp.\ valid discs $D_+^{(i)} := D_{\mathfrak{s}_i,d_+(\mathfrak{s}_i)-2v(2)}$ and $D_-^{(i)} := D_{\mathfrak{s}_i,d_-(\mathfrak{s}_i) + 2v(2)}$; these discs coincide if and only if $m_i = 4v(2)$).  Moreover, if $m_i \leq 4v(2)$ (resp.\ if $m_i > 4v(2)$), then the disc $D_-^{(i)}$ contributes a component of the special fiber $\SF{\Yrst}$ of abelian rank $1$ (resp.\ the discs $D_\pm^{(i)}$ each contribute a component of abelian rank $0$ and these components meet at $2$ nodes), and the components of $\SF{\Yrst}$ corresponding to $D_-^{(1)}$ and $D_-^{(2)}$ intersect at $1$ node.
        
        In fact, it is straightforward to compute that, with the quantity $w$ defined as in the theorem, when there are exactly $2$ even-cardinality clusters, then we have $w = 0$; the above statements can therefore be proved for each $i$ by applying the below arguments in the proof of part (b) (resp.\ (c)) of \Cref{thm g=2} to $\mathfrak{s}_i$ in the case that $m_i > 4v(2)$ (resp.\ $m_i \leq 4v(2)$) to obtain valid discs $D_\pm^{(i)}$ (resp.\ the valid disc $D_-^{(i)}$) with the claimed properties.
        \item In the case that there are $3$ even-cardinality clusters, the computation of valid discs is in general much less straightforward, but in most cases either \Cref{prop deep ubereven cluster} or \Cref{cor g pairs} can be applied to entirely determine the special fiber $\SF{\Yrst}$.
    \end{enumerate}
\end{rmk}

\begin{rmk} \label{rmk kappa}
    Let $\alpha \in \bar{K}$ be a root of $F$.  It is implicit in our proof of \Cref{thm g=2} that the rational number $\kappa(\alpha)$ is well defined in all contexts of the statement in which its precise value is relevant (more precisely, one can show that it does not depend on the choice of totally odd decomposition with no linear term as long as it is $< 2v(2)$, which is guaranteed to be the case outside of parts (a) and (d)(ii)).  We see from the formula for $R_3$ found in \S\ref{sec computations sufficiently odd deg5} that it can be computed as 
    \begin{equation}
        \kappa(\alpha) = v\Big(P_3(\alpha) - 2\sqrt{P_4(\alpha)}\sqrt{P_2(\alpha) - 2\sqrt{P_4(\alpha)P_0(\alpha)}}\Big)
    \end{equation}
    only for particular choices of the square roots in the above formula.
\end{rmk}

\begin{rmk} \label{rmk valuations of roots of F}
    We observe the following regarding valuations of roots of the polynomial $F$.
    \begin{enumerate}[(a)]
        \item The polynomial $F$ has degree $16$; its leading term has unit coefficient; and its constant term equals $(a_2 a_3 a_4 a_5)^4$, and hence, under the hypotheses of the theorem, it has valuation $4m$.
        \item  In light of \Cref{cor centers}, parts (c) and (d) of the theorem now allow us to deduce the valuations of the roots of the polynomial $F$.  If $m>0$, $w < \frac{1}{2}m$ and $w\le 4v(2) - m$, then the statement of \Cref{thm g=2}(c) implies that all roots of $F$ have valuation either $\frac{1}{2}w$ or $\frac{1}{2}(m - w)$; since there are $16$ roots whose valuations must add up to $4m$, we get that $8$ of the roots have valuation $\frac{1}{2}w$ while the other $8$ have valuation $\frac{1}{2}(m - w)$.  Similarly, if $m=0$ or $0<m \leq \min\{2w, \frac{8}{3}v(2)\}$, then the statement of \Cref{thm g=2}(d) implies that all roots of $F$ have valuation $\frac{1}{4}m$.
    \end{enumerate}
\end{rmk}

\begin{ex} \label{ex g=2 p=2}
    Let $Y$ be the hyperelliptic curve of genus $2$ over $\zz_2^\unr$ given by 
    $$y^2 = x(x - 16)(x - 1)(x^2 + x - 1),$$
    so that we have a cardinality-$2$ cluster $\mathfrak{s} = \{0, 16\}$ of relative (and absolute) depth $m = 4v(2)$.  It is straightforward to compute that $f^{\mathfrak{s}}_-(z) = 1-2z+z^3$ and so we have $w = v(-2 - 2\sqrt{0}) = v(2)$.  The hypothesis of \Cref{thm g=2}(b) clearly holds, and so we have valid discs $D_1 := D_- = D_{0,v(2)}$ and $D_2 := D_+ = D_{0,2v(2)}$ which are linked to $\mathfrak{s}$.  By applying the computations in \S\ref{sec computations sufficiently odd deg3}, we get totally odd decompositions of $f^{\mathfrak{s},0}_\pm$, which induce (as in \S\ref{sec depths separating roots std form}) the decomposition 
    $$f(x) = [x]^2 + [x^5 - 16 x^4 - 2 x^3 + 32 x^2 - 16 x],$$
    which according to \Cref{prop finding J from b0} is good at the discs $D_i$.     Using this decomposition and our knowledge of the depths of the valid discs $D_i$, following the computations in \S\ref{sec models hyperelliptic forming} we obtain that the changes in coordinate corresponding to these discs may be written as 
    \begin{equation}
        x = 2x_1 = 4x_2, \ \ \ y = 4y_1 + 2x_1 = 8y_2 + 4x_2.
    \end{equation}
    We now get equations for the corresponding models $\YY_1$ and $\YY_2$ as 
    \begin{align}
    \begin{split}
        \YY_1 &: 
        y_1^2 + x_1 y_1 = 2 x_1^5 - 2^4 x_1^4 - x_1^3 + 2^3 x_1^2 - 2 x_1\\
        \YY_2 &: 
        y_2^2 + x_2 y_2 = 2^{4} x_2^5 - 2^{6} x_2^4 - 2 x_2^3 + 2^{3} x_2^2 -  x_2
    \end{split}
    \end{align}
    The special fibers of these models are the $\bar{\ff}_2$-curves given by 
    \begin{equation}
        \SF{\YY_1}: y_1^2 + x_1 y_1 = x_1^3, \qquad y_2^2 + x_2 y_2 = x_2.
    \end{equation}
    The desingularizations of $\SF{\YY_1}$ and $\SF{\YY_2}$ are each smooth curves of genus $0$ and give rise to $2$ of the components of $\SF{\Yrst}$.  However, these are not all of the components of $\SF{\Yrst}$, as \Cref{thm g=2}(b) asserts the existence of another valid disc $D_3 := D_{\alpha_1,b_1}$ for some root $\alpha_1$ of $F$ with $v(\alpha_1) = \frac{1}{2}v(2)$ and $b_1 = 1 - \frac{1}{3}\kappa(\alpha_1)$.  Now through tedious but straightforward calculations, one can show that $v(P_3(\alpha_1)) = v(2)$ and $v(P_4(\alpha_1)) = \frac{1}{2}v(2)$, from which it follows from considering the cubic coefficient appearing in (\ref{eq rho degree 5}) that we have $\kappa(\alpha_1) = v(2)$ and so $b_1 = \frac{2}{3}v(2)$.
    
    For an appropriate part-square decomposition $f = q^2 + \rho$ that is totally odd with respect to the center $\alpha_1$, the change in coordinates corresponding to $D_3$ can be written as  
    $$x = 2^{2/3}x_3 + \alpha_1, \ \ \ y = 2^{3/2}y_3 + q_{\alpha_1,1}(2^{2/3}x_3)y_3.$$
    We now get an equation for the model $\YY_3$ corresponding to $D_3$ as 
    \begin{equation} \label{eq p=2 g=2 Y_3}
        \YY_3: 
         y_3^2 + 2^{-1/2} q_{\alpha_1,1}(2^{2/3}x_3)y_3 = 2^{-3}\rho_{\alpha_1,1}(2^{2/3}x_3).
    \end{equation}
    Note that using \Cref{lemma crushed components valuations}, we have
    \begin{align}
    \begin{split}
        v(q_{\alpha_1,1}(2^{2/3}x_3)) &= \frac{1}{2}v(f_{\alpha_1,1}(2^{2/3}x_3) - \rho_{\alpha_1,1}(2^{2/3}x_2)) = \frac{1}{2}v(f_{\alpha_1,1}(2^{2/3}x_3)) = v(\alpha_1) = \frac{1}{2}v(2), \\
        v(\rho_{\alpha_1,1}(2^{2/3}x_3)) &= 2v(2) + v(f_{\alpha_1,1}(2^{2/3}x_3)) = 2v(2) + 2v(\alpha_1) = 3v(2).
    \end{split}
    \end{align}
    One can now readily verify that the special fiber of $\YY_3$ is the $\bar{\ff}_2$-curve given by 
    \begin{equation}
        y_3^2 + c_1 y_3 = c_2 x_3^3,
    \end{equation}
    for some $c_1, c_2\in k^\times$, and its desingularization is a smooth curve of genus $1$ which gives rise to the remaining component of $\SF{\Yrst}$.  The configuration of the components of $\SF{\Yrst}$ is seen in \Cref{fig p2 g2 example2}.
\end{ex}

\begin{figure}
    \centering
    \vspace{-1cm}
    \tikzset{every picture/.style={line width=0.75pt}} 

\begin{tikzpicture}[x=0.75pt,y=0.75pt,yscale=-1,xscale=1]

\draw    (11,70) -- (235,70) ;
\draw    (29,109) .. controls (76,107.33) and (75,-37.67) .. (128,26.33) ;
\draw    (128,112.33) .. controls (164,-10.33) and (196,-25.67) .. (213,110.33) ;
\draw    (360,117) -- (394,6) ;
\draw    (477,71) -- (340.47,71) ;
\draw    (262.95,70) -- (320.47,70) ;
\draw [shift={(323.47,70)}, rotate = 180] [fill={rgb, 255:red, 0; green, 0; blue, 0 }  ][line width=0.08]  [draw opacity=0] (10.72,-5.15) -- (0,0) -- (10.72,5.15) -- (7.12,0) -- cycle    ;
\draw    (405,122.33) -- (439,11.33) ;

\draw (13,42.4) node [anchor=north west][inner sep=0.75pt]    {$V_{1}$};
\draw (57,91.4) node [anchor=north west][inner sep=0.75pt]    {$V_{3}$};
\draw (217,92.4) node [anchor=north west][inner sep=0.75pt]    {$V_{2}$};
\draw (456,51.4) node [anchor=north west][inner sep=0.75pt]    {$L_{1}$};
\draw (369,95.4) node [anchor=north west][inner sep=0.75pt]    {$L_{3}$};
\draw (414,100.73) node [anchor=north west][inner sep=0.75pt]    {$L_{2}$};

\end{tikzpicture}
    \caption{The special fiber $\SF{\Yrst}$, shown on the left, mapping to $\SF{\Xrst}$; each component $V_i$ of $\SF{\Yrst}$ maps to each component $L_i := \SF{\XX_{D_i}}$ of $\SF{\Xrst}$.}
    \label{fig p2 g2 example2}
\end{figure}

The rest of this subsection is devoted to proving \Cref{thm g=2}.

\subsubsection{Finding valid discs containing an even-cardinality cluster} \label{sec computations g=2 containing cluster}

Suppose that, in the settting of \Cref{thm g=2}, we have $m>0$, i.e.\ that we have a unique even-cardinality cluster $\mathfrak{s} := \{0, a_2\}$ of relative depth $m$; our goal for the moment is to find all valid discs which are linked to $\mathfrak{s}$.  We adopt the notation and constructions of \S\ref{sec depths separating roots} and get the polynomials $f^{\mathfrak{s}}_+(z) = 1 - z$ and  
\begin{equation}
    f^{\mathfrak{s}}_-(z) = (1 - a_3^{-1}z)(1 - a_4^{-1}z)(1 - a_5^{-1}z) = 1 + M_1 z + M_2 z^2 + M_3 z^3.
\end{equation}
Just as in the situation of \S\ref{sec computations g=1 containing cluster}, we have $\mathfrak{t}^{\mathfrak{s}}_+(b)=\truncate{b}$ and $b_0(\mathfrak{t}^{\mathfrak{s}}_+) = 2v(2)$.  Now applying the computations in \S\ref{sec computations sufficiently odd deg3}, we have a totally odd part-square decomposition $f_-^{\mathfrak{s},0} = [q_-]^2 + \rho_-$ where (for some choice of square roots of $M_2$) we have 
$$\rho_-(z) = (M_1 - 2\sqrt{M_2})z + M_3 z^3.$$
It is immediate to see that we have $v(M_3)=0$, and the formula $\mathfrak{t}^{\mathfrak{s}}_-(b) = \min\{3b, b + w,2v(2)\}$ follows, from which we get $b_0(\mathfrak{t}_-^{\mathfrak{s}})=\max\{ \frac{2}{3}v(2), 2v(2) - w\}$.
Now, using the formula $B_{f,\mathfrak{s}} = b_0(\mathfrak{t}_+^{\mathfrak{s}}) + b_0(\mathfrak{t}_-^{\mathfrak{s}})$ from \Cref{prop depth threshold}, we get $B_{f,\mathfrak{s}} = \max\{\frac{8}{3}v(2),4v(2)-w\}$.
If we assume that $\mathfrak{s}$ is viable (i.e., if we are in the setting $m>\max\{\frac{8}{3}v(2),4v(2)-w\}$ treated by \Cref{thm g=2}(a),(b)), the components of $\SF{\Yrst}$ corresponding to $D_\pm := D_{0,b_\pm(\mathfrak{s})}$ intersect at $2$ nodes (see \Cref{prop viable correspondence}).  Moreover, it follows from \Cref{lemma slopes of t} that $\lambda_+(\mathfrak{s}) = 1$, and it is easily checked from the valuations of the coefficients of $\rho_-$ that we have $\lambda_-(\mathfrak{s}) = 3$ if we moreover have $w \geq \frac{4}{3}v(2)$ (i.e.\ case (a)), whereas  $\lambda_-(\mathfrak{s}) = 1$ if $w< \frac{4}{3}v(2)$ (i.e.\ case (b)).  Now applying \Cref{prop genus of components}(b)(i),(ii) shows us that in case (a) of \Cref{thm g=2}, the components of $\SF{\Yrst}$ corresponding to $D_+$ and $D_-$ have abelian ranks $0$ and $1$ respectively if $m > \frac{8}{3}v(2)$.

In particular, we have proved the formula for $B_{f,\mathfrak{s}}$ at the start of the statement of \Cref{thm g=2}, as well as part (a) of the theorem and part of the statement of part (b). We also note for below use that from the formulas $\mathfrak{t}^{\mathfrak{s}}_+(b) = \truncate{b}$ and $\mathfrak{t}^{\mathfrak{s}}_-(b) = \min\{3b, b + w,2v(2)\}$, by \Cref{prop t^R is min of t^s and t^(R-s)}, we get the formula 
\begin{equation} \label{eq t_f formula g=2 with a 2-cluster}
    \mathfrak{t}^\RR(D_{0,b}) = \min\{3b, b + w, -b + m, 2v(2)\} \qquad \mathrm{for} \ b \in [d_-(\mathfrak{s}), d_+(\mathfrak{s})]=[0,m],
\end{equation}
as displayed in \Cref{fig cases g2}.

\subsubsection{Finding a center and a depth of a valid disc not containing any roots} \label{sec computations g=2 not containing cluster}

We retain all of the above notation and assumptions, except that we now allow the possibility that $m = 0$ (so that there is no even-cardinality cluster), and we set out to find and chararcterize the valid discs $D$ associated to $Y$ 
which either are linked to no cluster or are the unique valid disc linked to $\mathfrak{s}$. 

Below we will need a lemma to treat situations where $w = 0$ (which is possible only under the hypotheses of \Cref{thm g=2}(b),(c)).

\begin{lemma} \label{lemma w = 0}
    With the notation and hypotheses of $\Cref{thm g=2}$, suppose that we have $m>0$ and $w = 0$. Then there is a valid disc $D$ containing no roots of $f$ and which, for all $\alpha\in D$, satisfies
    \begin{equation}
        v(\alpha)=v(\alpha-a_2)=0, \qquad v(\alpha-a_3)=v(\alpha-a_4)=v(\alpha-a_5)=m'.
    \end{equation}
    Moreover, $D$ contributes a component of abelian rank $1$ to the special fiber $\SF{\Yrst}$ that meets the rest of $\SF{\Yrst}$ at $1$ node.
    \begin{proof}
        Let $D':=D_{0,0}$. The hypothesis $w = 0$, by definition of $w$, is equivalent to $v(M_1) = 0$, and one checks straightforwardly from formulas that it implies that $f$ has unit cubic coefficient. Moreover, the polynomial $f$ has unit quintic coefficient (because it is monic), while the presence of the cardinality-$2$ cluster $\mathfrak{s}$ implies that the linear term of $f$ has positive valuation. Hence, the decomposition $f=0^2+f$ is good at $D'$.  Moreover, when $m'=0$, by looking at the roots of $\bar{f}'$ we see that the inseparable curve $\SF{\YY_{D'}}\to \SF{\XX_{D'}}$ is singular exactly over $\overline{x}=0$ and over a second point $P$ to which none of the elements of $\Rinfty$ reduce, with $\mu(\XX_{D'},0)=\mu(\XX_{D'},P)=2$ (see \S\ref{sec models hyperelliptic inseparable}).  Now applying \Cref{prop partition rank} (combined with \Cref{prop relative position smooth models line}), we get the desired statement when $m'=0$.
        When we have $m'>0$, we instead let $D' = D_{a_i,m'}$ for $i=3,4,5$, and letting $\gamma'$ be an element of valuation $m'$, one easily sees that any normalized reduction of $f_{a_i,\gamma'}$ has no quintic term but does have nonzero linear and cubic terms. It follows that $\SF{\YY_{D'}}\to \SF{\XX_{D'}}$ is singular exactly over $\infty$ and over a second point $P$ to which none of the elements of $\Rinfty$ reduce, with $\mu(\XX_{D'},0)=\mu(\XX_{D'},P)=2$. Again, the desired statement follows via \Cref{prop partition rank}.
    \end{proof}
\end{lemma}

In the case treated by \Cref{thm g=2}(a), where $m > \frac{8}{3}v(2)$ and $w \geq \frac{4}{3}v(2)$, we have seen that there are $2$ valid discs linked to $\mathfrak{s}$: they have already been found and determined to contribute $1$ to the abelian rank and $1$ to the toric rank of $\SF{\Yrst}$; since we have $g = 2$ (so that the sum of the ranks must equal $2$; see \S\ref{sec preliminaries abelian toric unipotent}) and valid discs not linked to any cluster correspond to components of positive abelian rank by \Cref{cor crushed positive abelian rank}, it is clear that there is no valid disc which does not contain a root of $f$ or which is the unique one linked to $\mathfrak{s}$.  We therefore assume that the hypothesis of \Cref{thm g=2}(a) does not hold.

Suppose that we have $m>0$ and $w < \min\{\frac{1}{2}m, \frac{4}{3}v(2)\}$ (as is true for the cases treated by \Cref{thm g=2}(b),(c)); we will show that there is a valid disc $D_{\alpha_1,b_1}$ containing no root of $f$ such that $\alpha_1$ satisfies
\begin{equation}
    \label{eq alpha1}
    v(\alpha_1)=v(\alpha_1-a_2)=\frac{1}{2}w, \text{ and  } v(\alpha_1-a_3)=v(\alpha_1-a_4)=v(\alpha_1-a_5)=m'
\end{equation}
and which contributes a component of abelian rank 1 to $\SF{\Yrst}$. If $w = 0$, this follows from \Cref{lemma w = 0}.  We therefore assume for the rest of this paragraph that $w > 0$.  Then we have that $c \mapsto \mathfrak{t}^\RR(D_{0,c})$ is not differentiable at the input $c = \frac{1}{2}w$ with left and right derivatives equal to $3$ and $1$ respectively and that $\mathfrak{t}^\RR(D_{0,\frac{1}{2}w}) = \frac{3}{2}w < 2v(2)$.  Therefore, by \Cref{lemma bifurcation valid discs} (and \Cref{rmk bifurcation valid discs}), there is a center $\alpha_1 \in \bar{K}$ such that $v(\alpha_1) = \frac{1}{2}w$ and $D_{\alpha_1,b_1}$ is a valid disc which is not linked to any cluster for some $b_1>\frac{1}{2}w$.  In fact, we know from the left and right derivatives and applying \Cref{prop partition rank}(b) to the disc $D_{0,\frac{1}{2}w}$ that the abelian rank of the corresponding component of $\SF{\Yrst}$ must be $\frac{1}{2}(3 - 1) = 1$.  When we also have $w > 4v(2) - m$ (so that we are in the case treated by \Cref{thm g=2}(b)), we have already found that $\mathfrak{s}$ is viable and contributes $1$ to the toric rank of $\SF{\Yrst}$ and so we have found all of the valid discs as the ranks must add up to $g = 2$ (see \S\ref{sec preliminaries abelian toric unipotent}). 

In the case treated by \Cref{thm g=2}(c) where we moreover have $w \le 4v(2) - m$ and $w < \frac{1}{2}m$, the function $c \mapsto \mathfrak{t}^\RR(D_{0,c})$ is also not differentiable at the input $c = \frac{1}{2}(m - w)$ with left and right derivative equal to $1$ and $-1$ respectively, and we have $\mathfrak{t}^\RR(D_{0,\frac{1}{2}(m - w)}) =  \frac{1}{2}(m + w) \leq 2v(2)$.  Therefore, by applying \Cref{lemma bifurcation valid discs} and \Cref{rmk bifurcation valid discs} when $w<4v(2)-m$, and by observing that $b_-(\mathfrak{s}) = b_+(\mathfrak{s}) =  \frac{1}{2}(m - w)$ when $w=4v(2)-m$, we have that there is a center $\alpha_2$ satisfying
\begin{equation}
    \label{eq alpha2}
    v(\alpha_2)=v(\alpha_2-a_2)=\frac{1}{2}(m - w)>0, \text{ and  } v(\alpha_2-a_3)=v(\alpha_2-a_4)=v(\alpha_2-a_5)=0,
\end{equation}
and such that $D_{\alpha_2,b_2}$ is a valid disc which is not linked to any cluster (resp.\ is the unique valid disc linked to $\mathfrak{s}$) if $w < 4v(2) - m$ (resp.\ $w = 4v(2) - m$), for some $b_2 \geq  \frac{1}{2}(m - w)$.  By knowing the left and right derivatives and a similar application of \Cref{prop partition rank}(b) to the disc $D_{0,\frac{1}{2}(m-w)}$ (resp.\ using \Cref{prop genus of components}(b)(iii)), the abelian rank of the corresponding component of $\SF{\Yrst}$ must be $\frac{1}{2}(1 - (-1)) = 1$.  Since the $2$ valid discs we have found each contribute $1$ to the abelian rank of $\SF{\Yrst}$, we have again found all of the valid discs as $g = 2$ (see \S\ref{sec preliminaries abelian toric unipotent}).

Let us now address the case treated by \Cref{thm g=2}(d), in which we instead have $m=0$ or $0<m \leq \min\{2w, \frac{8}{3}v(2)\}$ and that the function $b \mapsto \mathfrak{t}^\RR(D_{0,b})$ is not differentiable at the input $b = \frac{1}{4}m$, with  $\mathfrak{t}^\RR(D_{0,\frac{1}{4}m}) = \frac{3}{4}m \leq 2v(2)$. We claim that, in this case, there exist two possibly coinciding valid discs $D_{\alpha_1,b_1}$ and $D_{\alpha_2,b_2}$ such that $\alpha_1$ and $\alpha_2$ satisfy the conditions
\begin{equation}
    \label{eq alpha1'}
    v(\alpha_1)=v(\alpha_1-a_2)=\frac{1}{4}m, \text{ and  } v(\alpha_1-a_3)=v(\alpha_1-a_4)=v(\alpha_1-a_5)=0,
\end{equation}
\begin{equation}
    \label{eq alpha2'}
    v(\alpha_2)=v(\alpha_2-a_2)=\frac{1}{4}m, \text{ and  } v(\alpha_2-a_3)=v(\alpha_2-a_4)=v(\alpha_2-a_5)=m'
\end{equation}
respectively, and such that they are linked to no cluster (when $m<\frac{8}{3}v(2)$), or they both coincide with the unique disc linked to $\mathfrak{s}$ (when $m=\frac{8}{3}v(2)$). When $m=\frac{8}{3}v(2)$, it is straightforward to see that the unique valid disc $D_{0,\frac{2}{3}v(2)}=D_{\alpha_1,b_1}=D_{\alpha_2,b_2}$ linked to $\mathfrak{s}$ contributes a component of abelian rank 2 (see \Cref{prop genus of components}(b)(iii)), which proves the claim directly above. In the $m<\frac{8}{3}v(2)$ case, no valid disc is linked to $\mathfrak{s}$, and hence all valid discs contain no roots of $f$. Since the genus is $g=2$ and the toric rank of $\SF{\Yrst}$ is 0 by \Cref{thm toric rank}, in light of \Cref{cor crushed positive abelian rank} we must have either two distinct valid discs each contributing a component of abelian rank 1 or a unique valid disc contributing a component of abelian rank 2. The fact that these valid discs have the form prescribed by the claim directly above can be easily proved by studying the singularities of $\SF{\YY_{D'}}\to \SF{\XX_{D'}}$ for the disc $D'=D_{0,\frac{m}{4}}$ (and also for the disc $D'=D_{0,m'}$ when $m'>0$) and then exploiting \Cref{prop inseparable tfae}.

We now set out to find a formula for $b_i$ for the valid discs $D_{\alpha_i,b_i}$ that we have found in all cases discussed above, where $i=1,2$ and the centers $\alpha_i$ satisfy (\ref{eq alpha1})-(\ref{eq alpha2}) or  (\ref{eq alpha1'})-(\ref{eq alpha2'}). We also want to know whether, in the case that $m=0$ or $0<m < \min\{2w,\frac{8}{3}v(2)\}$, the valid discs $D_{\alpha_1,b_1}$ and $D_{\alpha_2,b_2}$ give rise to two distinct components of $\SF{\Yrst}$ of abelian rank $1$ or they coincide and provide a unique component of abelian rank 2. 

Let us first remark that $\alpha_i$ can always be chosen to be a root of $F$ thanks to \Cref{cor centers}. Now, using \Cref{lemma crushed components valuations} and the definition of $\kappa(\alpha)$, letting $\nu=2v(\alpha_i)+3v(\alpha_i-a_3)$, we compute the formula 
\begin{equation} \label{eq t^R g=2 crushed}
    \mathfrak{t}^\RR(D_{\alpha_i,b}) = \truncate{\vfun_{\rho}(D_{\alpha_i,b}) - \vfun_{f}(D_{\alpha_i,b})} = \min\{3b + \kappa(\alpha_i)-\nu, 5b-\nu,2v(2)\}.
\end{equation}

We know from \Cref{thm summary depths valid discs}(b) that the depth $b_i$ is the first input at which  $b\mapsto \mathfrak{t}^\RR(D_{\alpha_i,b})$ attains $2v(2)$.  Meanwhile, it can be calculated using \Cref{prop genus of components}(a) and the formula in (\ref{eq t^R g=2 crushed}) that the component of $\SF{\Yrst}$ corresponding to $D_{\alpha_i,b_i}$ has abelian rank $1$ (resp.\ $2$) if we have $3b_i + \kappa(\alpha_i) - \nu < 5b_i - \nu$ (resp.\ $3b_i+\kappa(\alpha_i) - \nu \geq 5{b_i} - \nu$), or equivalently, if we have $\kappa(\alpha_i) < 2b_i$ (resp.\ $\kappa(\alpha_i) \geq 2b_i$).  We therefore have have the following two cases:
\begin{enumerate}
    \item the abelian rank of the component of $\SF{\Yrst}$ corresponding to $D_{\alpha_i,b_i}$ equals $1$; we then have 
    \begin{align}
    \begin{split}
        \mathfrak{t}^\RR(D_{\alpha_i,b_i}) = 3b_i+\kappa(\alpha_i)-\nu = 2v(2) &\implies b_i = \frac{1}{3}(\nu - \kappa(\alpha_i) + 2v(2)), \\ \mathrm{and} \ \ \ \kappa(\alpha_i) < 2b_i &\implies \kappa(\alpha_i) < \frac{2}{5}(\nu + 2v(2)); \qquad \mathrm{or} 
    \end{split}
    \end{align}
    \item the abelian rank of the component of $\SF{\Yrst}$ corresponding to $D_{\alpha_i,b_i}$ equals $2$; we then have 
    \begin{align}
    \begin{split}
        \mathfrak{t}^\RR(D_{\alpha_i,b_i}) = 5b_i-\nu = 2v(2) &\implies b_i = \frac{1}{5}(\nu + 2v(2)), \\ \mathrm{and} \ \ \ \kappa(\alpha_i) \geq 2b_i &\implies \kappa(\alpha_i) \geq \frac{2}{5}(\nu + 2v(2)).
    \end{split}
    \end{align}
\end{enumerate}
Now, on substituting the formulas for $\nu=2v(\alpha_1)+3v(\alpha_1-a_3)$ which we found above on a case-by-case basis, we get the following outcomes:
\begin{enumerate}
    \item if we have $w < \min\{\frac{1}{2}m, \frac{4}{3}v(2)\}$, then the valid disc $D_{\alpha_1,b_1}$ that we have found (corresponding to a component of abelian rank $1$), with $\alpha_1$ satisfying (\ref{eq alpha1}), has depth $b_1 = m'+\frac{1}{3}(w - \kappa(\alpha_1) + 2v(2))$;
    \item if moreover we have $w < \min\{\frac{1}{2}m, 4v(2) - m\}$, then the second valid disc $D_{\alpha_2,b_2}$ that we have found (also corresponding to a component of abelian rank $1$), with $\alpha_2$ satisfying (\ref{eq alpha2}), has depth $b_2 = \frac{1}{3}(m - w - \kappa(\alpha_2) + 2v(2))$; and 
    \item if instead we have $m=0$ or $0<m < \min\{2w, \frac{8}{3}v(2)\}$, then the valid discs $D_{\alpha_i,b_i}$ for $i=1,2$ that we have found, with $\alpha_i$ satisfying (\ref{eq alpha1'})-(\ref{eq alpha2'}), may be distinct and contribute each a component of abelian rank 1, or they may coincide and give a unique component of abelian rank 2.  The first case occurs if we have $\kappa(\alpha_1) < \frac{2}{5}(\frac{1}{2}m + 2v(2))$; in this case, we get $b_1 = \frac{1}{3}(\frac{1}{2}m - \kappa(\alpha_1) + 2v(2))$ and $b_2=m'+b_1$.  The second case (which can only happen if $m'=0$) occurs when $\kappa(\alpha_1) \ge \frac{2}{5}(\frac{1}{2}m + 2v(2))$, and we get $b_1 = b_2 = \frac{1}{5}(\frac{1}{2}m + 2v(2))$.
\end{enumerate}
This finishes the proof of \Cref{thm g=2}.

\bibliographystyle{plain}
\bibliography{bibfile}

\end{document}